\documentclass[11pt]{article}
\usepackage{fullpage}
\usepackage{algorithm}
\usepackage{algpseudocode}
\usepackage{amsmath,amsthm,amsfonts}
\usepackage{amsmath}
\usepackage{color}
\usepackage{mathrsfs}
\usepackage{enumitem}
\usepackage{bm}
\usepackage{multirow}
\usepackage{booktabs}
\usepackage{makecell}
\usepackage{graphicx}
\usepackage{subfigure}
\usepackage[font={small}]{caption}
\usepackage{comment}
\usepackage{cases}
\usepackage{epsfig}

\usepackage{tikz}
\usetikzlibrary{shapes.geometric, arrows, positioning}

\tikzset{
   mybox/.style  = {draw, rectangle, minimum width=3.4cm, minimum height=0.8cm, text centered, text width=3.2cm,   
  font=\normalsize},
  box/.style  = {draw, rectangle, minimum width=3cm, minimum height=0.8cm, text centered, text width=3.6cm,   
  font=\normalsize},
   myarrow/.style = {line width=0.2pt, draw=black, -triangle 60, postaction={draw, line width=0.2pt, shorten >=10pt,-}}
}

\tikzstyle{arrow} = [->, >=stealth, -triangle 60]

\usepackage[colorlinks, linkcolor=black, citecolor=blue, urlcolor=blue]{hyperref}

\numberwithin{equation}{section}

\newtheorem{thm}{Theorem}

\newtheorem{thm2}{Theorem}[section]
\newtheorem{lem}[thm2]{Lemma}
\newtheorem{coro}[thm2]{Corollary}
\newtheorem{prop}[thm2]{Proposition}

\theoremstyle{remark}
\newtheorem{rem}[thm2]{Remark}

\newcommand\ee{\mathrm{e}}
\newcommand\dd{\mathrm{d}}

\newcommand{\goto}{\rightarrow}

\renewcommand{\comment}[1]{}

\definecolor{wjs}{RGB}{0,0,255}

\allowdisplaybreaks

\title{Understanding the Acceleration Phenomenon via High-Resolution Differential Equations}

\begin{document}

\author{Bin~Shi\thanks{Florida International University, Miami. Email: bshi001@cs.fiu.edu}  \qquad Simon S.~Du\thanks{Carnegie Mellon University, Pittsburgh. Email: ssdu@cs.cmu.edu}  \qquad   Michael I.~Jordan\thanks{University of California, Berkeley. Email: 
jordan@cs.berkeley.edu (correspondence)} \qquad  Weijie J.~Su\thanks{University of Pennsylvania, Philadelphia. Email: suw@wharton.upenn.edu (correspondence)}}

\date{}
\maketitle

\begin{abstract}
Gradient-based optimization algorithms can be studied from the perspective of
limiting ordinary differential equations (ODEs).  Motivated by the fact that
existing ODEs do not distinguish between two fundamentally different 
algorithms---Nesterov's accelerated gradient method for strongly convex 
functions (NAG-\texttt{SC}) and Polyak's heavy-ball method---we study an
alternative limiting process that yields \textit{high-resolution ODEs}.
We show that these ODEs permit a general Lyapunov function framework for
the analysis of convergence in both continuous and discrete time.  We also
show that these ODEs are more accurate surrogates for the underlying algorithms;
in particular, they not only distinguish between NAG-\texttt{SC} and Polyak's 
heavy-ball method, but they allow the identification of a term that we 
refer to as ``gradient correction'' that is present in NAG-\texttt{SC} but 
not in the heavy-ball method and is responsible for the qualitative difference
in convergence of the two methods.  We also use the high-resolution ODE 
framework to study Nesterov's accelerated gradient method for (non-strongly) convex 
functions, uncovering a hitherto unknown result---that NAG-\texttt{C} minimizes
the squared gradient norm at an inverse cubic rate.  Finally, by modifying 
the high-resolution ODE of NAG-\texttt{C}, we obtain a family of new optimization 
methods that are shown to maintain the accelerated convergence rates of
NAG-\texttt{C} for smooth convex functions.
\end{abstract}


\section{Introduction}\label{sec: introduction}

Machine learning has become one of the major application areas for optimization 
algorithms during the past decade.  While there have been many kinds of applications,
to a wide variety of problems, the most prominent applications have involved large-scale
problems in which the objective function is the sum over terms associated with individual
data, such that stochastic gradients can be computed cheaply, while gradients are much 
more expensive and the computation (and/or storage) of Hessians is often infeasible.  
In this setting, simple first-order gradient descent algorithms have become dominant,
and the effort to make these algorithms applicable to a broad range of machine 
learning problems has triggered a flurry of new research in optimization, both 
methodological and theoretical.

We will be considering unconstrained minimization problems,
\begin{equation}\label{eq:convex_p}
\min_{x \in \mathbb{R}^n} ~~ f(x),
\end{equation}
where $f$ is a smooth convex function.  Perhaps the simplest first-order 
method for solving this problem is gradient descent. Taking a fixed step size 
$s$, gradient descent is implemented as the recursive rule
\[
x_{k+1} = x_{k} - s \nabla f(x_k),
\]
given an initial point $x_0$.

As has been known at least since the advent of conjugate gradient algorithms,
improvements to gradient descent can be obtained within a first-order framework
by using the history of past gradients.  Modern research on such extended 
first-order methods arguably dates to Polyak~\cite{polyak1964some,polyakintroduction}, 
whose \emph{heavy-ball method} incorporates a momentum term into the gradient 
step.  This approach allows past gradients to influence the current step, while 
avoiding the complexities of conjugate gradients and permitting a stronger
theoretical analysis.  Explicitly, starting from an initial point 
$x_{0},\; x_{1} \in \mathbb{R}^n$, the heavy-ball method updates the 
iterates according to
\begin{equation}\label{eqn: polyak_heavy_ball}
x_{k + 1} =  x_{k} +  \alpha \left( x_{k} - x_{k - 1} \right) -  s \nabla f( x_{k} ),
\end{equation}
where $\alpha > 0$ is the momentum coefficient.  While the heavy-ball method 
provably attains a faster rate of \emph{local} convergence than gradient 
descent near a minimum of $f$, it does not come with \emph{global} guarantees.
Indeed, \cite{lessard2016analysis} demonstrate that even for strongly convex 
functions the method can fail to converge for some choices of the step size.\footnote{\cite{polyak1964some} considers $s = 4/(\sqrt{L} + \sqrt{\mu})^2$ and $\alpha = ( 1 - \sqrt{\mu s})^{2}$. This momentum coefficient is basically the same as the choice $\alpha = \frac{1 - \sqrt{\mu s}}{1 + \sqrt{\mu s}}$ (adopted starting from Section~\ref{sec:myst-grad-corr}) if $s$ is small. }
 
The next major development in first-order methodology was due to Nesterov, who
discovered a class of \emph{accelerated gradient methods} that have a faster 
global convergence rate than gradient descent~\cite{nesterov1983method,
nesterov2013introductory}.  For a $\mu$-strongly convex objective $f$ with 
$L$-Lipschitz gradients, Nesterov's accelerated gradient method (NAG-\texttt{SC}) 
involves the following pair of update equations:
\begin{equation}\label{eqn: Nesterov_strongly}
\begin{aligned}
         & y_{k + 1} = x_{k} - s \nabla f( x_{k} ) \\
         & x_{k + 1} = y_{k + 1} + \frac{1 - \sqrt{\mu s} }{ 1 + \sqrt{\mu s} } \left( y_{k + 1} - y_{k} \right),
         \end{aligned}
\end{equation}
given an initial point $x_{0} = y_{0} \in \mathbb{R}^n$.  Equivalently, 
NAG-\texttt{SC} can be written in a single-variable form that is similar to 
the heavy-ball method:
\begin{equation}\label{eqn: Nesterov_strongly_single}
x_{k + 1} =  x_{k} + \frac{1 - \sqrt{\mu s} }{ 1+ \sqrt{\mu s} } \left( x_{k} - x_{k - 1} \right) -  s\nabla f( x_{k} ) - \frac{ 1 - \sqrt{\mu s} }{ 1 + \sqrt{\mu s} } \cdot  s \left( \nabla f(x_{k}) - \nabla f(x_{k - 1}) \right),
\end{equation}
starting from $x_{0}$ and $x_{1} = x_{0} - \frac{2 s \nabla f(x_{0})}{1 + \sqrt{\mu s}}$. 
Like the heavy-ball method, NAG-\texttt{SC} blends gradient and momentum contributions 
into its update direction, but defines a specific momentum coefficient 
$\frac{1-\sqrt{\mu s}}{1+\sqrt{\mu s}}$. Nesterov also developed the 
\emph{estimate sequence technique} to prove that NAG-\texttt{SC} achieves 
an accelerated linear convergence rate:
\[
f(x_{k}) - f(x^{\star}) \leq O\left( \left(1 - \sqrt{s\mu} \right)^{k} \right),
\] 
if the step size satisfies $0 < s \leq 1/L$. Moreover, for a (weakly) convex 
objective $f$ with $L$-Lipschitz gradients, Nesterov defined a related accelerated 
gradient method (NAG-\texttt{C}), that takes the following form:
\begin{equation}\label{eqn:nagm-c}
\begin{aligned}  
& y_{k + 1} = x_{k} -s  \nabla f(x_{k})\\
& x_{k + 1} = y_{k + 1} + \frac{k}{k+3} (y_{k + 1} - y_{k}),
\end{aligned}
\end{equation}
with $x_{0} = y_{0} \in \mathbb{R}^n$. The choice of momentum coefficient 
$\frac{k}{k+3}$, which tends to one, is fundamental to the estimate-sequence-based
argument used by Nesterov to establish the following inverse quadratic convergence rate:
\begin{equation}\label{eq:nac_k_2_conv}
f(x_{k}) - f(x^{\star}) \leq  O\left( \frac{1}{sk^2} \right),
\end{equation}
for any step size $s \le 1/L$.  Under an oracle model of optimization complexity, 
the convergence rates achieved by NAG-\texttt{SC} and NAG-\texttt{C} are 
\textit{optimal} for smooth strongly convex functions and smooth convex 
functions, respectively \cite{nemirovsky1983problem}.

\subsection{Gradient Correction: Small but Essential}
\label{sec:myst-grad-corr}
Throughout the present paper, we let $\alpha = \frac{1 - \sqrt{\mu s}}{1 + \sqrt{\mu s}}$ 
and $x_{1} = x_{0}-\frac{2s\nabla f(x_{0})}{1+\sqrt{\mu s}}$ to define a specific
implementation of the heavy-ball method in \eqref{eqn: polyak_heavy_ball}. This choice of the momentum coefficient and the second initial point renders 
the heavy-ball method and NAG-\texttt{SC} identical except for the last (small) term 
in \eqref{eqn: Nesterov_strongly_single}. Despite their close resemblance, however, 
the two methods are in fact fundamentally different, with contrasting convergence 
results (see, for example, \cite{bubeck2015convex}). Notably, the former algorithm
in general only achieves \emph{local} acceleration, while the latter achieves 
acceleration method for all initial values of the iterate \cite{lessard2016analysis}. 
As a numerical illustration, Figure~\ref{fig:nag-c-and-heavy} presents the trajectories
that arise from the two methods when minimizing an ill-conditioned convex quadratic 
function. We see that the heavy-ball method exhibits pronounced oscillations throughout 
the iterations, whereas NAG-\texttt{SC} is monotone in the function value once the 
iteration counter exceeds $50$.

This striking difference between the two methods can \textit{only} be attributed to 
the last term in \eqref{eqn: Nesterov_strongly_single}:
\begin{equation}\label{eq:gradient_cor}
\frac{ 1 - \sqrt{\mu s} }{ 1 + \sqrt{\mu s} } \cdot  s 
\left( \nabla f(x_{k}) - \nabla f(x_{k - 1}) \right),
\end{equation}
which we refer to henceforth as the \textit{gradient correction}\footnote{The 
gradient correction for NAG-\texttt{C} is $\frac{k}{k+3} \cdot s 
(\nabla f(x_k) - \nabla f(x_{k-1}))$, as seen from the single-variable 
form of NAG-\texttt{C}: $x_{k+1} = x_k + \frac{k}{k+3}(x_k - x_{k-1}) 
- s \nabla f(x_k) - \frac{k}{k+3} \cdot s (\nabla f(x_k) - \nabla f(x_{k-1}))$.}. 
This term corrects the update direction in NAG-\texttt{SC} by contrasting the 
gradients at consecutive iterates. Although an essential ingredient in 
NAG-\texttt{SC}, the effect of the gradient correction is unclear from the 
vantage point of the estimate-sequence technique used in Nesterov's proof. 
Accordingly, while the estimate-sequence technique delivers a proof of acceleration
for NAG-\texttt{SC}, it does not explain why the absence of the gradient correction 
prevents the heavy-ball method from achieving acceleration for strongly convex functions.

\begin{figure}[htp!]
\centering
\includegraphics[width=3.8in]{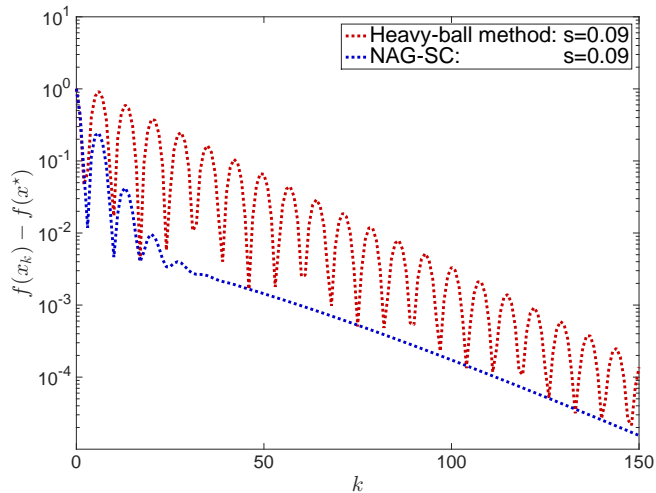}
\caption{A numerical comparison between NAG-\texttt{SC} and heavy-ball method. The objective function (ill-conditioned $\mu/L \ll 1$) is $f(x_{1}, x_{2}) = 5 \times 10^{-3} x_1^{2} + x_2^{2} $, with the initial iterate $(1, 1)$.} 
\label{fig:nag-c-and-heavy}
\end{figure}

A recent line of research has taken a different point of view on the theoretical
analysis of acceleration, formulating the problem in continuous time and obtaining
algorithms via discretization~\cite{su2014differential, krichene2015accelerated,
wibisono2016variational}).  This can be done by taking continuous-time limits
of existing algorithms to obtain ordinary differential equations (ODEs) that
can be analyzed using the rich toolbox associated with ODEs, including Lyapunov
functions\footnote{One can think of the Lyapunov function as a generalization of the idea of the energy of
a system. Then the method studies stability by looking at the rate of change of this measure of energy.}.  For instance, \cite{su2016differential} shows that 
\begin{equation}\label{eqn:ode_old_nagmc}
\ddot{X}(t) + \frac{3}{t} \dot{X}(t) + \nabla f(X(t)) = 0,
\end{equation}
with initial conditions $X(0) = x_0$ and $\dot X(0) = 0$, is the exact limit of 
NAG-\texttt{C}~\eqref{eqn:nagm-c} by taking the step size $s \rightarrow 0$. 
Alternatively, the starting point may be a Lagrangian or Hamiltonian 
framework~\cite{wibisono2016variational}.  In either case, the continuous-time
perspective not only provides analytical power and intuition, but it also provides 
design tools for new accelerated algorithms.

Unfortunately, existing continuous-time formulations of acceleration stop 
short of differentiating between the heavy-ball method and NAG-\texttt{SC}.
In particular, these two methods have the \textit{same} limiting ODE (see, 
for example, \cite{wilson2016lyapunov}):
\begin{equation}\label{eqn:ode_old_nagsc_hb}
\ddot{X}(t) + 2 \sqrt{\mu} \dot{X}(t) + \nabla f(X(t))= 0,
\end{equation}
and, as a consequence, this ODE does not provide any insight into the stronger
convergence results for NAG-\texttt{SC} as compared to the heavy-ball method. 
As will be shown in Section \ref{sec:techniques}, this is because the gradient 
correction $\frac{ 1 - \sqrt{\mu s} }{ 1 + \sqrt{\mu s} } s 
\left( \nabla f(x_{k}) - \nabla f(x_{k - 1}) \right) = O(s^{1.5})$ 
is an order-of-magnitude smaller than the other terms in 
\eqref{eqn: Nesterov_strongly_single} if $s = o(1)$. Consequently, 
the gradient correction is \textit{not} reflected in the \textit{low-resolution} 
ODE \eqref{eqn:ode_old_nagsc_hb} associated with NAG-\texttt{SC}, which is derived 
by simply taking $s \rightarrow 0$ in both \eqref{eqn: polyak_heavy_ball}
and \eqref{eqn: Nesterov_strongly_single}.


\subsection{Overview of Contributions}
\label{subsec: contribution}

Just as there is not a singled preferred way to discretize a differential equation, 
there is not a single preferred way to take a continuous-time limit of a difference 
equation.  Inspired by dimensional-analysis strategies widely used in fluid mechanics 
in which physical phenomena are investigated at multiple scales via the inclusion 
of various orders of perturbations~\cite{pedlosky2013geophysical}, we propose to 
incorporate $O(\sqrt{s})$ terms into the limiting process for obtaining an ODE, 
including the (Hessian-driven) gradient correction $\sqrt{s}  \nabla^{2} 
f(X) \dot{X}$ in \eqref{eq:gradient_cor}.  This will yield \emph{high-resolution 
ODEs} that differentiate between the NAG methods and the heavy-ball method. 

We list the high-resolution ODEs that we derive in the paper here\footnote{We note
that the form of the initial conditions is fixed for each ODE throughout the paper. 
For example, while $x_0$ is arbitrary, $X(0)$ and $\dot X(0)$ must always be equal 
to $x_0$ and $-2\sqrt{s} f(x_0)/(1+\sqrt{\mu s})$ respectively in the high-resolution 
ODE of the heavy-ball method. This is in accordance with the choice of $\alpha = 
\frac{1 - \sqrt{\mu s}}{1 + \sqrt{\mu s}}$ and 
$x_{1} = x_{0}-\frac{2s\nabla f(x_{0})}{1+\sqrt{\mu s}}$.}:
\begin{itemize}
\item[(a)]
The high-resolution ODE for the heavy-ball method~\eqref{eqn: polyak_heavy_ball}: 
\begin{align}\label{eqn: heavy_ball_first}
\ddot{X}(t) +2 \sqrt{\mu}  \dot{X}(t)  +  (1 + \sqrt{\mu s} ) \nabla f(X(t)) = 0,
\end{align}
with $X(0) = x_{0}$ and $\dot{X}(0) =  - \frac{2\sqrt{s} \nabla f(x_{0})}{1 + \sqrt{\mu s}}$. 

\item[(b)]
The high-resolution ODE for NAG-\texttt{SC}~\eqref{eqn: Nesterov_strongly}:
\begin{align}\label{eqn: nag-sc_first}
\ddot{X}(t) + 2 \sqrt{\mu} \dot{X}(t) + \sqrt{s} \nabla^{2} f(X(t)) \dot{X}(t) + \left( 1 + \sqrt{\mu s} \right)\nabla f(X(t))= 0,
\end{align}
with $X(0) = x_{0}$ and $\dot{X}(0) = - \frac{2\sqrt{s} \nabla f(x_{0})}{1 + \sqrt{\mu s}}$.

\item[(c)] 
The high-resolution ODE for NAG-\texttt{C}~\eqref{eqn:nagm-c}:
\begin{equation}\label{eqn: nag-c_first}
\ddot{X}(t) + \frac{3}{t} \dot{X}(t) + \sqrt{s} \nabla^{2} f(X(t)) \dot{X}(t) + \left(1 + \frac{3\sqrt{s}}{2t}\right) \nabla f(X(t)) = 0
\end{equation}
for $t \ge 3\sqrt{s}/2$, with $X(3\sqrt{s}/2) = x_{0}$ and $\dot{X}(3\sqrt{s}/2) = -\sqrt{s} \nabla f(x_{0})$.

\end{itemize}

High-resolution ODEs are more accurate continuous-time counterparts for the 
corresponding discrete algorithms than low-resolution ODEs, thus allowing 
for a better characterization of the accelerated methods. This is illustrated 
in Figure~\ref{fig: ODE_algorithm}, which presents trajectories and convergence of the discrete 
methods, and the low- and high-resolution ODEs. For both NAGs, the high-resolution 
ODEs are in much better agreement with the discrete methods than the low-resolution 
ODEs\footnote{Note that for the heavy-ball method, the trajectories of the 
high-resolution ODE and the low-resolution ODE are almost identical.}. Moreover, for NAG-\texttt{SC}, its high-resolution ODE captures the non-oscillation pattern while the low-resolution ODE does not.

\begin{figure}[t]
\begin{minipage}[t]{0.5\linewidth}
\centering
\includegraphics[width=3.2in]{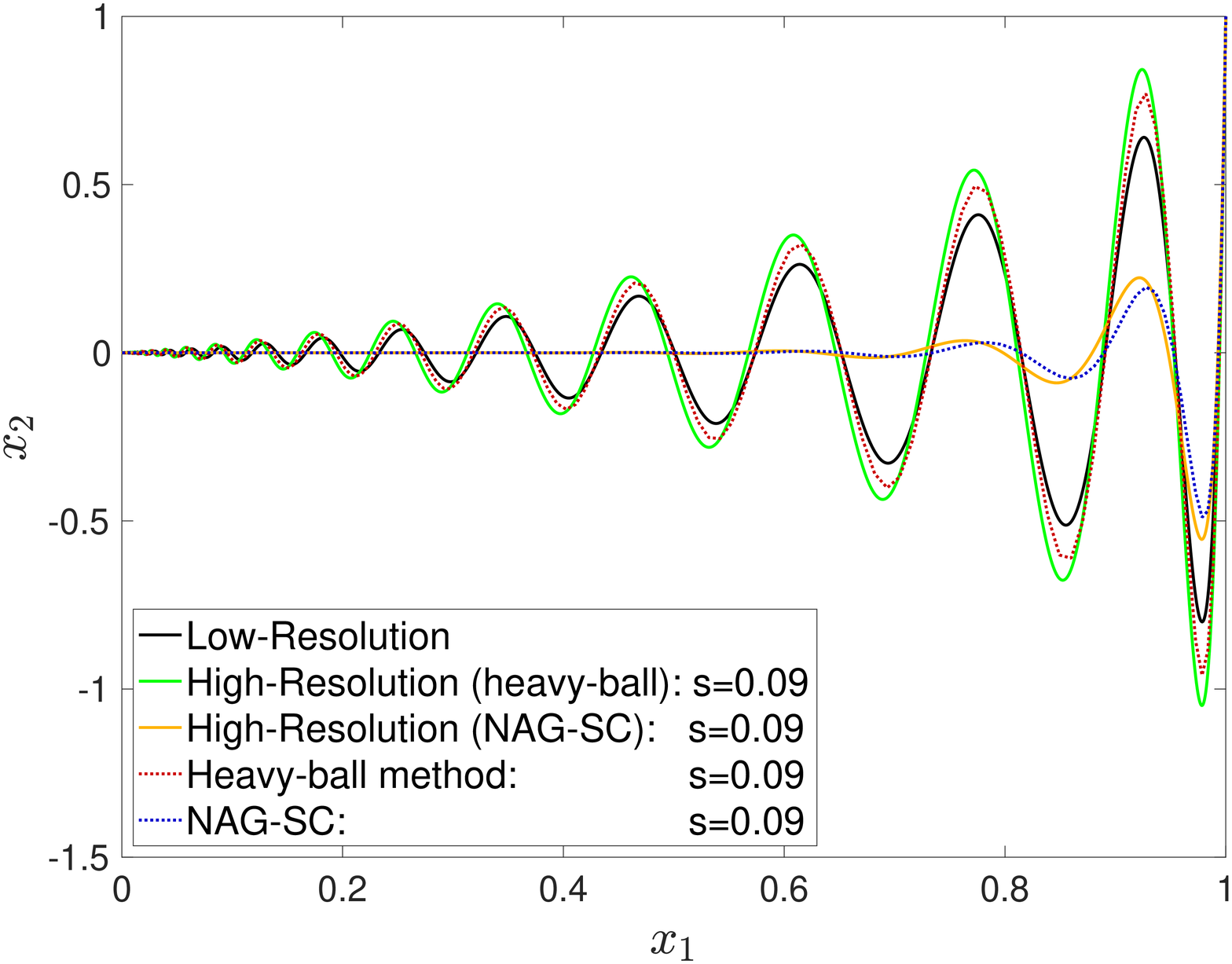}
\end{minipage}
\begin{minipage}[t]{0.5\linewidth}
\centering
\includegraphics[width=3.2in]{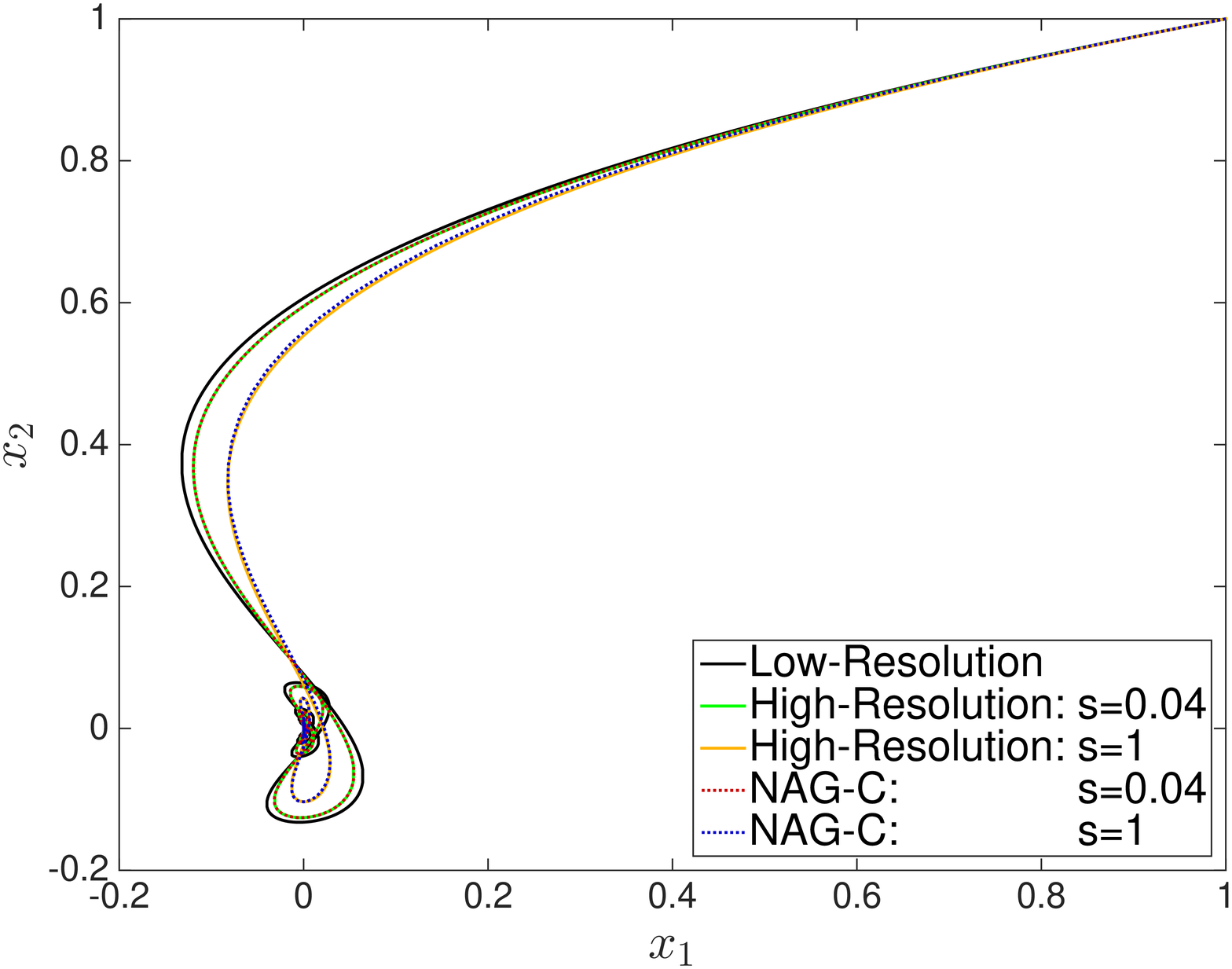}
\end{minipage}
\begin{minipage}[t]{0.5\linewidth}
\centering
\includegraphics[width=3.2in]{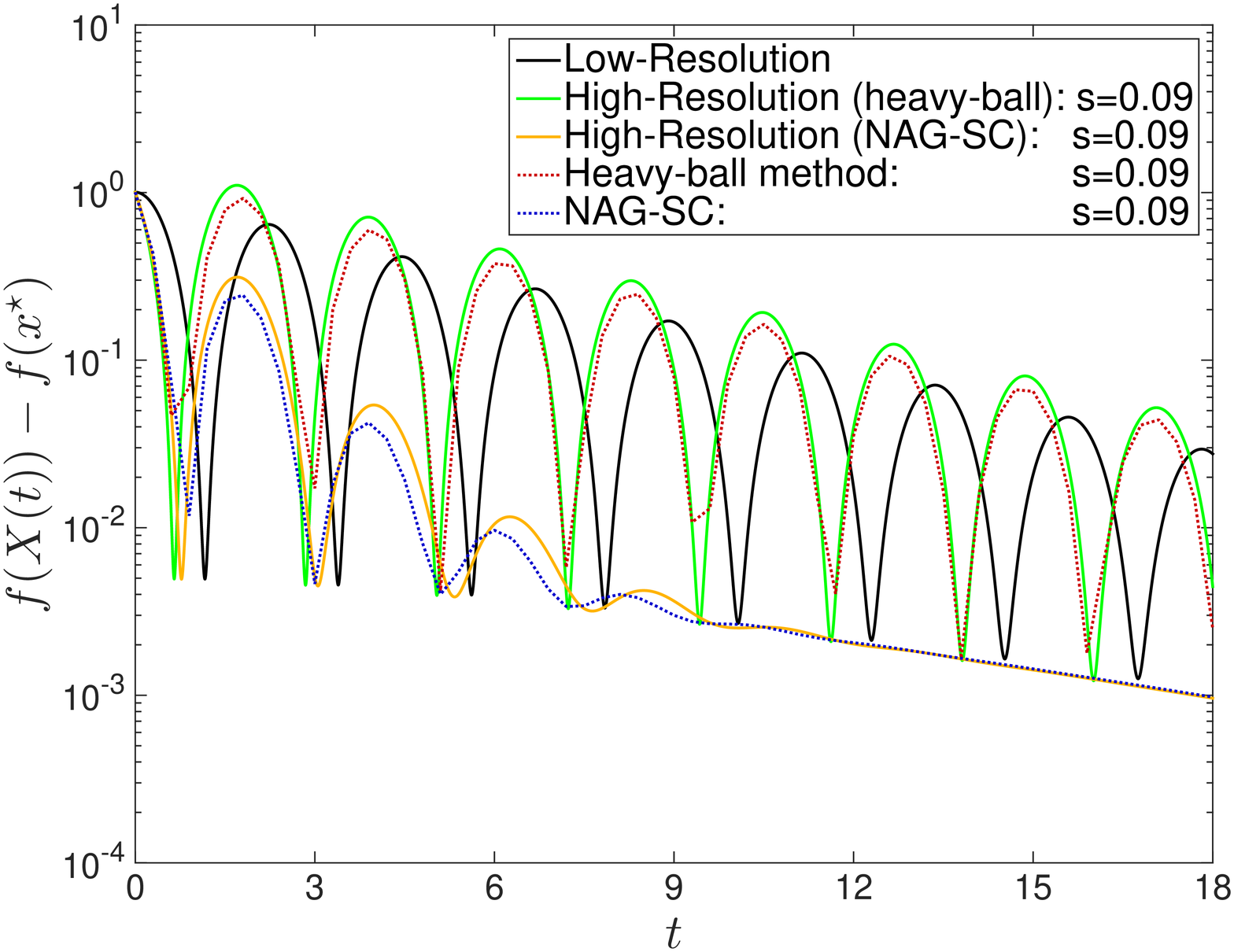}
\end{minipage}
\begin{minipage}[t]{0.5\linewidth}
\centering
\includegraphics[width=3.2in]{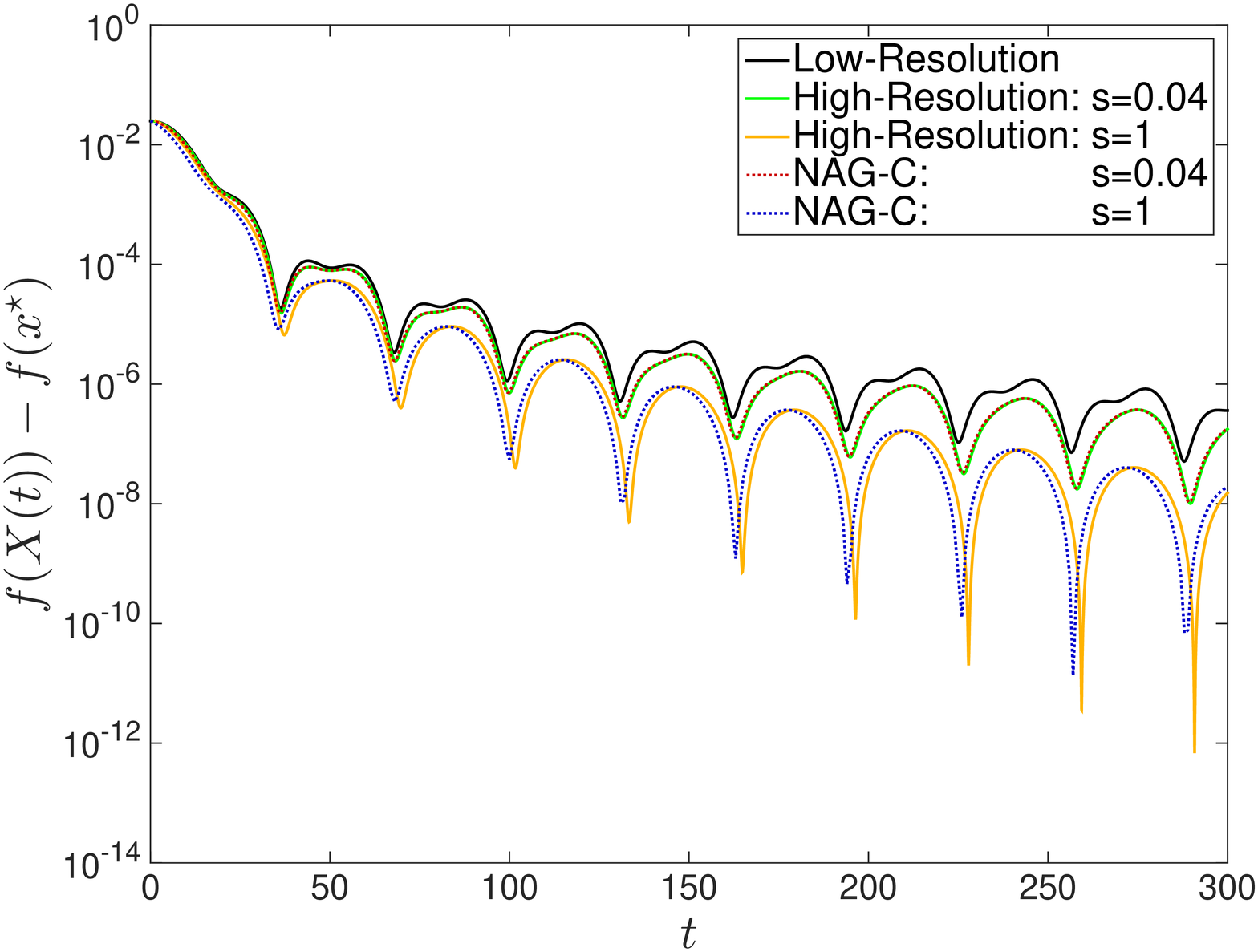}
\end{minipage}
\caption{Top left and bottom left: trajectories and errors of NAG-\texttt{SC} and the heavy-ball method for minimizing $f(x_{1}, x_{2}) = 5 \times 10^{-3} x_1^{2} + x_2^{2}$, from the initial value $(1, 1)$, the same setting as Figure~\ref{fig:nag-c-and-heavy}. Top right and bottom right: trajectories and errors of NAG-\texttt{C} for minimizing $f(x_{1}, x_{2}) = 2 \times 10^{-2}x_{1}^{2} + 5 \times 10^{-3} x_{2}^{2}$, from the initial value $(1, 1)$. For the two bottom plots, we use the identification $t = k\sqrt{s}$ between time and iterations for the x-axis. 
} 
\label{fig: ODE_algorithm}
\end{figure}

The three new ODEs include $O(\sqrt{s})$ terms that are not present in the
corresponding low-resolution ODEs (compare, for example, \eqref{eqn: nag-c_first} 
and \eqref{eqn:ode_old_nagmc}). Note also that if we let $s \goto 0$, each 
high-resolution ODE reduces to its low-resolution counterpart.  Thus, the 
difference between the heavy-ball method and NAG-\texttt{SC} is reflected 
only in their high-resolution ODEs: the gradient correction \eqref{eq:gradient_cor} 
of NAG-\texttt{SC} is preserved only in its high-resolution ODE in the form 
$\sqrt{s} \nabla^{2} f(X(t)) \dot{X}(t)$.  This term, which we refer to as 
the (Hessian-driven) gradient correction, is connected with the discrete 
gradient correction by the approximate identity:
\[
\frac{ 1 - \sqrt{\mu s} }{ 1 + \sqrt{\mu s} } \cdot s \left( \nabla f(x_{k}) - \nabla 
f(x_{k - 1}) \right) \approx s \nabla^2 f(x_{k})(x_{k} - x_{k - 1}) 
\approx s^{\frac32} \nabla^2 f(X(t)) \dot X(t)
\]
for small $s$, with the identification $t = k \sqrt{s}$. The gradient correction 
$\sqrt{s} \nabla^{2} f(X) \dot{X}$ in NAG-\texttt{C} arises in the same 
fashion\footnote{Henceforth, the dependence of $X$ on $t$ is suppressed when 
clear from the context.}. Interestingly, although both NAGs are first-order 
methods, their gradient corrections brings in second-order information from 
the objective function.  

Despite being small, the gradient correction has a fundamental effect on the 
behavior of both NAGs, and this effect is revealed by inspection of the 
high-resolution ODEs.  We provide two illustrations of this.

\begin{itemize}

\item \textbf{Effect of the gradient correction in acceleration.} 
Viewing the coefficient of $\dot X$ as a damping ratio, the ratio 
$2 \sqrt{\mu} + \sqrt{s} \nabla^{2} f(X)$ of $\dot X$ in the high-resolution 
ODE \eqref{eqn: nag-sc_first} of NAG-\texttt{SC} is \textit{adaptive} 
to the position $X$, in contrast to the \textit{fixed} damping ratio 
$2\sqrt{\mu}$ in the ODE \eqref{eqn: heavy_ball_first} for the heavy-ball 
method.  To appreciate the effect of this adaptivity, imagine that the 
velocity $\dot X$ is highly correlated with an eigenvector of $\nabla^{2} f(X)$ 
with a large eigenvalue, such that the large friction $(2 \sqrt{\mu} + 
\sqrt{s} \nabla^{2} f(X)) \dot X$ effectively ``decelerates'' along the 
trajectory of the ODE \eqref{eqn: nag-sc_first} of NAG-\texttt{SC}. 
This feature of NAG-\texttt{SC} is appealing as taking a cautious step 
in the presence of high curvature generally helps avoid oscillations. 
Figure~\ref{fig:nag-c-and-heavy} and the left plot of 
Figure~\ref{fig: ODE_algorithm} confirm the superiority of 
NAG-\texttt{SC} over the heavy-ball method in this respect.

If we can translate this argument to the discrete case we can understand
why NAG-\texttt{SC} achieves acceleration globally for strongly convex functions
but the heavy-ball method does not.  We will be able to make this translation
by leveraging the high-resolution ODEs to construct discrete-time Lyapunov 
functions that allow maximal step sizes to be characterized for the 
NAG-\texttt{SC} and the heavy-ball method.  The detailed analyses is given 
in Section \ref{sec:strongly}.

\item \textbf{Effect of gradient correction in gradient norm minimization.} 
We will also show how to exploit the high-resolution ODE of NAG-\texttt{C} 
to construct a continuous-time Lyapunov function to analyze convergence in 
the setting of a smooth convex objective with $L$-Lipschitz gradients.  
Interestingly, the time derivative of the Lyapunov function is not only
negative, but it is smaller than $-O(\sqrt{s} t^2 \|\nabla f(X)\|^2)$. 
This bound arises from the gradient correction and, indeed, it cannot be
obtained from the Lyapunov function studied in the low-resolution case by 
\cite{su2016differential}. This finer characterization in the high-resolution 
case allows us to establish a new phenomenon:
\[
\min_{0 \leq i \leq k} \left\| \nabla f(x_{i}) \right\|^{2} \leq O\left(\frac{L^{2}}{k^3}\right).
\]
That is, we discover that NAG-\texttt{C} achieves an inverse \emph{cubic} 
rate for minimizing the squared gradient norm.  By comparison, from 
\eqref{eq:nac_k_2_conv} and the $L$-Lipschitz continuity of $\nabla f$ 
we can only show that $\left\| \nabla f(x_k) \right\|^{2} \le O\left(L^{2}/k^2\right)$. 
See Section \ref{sec:nagm-c_analysis} for further elaboration on this
cubic rate for NAG-\texttt{C}.

\end{itemize}

As we will see, the high-resolution ODEs are based on a phase-space representation 
that provides a systematic framework for translating from continuous-time Lyapunov 
functions to discrete-time Lyapunov functions.  In sharp contrast, the process
for obtaining a discrete-time Lyapunov function for low-resolution ODEs presented
by \cite{su2016differential} relies on ``algebraic tricks'' (see, for example, 
Theorem 6 of \cite{su2016differential}). On a related note, a Hessian-driven damping term also appears in ODEs for modeling Newton's method \cite{alvarez2002second,attouch2012second,attouch2016fast}.

\subsection{Related Work}
\label{sec:rel}

There is a long history of using ODEs to analyze optimization 
methods \cite{helmke2012optimization, schropp2000dynamical, fiori2005quasi}. 
Recently, the work of \cite{su2014differential,su2016differential} has sparked 
a renewed interest in leveraging continuous dynamical systems to understand 
and design first-order methods and to provide more intuitive proofs for the 
discrete methods. Below is a rather incomplete review of recent work that 
uses continuous-time dynamical systems to study accelerated methods.

In the work of \cite{wibisono2016variational,wilson2016lyapunov,betancourt2018symplectic}, 
Lagrangian and Hamiltonian frameworks are used to generate a large class of 
continuous-time ODEs for a unified treatment of accelerated gradient-based 
methods. Indeed, \cite{wibisono2016variational} extend NAG-\texttt{C} to 
non-Euclidean settings, mirror descent and accelerated higher-order gradient 
methods, all from a single ``Bregman Lagrangian.'' In \cite{wilson2016lyapunov}, 
the connection between ODEs and discrete algorithms is further strengthened 
by establishing an equivalence between the estimate sequence technique and 
Lyapunov function techniques, allowing for a principled analysis of the
discretization of continuous-time ODEs.  Recent papers have considered
symplectic \cite{betancourt2018symplectic} and Runge--Kutta \cite{zhang2018direct}
schemes for discretization of the low-resolution ODEs.

An ODE-based analysis of mirror descent has been pursued in another line of 
work by \cite{krichene2015accelerated,krichene2016adaptive,krichene2017acceleration},
delivering new connections between acceleration and constrained optimization, 
averaging and stochastic mirror descent.

In addition to the perspective of continuous-time dynamical systems, there has 
also been work on the acceleration from a control-theoretic point of view 
\cite{lessard2016analysis,hu2017dissipativity,fazlyab2018analysis} 
and from a geometric point of view \cite{bubeck2015geometric, chen2017geometric}. 
See also \cite{o2015adaptive,flammarion2015averaging,ghadimi2016accelerated,diakonikolas2017approximate, lin2018catalyst,drusvyatskiy2018optimal} for a number of other recent contributions
to the study of the acceleration phenomenon.



\subsection{Organization and Notation}
\label{sec:organ-notat}

The remainder of the paper is organized as follows. In Section~\ref{sec:techniques}, 
we briefly introduce our high-resolution ODE-based analysis framework. This 
framework is used in Section~\ref{sec:strongly} to study the heavy-ball method 
and NAG-\texttt{SC} for smooth strongly convex functions.  In Section 
\ref{sec:nagm-c_analysis}, we turn our focus to NAG-\texttt{C} for a general 
smooth convex objective. In Section~\ref{sec:extension} we derive some 
extensions of NAG-\texttt{C}. We conclude the paper in Section~\ref{sec:con} 
with a list of future research directions. Most technical proofs are deferred 
to the Appendix.

We mostly follow the notation of \cite{nesterov2013introductory}, with slight 
modifications tailored to the present paper. Let $\mathcal{F}_{L}^1( \mathbb{R}^{n})$ 
be the class of $L$-smooth convex functions defined on $\mathbb{R}^n$; that is, 
$f \in \mathcal{F}^1_L$ if $f(y) \geq f(x) + \left\langle \nabla f(x), 
y - x \right\rangle$ for all $x, y \in \mathbb R^n$ and its gradient is 
$L$-Lipschitz continuous in the sense that 
\[
\left\| \nabla f(x) - \nabla f(y) \right\| \leq L \left\| x - y \right\|,
\]
where $\|\cdot\|$ denotes the standard Euclidean norm and $L > 0$ is the 
Lipschitz constant.  (Note that this implies that $\nabla f$ is also 
$L'$-Lipschitz for any $L' \geq L$.)  The function class 
$\mathcal{F}_{L}^2(\mathbb{R}^{n})$ is the subclass of 
$\mathcal{F}_{L}^1(\mathbb{R}^{n})$ such that each $f$ has a 
Lipschitz-continuous Hessian.  For $p = 1, 2$, let 
$\mathcal{S}_{\mu,L}^p(\mathbb{R}^{n})$ denote the subclass of 
$\mathcal{F}_{L}^p(\mathbb{R}^{n})$ such that each member $f$ is 
$\mu$-strongly convex for some $0 < \mu \le L$. That is, 
$f \in  \mathcal{S}_{\mu, L}^p(\mathbb{R}^{n})$ if 
$f \in  \mathcal{F}_{L}^p(\mathbb{R}^{n})$ and
\[
f(y) \geq f(x) + \left\langle \nabla f(x), y - x \right\rangle + \frac{\mu}{2} \left\| y - x \right\|^{2},
\]
for all $x, y \in \mathbb{R}^{n}$. Note that this is equivalent to 
the convexity of $f(x) - \frac{\mu}{2}\|x - x^\star\|^2$, where $x^\star$ 
denotes a minimizer of the objective $f$.


\section{The High-Resolution ODE Framework}
\label{sec:techniques}

This section introduces a high-resolution ODE framework for analyzing gradient-based 
methods, with NAG-\texttt{SC} being a guiding example.  Given a (discrete) optimization 
algorithm, the first step in this framework is to derive a high-resolution ODE using 
dimensional analysis, the next step is to construct a continuous-time Lyapunov function 
to analyze properties of the ODE, the third step is to derive a discrete-time Lyapunov
function from its continuous counterpart and the last step is to translate properties 
of the ODE into that of the original algorithm. The overall framework is illustrated 
in Figure \ref{fig:chart}.

\begin{figure}[!htp]
\begin{center}
\begin{tikzpicture}[node distance=4cm]
    \node (n00) [box, draw=black] {\small Algorithms};
    \node (n10) [box, draw=black,  above of=n00] {\small High-Resolution ODEs};
    \node (n11) [box, draw=black,  above of=n00, xshift=+8cm] {\small Continuous $\mathcal{E}(t)$};
    \node (n100) [box, draw=black,  below of=n11] {\small Discrete $\mathcal{E}(k)$};
    \node (n101) [mybox, draw=black,  above left of=n100, xshift = - 1.15cm, yshift = -1cm] {\small Nesterov's \, Acceleration};
    \node (n110) [mybox, draw=black,  below left of=n100, xshift = - 1.15cm, yshift = +1cm] {\small Gradient Norm Minimization};

    \draw [arrow] (n00) --node [left] {dimensional analysis} (n10);
    \draw [arrow] (n10) -- (n11);
    \draw [arrow] (n11) -- node [right] {phase-space representation} (n100);
    \draw [arrow] (n100)  to [bend right=25] (n101);
    \draw [arrow] (n100)  to [bend left=25] (n110);
    \draw [arrow] [dashed] (n00) -- (n100);
    
\end{tikzpicture}
\end{center}
\caption{An illustration of our high-resolution ODE framework. The three solid straight 
lines represent Steps 1, 2 and 3, and the two curved lines denote Step 4. The dashed 
line is used to emphasize that it is difficult, if not impractical, to construct discrete 
Lyapunov functions directly from the algorithms.}
\label{fig:chart}
\end{figure}
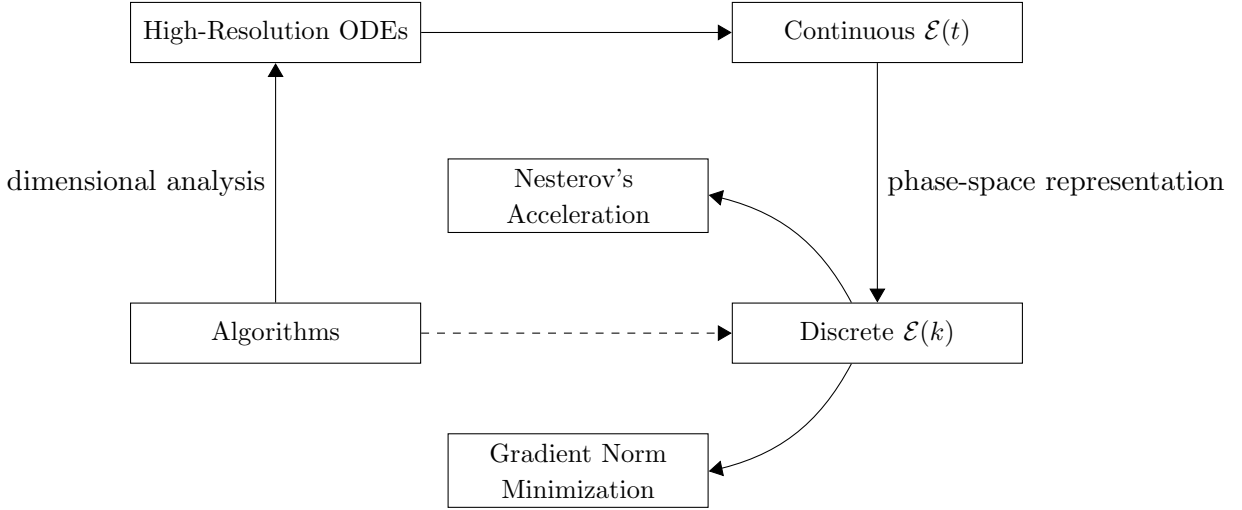

\subsection*{Step 1: Deriving High-Resolution ODEs}

Our focus is on the single-variable form~\eqref{eqn: Nesterov_strongly_single} of 
NAG-\texttt{SC}. For any nonnegative integer $k$, let $t_{k} = k\sqrt{s}$ and assume 
$x_k = X(t_k)$ for some sufficiently smooth curve $X(t)$.  Performing a Taylor expansion 
in powers of $\sqrt{s}$, we get
\begin{equation}
\label{eqn: 3rd_taylor_variable}
\begin{aligned}
& x_{k + 1} = X(t_{k+1}) = X(t_{k}) + \dot{X}(t_{k})\sqrt{s} + \frac{1}{2} \ddot{X}(t_{k})\left(\sqrt{s}\right)^{2}  + \frac{1}{6}\dddot{X}(t_{k})\left(\sqrt{s}\right)^{3} + O\left( \left(\sqrt{s}\right)^{4} \right) \\
& x_{k  - 1} =  X(t_{k-1}) = X(t_{k})  - \dot{X}(t_{k})\sqrt{s} + \frac{1}{2} \ddot{X}(t_{k})\left(\sqrt{s}\right)^{2} - \frac{1}{6}\dddot{X}(t_{k})\left(\sqrt{s}\right)^{3}  + O\left( \left(\sqrt{s}\right)^{4} \right).
\end{aligned}
\end{equation}
We now use a Taylor expansion for the gradient correction, which gives
\begin{equation}
\label{eqn: gradient_correction}
\nabla f(x_{k}) - \nabla f(x_{k - 1}) =  \nabla^{2} f(X(t_{k})) \dot{X}(t_{k})\sqrt{s} + O\left(\left( \sqrt{s} \right)^{2}\right). 
\end{equation}
Multiplying both sides of \eqref{eqn: Nesterov_strongly_single} by 
$\frac{1 + \sqrt{\mu s}}{1 - \sqrt{\mu s}} \cdot \frac1{s}$ and rearranging 
the equality, we can rewrite NAG-\texttt{SC} as
\begin{align}
\frac{x_{k + 1} + x_{k - 1} - 2x_{k}}{s} +\frac{2\sqrt{\mu s}}{1 - \sqrt{\mu s}} \cdot \frac{x_{k + 1} - x_{k}}{s} + \nabla f(x_{k}) - \nabla f(x_{k - 1}) + \frac{1 + \sqrt{\mu s}}{1 - \sqrt{\mu s}} \nabla f(x_{k}) = 0.\label{eqn:nagm-sc_rewrite}
\end{align}
Next, plugging~\eqref{eqn: 3rd_taylor_variable} and~\eqref{eqn: gradient_correction} 
into~\eqref{eqn:nagm-sc_rewrite}, we have\footnote{Note that we use the approximation 
$\frac{x_{k+1} + x_{k -1} - 2x_{k}}{s} = \ddot{X}(t_{k}) + O(s)$, whereas 
\cite{su2016differential} relies on the low-accuracy Taylor expansion 
$\frac{x_{k+1} + x_{k -1} - 2x_{k}}{s} = \ddot{X}(t_{k}) + o(1)$ in the 
derivation of the low-resolution ODE of NAG-\texttt{C}. We illustrate this
derivation of the three low-resolution ODEs in Appendix~\ref{sec: ODE_1};
they can be compared to the high-resolution ODEs that we derive here.}
\begin{multline*}
\ddot{X}(t_{k}) + O\left( \left(\sqrt{s}\right)^{2} \right)  + \frac{2\sqrt{\mu }}{1 - \sqrt{\mu s}} \left[ \dot{X}(t_{k})  + \frac{1}{2} \ddot{X}(t_{k})\sqrt{s} + O\left( \left(\sqrt{s}\right)^{2} \right)  \right]  \nonumber\\
  + \nabla^{2} f(X(t_{k})) \dot{X}(t_{k}) \sqrt{s} + O\left( \left(\sqrt{s}\right)^{2} \right) + \left(\frac{1 + \sqrt{\mu s}}{1 - \sqrt{\mu s}} \right)\nabla f(X(t_{k})) = 0,
\end{multline*} 
which can be rewritten as
\[
\frac{\ddot X (t_k)}{1 - \sqrt{\mu s}}   + \frac{2\sqrt{\mu}}{1 - \sqrt{\mu s}} \dot X(t_k) + \sqrt{s} \nabla^2 f(X(t_k)) \dot X(t_k) + \frac{1 + \sqrt{\mu s}}{1 - \sqrt{\mu s}} \nabla f(X(t_k)) + O(s) = 0.
\]
Multiplying both sides of the last display by $1-\sqrt{\mu s}$, we obtain 
the following high-resolution ODE of NAG-\texttt{SC}:
\[
\ddot{X} + 2\sqrt{\mu} \dot{X} + \sqrt{s}  \nabla^{2} f(X) \dot{X} + (1 + \sqrt{\mu s})\nabla f(X)= 0,
\]
where we ignore any $O(s)$ terms but retain the $O(\sqrt{s})$ terms (note 
that $(1 - \sqrt{\mu s})\sqrt{s} = \sqrt{s} + O(s)$).


Our analysis is inspired by dimensional analysis \cite{pedlosky2013geophysical}, 
a strategy widely used in physics to construct a series of differential equations 
that involve increasingly high-order terms corresponding to small perturbations.  
In more detail, taking a small $s$, one first derives a differential equation that 
consists only of $O(1)$ terms, then derives a differential equation consisting of 
both $O(1)$ and $O(\sqrt{s})$, and next, one proceeds to obtain a differential 
equation consisting of $O(1), O(\sqrt{s})$ and $O(s)$ terms. High-order terms in 
powers of $\sqrt{s}$ are introduced sequentially until the main characteristics 
of the original algorithms have been extracted from the resulting approximating 
differential equation. Thus, we aim to understand Nesterov acceleration by 
incorporating $O(\sqrt{s})$ terms into the ODE, including the (Hessian-driven) 
gradient correction $\sqrt{s}  \nabla^{2} f(X) \dot{X}$ which results from the 
(discrete) gradient correction \eqref{eq:gradient_cor} in the single-variable 
form \eqref{eqn: Nesterov_strongly_single} of NAG-\texttt{SC}.  We also show
(see Appendix~\ref{sec: ODE_2} for the detailed derivation) that this $O(\sqrt{s})$ 
term appears in the high-resolution ODE of NAG-\texttt{C}, but is not found in 
the high-resolution ODE of the heavy-ball method. 

As shown below, each ODE admits a unique global solution under mild conditions 
on the objective, and this holds for an arbitrary step size $s > 0$. The solution 
is accurate in approximating its associated optimization method if $s$ is small. 
To state the result, we use $C^2(I; \mathbb{R}^n)$ to denote the class of 
twice-continuously-differentiable maps from $I$ to $\mathbb{R}^n$ for 
$I = [0, \infty)$ (the heavy-ball method and NAG-\texttt{SC}) and 
$I = [1.5\sqrt{s}, \infty)$ (NAG-\texttt{C}).
\begin{prop}
\label{prop: exist_unique_SC}
For any $f \in \mathcal{S}_{\mu}^2(\mathbb{R}^n) := \cup_{L \geq \mu} 
\mathcal{S}^2_{\mu,L}(\mathbb{R}^n)$, each of the ODEs \eqref{eqn: heavy_ball_first} 
and \eqref{eqn: nag-sc_first} with the specified initial conditions has a 
unique global solution $X \in C^2([0, \infty); \mathbb{R}^n)$. Moreover, 
the two methods converge to their high-resolution ODEs, respectively, in the 
sense that
\[
\limsup_{s \rightarrow 0} \max_{0 \le k \le \frac{T}{\sqrt{s}}} \left\| x_k - X(k\sqrt{s}) \right\| = 0,
\]
for any fixed $T > 0$.
\end{prop}

In fact, Proposititon~\ref{prop: exist_unique_SC} holds for $T = \infty$ because both the discrete iterates and the ODE trajectories converge to the unique minimizer when the objective is stongly convex.

\begin{prop}
\label{prop: exist_unique_C}
For any $f \in \mathcal{F}^2(\mathbb{R}^n) := \cup_{L > 0} \mathcal{F}^2_L(\mathbb{R}^n)$, 
the ODE \eqref{eqn: nag-c_first} with the specified initial conditions has a unique 
global solution $X \in C^2([1.5\sqrt{s}, \infty); \mathbb{R}^n)$. Moreover, 
NAG-\texttt{C} converges to its high-resolution ODE in the sense that
\[
\limsup_{s \rightarrow 0} \max_{0 \le k \le \frac{T}{\sqrt{s}}} \left\| x_k - X(k\sqrt{s} + 1.5\sqrt{s}) \right\| = 0,
\]
for any fixed $T > 0$.
\end{prop}

The proofs of these propositions are given in Appendix~\ref{sec: proof_exist_unique_SC} 
and Appendix~\ref{sec: proof_exist_unique_C}.

\subsection*{Step 2: Analyzing ODEs Using Lyapunov Functions}
With these high-resolution ODEs in place, the next step is to construct 
Lyapunov functions for analyzing the dynamics of the corresponding ODEs, 
as is done in previous work \cite{su2016differential,wilson2016lyapunov,
lessard2016analysis}. For NAG-\texttt{SC}, we consider the Lyapunov function
\begin{equation}\label{eqn:nagm-sc_ode_lypaunov}
\mathcal{E}(t)  =  (1 + \sqrt{\mu s} )\left( f(X) - f(x^{\star})
\right) + \frac{1}{4} \| \dot{X} \|^{2}  + \frac{1}{4} \| \dot{X} 
+ 2 \sqrt{\mu} (X - x^{\star}) + \sqrt{s} \nabla f(X) \|^2.
\end{equation}
The first and second terms $(1 + \sqrt{\mu s})\left( f(X) - f(x^{\star}) \right)$ and $\frac{1}{4} \| \dot{X} \|^{2}$ can be regarded, respectively, as the potential energy and kinetic energy, and the last term is a mix. For the mixed term, it is interesting to note that the time derivative of $\dot{X} + 2 \sqrt{\mu} (X - x^{\star}) + \sqrt{s} \nabla f(X)$ equals $-(1+\sqrt{\mu s}) \nabla f(X)$.

The differentiability of $\mathcal{E}(t)$ will allow us to investigate 
properties of the ODE~\eqref{eqn: nag-sc_first} in a principled manner. 
For example, we will show that $\mathcal{E}(t)$ decreases exponentially 
along the trajectories of \eqref{eqn: nag-sc_first}, recovering the 
accelerated linear convergence rate of NAG-\texttt{SC}. Furthermore, 
a comparison between the Lyapunov function of NAG-\texttt{SC} and that
of the heavy-ball method will explain why the gradient correction 
$\sqrt{s}\nabla^{2} f(X) \dot{X}$ yields acceleration in the former
case.  This is discussed in Section~\ref{sec:continuous-time-1-s}.

\subsection*{Step 3: Constructing Discrete Lyapunov Functions}

Our framework make it possible to translate continuous Lyapunov functions 
into discrete Lyapunov functions via a phase-space representation (see, 
for example, \cite{arnol2013mathematical}).  We illustrate the procedure
in the case of NAG-\texttt{SC}.  The first step is formulate explicit 
position and velocity updates:
\begin{equation}\label{eqn: Nesterov_sc_symplectic}
\begin{aligned}
& x_{k} - x_{k - 1} =  \sqrt{s} v_{k - 1} \\
& v_{k} - v_{k - 1} = - \frac{2\sqrt{\mu s}}{1 - \sqrt{\mu s}} 
v_{k} - \sqrt{s}( \nabla f(x_{k}) - \nabla f(x_{k - 1}) ) - 
\frac{1 + \sqrt{\mu s}}{1 - \sqrt{\mu s}} \cdot \sqrt{s}\nabla f( x_{k} ),
\end{aligned} 
\end{equation}
where the velocity variable $v_k$ is defined as:
\[
v_k = \frac{x_{k+1} - x_k}{\sqrt{s}}.
\] 
The initial velocity is $v_{0} = - \frac{2\sqrt{s}}{1 + \sqrt{\mu s}} 
\nabla f(x_{0})$. Interestingly, this phase-space representation has the
flavor of symplectic discretization, in the sense that the update for 
$x_k - x_{k-1}$ is explicit (it only depends on the last iterate $v_{k-1}$) 
while the update for $v_{k} - v_{k - 1}$ is implicit (it depends on the 
current iterates $x_k$ and $v_{k}$)\footnote{Although this suggestion is
a heuristic one, it is also possible to rigorously derive a symplectic 
integrator of the high-resolution ODE of NAG-\texttt{SC}; this integrator
has the form:
\begin{equation}\nonumber
\begin{aligned}
& x_{k} - x_{k - 1} = \sqrt{s} v_{k - 1} \\
& v_{k} - v_{k - 1} = - 2\sqrt{\mu s}v_{k} - s\nabla^{2}f(x_{k})  v_{k}  - (1 + \sqrt{\mu s})\sqrt{s}\nabla f(x_{k}).
\end{aligned}
\end{equation}}.

The representation \eqref{eqn: Nesterov_sc_symplectic} suggests translating 
the continuous-time Lyapunov function \eqref{eqn:nagm-sc_ode_lypaunov} into 
a discrete-time Lyapunov function of the following form:
\begin{equation}\label{eqn:lypunov_NAGM-SC_strongly}
\begin{aligned}
\mathcal{E}(k) =  &  \underbrace{\frac{1 + \sqrt{\mu s} }{1 - \sqrt{\mu s} }  
\left( f(x_{k}) - f(x^{\star}) \right)}_{\mathbf{I}} + \underbrace{\frac{1}{4} 
\left\| v_{k} \right\|^{2}}_{\mathbf{II}}  + \underbrace{\frac{1}{4} 
\left\| v_{k} +   \frac{2\sqrt{\mu}}{1 - \sqrt{\mu s}} ( x_{k + 1} - x^{\star} ) 
+  \sqrt{s} \nabla f(x_{k}) \right\|^{2}}_{\mathbf{III}}\\
&  \underbrace{- \frac{s\left\| \nabla f(x_{k})\right\|^{2}}{2(1 - \sqrt{\mu s})}}_{\textbf{a negative term}},
\end{aligned}                           
\end{equation}
by replacing continuous terms (e.g., $\dot X$) by their discrete counterparts 
(e.g., $v_k$). Akin to the continuous \eqref{eqn:nagm-sc_ode_lypaunov}, here 
$\mathbf{I}$, $\mathbf{II}$, and $\mathbf{III}$ correspond to potential energy, 
kinetic energy, and mixed energy, respectively, from a mechanical perspective.
To better appreciate this translation, note that the factor 
$\frac{1 + \sqrt{\mu s}}{1 - \sqrt{\mu s}}$ in $\mathbf{I}$ results 
from the term $\frac{1 + \sqrt{\mu s}}{1 - \sqrt{\mu s}} \sqrt{s}\nabla 
f( x_{k} )$ in \eqref{eqn: Nesterov_sc_symplectic}. Likewise, 
$\frac{ 2\sqrt{\mu}}{1 - \sqrt{\mu s}}$ in $\mathbf{III}$ is from 
the term $\frac{2\sqrt{\mu s}}{1 - \sqrt{\mu s}}v_{k}$ in \eqref{eqn:
  Nesterov_sc_symplectic}. The need for the final (small) negative term 
is technical; we discuss it in Section \ref{sec:discrete-time}.




\subsection*{Step 4: Analyzing Algorithms Using Discrete Lyapunov Functions}

The last step is to map properties of high-resolution ODEs to corresponding 
properties of optimization methods. This step closely mimics Step 2 except that 
now the object is a discrete algorithm and the tool is a discrete Lyapunov 
function such as \eqref{eqn:lypunov_NAGM-SC_strongly}. Given that Step 2 has 
been performed, this translation is conceptually straightforward, albeit often 
calculation-intensive. For example, using the discrete Lyapunov function 
\eqref{eqn:lypunov_NAGM-SC_strongly}, we will recover the optimal linear rate of 
NAG-\texttt{SC} and gain insights into the fundamental effect of the gradient 
correction in accelerating NAG-\texttt{SC}. In addition, NAG-\texttt{C} is 
shown to minimize the squared gradient norm at an inverse cubic rate by a 
simple analysis of the decreasing rate of its discrete Lyapunov function.


\section{Gradient Correction for Acceleration}
\label{sec:strongly}

In this section, we use our high-resolution ODE framework to analyze NAG-\texttt{SC} 
and the heavy-ball method. Section~\ref{sec:continuous-time-1-s} focuses on the 
ODEs with an objective function $f \in \mathcal{S}^2_{\mu, L}(\mathbb{R}^{n})$, 
and in Section~\ref{sec:discrete-time} we extend the results to the discrete case 
for $f \in \mathcal{S}^1_{\mu, L}(\mathbb{R}^{n})$.  Finally, in Section 
\ref{sec:insights-into-accel} we offer a comparative study of NAG-\texttt{SC} 
and the heavy-ball method from a finite-difference viewpoint.

Throughout this section, the strategy is to analyze the two methods in parallel, 
thereby highlighting the differences between the two methods.  In particular, 
the comparison will demonstrate the vital role of the gradient correction, 
namely $\frac{ 1 - \sqrt{\mu s} }{ 1 + \sqrt{\mu s} } \cdot s 
\left( \nabla f(x_{k}) - \nabla f(x_{k - 1}) \right)$ in the discrete case 
and $\sqrt{s} \nabla^2 f(X) \dot X$ in the ODE case, in making NAG-\texttt{SC} 
an accelerated method.

\subsection{The ODE Case}
\label{sec:continuous-time-1-s}

The following theorem characterizes the convergence rate of the high-resolution 
ODE corresponding to NAG-\texttt{SC}.
\begin{thm}[Convergence of NAG-\texttt{SC} ODE]\label{thm: ODE_first-NAGM-SC_strongly}
Let $f \in \mathcal{S}^2_{\mu, L}(\mathbb{R}^n)$. For any step size $0 < s \le 1/L$, the solution $X = X(t)$ of the high-resolution ODE \eqref{eqn: nag-sc_first} satisfies
\begin{equation}\nonumber
f(X(t)) - f(x^{\star}) \leq \frac{2 \left\| x_{0} - x^{\star}  \right\|^{2}}{s} \mathrm{e}^{- \frac{ \sqrt{\mu} t}{4}}. 
\end{equation}
\end{thm}
The theorem states that the functional value $f(X)$ tends to the minimum 
$f(x^\star)$ at a linear rate. By setting $s = 1/L$, we obtain 
$f(X) - f(x^{\star}) \le 2L\left\| x_{0} - x^{\star}  \right\|^{2} 
\mathrm{e}^{- \frac{ \sqrt{\mu} t}{4}}$. 

The proof of Theorem~\ref{thm: ODE_first-NAGM-SC_strongly} is based on analyzing 
the Lyapunov function $\mathcal{E}(t)$ for the high-resolution ODE of NAG-\texttt{SC}.
Recall that $\mathcal{E}(t)$ defined in \eqref{eqn:nagm-sc_ode_lypaunov} is
\[
\mathcal{E}(t) = (1 + \sqrt{\mu s} )\left( f(X) - f(x^{\star}) \right) + \frac{1}{4} \| \dot{X} \|^{2}  + \frac{1}{4} \| \dot{X} + 2 \sqrt{\mu} (X - x^{\star}) + \sqrt{s} \nabla f(X)\|^2.
\]
The next lemma states the key property we need from this Lyapunov function
\begin{lem}[Lyapunov function for NAG-\texttt{SC} ODE]\label{lm:nag_sc_e}
Let $f \in \mathcal{S}^2_{\mu, L}(\mathbb{R}^n)$. For any step size $s > 0$, 
and with $X = X(t)$ being the solution to the high-resolution ODE 
\eqref{eqn: nag-sc_first}, the Lyapunov function~\eqref{eqn:nagm-sc_ode_lypaunov} 
satisfies
\begin{align}
\label{eqn: nag-sc_ode_conver-rate}
\frac{\dd \mathcal{E}(t)}{\dd t} \le -\frac{\sqrt{\mu}}{4}\mathcal{E}(t) -  \frac{\sqrt{s}}{2} \left[ \left\| \nabla f(X(t))\right\|^{2} + \dot{X}(t)^{\top}\nabla^{2}  f(X(t))\dot{X}(t) \right].
\end{align}
\end{lem}


The proof of this theorem relies on Lemma \ref{lm:nag_sc_e} through the 
inequality $\dot{\mathcal{E}}(t) \le - \frac{\sqrt{\mu}}{4}\mathcal{E}(t)$. 
The term $\frac{\sqrt{s}}{2} ( \left\| \nabla f(X)\right\|^{2} + \dot{X}^\top \nabla^{2}f(X)\dot{X} ) \ge 0$ plays no role at the moment, but Section~\ref{sec:discrete-time} will shed light on its profound effect in the discretization of the high-resolution ODE of NAG-\texttt{SC}. 

\begin{proof}[Proof of Theorem~\ref{thm: ODE_first-NAGM-SC_strongly}]

Lemma \ref{lm:nag_sc_e} implies $\dot{\mathcal{E}}(t) \le - \frac{\sqrt{\mu}}{4}\mathcal{E}(t)$, which amounts to
\[
\frac{\dd }{\dd t} \left( \mathcal{E}(t) \mathrm{e}^{\frac{\sqrt{\mu} t}{4}} \right) \le 0.
\]
By integrating out $t$, we get
\begin{equation}\label{eq:e_exp_bd}
\mathcal{E}(t) \le \mathrm{e}^{-\frac{\sqrt{\mu} t}{4}} \mathcal{E}(0).
\end{equation}
Recognizing the initial conditions $X(0) = x_{0}$ and $\dot X(0) = -\frac{2\sqrt{s} \nabla f(x_{0})}{1 + \sqrt{\mu s}}$, we write \eqref{eq:e_exp_bd} as
\[
\begin{aligned}
f(X) - f(x^{\star}) \leq & \mathrm{e}^{- \frac{ \sqrt{\mu} t}{4}} \left[   f(x_0) - f(x^{\star}) + \frac{s}{\left( 1 + \sqrt{\mu s} \right)^{3}} \left\| \nabla f(x_{0}) \right\|^{2}  \right. \\
	&\qquad \quad \left.+ \frac{1}{4(1 + \sqrt{\mu s})} \left\| 2\sqrt{\mu} ( x_0 - x^{\star} )- \frac{1 - \sqrt{\mu s}}{1 + \sqrt{\mu s}} \cdot \sqrt{s} \nabla f(x_0) \right\|^{2} \right]. 
\end{aligned}
\]
Since $f \in \mathcal{S}_{\mu, L}^{2}$, we have that 
$\|\nabla f(x_0)\| \le L \|x_0 - x^\star\|$ and 
$f(x_0) - f(x^\star) \le L\|x_0 - x^\star\|^2/2$. 
Together with the Cauchy--Schwarz inequality, the two inequalities yield
\[
\begin{aligned}
f(X) - f(x^{\star}) &\leq \left[ f(x_{0}) - f(x^{\star}) + \frac{2 + (1 - \sqrt{\mu s})^{2}}{2(1 + \sqrt{\mu s})^{3}} \cdot s\left\| \nabla f(x_{0})\right\|^{2} + \frac{2\mu}{1 + \sqrt{\mu s}} \left\| x_{0} - x^{\star} \right\|^{2}  \right] \mathrm{e}^{- \frac{ \sqrt{\mu} t}{4}}\\
&\leq \left[ \frac{L}{2} + \frac{3 - 2\sqrt{\mu s} + \mu s}{2(1 + \sqrt{\mu s})^{3}} \cdot sL^{2} + \frac{2\mu}{1 + \sqrt{\mu s}}  \right]  \left\| x_{0} - x^{\star} \right\|^{2} \mathrm{e}^{- \frac{ \sqrt{\mu} t}{4}},
\end{aligned}
\]
which is valid for all $s > 0$. To simplify the coefficient of 
$\left\| x_{0} - x^{\star} \right\|^{2} \mathrm{e}^{- \frac{ \sqrt{\mu} t}{4}}$, 
note that $L$ can be replaced by $1/s$ in the analysis since $s \le 1/L$. 
It follows that
\[
f(X(t)) - f(x^{\star}) \leq \left[ \frac{1}{2} + \frac{3 - 2\sqrt{\mu s} + \mu s}{2(1 + \sqrt{\mu s})^{3}} + \frac{2\mu s}{1 + \sqrt{\mu s}}  \right]  \frac{\left\| x_{0} - x^{\star} \right\|^{2} \mathrm{e}^{- \frac{ \sqrt{\mu} t}{4}}}{s}. 
\]
Furthermore, a bit of analysis reveals that
\[
\frac{1}{2} + \frac{3 - 2\sqrt{\mu s} + \mu s}{2(1 + \sqrt{\mu s})^{3}} + \frac{2\mu s}{1 + \sqrt{\mu s}}  < 2,
\]
since $\mu s \le \mu/L \le 1$, and this step completes the proof of 
Theorem~\ref{thm: ODE_first-NAGM-SC_strongly}.
\end{proof}

We now consider the heavy-ball method~(\ref{eqn: polyak_heavy_ball}).  Recall that 
the momentum coefficient $\alpha$ is set to $\frac{1 - \sqrt{\mu s}}{1 + \sqrt{\mu s}}$. 
The following theorem characterizes the rate of convergence of this method.

\begin{thm}[Convergence of heavy-ball ODE]\label{thm: ODE_first-PHBM_strongly}
Let $f \in \mathcal{S}_{\mu, L}^{2}(\mathbb{R}^{n})$. For any step size  $0 < s \le 1/L$,  
the solution $X = X(t)$ of the \textit{high-resolution ODE}~\eqref{eqn: heavy_ball_first} 
satisfies
\[
f(X(t)) - f(x^{\star}) \leq \frac{7\left\| x_0 - x^{\star}  \right\|^{2}}{2s}  
\mathrm{e}^{- \frac{ \sqrt{\mu} t}{4}}.
\]
\end{thm}

As in the case of NAG-\texttt{SC}, the proof of Theorem \ref{thm: ODE_first-PHBM_strongly} is based on a Lyapunov function:
\begin{align}\label{eqn: energy_heavy-b_ode}
\mathcal{E}( t ) =  (1 + \sqrt{\mu s} ) \left( f(X) - f(x^{\star}) \right) + \frac{1}{4} \| \dot{X} \|^{2} + \frac{1}{4} \|  \dot{X}  + 2 \sqrt{\mu} (X - x^{\star})\|^2,
\end{align}
which is the same as the Lyapunov function \eqref{eqn:nagm-sc_ode_lypaunov} for NAG-\texttt{SC} except for the lack of the $\sqrt{s} \nabla f(X)$ term. In particular, \eqref{eqn:nagm-sc_ode_lypaunov} and \eqref{eqn: energy_heavy-b_ode} are identical if $s = 0$. The following lemma considers the decay rate of \eqref{eqn: energy_heavy-b_ode}.

\begin{lem}[Lyapunov function for the heavy-ball ODE]\label{lm:heavyball_e}
Let $f\in \mathcal{S}^{2}_{\mu, L}(\mathbb{R}^n)$. For any step size $s>0$, 
the Lyapunov function~\eqref{eqn: energy_heavy-b_ode} for the high-resolution 
ODE~\eqref{eqn: heavy_ball_first} satisfies
\[
\frac{\dd \mathcal{E}(t)}{\dd t} \le -\frac{\sqrt{\mu}}{4}\mathcal{E}(t).
\]
\end{lem}

The proof of Theorem~\ref{thm: ODE_first-PHBM_strongly} follows the same 
strategy as the proof of Theorem~\ref{thm: ODE_first-NAGM-SC_strongly}. 
In brief, Lemma~\ref{lm:heavyball_e} gives $\mathcal{E}(t) \le 
\mathrm{e}^{-\sqrt{\mu} t/4} \mathcal{E}(0)$ by integrating over the 
time parameter $t$. Recognizing the initial conditions
\[
X(0) = x_{0}, \quad \dot{X}(0) =  - \frac{2\sqrt{s} \nabla f(x_{0})}{1 + \sqrt{\mu s}}
\]
in the high-resolution ODE of the heavy-ball method and using the $L$-smoothness 
of $\nabla f$, Lemma~\ref{lm:heavyball_e} yields
\[
f(X) - f(x^{\star}) \leq \left[ \frac{1}{2} + \frac{3}{( 1 + \sqrt{\mu s} )^{3}} 
+ \frac{2(\mu s)}{1 + \sqrt{\mu s}} \right]  \frac{\left\| x_{0} - x^{\star}\right\|^{2} 
\mathrm{e}^{- \frac{ \sqrt{\mu} t}{4}}}{s},
\]
if the step size $s \leq 1/L$. Finally, since $0 < \mu s \le \mu/L \le 1$, 
the coefficient satisfies
\[
 \frac{1}{2} + \frac{3}{( 1 + \sqrt{\mu s} )^{3}} + \frac{2\mu s}{1 + \sqrt{\mu s}}< \frac{7}{2}.
\]

The proofs of Lemma~\ref{lm:nag_sc_e} and Lemma~\ref{lm:heavyball_e} share similar ideas. 
In view of this, we present only the proof of the former here, deferring the proof of 
Lemma~\ref{lm:heavyball_e} to Appendix~\ref{subsec: heavy_ball_continuous}.

\begin{proof}[Proof of Lemma \ref{lm:nag_sc_e}]
Along trajectories of \eqref{eqn: nag-sc_first} the Lyapunov 
function~(\ref{eqn:nagm-sc_ode_lypaunov}) satisfies
\begin{equation}\label{eq:eq_key_diff_app}
\begin{aligned}
	\frac{\dd \mathcal{E}}{\dd t} & =(1 + \sqrt{\mu s} ) \langle \nabla f(X), \dot{X} \rangle + \frac{1}{2}\left\langle \dot{X}, - 2 \sqrt{\mu} \dot{X} - \sqrt{s} \nabla^{2} f(X) \dot{X} - (1 + \sqrt{\mu s} )\nabla f(X) \right\rangle  \\
	& \quad\quad + \frac{1}{2} \left\langle \dot{X} + 2\sqrt{\mu} \left( X - x^{\star} \right) +  \sqrt{s} \nabla f(X), - (1 + \sqrt{\mu s}) \nabla f(X) \right\rangle  \\
	& = - \sqrt{\mu} \left( \| \dot{X}\|^{2} + (1 + \sqrt{\mu s} ) \left\langle \nabla f(X), X - x^{\star} \right\rangle + \frac{s}{2}\left\| \nabla f(X)\right\|^2\right) \\
         &\quad\quad -  \frac{\sqrt{s}}{2} \left[ \left\| \nabla f(X)\right\|^{2} + \dot{X}^{\top}\nabla^{2}  f(X)\dot{X} \right]\\
	& \le - \sqrt{\mu} \left( \| \dot{X}\|^{2} + (1 + \sqrt{\mu s} ) \left\langle \nabla f(X), X - x^{\star} \right\rangle + \frac{s}{2}\left\| \nabla f(X)\right\|^2\right).
\end{aligned}
\end{equation}
Furthermore, $\left\langle \nabla f(X), X - x^{\star} \right\rangle$ is greater than or equal to both $f(X) - f(x^\star) + \frac{\mu}{2} \|X - x^\star\|^2$ and $\mu \|X - x^\star\|^2$ 
due to the $\mu$-strong convexity of $f$. This yields
\begin{align*}
(1 + \sqrt{\mu s} ) \left\langle \nabla f(X), X - x^{\star} \right\rangle &\ge \frac{1 + \sqrt{\mu s}}{2} \left\langle \nabla f(X), X - x^{\star} \right\rangle + \frac12 \left\langle \nabla f(X), X - x^{\star} \right\rangle\\
&\ge \frac{1 + \sqrt{\mu s}}{2} \left[ f(X) - f(x^\star) + \frac{\mu}{2} \|X - x^\star\|^2\right] + \frac{\mu}{2} \|X - x^\star\|^2\\
&\ge \frac{1 + \sqrt{\mu s}}{2} (f(X) - f(x^\star)) + \frac{3\mu}{4} \|X - x^\star\|^2,
\end{align*}
which together with \eqref{eq:eq_key_diff_app} suggests that the time derivative of this Lyapunov function can be bounded as
\begin{equation}\label{eq:e_t_bound_1}
\frac{\dd \mathcal{E}}{\dd t} \leq - \sqrt{\mu} \left(\frac{1 + \sqrt{\mu s} }{2} (f(X) - f(x^{\star})) + \| \dot{X} \|^{2} +  \frac{3\mu}{4} \left\| X - x^{\star} \right\|^{2} + \frac{s}{2} \left\| \nabla f(X)\right\|^{2} \right).
\end{equation}
Next, the Cauchy--Schwarz inequality yields
\[
\left\| 2\sqrt{\mu} ( X - x^{\star} ) + \dot{X} + \sqrt{s} \nabla f(X)\right\|^{2} \leq 3 \left( 4 \mu \left\| X - x^{\star} \right\|^{2} + \| \dot{X} \|^{2} + s \left\| \nabla f(X) \right\|^{2} \right),
\]
from which it follows that
\begin{align}
\label{eqn: estimate_ef_sc_ode}
\mathcal{E}(t) \leq \left(1 + \sqrt{\mu s} \right) \left(f(X) - f(x^{\star}) \right) + \| \dot{X} \|^{2} +  3 \mu \left\|X - x^{\star}  \right\|^{2} + \frac{3s}{4} \left\| \nabla f(X)\right\|^{2}.
\end{align} 
Combining \eqref{eq:e_t_bound_1} and \eqref{eqn: estimate_ef_sc_ode} 
completes the proof of the theorem.

\end{proof}

\begin{rem}\label{rem:gc_effect}
The only inequality in \eqref{eq:eq_key_diff_app} is due to the term $\frac{\sqrt{s}}{2} (\left\| \nabla f(X)\right\|^{2} + \dot{X}^{\top}\nabla^{2}  f(X)\dot{X} )$, which is discussed right after the statement of Lemma~\ref{lm:nag_sc_e}. This term results from the gradient correction $\sqrt{s} \nabla^2 f(X) \dot X$ in the NAG-\texttt{SC} ODE. For comparison, this term does not appear in Lemma~\ref{lm:heavyball_e} in the case of the heavy-ball method as its ODE does not include the gradient correction and, accordingly, its Lyapunov function \eqref{eqn: energy_heavy-b_ode} is free of the $\sqrt{s} \nabla f(X)$ term.

\end{rem}

\subsection{The Discrete Case}
\label{sec:discrete-time}
This section carries over the results in Section \ref{sec:continuous-time-1-s} to the two discrete algorithms, namely NAG-\texttt{SC} and the heavy-ball method. Here we consider an objective $f \in \mathcal{S}_{\mu,L}^1(\mathbb{R}^n)$ since second-order differentiability of $f$ is not required in the two discrete methods. Recall that both methods start with an arbitrary $x_0$ and $x_1 = x_0 - \frac{2 s\nabla f(x_0)}{1 + \sqrt{\mu s}}$. 

\begin{thm}[Convergence of NAG-\texttt{SC}]\label{thm: strongly_NAGM-SC}
Let $f \in \mathcal{S}_{\mu,L}^1(\mathbb{R}^n)$. If the step size is set to $s = 1/(4L)$, the iterates $\{x_k\}_{k=0}^{\infty}$ generated by NAG-\texttt{SC} \eqref{eqn: Nesterov_strongly} satisfy
\[
f(x_{k}) - f(x^\star) \leq \frac{5 L \left\| x_{0} - x^{\star} \right\|^{2}}{\left( 1 + \frac{1}{12}\sqrt{\mu/L} \right)^k},
\]
for all $k \ge 0$.
\end{thm}

In brief, the theorem states that $\log(f(x_k) - f(x^\star)) \le - O(k\sqrt{\mu/L})$, 
which matches the optimal rate for minimizing smooth strongly convex functions using 
only first-order information \cite{nesterov2013introductory}. More precisely, 
\cite{nesterov2013introductory} shows that $f(x_k) - f(x^\star) = 
O((1 - \sqrt{\mu/L})^k)$ by taking $s = 1/L$ in NAG-\texttt{SC}. Although 
this optimal rate of NAG-\texttt{SC} is well known in the litetature, this is the 
first Lyapunov-function-based proof of this result.

As indicated in Section \ref{sec:techniques}, the proof of 
Theorem \ref{thm: strongly_NAGM-SC} rests on the discrete Lyapunov 
function \eqref{eqn:lypunov_NAGM-SC_strongly}:
\begin{equation*}
\begin{aligned}
\mathcal{E}(k) =  & \frac{1 + \sqrt{\mu s} }{1 - \sqrt{\mu s} }  \left( f(x_{k}) - f(x^{\star}) \right) + \frac{1}{4} \left\| v_{k} \right\|^{2}  + \frac{1}{4} \left\| v_{k} + \frac{2\sqrt{\mu}}{1 - \sqrt{\mu s}} ( x_{k + 1} - x^{\star} ) + \sqrt{s} \nabla f(x_{k}) \right\|^{2}\\
&  - \frac{s\left\| \nabla f(x_{k})\right\|^{2}}{2(1 - \sqrt{\mu s})}.
\end{aligned}                           
\end{equation*}
Recall that this functional is derived by writing NAG-\texttt{SC} in the 
phase-space representation \eqref{eqn: Nesterov_sc_symplectic}. Analogous 
to Lemma~\ref{lm:nag_sc_e}, the following lemma gives an upper bound on 
the difference $\mathcal{E}(k+1) - \mathcal{E}(k)$.

\begin{lem}[Lyapunov function for NAG-\texttt{SC}]\label{lm:nag_sc_ek1}
Let $f \in \mathcal{S}_{\mu, L}^1(\mathbb{R}^n)$. Taking any step size $0 < s \le 1/(4L)$,  the discrete Lyapunov function~\eqref{eqn:lypunov_NAGM-SC_strongly} with $\{x_{k}\}_{k = 0}^{\infty}$ generated by NAG-\texttt{SC} satisfies
\begin{equation}\nonumber
\mathcal{E}(k + 1) - \mathcal{E}(k) \le - \frac{\sqrt{\mu s}}{6}\mathcal{E}(k + 1).
\end{equation}
\end{lem}

The form of the inequality ensured by Lemma \ref{lm:nag_sc_ek1} is consistent with that of Lemma \ref{lm:nag_sc_e}. Alternatively, it can be written as $\mathcal{E}(k + 1) \le \frac{1}{1+\frac{\sqrt{\mu s}}{6}} \mathcal{E}(k)$. With Lemma \ref{lm:nag_sc_ek1} in place, we give the proof of Theorem \ref{thm: strongly_NAGM-SC}.
\begin{proof}[Proof of Theorem~\ref{thm: strongly_NAGM-SC}]
Given $s = 1/(4L)$, we have
\begin{equation}\label{eq:e_f_diff_rel}
f(x_{k}) - f(x^{\star}) \leq \frac{4 (1- \sqrt{\mu/(4L)} )}{3 + 4\sqrt{\mu/(4L)}}\mathcal{E}(k).
\end{equation}
To see this, first note that
\[
\mathcal{E}(k) \ge  \frac{1 + \sqrt{\mu/(4L)} }{1 - \sqrt{\mu/(4L)} }  \left( f(x_{k}) - f(x^{\star}) \right) - \frac{\left\| \nabla f(x_{k})\right\|^{2}}{8L(1 - \sqrt{\mu/(4L)})}
\]
and
\[
 \frac1{2L} \left\| \nabla f(x_{k})\right\|^2 \le  f(x_k) - f(x^\star).
\]
Combining these two inequalities, we get
\[
\mathcal{E}(k) \ge  \frac{1 + \sqrt{\mu /(4L)} }{1 - \sqrt{\mu /(4L)} }  \left( f(x_{k}) - f(x^{\star}) \right) - \frac{f(x_k) - f(x^{\star})}{4(1 - \sqrt{\mu /(4L)})} = \frac{3 +4 \sqrt{\mu /(4L)}}{4 (1 - \sqrt{\mu /(4L)})} (f(x_k) - f(x^{\star})),
\]
which gives \eqref{eq:e_f_diff_rel}.

Next, we inductively apply Lemma \ref{lm:nag_sc_ek1}, yielding
\begin{equation}\label{eq:e_0_upperx}
\mathcal{E}(k) \le \frac{\mathcal{E}(0)}{\left(1 + \frac{\sqrt{\mu s}}{6}\right)^k} = \frac{\mathcal{E}(0)}{\left( 1 + \frac{1}{12}\sqrt{\mu/L} \right)^k}.
\end{equation}
Recognizing the initial velocity $v_{0} = - \frac{2\sqrt{s} \nabla f(x_{0})}{1 + \sqrt{\mu s}}$ in NAG-\texttt{SC}, one can show that
\begin{equation}\label{eq:eqn_4l_f}
\begin{aligned}
\mathcal{E}(0) &\le \frac{1 + \sqrt{\mu s} }{1 - \sqrt{\mu s}} \left( f(x_{0}) - f(x^{\star})\right) + \frac{s}{(1 + \sqrt{\mu s} )^{2}}  \left\| \nabla f(x_{0})\right\|^{2} \\
&\quad\quad + \frac14\left\| \frac{2\sqrt{\mu}}{1 - \sqrt{\mu s}} (x_{0} - x^{\star}) -\frac{1 + \sqrt{\mu s} }{1 - \sqrt{\mu s} } \sqrt{s} \nabla f(x_{0})\right\|^2\\ 
                        &\le \left[\frac12 \left( \frac{1 + \sqrt{\mu s}}{1 - \sqrt{\mu s}}\right) + \frac{Ls}{(1 + \sqrt{\mu s})^{2}} + \frac{2\mu/L}{(1 - \sqrt{\mu s})^{2}} + \frac{Ls}{2}\left( \frac{1 + \sqrt{\mu s}}{1 - \sqrt{\mu s}}\right)^{2}\right] \cdot L \left\| x_{0} - x^{\star} \right\|^{2}.
\end{aligned}
\end{equation}
Taking $s = 1/(4L)$ in \eqref{eq:eqn_4l_f}, it follows from \eqref{eq:e_f_diff_rel} and \eqref{eq:e_0_upperx} that
\[
f(x_{k}) - f(x^{\star}) \le \frac{C_{\mu/L} \, L \left\| x_{0} - x^{\star} \right\|^{2}}{\left( 1 + \frac{1}{12}\sqrt{\mu/L} \right)^k}.
\]
Here the constant factor $C_{\mu/L}$ is a short-hand for
\[
\frac{4 \left(1 - \sqrt{\mu /(4L)}\right)}{3 +4 \sqrt{\mu /(4L)}} \cdot \left[\frac{1 + \sqrt{\mu/(4L)}}{2 - 2\sqrt{\mu/(4L)}} + \frac{1}{4(1 + \sqrt{\mu/(4L)})^{2}} + \frac{2\mu/L}{(1 - \sqrt{\mu/(4L)})^{2}} + \frac{1}{8}\left( \frac{1 + \sqrt{\mu/(4L)}}{1 - \sqrt{\mu/(4L)}}\right)^{2}\right],
\]
which is less than five by making use of the fact that $\mu/L \le 1$. 
This completes the proof.

\end{proof}

We now turn to the heavy-ball method~(\ref{eqn: polyak_heavy_ball}). 
Recall that $\alpha = \frac{1-\sqrt{\mu s}}{1+\sqrt{\mu s}}$ and 
$x_{1} = x_{0}-\frac{2s\nabla f(x_{0})}{1+\sqrt{\mu s}}$.
\begin{thm}[Convergence of heavy-ball method]
\label{thm: strongly_PHBM}
Let $f \in \mathcal{S}^1_{\mu,L}(\mathbb{R}^n)$. If the step size 
is set to $s = \mu/(16L^2)$, the iterates $\{x_k\}_{k=0}^{\infty}$ 
generated by the heavy-ball method satisfy
\[
f(x_{k}) - f(x_{0}) \leq  \frac{5L\left\| x_{0} - x^{\star} \right\|^{2}}{\left( 1 +  \frac{\mu}{16L} \right)^k}
\]
for all $k \ge 0$.
\end{thm}

The heavy-ball method minimizes the objective at the rate 
$\log(f(x_k) - f(x^\star)) \le - O(k\mu/L)$, as opposed to the optimal 
rate $-O(k\sqrt{\mu/L})$ obtained by NAG-\texttt{SC}. Thus, the acceleration 
phenomenon is not observed in the heavy-ball method for minimizing functions 
in the class $\mathcal{S}^1_{\mu,L}(\mathbb{R}^n)$. This difference is, on 
the surface, attributed to the much smaller step size $s = \mu/(16L^2)$ in 
Theorem~\ref{thm: strongly_PHBM} than the ($s = 1/(4L)$) in 
Theorem~\ref{thm: strongly_NAGM-SC}. Further discussion of this difference 
is given after Lemma~\ref{lm:heavy_dis_add} and in Section~\ref{sec:insights-into-accel}.

In addition to allowing us to complete the proof of Theorem \ref{thm: strongly_PHBM}, 
Lemma \ref{lm:heavy_dis_add} will shed light on why the heavy-ball method needs a 
more conservative step size. To state this lemma, we consider the discrete Lyapunov 
function defined as
\begin{equation}\label{eqn: lypunov_PHBM_strongly}
\mathcal{E}(k) = \frac{1 + \sqrt{\mu s} }{1 - \sqrt{\mu s} } \left( f(x_{k}) - f(x^{\star}) \right) + \frac{1}{4} \left\| v_{k} \right\|^{2} + \frac{1}{4} \left\|v_{k} + \frac{2\sqrt{\mu}}{1 - \sqrt{\mu s}}(x_{k + 1} - x^{\star}) \right\|^{2},
\end{equation}
which is derived by discretizing the continuous Lyapunov function 
\eqref{eqn: energy_heavy-b_ode} using the phase-space representation of 
the heavy-ball method:
\begin{equation}
\label{eqn: polyak_heavy_ball_symplectic}
\begin{aligned}
 & x_{k} - x_{k - 1} =  \sqrt{s} v_{k - 1} \\
 & v_{k} - v_{k - 1} = - \frac{ 2 \sqrt{\mu s} }{ 1 - \sqrt{\mu s} }  
v_{k}  -  \frac{ 1 + \sqrt{\mu s}  }{1 - \sqrt{\mu s}  } \cdot \sqrt{s} \nabla f( x_{k} ).
\end{aligned}
\end{equation} 

\begin{lem}[Lyapunov function for the heavy-ball method]\label{lm:heavy_dis_add}
Let $f \in \mathcal{S}_{\mu,L}^{1}(\mathbb{R}^n)$. For any step size $s > 0$, the discrete Lyapunov function~\eqref{eqn: lypunov_PHBM_strongly} with $\{x_{k}\}_{k = 0}^{\infty}$ generated by the heavy-ball method satisfies
\begin{equation}\label{eq:e_diff_neg_sec}
\begin{aligned}
\mathcal{E}(k+1) - &\mathcal{E}(k) \le - \sqrt{\mu s} \min\left\{ \frac{1 - \sqrt{\mu s}}{1 + \sqrt{\mu s}}, \frac{1}{4}\right\}\mathcal{E}(k+1) \\
&- \Bigg[ \frac{3\sqrt{\mu s}}{4} \left( \frac{1 + \sqrt{\mu s} }{1 - \sqrt{\mu s} } \right) \left( f(x_{k + 1}) - f(x^{\star}) \right)  - \frac{s}{2} \left( \frac{1 + \sqrt{\mu s} }{1 - \sqrt{\mu s} } \right)^2 \left\| \nabla f(x_{k + 1}) \right\|^{2} \Bigg].
\end{aligned}
\end{equation}
\end{lem}

The proof of Lemma~\ref{lm:heavy_dis_add} can be found in 
Appendix~\ref{subsec: heavy_ball_discrete}. To apply this lemma to 
prove Theorem~\ref{thm: strongly_PHBM}, we need to ensure
\begin{align}\label{eqn: condition_strongly_PHBM}
\frac{3\sqrt{\mu s}}{4}\left( \frac{1 + \sqrt{\mu s} }{1 - \sqrt{\mu s} } \right) \left( f(x_{k + 1}) - f(x^{\star}) \right)  - \frac{s}{2} \left( \frac{1 + \sqrt{\mu s} }{1 - \sqrt{\mu s} } \right)^{2}\left\| \nabla f(x_{k + 1}) \right\|^{2} \ge 0.
\end{align}
A sufficient and necessary condition for \eqref{eqn: condition_strongly_PHBM} is
\begin{equation}\label{eqn: condition_strongly_PHBM_22}
\frac{3\sqrt{\mu s}}{4}\left( f(x_{k + 1}) - f(x^{\star}) \right) -  \left( \frac{1 + \sqrt{\mu s} }{1 - \sqrt{\mu s} } \right) s L  \left( f(x_{k + 1}) - f(x^{\star}) \right) \geq 0.
\end{equation}
This is because $\left\| \nabla f(x_{k + 1}) \right\|^{2} \leq 2L \left( f(x_{k + 1}) - f(x^{\star}) \right)$, which can be further reduced to an equality (for example, $f(x) = \frac{L}{2} \|x\|^2$). Thus, the step size $s$ must obey
\[
s = O\left(\frac{\mu}{L^2} \right).
\]
In particular, the choice of $s = \frac{\mu}{16L^2}$ fulfills \eqref{eqn: condition_strongly_PHBM_22} and, as a consequence, Lemma \ref{lm:heavy_dis_add} implies
\[
\mathcal{E}(k+1) - \mathcal{E}(k) \le -\frac{\mu}{16 L} \mathcal{E}(k+1).
\]
The remainder of the proof of Theorem \ref{thm: strongly_PHBM} is similar to 
that of Theorem \ref{thm: strongly_NAGM-SC} and is therefore omitted. As an 
aside, \cite{polyak1964some} uses $s = 4 /(\sqrt{L} + \sqrt{\mu})^{2}$ for 
\textit{local} accelerated convergence of the heavy-ball method. This choice 
of step size is larger than our step size $s = \frac{\mu}{16 L^2}$, which 
yields a non-accelerated but global convergence rate.

The term $\frac{s}{2} \left( \frac{1 + \sqrt{\mu s} }{1 - \sqrt{\mu s} } \right)^2 \left\| \nabla f(x_{k + 1}) \right\|^{2}$ in \eqref{eq:e_diff_neg_sec} that arises from 
finite differencing of \eqref{eqn: lypunov_PHBM_strongly} is a (small) term 
of order $O(s)$ and, as a consequence, this term is not reflected in 
Lemma \ref{lm:heavyball_e}. In relating to the case of NAG-\texttt{SC}, 
one would be tempted to ask why this term does not appear in 
Lemma \ref{lm:nag_sc_ek1}. In fact, a similar term can be found in 
$\mathcal{E}(k+1) - \mathcal{E}(k)$ by taking a closer look at the 
proof of Lemma \ref{lm:nag_sc_ek1}. However, this term is canceled out 
by the discrete version of the quadratic term 
$\frac{\sqrt{s}}{2} (\left\| \nabla f(X)\right\|^{2} + \dot{X}^\top \nabla^{2}f(X)\dot{X} )$ 
in Lemma \ref{lm:nag_sc_e} and is, therefore, not present in the statement of 
Lemma \ref{lm:nag_sc_ek1}. Note that this quadratic term results from the 
gradient correction (see Remark \ref{rem:gc_effect}). In light of the above, 
the gradient correction is the key ingredient that allows for a larger step 
size in NAG-\texttt{SC}, which is necessary for achieving acceleration.

For completeness, we finish Section~\ref{sec:discrete-time} by proving Lemma 
\ref{lm:nag_sc_ek1}.
\begin{proof}[Proof of Lemma \ref{lm:nag_sc_ek1}]
Using the Cauchy--Schwarz inequality, we have\footnote{See the definition of $\mathbf{III}$ in \eqref{eqn:lypunov_NAGM-SC_strongly}.}
\begin{align*}
\mathbf{III} = & \frac14 \left\|  \left(\frac{1 + \sqrt{\mu s} }{1 - \sqrt{\mu s}}\right) v_{k} + \frac{2\sqrt{\mu}}{1 - \sqrt{\mu s}} ( x_{k} - x^{\star} ) +   \sqrt{s} \nabla f(x_{k}) \right\|^{2} \\
	\leq & \frac34 \left[  \left( \frac{1 + \sqrt{\mu s} }{1 - \sqrt{\mu s} }\right)^{2} \left\| v_{k} \right\|^{2} + \frac{ 4 \mu }{(1 - \sqrt{\mu s})^{2}}  \left\| x_{k} - x^{\star} \right\|^{2} + s \left\| \nabla f(x_{k}) \right\|^{2}\right],
\end{align*}
which, together with the inequality
\[
\begin{aligned}
\frac{3s}{4}\left\| \nabla f(x_{k}) \right\|^{2} - \frac{s\left\| \nabla f(x_{k})\right\|^{2}}{2(1 - \sqrt{\mu s})} &= \frac{s}{4}\left\| \nabla f(x_{k}) \right\|^{2} + \frac{s}{2}\left\| \nabla f(x_{k}) \right\|^{2} - \frac{s\left\| \nabla f(x_{k})\right\|^{2}}{2(1 - \sqrt{\mu s})}\\
&\le \frac{Ls}{2} \left( f(x_{k}) - f(x^{\star}) \right) - \frac{s\sqrt{\mu s}\left\| \nabla f(x_{k})\right\|^{2}}{2(1 - \sqrt{\mu s})},
\end{aligned}
\]
for $f \in \mathcal{S}_{\mu, L}^{1}(\mathbb{R}^n)$, shows that the Lyapunov function~(\ref{eqn:lypunov_NAGM-SC_strongly}) satisfies
\begin{equation}\label{eqn: nag-sc_energy_estimate}
\begin{aligned}
\mathcal{E}(k) \leq &  \left( \frac{ 1}{1 - \sqrt{\mu s}} + \frac{Ls}{2}\right) \left( f(x_{k}) - f(x^{\star}) \right) + \frac{1 + \sqrt{\mu s} + \mu s }{ (1 - \sqrt{\mu s} )^{2} } \left\| v_{k} \right\|^{2} \\
	&  + \frac{3\mu}{ (1 - \sqrt{\mu s} )^{2} } \left\| x_{k} - x^{\star} \right\|^{2} + \frac{ \sqrt{\mu s} }{1 - \sqrt{\mu s}}  \left( f(x_{k}) - f(x^{\star}) - \frac{s}{2} \left\| \nabla f(x_{k})\right\|^{2} \right).
\end{aligned}
\end{equation}

Next, as shown in Appendix~\ref{subsec:proof-lemma-nsc}, the inequality
\begin{equation}\label{eq:proof_app_is}
\begin{aligned}
\mathcal{E}&(k + 1) - \mathcal{E}(k)  \leq  - \sqrt{\mu s} \left[ \frac{1 - 2Ls}{ \left( 1 - \sqrt{\mu s}  \right)^{2} }  \left( f(x_{k + 1}) - f(x^{\star})  \right) +  \frac{1}{ 1 - \sqrt{\mu s} } \left\| v_{k + 1} \right\|^{2}   \right. \\
	&\left. + \frac{\mu}{2(1 - \sqrt{\mu s} )^{2}} \left\| x_{k + 1} - x^{\star} \right\|^{2} +  \frac{\sqrt{\mu s} }{(1 - \sqrt{\mu s})^{2}}  \left( f(x_{k + 1}) - f(x^{\star}) - \frac{s}{2} \left\| \nabla f(x_{k + 1}) \right\|^{2} \right)\right]
\end{aligned}
\end{equation}
holds for $s \leq 1/(2L)$. Comparing the coefficients of the same terms 
in \eqref{eqn: nag-sc_energy_estimate} for $\mathcal{E}(k+1)$ and 
\eqref{eq:proof_app_is}, we conclude that the first difference of the 
discrete Lyapunov function~(\ref{eqn:lypunov_NAGM-SC_strongly}) must satisfy
\begin{align*}
	\mathcal{E}(k + 1) - \mathcal{E}(k)
	& \leq  - \sqrt{\mu s} \min\left\{ \frac{1 - 2Ls}{1 - \sqrt{\mu s} + \frac{Ls}{2}\left(1 - \sqrt{\mu s}\right)^{2} },  \frac{1 -  \sqrt{\mu s}}{1 + \sqrt{\mu s} + \mu s}, \frac{1}{6}, \frac{1}{1 - \sqrt{\mu s}}\right\} \mathcal{E}(k + 1) \\
	& \leq  - \sqrt{\mu s} \min\left\{ \frac{1 - 2Ls}{1 + \frac{Ls}{2} }, \frac{1 -  \sqrt{\mu s}}{1 + \sqrt{\mu s} + \mu s}, \frac{1}{6}, \frac{1}{1 - \sqrt{\mu s}}\right\} \mathcal{E}(k + 1) \\
	& =- \frac{\sqrt{\mu s}}{6}\mathcal{E}(k + 1),
\end{align*}
since $s \leq 1/(4L)$.
\end{proof}

\subsection{A Numerical Stability Perspective on Acceleration}
\label{sec:insights-into-accel}

As shown in Section \ref{sec:discrete-time}, the gradient correction is the 
fundamental cause of the difference in convergence rates between the heavy-ball 
method and NAG-\texttt{SC}. This section aims to further elucidate this distinction 
from the viewpoint of numerical stability. A numerical scheme is said to be stable 
if, roughly speaking, this scheme does not magnify errors in the input data. 
Accordingly, we address the question of what values of the step size $s$ are 
allowed for solving the high-resolution ODEs \eqref{eqn: heavy_ball_first} and 
\eqref{eqn: nag-sc_first} in a stable fashion. While various discretization 
schemes on low-resolution ODEs have been explored in \cite{wibisono2016variational,
wilson2016lyapunov,zhang2018direct}, we limit our attention to the forward 
Euler scheme to simplify the discussion (see \cite{stoer2013introduction} 
for an exposition on discretization schemes).

For the heavy-ball method, the forward Euler scheme applied to 
\eqref{eqn: heavy_ball_first} is
\begin{equation}\label{eq:euler_heavy}
\frac{X(t + \sqrt{s}) - 2X(t) + X(t - \sqrt{s})}{s} + 2\sqrt{\mu} \cdot \frac{X(t) - X(t - \sqrt{s})}{\sqrt{s}} + (1 + \sqrt{\mu s} ) \nabla f(X(t - \sqrt{s})) = 0.
\end{equation}
Using the approximation $\nabla f(X(t-\sqrt{s}) + \epsilon) \approx 
\nabla f(X(t-\sqrt{s})) + \nabla^2 f(X(t-\sqrt{s})) \epsilon$ for a 
small perturbation $\epsilon$, we get the characteristic equation of 
\eqref{eq:euler_heavy}:
\[
\det \left(\lambda^2\bm{I} - (2 - 2\sqrt{\mu s} )\lambda \bm{I} + ( 1 - 2\sqrt{\mu s}) \bm{I}  + (1+\sqrt{\mu s} ) s \nabla^2 f(X(t-\sqrt{s})) \right)= 0,
\]
where $\bm{I}$ denotes the $n \times n$ identity matrix. The numerical stability 
of \eqref{eq:euler_heavy} requires the roots of the characteristic equation 
to be no larger than one in absolute value.  Therefore, a necessary condition 
for the stability is that\footnote{The notation $ A \preceq B$ indicates that 
$B - A$ is positive semidefinite for symmetric matrices $A$ and $B$.}
\begin{equation}\label{eq:cha_root}
(1 - 2\sqrt{\mu s})\bm{I}  + (1+\sqrt{\mu s} ) s \nabla^2 f(X(t-\sqrt{s})) \preceq \bm{I}.
\end{equation}
By the $L$-smoothness of $f$, the largest singular value of $\nabla^2 f(X(t - \sqrt{s}))$ can be as large as $L$. Therefore, \eqref{eq:cha_root} is guaranteed in the worst case analysis only if
\[
(1+\sqrt{\mu s} ) s L \le 2\sqrt{\mu s},
\]
which shows that the step size must obey
\begin{equation}\label{eq:num_l_mu}
s \le O \left( \frac{\mu}{L^{2}}\right).
\end{equation}

Next, we turn to the high-resolution ODE \eqref{eqn: nag-sc_first} 
of NAG-\texttt{SC}, for which the forward Euler scheme reads
\begin{equation}\label{eq:euler_sc}
\begin{aligned}
\frac{X(t + \sqrt{s}) - 2X(t) + X(t-\sqrt{s})}{s} &+ (2\sqrt{\mu} + \sqrt{s} \nabla^2 f(X(t-\sqrt{s}))) \cdot \frac{X(t) - X(t-\sqrt{s})}{\sqrt{s}}\\ 
&+ (1 + \sqrt{\mu s} ) \nabla f(X(t-\sqrt{s})) = 0.
\end{aligned}
\end{equation}
Its characteristic equation is
\begin{equation}\nonumber
\det \left(\lambda^2 \bm{I} - (2 - 2\sqrt{\mu s}  - s \nabla^2 f(X(t-\sqrt{s})))\lambda \bm{I} + (1 - 2\sqrt{\mu s}) \bm{I} + \sqrt{\mu s^{3}} \nabla^2 f(X(t-\sqrt{s})) \right)= 0,
\end{equation}
which, as earlier, suggests that the numerical stability condition of \eqref{eq:euler_sc} is
\[
(1 - 2\sqrt{\mu s}) \bm{I} + \sqrt{\mu s^{3}} \nabla^2 f(X(t-\sqrt{s})) \preceq \bm{I}.
\]
This inequality is ensured by setting the step size
\begin{equation}\label{eq:num_l1}
s = O \left( \frac1{L}\right).
\end{equation}

As constraints on the step sizes, both \eqref{eq:num_l_mu} and \eqref{eq:num_l1} 
are in agreement with the discussion in Section \ref{sec:discrete-time}, albeit 
from a different perspective. In short, a comparison between \eqref{eq:euler_heavy} 
and \eqref{eq:euler_sc} reveals that the Hessian $\sqrt{s} \nabla^2 f(X(t-\sqrt{s}))$ 
makes the forward Euler scheme for the NAG-\texttt{SC} ODE numerically stable with 
a larger step size, namely $s = O(1/L)$. This is yet another reflection of the 
vital importance of the gradient correction in yielding acceleration for 
NAG-\texttt{SC}.


\section{Gradient Correction for Gradient Norm Minimization}
\label{sec:nagm-c_analysis}

In this section, we extend the use of the high-resolution ODE framework 
to NAG-\texttt{C} \eqref{eqn:nagm-c} in the setting of minimizing an $L$-smooth 
convex function $f$. The main result is an  improved rate of NAG-\texttt{SC} 
for minimizing the squared gradient norm. Indeed, we show that NAG-\texttt{C} 
achieves the $O(L^2/k^3)$ rate of convergence for minimizing $\|\nabla f(x_k)\|^2$. 
To the best of our knowledge, this is the \textit{sharpest} known bound for 
this problem using NAG-\texttt{C} \textit{without} any modification. 
Moreover, we will show that the gradient correction in NAG-\texttt{C} is 
responsible for this rate and, as it is therefore unsurprising that this 
inverse cubic rate was not perceived within the low-resolution ODE frameworks
such as that of \cite{su2016differential}. In Section~\ref{sec:new-accel-meth}, 
we propose a new accelerated method with the same rate $O(L^2/k^3)$ and briefly 
discuss the benefit of the phase-space representation in simplifying technical 
proofs.

\subsection{The ODE Case}
\label{sec:continuous-time}
We begin by studying the high-resolution ODE \eqref{eqn: nag-c_first} corresponding
to NAG-\texttt{C} with an objective $f \in \mathcal{F}_{L}^2(\mathbb{R}^{n})$ and 
an arbitrary step size $s > 0$. For convenience, let $t_0 = 1.5 \sqrt{s}$.

\begin{thm}\label{thm: first-order_NAGM-C_ode}
Assume $f \in \mathcal{F}_{L}^2(\mathbb{R}^n)$ and let $X = X(t)$ be the solution to the ODE \eqref{eqn: nag-c_first}. The squared gradient norm satisfies
\begin{equation}\nonumber
\inf_{t_0 \leq u \leq t} \left\| \nabla f(X(u))\right\|^{2} \le \frac{(12 + 9 sL)\|x_{0} - x^\star\|^2}{2\sqrt{s} (t^3 - t_{0}^3)},
\end{equation}
for all $t > t_0$.
\end{thm}
By taking the step size $s = 1/L$, this theorem shows that
\[
\inf_{t_0 \leq u \leq t} \left\| \nabla f(X(u))\right\|^2 = O(\sqrt{L}/t^3),
\] 
where the infimum operator is necessary as the squared gradient norm is 
generally not decreasing in $t$. In contrast, directly combining the convergence 
rate of the function value (see Corollary \ref{coro:c_f_bound}) and inequality 
$\|\nabla f(X)\|^2 \le 2L (f(X) - f(x^\star))$ only gives a $O(L/t^2)$ rate for 
squared gradient norm minimization.

The proof of the theorem is based on the continuous Lyapunov function
\begin{equation}\label{eqn: lypunov_NAGM-C_first-order_ode2}
\mathcal{E}(t) = t\left(t + \frac{\sqrt{s}}{2}\right) \left( f(X) - f(x^{\star}) \right) + \frac{1}{2} \| t \dot{X} + 2 (X - x^{\star}) + t\sqrt{s}\nabla f(X) \|^{2},
\end{equation}
which reduces to the continuous Lyapunov function in \cite{su2016differential} 
when setting $s = 0$.

\begin{lem}\label{lm:c_ener_neg}
Let $f \in \mathcal{F}_{L}^2(\mathbb{R}^n)$. The Lyapunov function defined in 
\eqref{eqn: lypunov_NAGM-C_first-order_ode2} with $X = X(t)$ being the solution 
to the ODE \eqref{eqn: nag-c_first} satisfies
\begin{equation}\label{eqn: estimate_derivative_ode_nag-c}
\frac{\dd \mathcal{E}(t)}{\dd t}   \le  -\left[ \sqrt{s} t^2 + \left( \frac{1}{L} +\frac{s}{2} \right)t + \frac{\sqrt{s}}{2L}  \right] \left\| \nabla f(X) \right\|^{2}
\end{equation}
for all $t \ge t_0$.
\end{lem}

The decreasing rate of $\mathcal{E}(t)$ as specified in the lemma is sufficient 
for the proof of Theorem \ref{thm: first-order_NAGM-C_ode}. 
First, note that Lemma \ref{lm:c_ener_neg} readily gives
\[
\begin{aligned}
\int_{t_0}^{t}   \left[ \sqrt{s} u^2 + \left( \frac{1}{L} +\frac{s}{2} \right)u + \frac{\sqrt{s}}{2L}  \right] \left\| \nabla f(X(u)) \right\|^{2} \dd u &\le -\int_{t_0}^{t} \frac{\dd \mathcal{E}(u)}{\dd u} \dd u\\
&= \mathcal{E}(t_0) - \mathcal{E}(t) \\
& \le \mathcal{E}(t_0),
\end{aligned}
\]
where the last step is due to the fact $\mathcal{E}(t) \ge 0$. Thus, it follows that
\begin{equation}\label{eq:nabla_f_inf}
\begin{aligned}
\inf_{t_0 \leq u \leq t} \left\| \nabla f(X(u))\right\|^{2} &\le \frac{\int_{t_0}^{t}   \left[ \sqrt{s} u^2 + \left( \frac{1}{L} +\frac{s}{2} \right)u + \frac{\sqrt{s}}{2L}  \right] \left\| \nabla f(X(u)) \right\|^{2} \dd u}{\int_{t_0}^{t}\sqrt{s} u^2 + \left( \frac{1}{L} +\frac{s}{2} \right)u + \frac{\sqrt{s}}{2L} \dd u}\\
& \le \frac{\mathcal{E}(t_0)}{\sqrt{s} (t^3 - t_0^3)/3 + \left( \frac{1}{L} +\frac{s}{2} \right)(t^2 - t_0^2)/2 + \frac{\sqrt{s}}{2L} (t - t_0)}.
\end{aligned}
\end{equation}
Recognizing the initial conditions of the ODE \eqref{eqn: nag-c_first}, we get
\[
\begin{aligned}
\mathcal{E}(t_0) &= t_0(t_0 + \sqrt{s}/2) (f(x_0) - f(x^\star)) + \frac12 \left\| - t_0 \sqrt{s} \nabla f(x_0) + 2(x_0 - x^\star)  +  t_0 \sqrt{s} \nabla f(x_0) \right\|^2\\
&\le 3s \cdot \frac{L}{2} \|x_0 - x^\star\|^2 + 2\left\| x_0 - x^\star\right\|^2,
\end{aligned}
\]
which together with \eqref{eq:nabla_f_inf} gives
\begin{equation}\label{eq:grad_ode_full}
\inf_{t_0 \leq u \leq t} \left\| \nabla f(X(u))\right\|^{2} \le  \frac{(2 + 1.5sL)\left\| x_0 - x^\star\right\|^2}{\sqrt{s} (t^3 - t_0^3)/3 + \left( \frac{1}{L} +\frac{s}{2} \right)(t^2 - t_0^2)/2 + \frac{\sqrt{s}}{2L} (t - t_0)}.
\end{equation}
This bound reduces to the one claimed by Theorem \ref{thm: first-order_NAGM-C_ode} 
by only keeping the first term $\sqrt{s} (t^3 - t_0^3)/3$ in the denominator.

The gradient correction $\sqrt{s} \nabla^{2} f(X) \dot{X}$ in the high-resolution 
ODE~\eqref{eqn: nag-c_first} plays a pivotal role in Lemma~\ref{lm:c_ener_neg} and 
is, thus, key to Theorem~\ref{thm: first-order_NAGM-C_ode}. As will be seen in the 
proof of the lemma, the factor $\left\| \nabla f(X) \right\|^{2}$ in 
\eqref{eqn: estimate_derivative_ode_nag-c} results from the term 
$t\sqrt{s}\nabla f(X)$ in the Lyapunov function 
\eqref{eqn: lypunov_NAGM-C_first-order_ode2}, which arises from the 
gradient correction in the ODE~\eqref{eqn: nag-c_first}. In light of this, 
the low-resolution ODE~\eqref{eqn:ode_old_nagmc} of NAG-\texttt{C} cannot 
yield a result similar to Lemma~\ref{lm:c_ener_neg} and; furthermore, we 
conjecture that the $O(\sqrt{L}/t^3)$ rate does applies to this ODE. 
Section~\ref{sec:discrete-case} will discuss this point further in the discrete case.

In passing, it is worth pointing out that the analysis above applies to the case of $s = 0$. In this case, we have $t_0 = 0$, and \eqref{eq:grad_ode_full} turns out to be
\[
\inf_{0 \leq u \leq t} \left\| \nabla f(X(u))\right\|^{2} \le  \frac{4L\left\| x_0 - x^\star\right\|^2}{t^2}.
\]
This result is similar to that of the low-resolution ODE in \cite{su2016differential}\footnote{To see this, recall that \cite{su2016differential} shows that $f(X(t)) - f(x^\star) \le \frac{2\|x_0 - x^\star\|^2}{t^2}$, where $X = X(t)$ is the solution to \eqref{eq:grad_ode_full} with $s = 0$. Using the $L$-smoothness of $f$, we get $\|\nabla f(X(t))\|^2 \le 2L(f(X(t)) - f(x^\star)) \le \frac{4L\|x_0 - x^\star\|^2}{t^2}$.}.

This section is concluded with the proof of Lemma \ref{lm:c_ener_neg}.
\begin{proof}[Proof of Lemma \ref{lm:c_ener_neg}]
The time derivative of the Lyapunov function~(\ref{eqn: lypunov_NAGM-C_first-order_ode2}) obeys
\begin{align*}
\frac{\dd \mathcal{E}(t)}{\dd t} & = \left(2t + \frac{\sqrt{s}}{2}\right) \left( f(X) - f(x^{\star}) \right) +t \left(t  +\frac{\sqrt{s}}{2} \right)\left\langle \nabla f(X), \dot{X} \right\rangle \\
                                         &\quad + \left\langle t\dot{X} + 2(X - x^{\star}) +  t\sqrt{s} \nabla f(X), -\left(\frac{\sqrt{s}}{2} + t\right) \nabla f(X) \right\rangle    \\
                                         & = \left(2t + \frac{ \sqrt{s}}{2}\right) \left( f(X) - f(x^{\star}) \right) - (\sqrt{s} + 2t) \left\langle X - x^{\star}, \nabla f(X) \right\rangle \\
                                         & \quad - \sqrt{s}t \left( t + \frac{\sqrt{s}}{2} \right)\left\| \nabla f(X) \right\|^{2}. 
\end{align*}
Making use of the basic inequality $f(x^{\star}) \geq f(X) + \left\langle \nabla f(X), x^{\star} - X \right\rangle + \frac{1}{2L} \left\| \nabla f(X) \right\|^{2}$ for $L$-smooth $f$, the expression of $\frac{\dd \mathcal{E}}{\dd t}$ above satisfies
\begin{align*}
\frac{\dd \mathcal{E}}{\dd t}   & \leq -\frac{\sqrt{s}}{2} \left( f(X) - f(x^{\star}) \right) - \left(\sqrt{s}t + \frac{1}{L}\right)\left(t + \frac{\sqrt{s}}{2}\right)\left\| \nabla f(X) \right\|^{2} \\
                                                  & \leq  - \left(\sqrt{s}t + \frac{1}{L}\right)\left(t + \frac{\sqrt{s}}{2}\right)\left\| \nabla f(X) \right\|^{2} \\
                                                  & =      - \left[ \sqrt{s} t^2 + \left( \frac{1}{L} +\frac{s}{2} \right)t + \frac{\sqrt{s}}{2L}  \right]\left\| \nabla f(X) \right\|^{2} .
\end{align*}
\end{proof}

Note that Lemma \ref{lm:c_ener_neg} shows $\mathcal{E}(t)$ is a decreasing 
function, from which we get
\begin{align*}
f(X) - f(x^{\star}) \leq \frac{\mathcal{E}(t_0)}{t \left(t + \frac{\sqrt{s}}{2}\right) }  =  \frac{3s (f(x_0) - f(x^{\star})) + 2 \left\|x_0 - x^{\star} \right\|^{2}}{t\left(t + \frac{\sqrt{s}}{2}\right) }
\end{align*}
by recognizing the initial conditions of the high-resolution ODE \eqref{eqn: nag-c_first}. 
This gives the following corollary.
\begin{coro}\label{coro:c_f_bound}
Under the same assumptions as in Theorem \ref{thm: first-order_NAGM-C_ode}, 
for any $t > t_0$, we have
\begin{align}\nonumber
f(X(t)) - f(x^{\star}) \leq  \frac{(4 + 3sL)\left\| x_{0} - x^{\star} \right\|^{2}}{t\left(2t + \sqrt{s}\right)}.
\end{align}
\end{coro}

\subsection{The Discrete Case}
\label{sec:discrete-case}

We now turn to the discrete NAG-\texttt{C} \eqref{eqn:nagm-c} for minimizing 
an objective $f \in \mathcal{F}_L^1(\mathbb{R}^{n})$. Recall that this 
algorithm starts from any $x_0$ and $y_0 = x_0$. The discrete counterpart of 
Theorem \ref{thm: first-order_NAGM-C_ode} is as follows.

\begin{thm}\label{thm:NAGM-C_original}
Let $f \in \mathcal{F}_{L}^1(\mathbb{R}^n)$. For any step size $0 < s \leq 1/(3L)$, the iterates $\left\{ x_{k} \right\}_{k = 0}^{\infty}$ generated by NAG-\texttt{C} obey
\[
\min_{0 \leq i \leq k}\left\| \nabla f(x_{i}) \right\|^{2} \leq \frac{8568  \left\| x_{0} - x^{\star}\right\|^{2}}{s^{2}(k + 1)^{3}},
\]
for all $k \ge 0$.  In additional, we have
\[
f(x_{k}) - f(x^{\star}) \leq \frac{119 \left\| x_{0} - x^{\star} \right\|^{2}}{s (k + 1)^{2}},
\]
for all $k \geq 0$.
\end{thm}

Taking $s = 1/(3L)$, Theorem \ref{thm:NAGM-C_original} shows that NAG-\texttt{C} 
minimizes the squared gradient norm at the rate $O(L^2/k^3)$. This theoretical 
prediction is in agreement with two numerical examples illustrated in 
Figure \ref{fig:grad_norm}. To our knowledge, the bound $O(L^2/k^3)$ is 
sharper than any existing bounds in the literature for NAG-\texttt{C} for 
squared gradient norm minimization. In fact, the convergence result 
$f(x_k) - f(x^\star) = O(L/k^2)$ for NAG-\texttt{C} and the $L$-smoothness 
of the objective immediately give $\|\nabla f(x_k)\|^2 \le O(L^2/k^2)$. 
This well-known but loose bound can be improved by using a recent result 
from \cite{attouch2016rate}, which shows that a slightly modified version
NAG-\texttt{C} satisfies $f(x_k) - f(x^\star) = o(L/k^2)$ (see 
Section \ref{sec:super-crit-frict} for more discussion of this 
improved rate). This reveals
\[
\|\nabla f(x_k)\|^2 \le o\left( \frac{L^{2}}{k^2} \right),
\]
which, however, remains looser than that of Theorem \ref{thm:NAGM-C_original}. 
In addition, the rate $o(L^{2}/k^2)$ is not valid for $k \le n/2$ and, 
as such, the bound $o(L^{2}/k^2)$ on the squared gradient norm is 
\textit{dimension-dependent} \cite{attouch2016rate}. For completeness, 
the rate $O(L^{2}/k^{3})$ can be achieved by introducing an additional 
sequence of iterates and a more aggressive step size policy in a variant 
of NAG-\texttt{C} \cite{ghadimi2016accelerated}. In stark contrast, 
our result shows that no adjustments are needed for NAG-\texttt{C} 
to yield an accelerated convergence rate for minimizing the gradient norm.

\begin{figure}[htp!]
\begin{minipage}[t]{0.5\linewidth}
\centering
\includegraphics[width=3.2in]{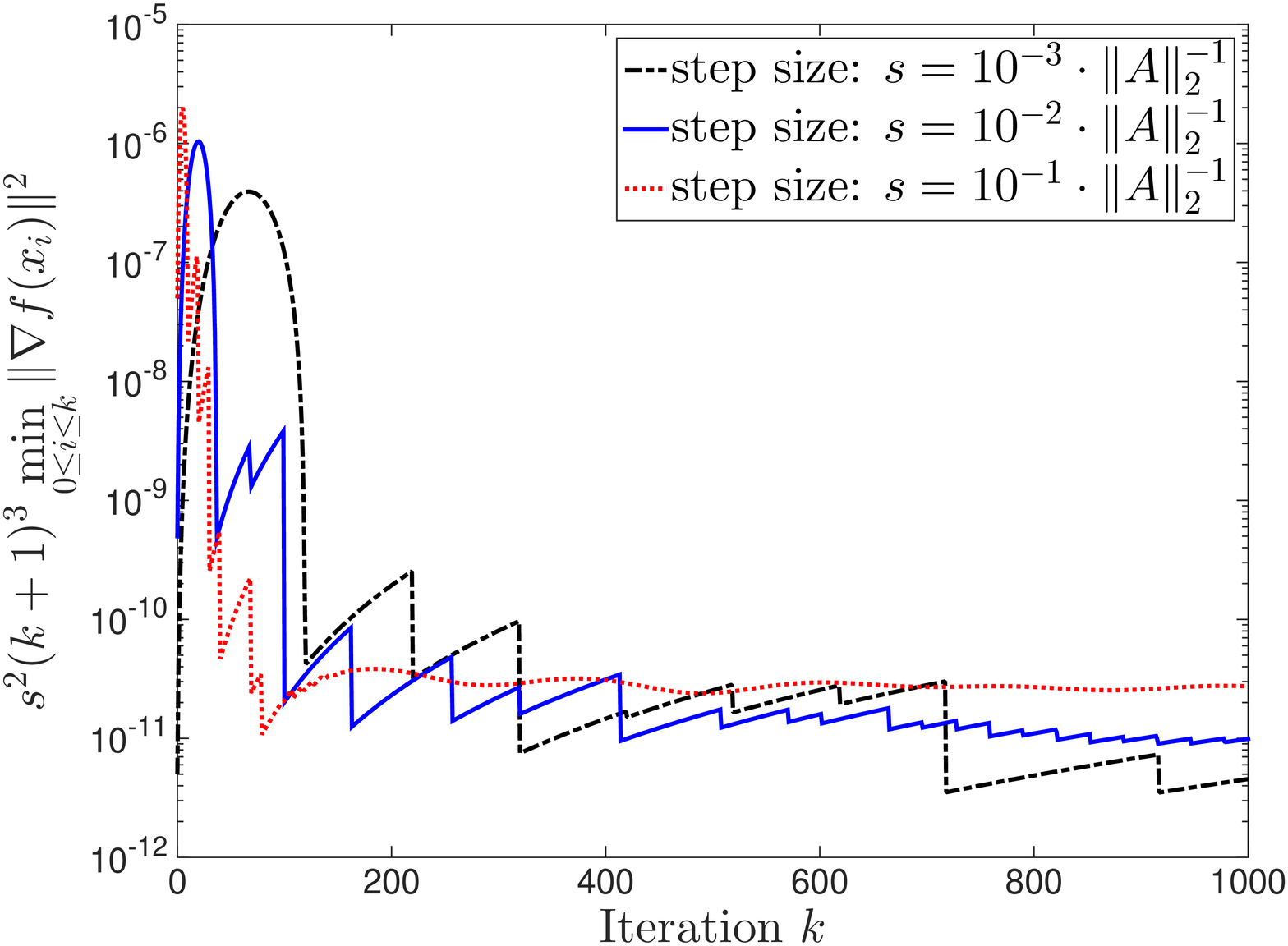}
\end{minipage}
\begin{minipage}[t]{0.5\linewidth}
\centering
\includegraphics[width=3.2in]{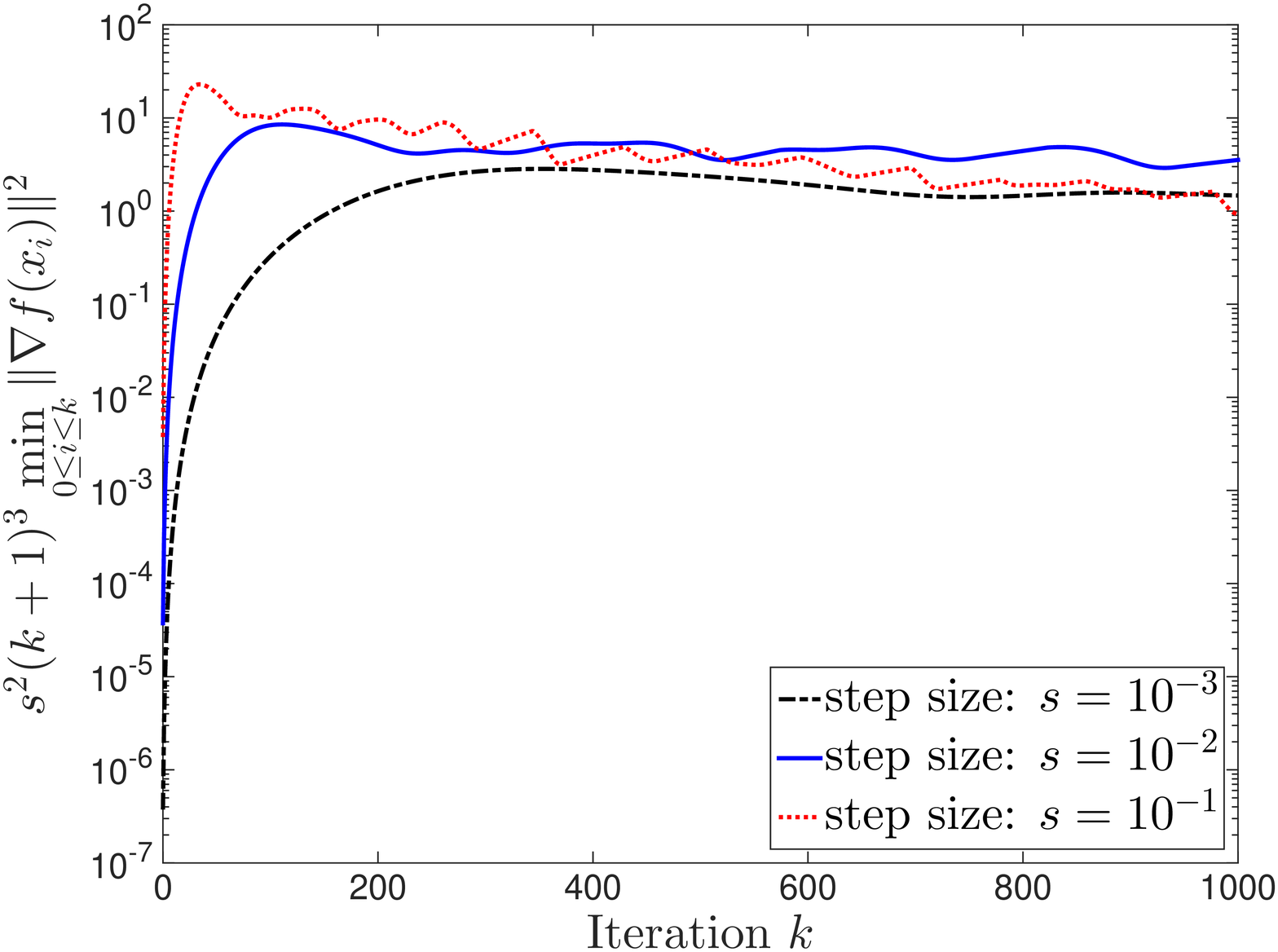}
\end{minipage}
\caption{Scaled squared gradient norm $s^2(k+1)^3 \min_{0 \le i \le k} \|\nabla f(x_i)\|^2$ of NAG-\texttt{C}. In both plots, the scaled squared gradient norm stays bounded as $k \rightarrow \infty$. Left: $f(x) = \frac12\left\langle Ax, x\right\rangle + \left\langle b, x\right\rangle $, where $A = T' T$ is a $500 \times 500$ positive semidefinite matrix and $b$ is $1 \times 500$.  All entries of $b, \; T \in \mathbb{R}^{500 \times 500}$ are i.i.d.~uniform random variables on $(0, 1)$, and $\|\cdot\|_{2}$ denotes the matrix spectral norm. Right: $f(x) = \rho \log \left\{ \sum\limits_{i = 1}^{200} \exp \left[\left( \left\langle a_{i}, x\right\rangle - b_{i} \right)/\rho\right]\right\}$, where $A = [a_{1}, \ldots, a_{200}]'$ is a $200 \times 50$ matrix and $b$ is a $200 \times 1$ column vector. All entries of $A$ and $b$ are i.i.d.-sampled from $\mathcal{N}(0,1)$ and $\rho = 20$.} 
\label{fig:grad_norm}
\end{figure}

An $\Omega(L^2/k^4)$ lower bound has been established by~\cite{nesterov2012make} 
as the optimal convergence rate for minimizing $\|\nabla f\|^2$ with access to only 
first-order information.  (For completeness, Appendix~\ref{subsec: lower_bound} 
presents an exposition of this fundamental barrier.) In the same paper, 
a regularization technique is used in conjunction with NAG-\texttt{SC} to 
obtain a matching upper bound (up to a logarithmic factor). This method, 
however, takes as input the distance between the initial point and the 
minimizer, which is not practical in general \cite{kim2018optimizing}.

Returning to Theorem \ref{thm:NAGM-C_original}, we present a proof of this 
theorem using a Lyapunov function argument. By way of comparison, we remark 
that Nesterov's estimate sequence technique is unlikely to be useful for 
characterizing the convergence of the gradient norm as this technique is 
essentially based on local quadratic approximations. The phase-space 
representation of NAG-\texttt{C}~(\ref{eqn:nagm-c}) takes the following form:
\begin{equation}\label{eqn: Nesterov_convex_symplectic2}
\begin{aligned}
& x_{k} - x_{k - 1} =  \sqrt{s} v_{k - 1} \\
& v_{k} - v_{k - 1} = - \frac{3}{k}  v_{k} - \sqrt{s}( \nabla f(x_{k}) - \nabla f(x_{k - 1}) ) -  \left( 1 + \frac{3}{k} \right) \sqrt{s} \nabla f( x_{k}), 
\end{aligned}
\end{equation}  
for any initial position $x_{0}$ and the initial velocity 
$v_{0} = - \sqrt{s} \nabla f(x_{0})$. This representation allows 
us to discretize the continuous Lyapunov function 
\eqref{eqn: lypunov_NAGM-C_first-order_ode2} into
\begin{equation}\label{eqn: NAGM-C_original_lypunov}
\mathcal{E}(k) = s (k + 3)(k + 1)\left( f(x_{k}) - f(x^{\star}) \right) + \frac{1}{2} \left\| (k + 1)\sqrt{s} v_{k} + 2(x_{k + 1} - x^{\star}) + (k + 1) s \nabla f(x_{k}) \right\|^{2}.
\end{equation}
The following lemma characterizes the dynamics of this Lyapunov function.

\begin{lem}\label{lm:e_decay}
Under the assumptions of Theorem \ref{thm:NAGM-C_original}, we have
\begin{equation}\nonumber
\mathcal{E}(k + 1) - \mathcal{E}(k) \le - \frac{s^2\left( (k + 3)(k - 1) - Ls(k + 3) (k + 1)\right)}{2} \left\| \nabla f(x_{k + 1}) \right\|^{2}
\end{equation}
for all $k \ge 0$. 
\end{lem}

Next, we provide the proof of Theorem~\ref{thm:NAGM-C_original}.
\begin{proof}[Proof of Theorem~\ref{thm:NAGM-C_original}]
We start with the fact that
\begin{equation}\label{eq:k_3_k_1ss}
(k + 3)(k - 1) - Ls(k + 3) (k + 1) \ge 0,
\end{equation}
for $k \ge 2$. To show this, note that it suffices to guarantee
\begin{equation}\label{eqn: NAGM-C-condition}
s \leq \frac{1}{L} \cdot \frac{k - 1}{k + 1},
\end{equation}
which is self-evident since $s \le 1/(3L)$ by assumption.

Next, by a telescoping-sum argument, Lemma \ref{lm:e_decay} leads to 
the following inequalities for $k \ge 4$: 
\begin{equation}\label{eq:sadsvdfvdfb}
\begin{aligned}
 \mathcal{E}(k) - \mathcal{E}(3)   & =\sum_{i=3}^{k - 1} \left(\mathcal{E}(i + 1) - \mathcal{E}(i)\right)\\
                                                     & \le \sum_{i=3}^{k - 1} - \frac{s^2}{2}\left[ (i + 3)(i - 1) - Ls(i + 3) (i + 1)\right] \left\| \nabla f(x_{i + 1}) \right\|^{2} \\
                                                     & \le  - \frac{s^{2}}{2} \min_{4 \leq i \leq k}\left\| \nabla f(x_{i}) \right\|^{2} \sum_{i=3}^{k - 1} \left[ (i + 3)(i - 1) - Ls(i + 3) (i + 1)\right]   \\
                                                     & \le - \frac{s^{2}}{2} \min_{4 \leq i \leq k}\left\| \nabla f(x_{i}) \right\|^{2} \sum_{i=3}^{k - 1} \left[ (i + 3)(i - 1) - \frac{1}{3}(i + 3) (i + 1)\right],
\end{aligned}
\end{equation}
where the second inequality is due to \eqref{eq:k_3_k_1ss}. To further simplify the bound, observe that
\[
\sum_{i=3}^{k - 1} \left[ (i + 3)(i - 1) - \frac{1}{3}(i + 3) (i + 1)\right] = \frac{2k^{3} - 38k + 60}{9} \ge \frac{(k+1)^3}{36},
\]
for $k \geq 4$. Plugging this inequality into \eqref{eq:sadsvdfvdfb} yields
\[
\mathcal{E}(k) - \mathcal{E}(3) \le - \frac{s^2 (k + 1)^3}{72} \min_{4 \leq i \leq k}\left\| \nabla f(x_{i}) \right\|^{2},
\]
which gives
\begin{equation}\label{eqn: estimate_gradsquare}
\min_{4 \leq i \leq k}\left\| \nabla f(x_{i}) \right\|^{2} \leq \frac{72( \mathcal{E}(3) - \mathcal{E}(k))}{s^{2}(k + 1)^{3} } \le \frac{72 \mathcal{E}(3)}{s^{2}(k + 1)^{3} }.
\end{equation}
It is shown in Appendix~\ref{subsec: technical_detail} that
\[
\mathcal{E}(3) \leq \mathcal{E}(2) \le 119 \left\| x_{0} - x^{\star}\right\|^{2},
\]
for $s \le 1/(3L)$. As a consequence of this, \eqref{eqn: estimate_gradsquare} gives
\begin{equation}\label{eqn: estimate_gradsquare_>=4}
\min_{4 \leq i \leq k}\left\| \nabla f(x_{i}) \right\|^{2} \leq \frac{8568  \left\| x_{0} - x^{\star}\right\|^{2}}{s^{2}(k + 1)^{3}}.
\end{equation}
For completeness, Appendix~\ref{subsec: technical_detail} proves, via
a brute-force calculation, that $\left\| \nabla f(x_0) \right\|^{2}, 
\left\| \nabla f(x_1) \right\|^{2}, \left\| \nabla f(x_2) \right\|^{2}$, 
and $\left\| \nabla f(x_3) \right\|^{2}$ are all bounded above by the 
right-hand side of \eqref{eqn: estimate_gradsquare_>=4}.  This completes 
the proof of the first inequality claimed by Theorem~\ref{thm:NAGM-C_original}.

For the second claim in Theorem~\ref{thm:NAGM-C_original}, the definition 
of the Lyapunov function and its decreasing property ensured by \eqref{eq:k_3_k_1ss} implies
\begin{equation}\label{eqn: functiona_value_>=2}
f(x_{k}) - f(x^{\star}) \leq \frac{\mathcal{E}(k)}{s(k + 3)(k+1)} \leq \frac{\mathcal{E}(2)}{s(k + 3)(k+1)} \leq \frac{119\left\| x_{0} - x^{\star}\right\|^{2}}{s (k + 1)^{2}},
\end{equation}
for all $k \geq 2$. Appendix~\ref{subsec: technical_detail} establishes 
that $f(x_{0}) - f(x^{\star})$ and $f(x_{1}) - f(x^{\star})$ are bounded 
by the right-hand side of \eqref{eqn: functiona_value_>=2}. This completes 
the proof.

\end{proof}

Now, we prove Lemma~\ref{lm:e_decay}.

\begin{proof}[Proof of Lemma \ref{lm:e_decay}]

The difference of the Lyapunov function~(\ref{eqn: NAGM-C_original_lypunov}) satisfies
\begin{align*}
\mathcal{E}(k + 1) - \mathcal{E}(k) & = s (k + 3)(k + 1)\left( f(x_{k + 1}) - f(x_{k}) \right) +  s(2k + 5) \left( f(x_{k + 1}) - f(x^{\star}) \right) \\
& \quad + \left\langle 2(x_{k + 2} - x_{k + 1}) + \sqrt{s}(k + 2) (v_{k + 1} + \sqrt{s} \nabla f(x_{k + 1})) - \sqrt{s} (k + 1) (v_{k} + \sqrt{s} \nabla f(x_{k})), \right.\\
&\qquad \left. 2 (x_{k + 2} - x^{\star}) + (k + 2)\sqrt{s} (v_{k + 1} + \sqrt{s} \nabla f(x_{k + 1}))\right\rangle \\
& \quad - \frac{1}{2} \left\|2(x_{k + 2} - x_{k + 1}) + \sqrt{s} (k + 2)(v_{k + 1} + \sqrt{s} \nabla f(x_{k + 1})) - (k + 1)\sqrt{s} (v_{k} + \sqrt{s} \nabla f(x_{k})) \right\|^{2} \\
& = s(k + 3)(k + 1) \left( f(x_{k + 1}) - f(x_{k}) \right) + s(2k + 5)\left( f(x_{k + 1}) - f(x^{\star}) \right) \\
& \quad + \left\langle  - s (k + 3) \nabla f(x_{k + 1}), 2 (x_{k + 2} - x^{\star}) + \sqrt{s}(k + 2)  (v_{k + 1} + \sqrt{s}  \nabla f(x_{k + 1}))\right\rangle \\
& \quad - \frac{1}{2} \left\|  s(k + 3)   \nabla f(x_{k + 1}) \right\|^{2} \\
& = s(k + 3)(k + 1) \left( f(x_{k + 1}) - f(x_{k}) \right) + s(2k + 5)\left( f(x_{k + 1}) - f(x^{\star}) \right) \\
& \quad  - s^{\frac32}(k + 3)(k + 4)  \left\langle   \nabla f(x_{k + 1}), v_{k+1}\right\rangle - 2s (k+3)\left\langle \nabla f(x_{k + 1}), x_{k + 1} - x^{\star} \right\rangle\\
& \quad - s^{2}(k + 3)(k + 2)   \left\| \nabla f(x_{k + 1}) \right\|^{2} - \frac{s^{2}}{2} (k + 3)^{2}  \left\| \nabla f(x_{k + 1}) \right\|^{2},	
\end{align*}  
where the last two equalities are due to
\begin{equation}\label{eq:phase_secondline}
(k + 3)\left( v_{k} + \sqrt{s} \nabla f(x_{k}) \right) -  k \left( v_{k - 1} + \sqrt{s} \nabla f(x_{k - 1}) \right) = -k\sqrt{s} \nabla f(x_{k}),
\end{equation}
which follows from the phase-space representation~\eqref{eqn: Nesterov_convex_symplectic2}. Rearranging the identity for $\mathcal{E}(k + 1) - \mathcal{E}(k)$, we get
\begin{equation}\label{eq:dsfvsdcsdcd}
\begin{aligned}
\mathcal{E}(k + 1) - \mathcal{E}(k) &= s(k + 3)(k + 1) \left( f(x_{k + 1}) - f(x_{k}) \right) - s^{\frac32} (k + 3)(k + 4)\left\langle   \nabla f(x_{k + 1}), v_{k+1}\right\rangle\\
& \quad  + s(2k + 5)\left( f(x_{k + 1}) - f(x^{\star}) \right)  - s(2k+6)\left\langle \nabla f(x_{k + 1}), x_{k + 1} - x^{\star} \right\rangle\\
& \quad - \frac{s^2 (k+3)(3k+7)}{2}\left\| \nabla f(x_{k + 1}) \right\|^{2}.
\end{aligned}
\end{equation}

The next step is to recognize that the convexity and the $L$-smoothness of $f$ gives
\[
\begin{aligned}
& f(x_{k + 1}) - f(x_{k}) \leq \left\langle \nabla f(x_{k + 1}), x_{k+1} - x_{k} \right\rangle - \frac{1}{2L}\left\| \nabla f(x_{k + 1}) - \nabla f(x_{k})  \right\|^{2} \\
& f(x_{k + 1}) - f(x^{\star}) \leq \left\langle \nabla f(x_{k + 1}), x_{k+1} - x^{\star} \right\rangle.
\end{aligned}
\]
Plugging these two inequalities into \eqref{eq:dsfvsdcsdcd}, we have
\begin{align*}                                                        
\mathcal{E}(k + 1) - \mathcal{E}(k)	& \leq - s^{\frac32} (k + 3)\left\langle \nabla f(x_{k + 1}), (k + 4)v_{k + 1} - (k + 1)v_{k}\right\rangle\\
	& \quad - \frac{s}{2L} (k + 3)(k + 1)\left\| \nabla f(x_{k + 1}) - \nabla f(x_{k})  \right\|^{2}- s \left\langle \nabla f(x_{k + 1}), x_{k + 1} - x^{\star} \right\rangle \\
	& \quad - \frac{s^2 (k+3)(3k+7)}{2} \left\| \nabla f(x_{k + 1}) \right\|^{2}\\
& \leq - s^{\frac32} (k + 3)\left\langle \nabla f(x_{k + 1}), (k + 4)v_{k + 1} - (k + 1)v_{k}\right\rangle\\
	& \quad - \frac{s}{2L} (k + 3)(k + 1)\left\| \nabla f(x_{k + 1}) - \nabla f(x_{k})  \right\|^{2} - \frac{s^2 (k+3)(3k+7)}{2} \left\| \nabla f(x_{k + 1}) \right\|^{2},
\end{align*}
where the second inequality uses the fact that $\left\langle \nabla f(x_{k + 1}), x_{k + 1} - x^{\star} \right\rangle \ge 0$.

To further bound $\mathcal{E}(k+1) - \mathcal{E}(k)$, making use of \eqref{eq:phase_secondline} with $k+1$ in place of $k$, we get
\begin{align*}
\mathcal{E}(k + 1) - \mathcal{E}(k) & \leq  s^2(k + 3)(k + 1) \left\langle \nabla f(x_{k + 1}),  \nabla f(x_{k + 1}) -  \nabla f(x_{k}) \right\rangle \\
	& \quad - \frac{s}{2L} (k + 3)(k + 1)\left\| \nabla f(x_{k + 1}) - \nabla f(x_{k})  \right\|^{2} \\
	& \quad -  s^2 \left( \frac{(k+3)(3k+7)}{2} - (k + 3) (k + 4)\right)\left\| \nabla f(x_{k + 1}) \right\|^{2} \\
        & =  \frac{L s^3(k+3)(k+1)}{2} \|\nabla f(x_{k+1})\|^2 - \frac{s (k+3)(k+1)}{2L} \left\|(1 - Ls)\nabla f(x_{k+1}) - \nabla f(x_k) \right\|^2 \\
	& \quad -  \frac{s^2(k+3)(k-1)}{2} \left\| \nabla f(x_{k + 1}) \right\|^{2} \\
	& \leq - \frac{s^2}{2}\left[ (k + 3)(k - 1) - Ls(k + 3) (k + 1)\right] \left\| \nabla f(x_{k + 1}) \right\|^{2}.
\end{align*}
This completes the proof.

\end{proof}

In passing, we remark that the gradient correction sheds light on the superiority 
of the high-resolution ODE over its low-resolution counterpart, just as in 
Section~\ref{sec:strongly}. Indeed, the absence of the gradient correction in 
the low-resolution ODE leads to the lack of the term $(k+1)s \nabla f(x_k)$ in 
the Lyapunov function (see Section 4 of \cite{su2016differential}), as opposed 
to the high-resolution Lyapunov function \eqref{eqn: NAGM-C_original_lypunov}. 
Accordingly, it is unlikely to carry over the bound $\mathcal{E}(k+1) - \mathcal{E}(k) 
\le - O(s^2 k^2 \|\nabla f(x_{k+1})\|^2)$ of Lemma \ref{lm:e_decay} to the 
low-resolution case and, consequently, the low-resolution ODE approach pioneered 
by \cite{su2016differential} is insufficient to obtain the $O(L^2/k^3)$ rate 
for squared gradient norm minimization.

\subsection{A Modified NAG-\texttt{C} without a Phase-Space Representation}
\label{sec:new-accel-meth}

This section proposes a new accelerated method that also achieves the $O(L^2/k^3)$ rate for minimizing the squared gradient norm. This method takes the following form:
\begin{equation}\label{eqn: modified-NAGM-C}
\begin{aligned}
          & y_{k + 1} = x_{k} - s \nabla f(x_{k}) \\
          & x_{k + 1} = y_{k + 1} + \frac{k}{k + 3}(y_{k + 1} - y_{k}) - s \left(\frac{k}{k + 3} \nabla f(y_{k + 1}) - \frac{k - 1}{k + 3} \nabla f(y_{k}) \right),
\end{aligned}
\end{equation}
starting with $x_{0}$ and $y_{0} = x_0$. As shown by the following theorem, 
this new method has the same convergence rates as NAG-\texttt{C}. 

\begin{thm}\label{thm: modified-NAGM-C}
Let $f \in \mathcal{F}_{L}^{1}(\mathbb{R}^n)$. Taking any step size $0 < s \leq 1/L$, the iterates $\{ (x_{k}, y_{k}) \}_{k = 0}^{\infty}$ generated by the modified NAG-\texttt{C} \eqref{eqn: modified-NAGM-C} satisfy
\begin{equation}\nonumber
\begin{aligned}
          & \min_{0 \leq i \leq k} \left\| \nabla f(x_{i}) + \nabla f(y_{i}) \right\|^{2} \leq \frac{882 \left\| x_{0} - x^{\star} \right\|^{2}}{s^2(k+ 1)^{3}} \\
          & f(y_{k}) - f(x^{\star}) \leq  \frac{21 \left\| x_{0} - x^{\star} \right\|^{2}}{s(k + 1)^{2}},
\end{aligned}
\end{equation}
for all $k \ge 0$. 
\end{thm}

We refer readers to Appendix~\ref{subsec: position_acceleration} for the 
proof of Theorem~\ref{thm: modified-NAGM-C}, which is, as earlier, based 
on a Lyapunov function. However, since both $f(x_k)$ and $f(y_k)$ appear 
in the iteration, \eqref{eqn: modified-NAGM-C} does not admit a phase-space 
representation. As a consequence, the construction of the Lyapunov function 
is complex; we arrived at it via trial and error. Our initial aim was 
to seek possible improved rates of the original NAG-\texttt{C} without 
using the phase-space representation, but the enormous challenges arising 
in this process motivated us to (1) modify NAG-\texttt{C} to the current 
\eqref{eqn: modified-NAGM-C}, and (2) to adopt the phase-space representation. 
Employing the phase-space representation yields a simple proof of the 
$O(L^2/k^3)$ rate for the original NAG-\texttt{C} and this technique turned 
out to be useful for other accelerated methods.


\section{Extensions}
\label{sec:extension}
Motivated by the high-resolution ODE \eqref{eqn: nag-c_first} of NAG-\texttt{C}, 
this section considers a family of generalized high-resolution ODEs that take the form
\begin{equation}\label{eqn: generalize_NAGM-C_ode}
\ddot{X} + \frac{\alpha}{t} \dot{X} + \beta\sqrt{s} \nabla^{2} f(X) \dot{X} + \left( 1 + \frac{\alpha \sqrt{s}}{2t}\right) \nabla f(X) = 0,
\end{equation}
for $t \ge \alpha \sqrt{s}/2$, with initial conditions 
$X(\alpha\sqrt{s}/2) = x_{0}$ and $\dot{X}(\alpha\sqrt{s}/2) =  - \sqrt{s} \nabla f(x_{0})$. 
As demonstrated in \cite{su2016differential,attouch2017rate,vassilis2018differential}, 
the low-resolution counterpart (that is, set $s = 0$) of \eqref{eqn: generalize_NAGM-C_ode} 
achieves acceleration if and only if $\alpha \ge 3$. Accordingly, we focus on the 
case where the friction parameter $\alpha \ge 3$ and the gradient correction parameter 
$\beta > 0$. An investigation of the case of $\alpha < 3$ is left for future work.

By discretizing the ODE \eqref{eqn: generalize_NAGM-C_ode}, we obtain a family 
of new accelerated methods for minimizing smooth convex functions:
\begin{equation}\label{eqn: generalize_NAG-C_position}
\begin{aligned}
          & y_{ k + 1} = x_{k} - \beta s\nabla f(x_{k}) \\
          & x_{ k + 1} = x_{k}  - s \nabla f(x_{k})+ \frac{k}{k + \alpha} (y_{k+1} - y_{k}),
\end{aligned}
\end{equation}
starting with $x_{0} = y_{0}$. The second line of the iteration is equivalent to
\[
x_{ k + 1} = \left(1 -  \frac1{\beta}\right)x_{k} + \frac1{\beta} y_{k+1} + \frac{k}{k + \alpha} (y_{k+1} - y_{k}).
\]
In Section \ref{sec:critical-friction}, we study the convergence rates of this 
family of generalized NAC-\texttt{C} algorithms along the lines of 
Section~\ref{sec:nagm-c_analysis}. To further our understanding of 
\eqref{eqn: generalize_NAG-C_position}, Section \ref{sec:super-crit-frict} 
shows that this method in the super-critical regime (that is, $\alpha > 3$) 
converges to the optimum actually faster than $O(1/(sk^2))$. As earlier, 
the proofs of all the results follow the high-resolution ODE framework 
introduced in Section~\ref{sec:techniques}. Proofs are deferred to 
Appendix~\ref{sec: appendix2}. Finally, we note that Section~\ref{sec:con} 
briefly sketches the extensions along this direction for NAG-\texttt{SC}.

\subsection{Convergence Rates}
\label{sec:critical-friction}

The theorem below characterizes the convergence rates of the generalized NAG-\texttt{C}~(\ref{eqn: generalize_NAG-C_position}).
\begin{thm}\label{thm:c_vary_alpha3}
Let $f \in \mathcal{F}^1_L(\mathbb{R}^n), \alpha \ge 3$, and $\beta > \frac12$. There exists $c_{\alpha,\beta} > 0$ such that, taking any step size $0 < s \leq c_{\alpha,\beta}/L$, the iterates $\{x_k\}_{k=0}^{\infty}$ generated by the generalized NAG-\texttt{C} \eqref{eqn: generalize_NAG-C_position} obey
\begin{equation}\label{eq:c_vary_alpha3_11}
\min_{0 \le i \le k} \|\nabla f(x_i)\|^2 \le \frac{C_{\alpha,\beta}  \|x_{0} - x^\star\|^{2}}{s^{2}(k+1)^3},
\end{equation}
for all $k \ge 0$. In addition, we have
\[
f(x_k) - f(x^\star) \le \frac{C_{\alpha,\beta}  \|x_{0} - x^\star\|^2}{s(k+1)^2},
\]
for all $k \ge 0$. The constants $c_{\alpha,\beta}$ and $C_{\alpha,\beta}$ only 
depend on $\alpha$ and $\beta$. 
\end{thm}
The proof of Theorem~\ref{thm:c_vary_alpha3} is given in 
Appendix~\ref{subsec: proof_gamma_=3} for $\alpha = 3$ and 
Appendix~\ref{subsec: proof_gamma_>3} for $\alpha > 3$.  
This theorem shows that the generalized NAG-\texttt{C} achieves 
the same rates as the original NAG-\texttt{C} in both squared 
gradient norm and function value minimization. The constraint 
$\beta > \frac12$ reveals that further leveraging of the gradient 
correction does not hurt acceleration, but perhaps not the other 
way around (note that NAG-\texttt{C} in its original form corresponds 
to $\beta = 1$). It is an open question whether this constraint 
is a technical artifact or is fundamental to acceleration.

\subsection{Faster Convergence in Super-Critical Regime}
\label{sec:super-crit-frict}
We turn to the case in which $\alpha > 3$, where we show that the 
generalized NAG-\texttt{C} in this regime attains a faster rate for 
minimizing the function value. The following proposition provides a 
technical inequality that motivates the derivation of the improved rate.

\begin{prop}\label{thm: OA_generalize_>3}
Let $f \in \mathcal{F}^1_L(\mathbb{R}^n), \alpha > 3$, and $\beta > \frac12$. There exists $c_{\alpha,\beta}' > 0$ such that, taking any step size $0 < s \leq c_{\alpha,\beta}'/L$, the iterates $\{x_k\}_{k=0}^{\infty}$ generated by the generalized NAG-\texttt{C} \eqref{eqn: generalize_NAG-C_position} obey
\[
\sum_{k = 0}^{\infty} \left[  (k + 1) \left( f(x_{k}) - f(x^{\star}) \right)  + s (k + 1)^2 \left\| \nabla f(x_{k})\right\|^{2} \right] \leq \frac{C'_{\alpha, \beta}\left\| x_{0} - x^{\star}\right\|^{2}}{s},
\]
where the constants $c'_{\alpha, \beta}$ and $C'_{\alpha, \beta}$ only depend on $\alpha$ and $\beta$.
\end{prop}

In relating to Theorem \ref{thm:c_vary_alpha3}, one can show that Proposition \ref{thm: OA_generalize_>3} in fact implies \eqref{eq:c_vary_alpha3_11} in Theorem~\ref{thm:c_vary_alpha3}. To see this, note that for $k \ge 1$, one has
\[
\min_{0 \le i \le k} \|\nabla f(x_i)\|^2 \le \frac{\sum_{i = 0}^{k} s (i + 1)^2 \left\| \nabla f(x_i)\right\|^{2}}{\sum_{i = 0}^{k} s (i + 1)^2} \le \frac{\frac{C'_{\alpha, \beta}\left\| x_{0} - x^{\star}\right\|^{2}}{s}}{\frac{s}{6}(k + 1)(k+2)(2k+1)} = O \left(\frac{\left\| x_{0} - x^{\star}\right\|^2}{s^2 k^3} \right),
\] 
where the second inequality follows from Proposition~\ref{thm: OA_generalize_>3}.

Proposition \ref{thm: OA_generalize_>3} can be thought of as a generalization of Theorem 6 of \cite{su2016differential}. In particular, this result implies an intriguing and important message. To see this, first note that, by taking $s = O(1/L)$, Proposition \ref{thm: OA_generalize_>3} gives
\begin{equation}\label{eqn: function_value_k^2}
\sum_{k=0}^{\infty} (k + 1) \left( f(x_{k}) - f(x^{\star}) \right) = O(L\left\| x_{0} - x^{\star} \right\|^{2}),
\end{equation} 
which would not be valid if $f(x_{k}) - f(x^{\star}) \ge c L\left\| x_{0} - x^{\star} \right\|^{2}/k^2$ for a constant $c > 0$. Thus, it is tempting to suggest that there might exist a faster convergence rate in the sense that
\begin{equation}\label{eqn: faster_conver}
f(x_{k}) - f(x^{\star}) \leq o\left( \frac{L\left\| x_{0} - x^{\star} \right\|^{2}}{k^{2}}\right).
\end{equation}
This faster rate is indeed achievable as we show next, though there are examples 
where \eqref{eqn: function_value_k^2} and $f(x_{k}) - f(x^{\star}) = 
O(L\left\| x_{0} - x^{\star} \right\|^{2}/k^2)$ are both satisfied but 
\eqref{eqn: faster_conver} does not hold (a counterexample is given in 
Appendx~\ref{subsec: counter}).

\begin{thm}\label{thm:faster_rate_0}
Under the same assumptions as in Proposition \ref{thm: OA_generalize_>3}, taking the step size $s = c'_{\alpha,\beta}/L$, the iterates $\{x_k\}_{k=0}^{\infty}$ generated by the generalized NAG-\texttt{C} \eqref{eqn: generalize_NAG-C_position} starting from any $x_0 \ne x^\star$ satisfy
\begin{equation}\nonumber
\lim_{k \goto \infty}\frac{k^2 (f(x_{k}) - f(x^{\star}))}{L\left\| x_{0} - x^{\star} \right\|^{2}} = 0.
\end{equation}
\end{thm}

\begin{figure}[htb!]
\begin{minipage}[t]{0.5\linewidth}
\centering
\includegraphics[width=3.2in]{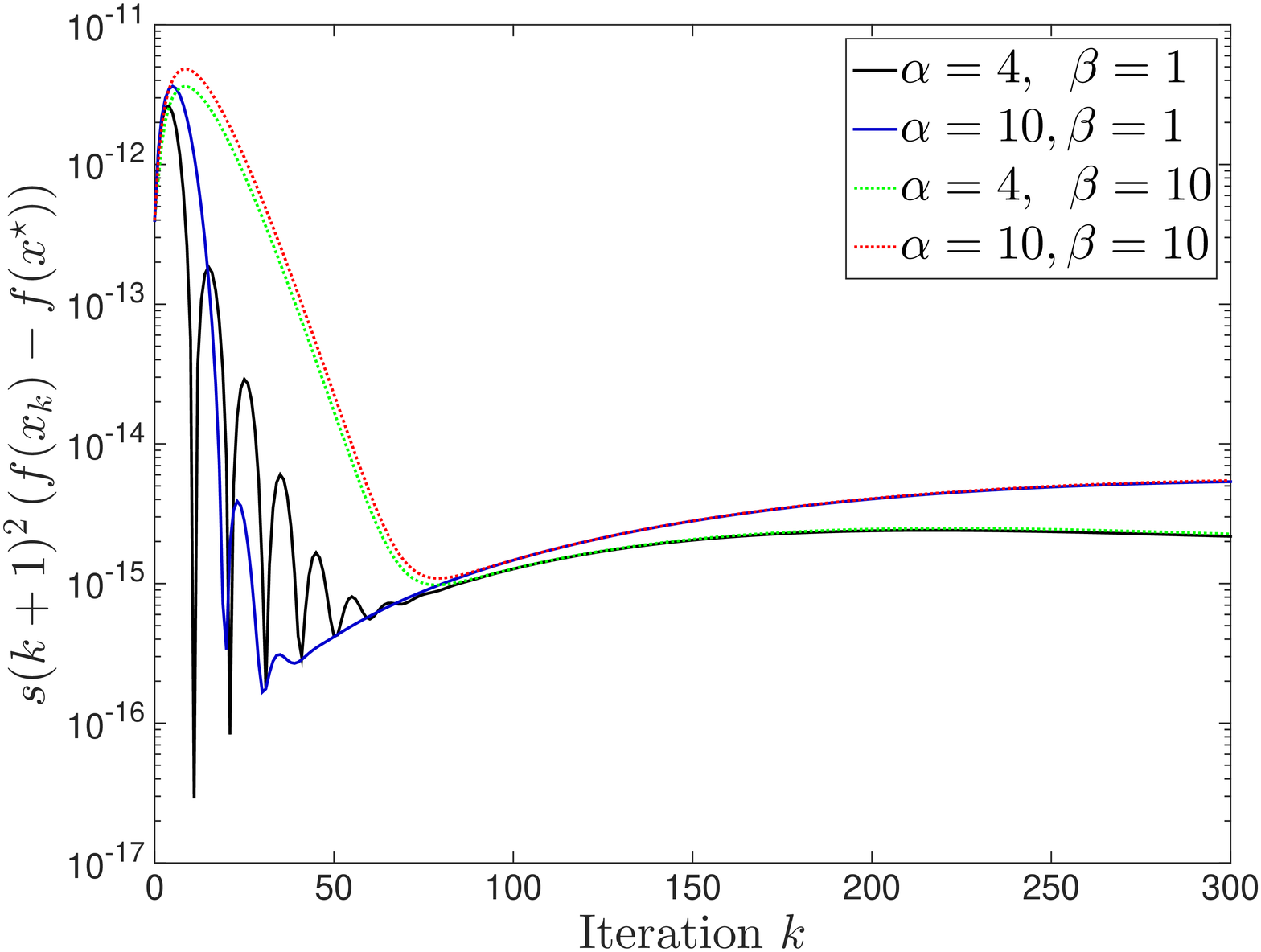}
\end{minipage}
\begin{minipage}[t]{0.5\linewidth}
\centering
\includegraphics[width=3.2in]{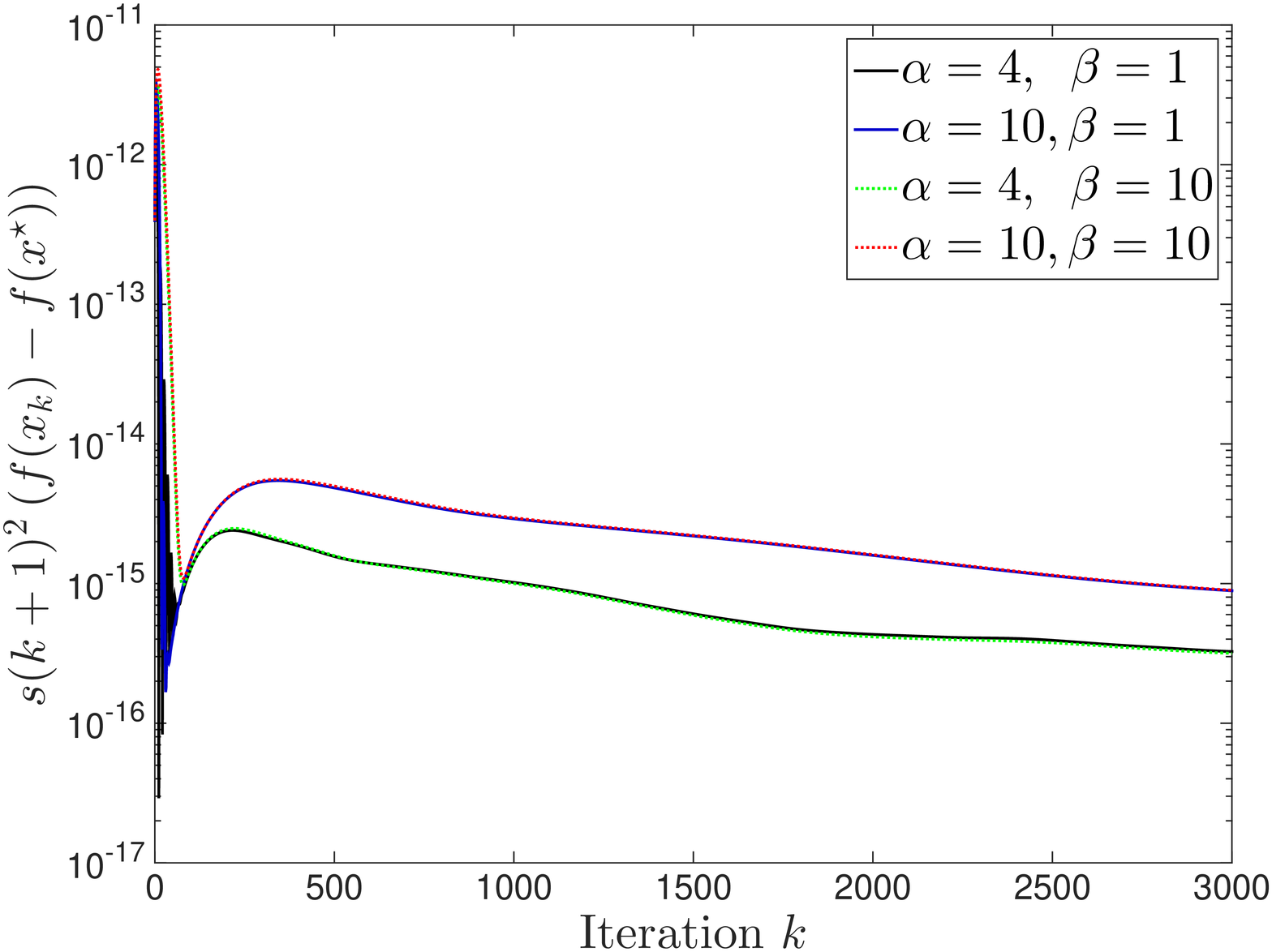}
\end{minipage}
\caption{Scaled error $s(k+1)^2 (f(x_k) - f(x^\star))$ of the generalized NAG-\texttt{C}~\eqref{eqn: generalize_NAG-C_position} with various $(\alpha, \beta)$. The setting is the same as the left plot of Figure~\ref{fig:grad_norm}, with the objective $f(x) = \frac12 \left\langle Ax, x\right\rangle + \left\langle b, x\right\rangle $. The step size is $s = 10^{-1}\|A\|_{2}^{-1}$. The left shows the short-time behaviors of the methods, while the right focuses on the long-time behaviors. The scaled error curves with the same $\beta$ are very close to each other in the short-time regime, but in the long-time regime, the scaled error curves with the same $\alpha$ almost overlap. The four scaled error curves slowly tend to zero.}
\label{fig:generalization_quadratic}
\end{figure}

\begin{figure}[htb!]
\begin{minipage}[t]{0.5\linewidth}
\centering
\includegraphics[width=3in]{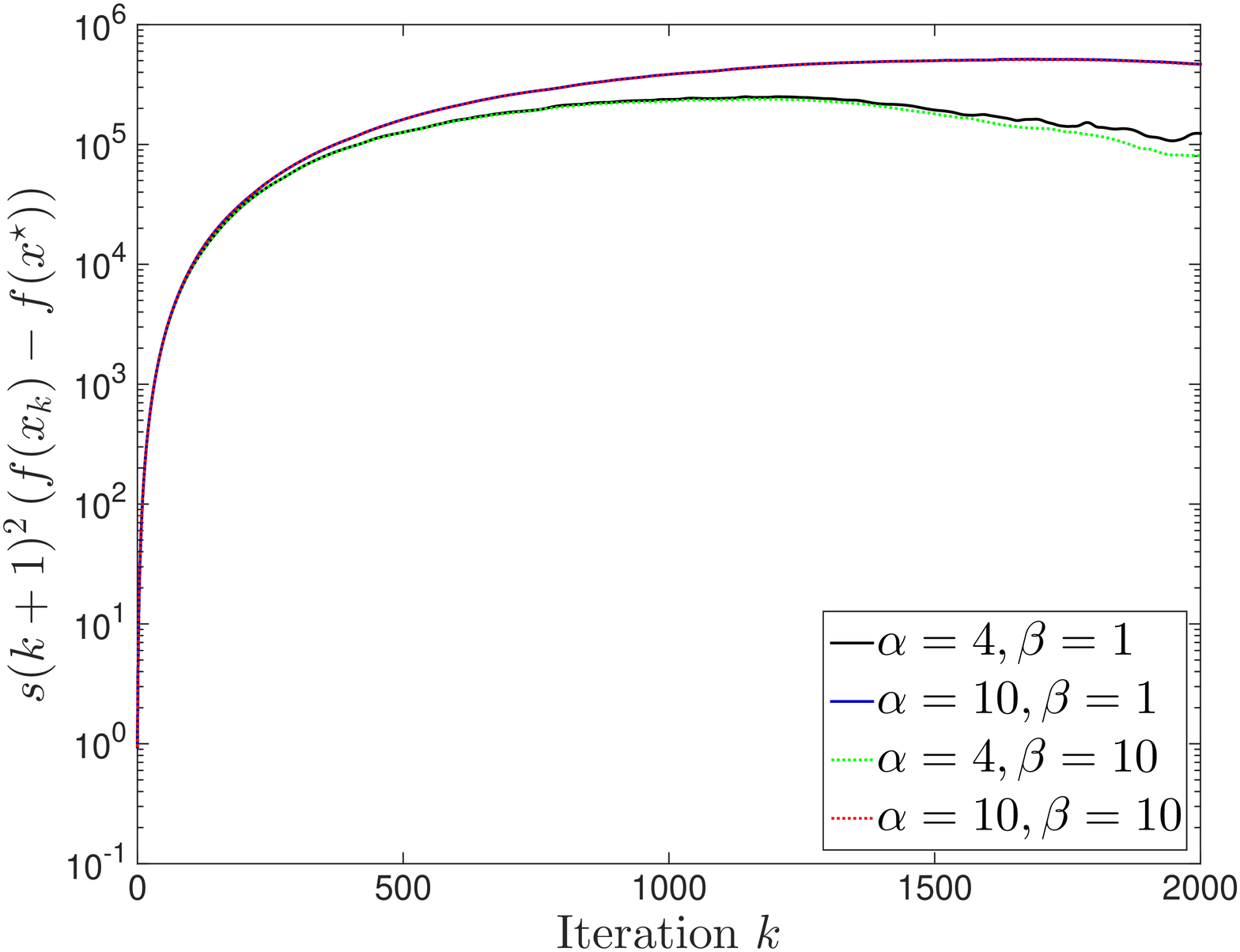}
\end{minipage}
\begin{minipage}[t]{0.5\linewidth}
\centering
\includegraphics[width=3in]{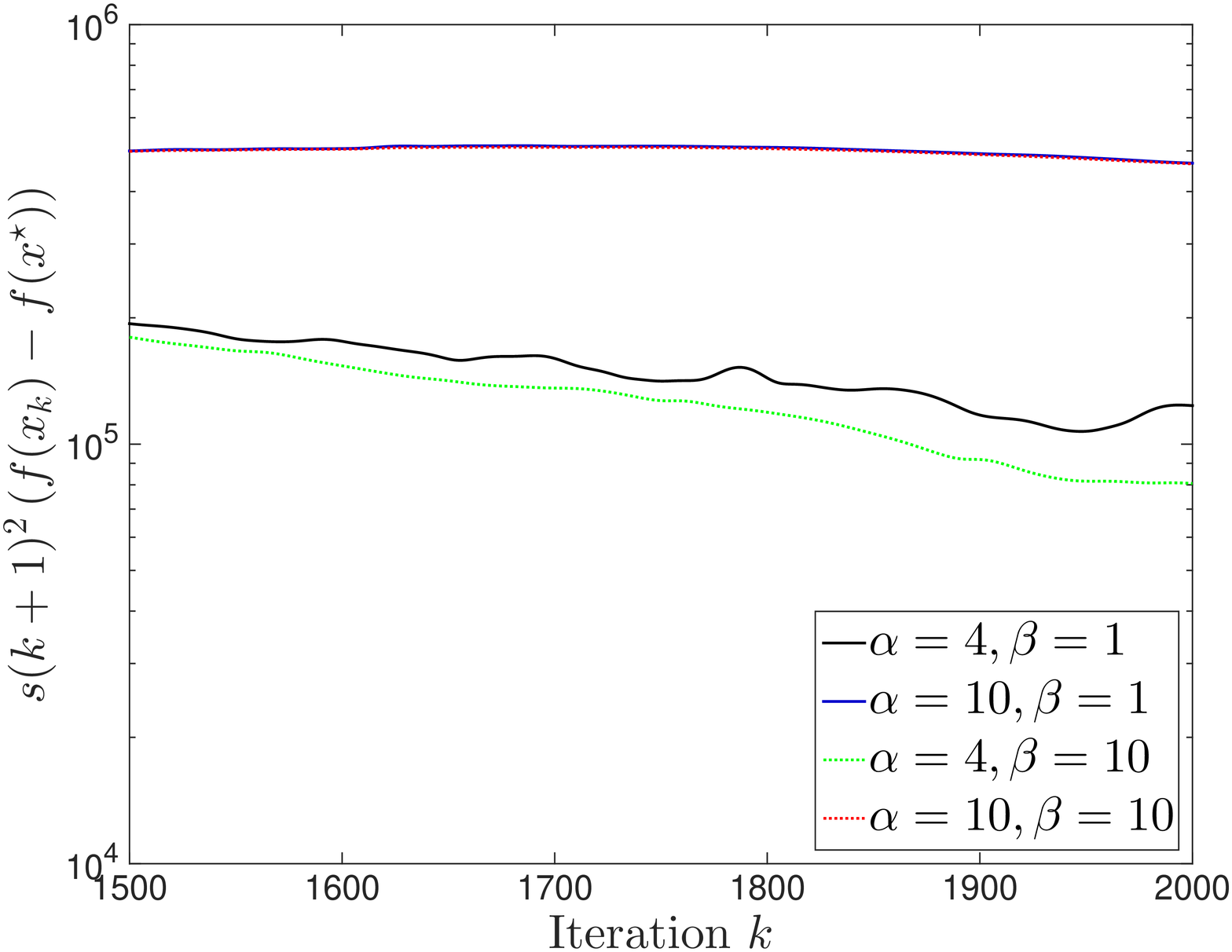}
\end{minipage}
\caption{Scaled error $s(k+1)^2 (f(x_k) - f(x^\star))$ of the generalized NAG-\texttt{C}~\eqref{eqn: generalize_NAG-C_position} with various $(\alpha, \beta)$. The setting is the same as the right plot of Figure~\ref{fig:grad_norm}, with the objective $f(x) = \rho \log \left\{ \sum\limits_{i = 1}^{200} \exp \left[\left( \left\langle a_{i}, x\right\rangle - b_{i} \right)/\rho\right]\right\}$. The step size is $s = 0.1$. This set of simulation studies implies that the convergence in Theorem~\ref{thm:faster_rate_0} is slow for some problems.}
\label{fig: generalization_log-sum-exap}
\end{figure}

Figures~\ref{fig:generalization_quadratic} and \ref{fig: generalization_log-sum-exap} 
present several numerical studies concerning the prediction of 
Theorem~\ref{thm:faster_rate_0}. For a fixed dimension $n$, the convergence 
in Theorem \ref{thm:faster_rate_0} is uniform over functions in 
$\mathcal{F}^1 = \cup_{L > 0} \mathcal{F}_L^1$ and, consequently, is 
independent of the Lipschitz constant $L$ and the initial point $x_0$. 
In addition to following the high-resolution ODE framework, the proof 
of this theorem reposes on the finiteness of the series in 
Proposition \ref{thm: OA_generalize_>3}.  See Appendix~\ref{subsec: proof_gamma_>3} and 
Appendix~\ref{subsec: high_friction_faster} for the full proofs of the 
proposition and the theorem, respectively. 

In the literature, \cite{attouch2016rate, may2017asymptotic,attouch2018fast} use low-resolution ODEs to establish the faster rate $o(1/k^2)$ for the generalized NAG-\texttt{C} \eqref{eqn: generalize_NAG-C_position} in the special case of $\beta = 1$. In contrast, our proof of Theorem~\ref{thm:faster_rate_0} is more general and applies to a broader class of methods.

In passing, we make the observation that Proposition \ref{thm: OA_generalize_>3} reveals that
\[
\sum_{k = 1}^{\infty} s k^2 \left\| \nabla f(x_{k})\right\|^{2}  \leq \frac{C'_{\alpha, \beta}\left\| x_{0} - x^{\star}\right\|^{2}}{s},
\]
which would not hold if $\min_{0 \le i \le k} \|\nabla f(x_i)\|^2 \ge c \|x_0 - x^\star\|^2/(s^2 k^3)$ for all $k$ and a constant $c > 0$. In view of the above, it might be true that the rate of the generalized NAG-\texttt{C} for minimizing the squared gradient norm can be improved to
\[
\min_{0 \le i \le k} \|\nabla f(x_i)\|^2 = o\left( \frac{\|x_0 - x^\star\|^2}{s^2k^3} \right).
\]
We leave the confirmation or disconfirmation of this asymptotic result for future research.


\section{Discussion}
\label{sec:con}

In this paper, we have proposed high-resolution ODEs for modeling 
three first-order optimization methods---the heavy-ball method, 
NAG-\texttt{SC}, and NAG-\texttt{C}.  These new ODEs are more
faithful surrogates for the corresponding discrete optimization 
methods than existing ODEs in the literature, thus serving as 
a more effective tool for understanding, analyzing, and generalizing 
first-order methods. Using this tool, we identified a term that we
refer to as ``gradient correction'' in NAG-\texttt{SC} and in its 
high-resolution ODE, and we demonstrate its critical effect in 
making NAG-\texttt{SC} an accelerated method, as compared to the 
heavy-ball method. We also showed via the high-resolution ODE of 
NAG-\texttt{C} that this method minimizes the squared norm of the 
gradient at a faster rate than expected for smooth convex functions, 
and again the gradient correction is the key to this rate. Finally, the analysis of this tool suggested a new family of accelerated methods 
with the same optimal convergence rates as NAG-\texttt{C}.

The aforementioned results are obtained using the high-resolution 
ODEs in conjunction with a new framework for translating findings 
concerning the amenable ODEs into those of the less ``user-friendly'' 
discrete methods. This framework encodes an optimization property 
under investigation to a continuous-time Lyapunov function for an 
ODE and a discrete-time Lyapunov function for the discrete method. 
As an appealing feature of this framework, the transformation from 
the continuous Lyapunov function to its discrete version is through 
a phase-space representation. This representation links continuous 
objects such as position and velocity variables to their discrete 
counterparts in a faithful manner, permitting a transparent analysis 
of the three discrete methods that we studied.

There are a number of avenues open for future research using the 
high-resolution ODE framework. First, the discussion of Section 
\ref{sec:extension} can carry over to the heavy-ball method and 
NAG-\texttt{SC}, which correspond to the high-resolution ODE
\[
\ddot{X}(t) + 2 \sqrt{\mu} \dot{X}(t) + \beta\sqrt{s} \nabla^{2} f(X(t)) \dot{X}(t) + \left( 1 + \sqrt{\mu s} \right)\nabla f(X(t))= 0
\]
with $\beta = 0$ and $\beta = 1$, respectively. This ODE with a 
general $0 < \beta < 1$ corresponds to a new algorithm that can be 
thought of as an interpolation between the two methods. It is of
interest to investigate the convergence properties of this class of 
algorithms.  Second, we recognize that new optimization algorithms are 
obtained in \cite{wibisono2016variational,wilson2016lyapunov} by using 
different discretization schemes on low-resolution ODE. Hence, a 
direction of interest is to apply the techniques therein to our 
high-resolution ODEs and to explore possible appealing properties 
of the new methods. Third, the technique of dimensional analysis, 
which we have used to derive high-resolution ODEs, can be further 
used to incorporate even higher-order powers of $\sqrt{s}$ into the ODEs. 
This might lead to further fine-grained findings concerning the discrete 
methods.

More broadly, we wish to remark on possible extensions of the 
high-resolution ODE framework beyond smooth convex optimization 
in the Euclidean setting. In the non-Euclidean case, it would be 
interesting to derive a high-resolution ODE for mirror descent 
\cite{krichene2015accelerated,wibisono2016variational}. This 
framework might also admit extensions to non-smooth optimization 
and stochastic optimization, where the ODEs are replaced, respectively, 
by differential inclusions \cite{osher2016sparse,vassilis2018differential} 
and stochastic differential equations \cite{krichene2017acceleration,hu2017diffusion,
li2017stochastic,liang2017statistical,xu2018continuous,he2018differential,
gao2018global}. Finally, recognizing that the high-resolution ODEs are
well-defined for non-convex functions, we believe that this framework will 
provide more accurate characterization of local behaviors of first-order 
algorithms near saddle points \cite{jin2017escape,du2017gradient,hu2017global}. 
On a related note, given the centrality of the problem of finding an 
approximate stationary point in the non-convex setting 
\cite{carmon2017lower,carmon2017lower2,allen2018make}, it is 
worth using the high-resolution ODE framework to explore possible 
applications of the faster rate for minimizing the squared gradient norm
that we have uncovered.


{\small
\subsection*{Acknowledgements}
B.~S.~is indebted to Xiaoping Yuan for teaching him the modern theory of ordinary differential equations and would like to thank Rui Xin Huang for teaching him how to leverage intuitions from physics to understand differential equations. We would like to thank Nicolas Flammarion for suggesting references. This work was supported in part by the NSF via grant CCF-1763314 and Army Research Office via grant W911NF-17-1-0304.

\bibliographystyle{alpha}
\bibliography{sigproc}
}

\newpage
\appendix

\section{Technical Details in Section~\ref{sec:techniques}}
\subsection{Derivation of High-Resolution ODEs}
\label{sec: ODE_2}
In this section, we formally derive the high-resolution ODEs of the heavy-ball method and NAG-\texttt{C}.  Let $t_{k} = k \sqrt{s}$. For the moment, let $X(t)$ be a sufficiently smooth map from $[0, \infty)$ (the heavy-ball method) or $[1.5\sqrt{s}, \infty)$ (NAG-\texttt{C}) to $\mathbb{R}^{n}$, with the correspondence $X(t_{k}) = X(k \sqrt{s}) = x_{k}$, where $\{x_{k}\}_{k = 0}^{\infty}$ is the sequence of iterates generated by the heavy-ball method or NAG-\texttt{C}, depending on the context.

\noindent{\bf The heavy-ball method.} For any function $f(x) \in \mathcal{S}_{\mu, L}^{2}(\mathbb{R}^{n})$, setting $\alpha = \frac{1 - \sqrt{\mu s}}{1 + \sqrt{\mu s}}$, multiplying both sides of~\eqref{eqn: polyak_heavy_ball} by $\frac{1 + \sqrt{\mu s}}{1 - \sqrt{\mu s}} \cdot \frac{1}{s}$ and rearranging the equality, we obtain                
                            \begin{align}\label{eqn: heavy-b_rewrite}
                            \frac{x_{k + 1} + x_{k - 1} - 2x_{k}}{s} + \frac{2 \sqrt{\mu s}}{ 1 - \sqrt{\mu s} } \frac{x_{k + 1} - x_{k}}{s}  + \frac{ 1 + \sqrt{\mu s} }{ 1- \sqrt{\mu s} } \nabla f(x_{k}) = 0.
                            \end{align}
                             Plugging~(\ref{eqn: 3rd_taylor_variable}) into~(\ref{eqn: heavy-b_rewrite}), we have
                            $$
                            \begin{aligned}
                            \ddot{X}(t_{k})  + O\left(\sqrt{s}\right)  + \frac{2 \sqrt{\mu}}{ 1 - \sqrt{\mu s} } \left[ \dot{X}(t_{k})  + \frac{1}{2}\sqrt{s} \ddot{X}(t_{k}) + O\left(\left(\sqrt{s}\right)^{2}\right) \right]
                                                    + \frac{ 1 + \sqrt{\mu s} }{ 1 - \sqrt{\mu s} } \nabla f(X(t_{k})) = 0.
                            \end{aligned} 
                            $$
                           By only ignoring the $O(s)$ term, we obtain the high-resolution ODE~(\ref{eqn: heavy_ball_first}) for the heavy-ball method
                            $$
                            \ddot{X} + 2 \sqrt{\mu} \dot{X} + \left(1 + \sqrt{\mu s}\right)\nabla f(X) =0.
                            $$  

\noindent{\bf NAG-\texttt{C}.} For any function $f(x) \in \mathcal{F}_{L}^{2}(\mathbb{R}^{n})$, multiplying both sides of~(\ref{eqn:nagm-c}) by $\frac{1 + \sqrt{\mu s}}{1 - \sqrt{\mu s}} \cdot \frac{1}{s}$ and rearranging the equality, we get 
                           \begin{align}
                           \label{eqn:nagm-c_rewrite}
                            \frac{x_{k + 1} + x_{k - 1} - 2x_{k}}{s} + \frac{3}{k} \cdot \frac{x_{k + 1} - x_{k}}{s} + \left( \nabla f(x_{k}) - \nabla f(x_{k - 1}) \right) +\left( 1 + \frac{3}{k} \right) \nabla f(x_{k}) = 0.
                            \end{align}
                             For convenience, we slightly change the definition  $t_{k} = k\sqrt{s} + (3/2)\sqrt{s}$ instead of $t_{k} = k\sqrt{s}$. 
                             Plugging~(\ref{eqn: 3rd_taylor_variable}) into~(\ref{eqn:nagm-c_rewrite}), we have
                            \begin{multline*}
                            \ddot{X}(t_{k})  + O\left(\left( \sqrt{s} \right)^{2} \right)  + \frac{3}{t_{k} - (3/2)\sqrt{s}} \left[ \dot{X}(t_{k})  + \frac{1}{2}\sqrt{s} \ddot{X}(t_{k}) + O\left(\left(\sqrt{s}\right)^{2}\right) \right] \\
                                                    + \nabla^{2} f(X(t_{k})) \dot{X}(t_{k}) \sqrt{s} + O\left(\left( \sqrt{s} \right)^{2} \right)+ \frac{ t_{k} + (3/2)\sqrt{ s} }{ t_{k} - (3/2)\sqrt{ s} } \nabla f(X(t_{k})) = 0.
                            \end{multline*} 
                            Ignoring any $O(s)$ terms, we obtain the high-resolution ODE~(\ref{eqn: nag-c_first}) for NAG-\texttt{C}
                            $$
                            \ddot{X} + \frac{3}{t} \dot{X} + \sqrt{s} \nabla^{2}f(X) \dot{X} +\left( 1 + \frac{3\sqrt{s}}{2t} \right) \nabla f(X) =0.
                            $$

\subsection{Derivation of Low-Resolution ODEs}
\label{sec: ODE_1}
In this section, we derive low-resolution ODEs of accelerated gradient methods for comparison. The results presented here are well-known in the literature and the purpose is for ease of reading. In~\cite{su2016differential},  the second-order Taylor expansions at both $x_{k - 1}$ and $x_{k + 1}$ with the step size $\sqrt{s}$ are,
\begin{equation}
\label{eqn: 2nd_taylor_variable}
\left\{ \begin{aligned}
& x_{k + 1} =  X\left( (k + 1)\sqrt{s} \right) = X(t_{k}) + \dot{X}(t_{k})\sqrt{s} + \frac{1}{2} \ddot{X}(t_{k})\left(\sqrt{s}\right)^{2} + O\left( \left(\sqrt{s}\right)^{3} \right) \\
& x_{k  - 1} =  X\left( (k  - 1)\sqrt{s} \right) = X(t_{k})  - \dot{X}(t_{k})\sqrt{s} + \frac{1}{2} \ddot{X}(t_{k})\left(\sqrt{s}\right)^{2} + O\left( \left(\sqrt{s}\right)^{3} \right) .
\end{aligned} \right.
\end{equation}
With the Taylor expansion~(\ref{eqn: 2nd_taylor_variable}), we obtain the gradient correction 
\begin{equation}
\label{eqn: gradient_correction1}
\nabla f(x_{k}) - \nabla f(x_{k - 1}) =  \nabla^{2} f(X(t_{k})) \dot{X}(t_{k})\sqrt{s} + O\left(\left( \sqrt{s} \right)^{2}\right) = O\left(\sqrt{s}\right). 
\end{equation}
From~\eqref{eqn: 2nd_taylor_variable} and~\eqref{eqn: gradient_correction1}, we can derive the following low-resolution ODEs.

\begin{enumerate}[label = \textbf{(\arabic*)}]
  \item For any function $f(x) \in \mathcal{S}_{\mu, L}^{1}(\mathbb{R}^{n})$. 
         \begin{enumerate}[label = \textbf{(\alph*)}]
                   \item Recall the equivalent form~(\ref{eqn:nagm-sc_rewrite}) of NAG-\texttt{SC}~(\ref{eqn: Nesterov_strongly}) is
                            \begin{align*}
                            \frac{x_{k + 1} + x_{k - 1} - 2x_{k}}{s} + \frac{2 \sqrt{\mu s}}{ 1 - \sqrt{\mu s} } \frac{x_{k + 1} - x_{k}}{s} + \left( \nabla f(x_{k}) - \nabla f(x_{k - 1}) \right) + \frac{ 1+ \sqrt{\mu s} }{ 1 - \sqrt{\mu s} } \nabla f(x_{k}) = 0.
                            \end{align*}
                            Plugging~(\ref{eqn: 2nd_taylor_variable}) and~\eqref{eqn: gradient_correction1} into~(\ref{eqn:nagm-sc_rewrite}), we have
                            \begin{multline*}
                            \ddot{X}(t_{k})  + O\left( \sqrt{s} \right)  + \frac{2 \sqrt{\mu}}{ 1 - \sqrt{\mu s} } \left[ \dot{X}(t_{k}) + \frac{1}{2}\ddot{X}\sqrt{s} + O\left( \left(\sqrt{s}\right)^{2} \right)  \right] \\
                                                    +  O\left(  \sqrt{s}\right)  +\left( 1 + O(\sqrt{s})\right)\nabla f(X(t_{k})) = 0.
                            \end{multline*}         
                            Hence, taking $s \rightarrow 0$, we obtain the low-resolution ODE~(\ref{eqn:ode_old_nagsc_hb}) of NAG-\texttt{SC} 
                            \[
                            \ddot{X} + 2 \sqrt{\mu} \dot{X} + \nabla f(X) =0.
                            \] 
                 \item  Recall the equivalent form~(\ref{eqn: heavy-b_rewrite}) of the heavy-ball method~(\ref{eqn: polyak_heavy_ball}) is
                            \[
                            \frac{x_{k + 1} + x_{k - 1} - 2x_{k}}{s} + \frac{2 \sqrt{\mu s}}{ 1 - \sqrt{\mu s} } \frac{x_{k + 1} - x_{k}}{s}  + \frac{ 1 + \sqrt{\mu s} }{ 1- \sqrt{\mu s} } \nabla f(x_{k}) = 0.
                            \]
                             Plugging~(\ref{eqn: 2nd_taylor_variable}) and~\eqref{eqn: gradient_correction1} into~(\ref{eqn: heavy-b_rewrite}), we have
                            $$
                            \begin{aligned}
                            \ddot{X}(t_{k})  + O\left(\sqrt{s}\right)  + \frac{2 \sqrt{\mu}}{ 1 - \sqrt{\mu s} } \left[ \dot{X}(t_{k})  + \frac{1}{2}\sqrt{s} \ddot{X}(t_{k}) + O\left( \left( \sqrt{s} \right)^{2}\right) \right]
                                                    + \frac{ 1 + \sqrt{\mu s} }{ 1 - \sqrt{\mu s} } \nabla f(X(t_{k})) = 0.
                            \end{aligned} 
                            $$
                             Hence, taking $s \rightarrow 0$, we obtain  the low-resolution ODE~(\ref{eqn:ode_old_nagsc_hb}) of the heavy-ball method
                            $$
                            \ddot{X} + 2 \sqrt{\mu} \dot{X} + \nabla f(X) =0.
                             $$  
  \end{enumerate}
 Notably,  NAG-\texttt{SC} and the heavy-ball method share the same low-resolution ODE~\eqref{eqn:ode_old_nagsc_hb}, which is almost consistent with~\eqref{eqn: heavy_ball_first}. Thus the low-resolution ODE fails to capture the information from the ``gradient correction" of NAG-\texttt{SC}.
  \item  For any function $f(x) \in \mathcal{F}_{L}^{1}(\mathbb{R}^{n})$,  recall the equivalent form~(\ref{eqn:nagm-c_rewrite}) of NAG-\texttt{C}~\eqref{eqn:nagm-c} is
                            \[
                            \frac{x_{k + 1} + x_{k - 1} - 2x_{k}}{s} + \frac{3}{k} \cdot \frac{x_{k + 1} - x_{k}}{s} + \left( \nabla f(x_{k}) - \nabla f(x_{k - 1}) \right) + \left( 1 + \frac{3}{k} \right) \nabla f(x_{k}) = 0.
                            \]
                            Plugging~(\ref{eqn: 2nd_taylor_variable}) and~\eqref{eqn: gradient_correction1} into~(\ref{eqn:nagm-c_rewrite}), we have
                            \begin{multline*}
                            \ddot{X}(t_{k})  + O\left( \sqrt{s} \right)  + \frac{3}{t_{k}} \cdot \left[ \dot{X}(t_{k}) + \frac{1}{2}\ddot{X}(t_{k})\sqrt{s} + O\left( \left(\sqrt{s}\right)^{2} \right)\right] \\
                                                     + O\left(\sqrt{s}\right) + \left( 1 + \frac{3\sqrt{s}}{t_k} \right) \nabla f(X(t_{k})) = 0.
                            \end{multline*} 
                            Thus,  by taking $s \rightarrow 0$, we obtain  the low-resolution ODE~(\ref{eqn:ode_old_nagmc}) of NAG-\texttt{C} 
                           \[
                            \ddot{X} + \frac{3}{t} \dot{X} + \nabla f(X) =0,
                            \]  
                            which is the same as \cite{su2016differential}.                                         

\end{enumerate}

\subsection{Solution Approximating Optimization Algorithms}
To investigate the property about the high-resolution ODEs~\eqref{eqn: heavy_ball_first},~\eqref{eqn: nag-sc_first} and~\eqref{eqn: nag-c_first}, we need to state the relationship between them and their low-resolution corresponding ODEs. Here, we denote the solution to high-order ODE by $X_s = X_s(t)$. 
Actually, the low-resolution ODE is the special case of high-resolution ODE with $s = 0$. Take NAG-\texttt{SC} for example
\[
\begin{aligned}
& \ddot{X}_{s} + \mu \dot{X}_{s} + \sqrt{s} \nabla f(X_{s}) \dot{X}_{s} + (1 + \sqrt{\mu s}) \nabla f(X_{s}) = 0\\
& X_{s}(0) = x_{0}, \quad \dot{X}_{s}(0) = - \frac{2\sqrt{s} \nabla f(x_0)}{1 + \sqrt{\mu s}}.
\end{aligned}
\]
In other words, we consider a family of ODEs about the step size parameter $s$.

\subsubsection{Proof of Proposition~\ref{prop: exist_unique_SC}}
\label{sec: proof_exist_unique_SC}

\paragraph{Global Existence and Uniqueness}
To prove the global existence and uniqueness of solution to the high-resolution ODEs~\eqref{eqn: heavy_ball_first} and~\eqref{eqn: nag-sc_first}, we first emphasize a fact that if $X_s = X_s(t)$ is the solution of~\eqref{eqn: heavy_ball_first} or~\eqref{eqn: nag-sc_first}, there exists some constant $\mathcal{C}_1 > 0$ such that
\begin{align}
\label{eqn: velocity_bound}
\sup_{0 \leq t < \infty}\left\| \dot{X}_s(t)\right\| \leq \mathcal{C}_1, 
\end{align}
which is only according to the following Lyapunov function
\begin{align}
\label{eqn:ef_simple_sc}
\mathcal{E}(t) = (1 + \sqrt{\mu s}) \left( f(X_s) - f(x^{\star}) \right)+ \frac{1}{2} \| \dot{X}_s \|^{2}.
\end{align}

Now, we proceed to prove  the global existence and uniqueness of solution to the high-resolution ODEs~\eqref{eqn: heavy_ball_first} and~\eqref{eqn: nag-sc_first}.  Recall initial value problem (IVP) for first-order ODE system in $\mathbb{R}^{m}$  as 
\begin{equation}
\label{eqn: IVP_auto}
\dot{x} = b(x), \quad x(0) = x_{0},
\end{equation}
of which the classical theory about global existence and uniqueness of solution is shown as below.
\begin{thm}[Chillingworth, Chapter 3.1, Theorem 4~\cite{perko2013differential}]
\label{thm: existence_uniqueness}
Let $M \in \mathbb{R}^{m}$ be a compact manifold and $b \in C^{1}(M)$. If the vector field $b$ satisfies the global Lipschitz condition
\[
\left\| b(x) - b(y) \right\| \leq \mathfrak{L} \left\| x - y\right\|
\]
for all $x, y \in M$. Then for any $x_{0} \in M$, the IVP~\eqref{eqn: IVP_auto} has a unique solution $x(t)$ defined for all $t \in \mathbb{R}$.
\end{thm}
Apparently, the set $M_{\mathcal{C}_1} = \left\{ \left. (X_s, \dot{X}_s) \in \mathbb{R}^{2n} \right| \| \dot{X}_s \| \leq \mathcal{C}_1 \right\}$ is a compact manifold satisfying Theorem~\ref{thm: existence_uniqueness} with $m = 2n$.

\begin{itemize}
\item For the heavy-ball method, the phase-space representation of  high-resolution ODE~\eqref{eqn: heavy_ball_first} is
\begin{equation}
\label{eqn: high_ode_heavy-b_phase}
\frac{\dd}{\dd t} \left( \begin{aligned}
& X_s \\ 
& \dot{X}_s 
\end{aligned} \right ) = \left( \begin{aligned}
& \dot{X}_s \\ 
-\mu &  \dot{X}_s - (1 + \sqrt{\mu s}) \nabla f(X_s)
\end{aligned} \right).
\end{equation}
For any $(X_s, \dot{X}_s)^\top, (Y_s, \dot{Y}_s)^\top \in M_{\mathcal{C}_1}$, we have 
\begin{align}
\label{eqn: Lipschitz_heavy-ball_ODE}
 & \left\| \left( \begin{aligned}
& \dot{X}_s \\ 
-\mu & \dot{X}_s - (1 + \sqrt{\mu s}) \nabla f(X_s)
\end{aligned} \right) - \left( \begin{aligned}
& \dot{Y}_s \\ 
-\mu &  \dot{Y}_s - (1 + \sqrt{\mu s}) \nabla f(Y_s) \end{aligned} \right)\right\| \nonumber \\
=& 
\left\| \left( \begin{aligned}
&\dot{X}_s - \dot{Y}_s \\ 
-\mu( &   \dot{X}_s - \dot{Y} _s)
\end{aligned} \right) \right\| 
+
(1 + \sqrt{\mu s}) \left\| \left( \begin{aligned}
& 0 \\ 
&  \nabla f(X_s) -  \nabla f(Y_s)
\end{aligned} \right)\right\| \nonumber \\
 \leq & \sqrt{ 1 + \mu^{2} } \left\| \dot{X}_s - \dot{Y}_s \right\| + (1 + \sqrt{\mu s}) L \left\|  X_s - Y_s \right\| \nonumber \\
\leq &  2 \max\left\{ \sqrt{ 1 + \mu^{2} }, (1 + \sqrt{\mu s}) L  \right\} \left\| \left( \begin{aligned}
& X_s \\ 
& \dot{X}_s 
\end{aligned} \right) - \left( \begin{aligned}
& Y_s \\ 
& \dot{Y}_s 
\end{aligned} \right)\right\|. 
\end{align}

\item For NAG-\texttt{SC}, the phase-space representation of  high-resolution ODE~\eqref{eqn: nag-sc_first} is
\begin{equation}
\label{eqn: high_ode_nag-sc_phase}
\frac{\dd}{\dd t} \left( \begin{aligned}
& X_s \\ 
& \dot{X}_s 
\end{aligned} \right )= \left( \begin{aligned}
& \dot{X}_s \\ 
-\mu &  \dot{X}_s - \sqrt{s} \nabla^{2}f(X_s) \dot{X}_s- (1 + \sqrt{\mu s}) \nabla f(X_s)
\end{aligned} \right).
\end{equation}
%
%
For any $(X_s, \dot{X}_s)^\top, (Y_s, \dot{Y}_s)^\top \in M_{\mathcal{C}_1}$, we have 
\begin{align}
\label{eqn: Lipschitz_nag-sc_ODE}
 & \left\| \left( \begin{aligned}
& \dot{X}_s \\ 
-\mu &  \dot{X}_s - \sqrt{s} \nabla^{2}f(X_s) \dot{X}_s - (1 + \sqrt{\mu s}) \nabla f(X_s)
\end{aligned} \right) - \left( \begin{aligned}
& \dot{Y}_s \\ 
-\mu &  \dot{Y}_s - \sqrt{s} \nabla^{2}f(Y_s) \dot{Y}_s  - (1 + \sqrt{\mu s}) \nabla f(Y_s) \end{aligned} \right)\right\| \nonumber \\
\leq & 
\left\| \left( \begin{aligned}
& \dot{X}_s - \dot{Y}_s \\ 
-\left( \mu \bm{I}+  \sqrt{s} \nabla^2 f(X_s)\right)( &   \dot{X}_s - \dot{Y}_s )
\end{aligned} \right) \right\| 
+
\sqrt{s}  \left\| \left( \begin{aligned}
& \qquad 0 \\ 
&  \left( \nabla^2 f(X_s) -  \nabla^2 f(Y_s)\right) \dot{Y}_s
\end{aligned} \right)\right\| \nonumber \\
& \qquad +
(1 + \sqrt{\mu s}) \left\| \left( \begin{aligned}
& \quad 0 \\ 
&  \nabla f(X_s) -  \nabla f(Y_s)
\end{aligned} \right)\right\| \nonumber \\
 \leq & \sqrt{ 1 + 2\mu^{2}  + 2sL^{2}} \left\|  \dot{X}_s - \dot{Y}_s \right\| + \left[ \sqrt{s} \mathcal{C}_1 L'+ (1 + \sqrt{\mu s}) L\right] \left\|  X_s - Y_s \right\| \nonumber \\
\leq &   2\max\left\{ \sqrt{ 1 + 2\mu^{2}  + 2sL^{2}}, \sqrt{s} \mathcal{C}_1 L'+ (1 + \sqrt{\mu s}) L \right\} \left\| \left( \begin{aligned}
& X_s \\ 
& \dot{X}_s 
\end{aligned} \right) - \left( \begin{aligned}
& Y_s \\ 
& \dot{Y}_s 
\end{aligned} \right)\right\|. 
\end{align}
\end{itemize}

Based on the phase-space representation~\eqref{eqn: high_ode_heavy-b_phase} and~\eqref{eqn: high_ode_nag-sc_phase}, together with the Lipschitz condition~\eqref{eqn: Lipschitz_heavy-ball_ODE} and~\eqref{eqn: Lipschitz_nag-sc_ODE}, Theorem~\ref{thm: existence_uniqueness} leads to the following Corollary.
\begin{coro}
\label{coro: nag-sc-heavy-ball}
For any $f \in \mathcal{S}_{\mu}^2(\mathbb{R}^n) := \cup_{L \geq \mu} \mathcal{S}^2_{\mu,L}(\mathbb{R}^n)$, each of the two ODEs \eqref{eqn: heavy_ball_first} and \eqref{eqn: nag-sc_first} with the specified initial conditions has a unique global solution $X \in C^2(I; \mathbb{R}^n)$
\end{coro}

%

\paragraph{Approximation}
Based on the Lyapunov function~\eqref{eqn:ef_simple_sc}, the gradient norm is bounded along the solution of~\eqref{eqn: heavy_ball_first} or~\eqref{eqn: nag-sc_first}, that is,
\begin{align}
\label{eqn: grad_norm_bound}
\sup_{0 \leq t < \infty}\left\| \nabla f(X_s(t))\right\| \leq \mathcal{C}_2. 
\end{align}
Recall the low-resolution ODE~\eqref{eqn:ode_old_nagsc_hb}, the phase-space representation is proposed as
\begin{equation}
\label{eqn: high_ode_nag-sc_phase_low}
\frac{\dd}{\dd t}\left( \begin{aligned}
& X \\ 
& \dot{X} 
\end{aligned} \right )
 = \left( \begin{aligned}
& \dot{X} \\ 
-\mu &  \dot{X} - \nabla f(X)
\end{aligned} \right).
\end{equation}
Similarly, using a Lyapunov function argument, we can show that if $X = X(t)$ is a solution of~\eqref{eqn:ode_old_nagsc_hb}, we have
\begin{align}
\label{eqn: velocity_bound_low}
\sup_{0 \leq t < \infty}\left\| \dot{X}(t)\right\| \leq \mathcal{C}_3.
\end{align}
Simple calculation tells us that there exists some constant $\mathcal{L}_1 > 0$ such that
\begin{align}
\label{eqn:Lip_low}
 \left\| \left( \begin{aligned}
& \dot{X} \\ 
-\mu &  \dot{X}  -  \nabla f(X)
\end{aligned} \right) - \left( \begin{aligned}
& \dot{Y} \\ 
-\mu &  \dot{Y}  -  \nabla f(Y) \end{aligned} \right)\right\| \leq \mathcal{L}_{1}  \left\| \left( \begin{aligned}
& X \\ 
& \dot{X} 
\end{aligned} \right) - \left( \begin{aligned}
& Y \\ 
& \dot{Y} 
\end{aligned} \right)\right\|.
\end{align}
Now, we proceed to show the approximation.
\begin{lem}
\label{lem: approximating_sc}
Let the solution to high-resolution ODEs \eqref{eqn: heavy_ball_first} and \eqref{eqn: nag-sc_first} as $X = X_{s}(t)$ and that of~\eqref{eqn:ode_old_nagsc_hb} as $X = X(t)$, then we have

\begin{align}
\label{eqn approximating_sc}
\lim_{s \rightarrow 0} \max_{0 \le t \le T} \left\| X_s(t) - X(t) \right\| = 0
\end{align}
for any fixed $T > 0$
\end{lem}
In order to prove~\eqref{eqn approximating_sc}, we  prove a stronger result as
\begin{align}
\label{eqn approximating_sc_energy}
\lim_{s \rightarrow 0} \max_{0 \le t \le T}\left(  \left\| X_s(t) - X(t) \right\|^{2} + \| \dot{X}_s(t) - \dot{X}(t) \|^{2}\right) = 0.
\end{align}
Before we start to prove~\eqref{eqn approximating_sc_energy}, we first describe the standard Gronwall-inequality as below. 
\begin{lem}
\label{lem: Granwall-inequality}
Let $m(t)$, $t \in [0, T]$, be a nonnegative function satisfying the relation
\[
m(t) \leq C + \alpha \int_{0}^{t} m(s) \dd s, \quad t \in [0, T],
\]
with $C, \alpha > 0$. Then
\[
m(t) \leq C \ee^{\alpha t}
\]
for any $t \in [0, T]$.
\end{lem}
The proof is only according to simple calculus, here we omit it. 
\begin{proof}[Proof of Lemma~\ref{lem: approximating_sc}]
We separate it into two parts.
\begin{itemize}
\item For the heavy-ball method, the phase-space representations~\eqref{eqn: high_ode_heavy-b_phase} and~\eqref{eqn: high_ode_nag-sc_phase_low} tell us that 
         \begin{align*}
         \frac{\dd}{\dd t}\left( \begin{aligned}
& X_s - X \\ 
& \dot{X}_s - \dot{X}
\end{aligned} \right ) = \left( \begin{aligned}
&\quad \dot{X}_s - \dot{X} \\ 
-\mu  &\left( \dot{X}_s - \dot{X} \right)  - \left(\nabla f(X_s) - \nabla f(X)\right)
\end{aligned} \right) - \sqrt{\mu s} \left(\begin{aligned} &0 \\ &\nabla f(X_s) \end{aligned}\right)
         \end{align*}
         By the boundedness~\eqref{eqn: grad_norm_bound},~\eqref{eqn: velocity_bound} and~\eqref{eqn: velocity_bound_low} and the inequality~\eqref{eqn:Lip_low}, we have
         \begin{align*}
            &\left\| X_s(t) - X(t) \right\|^{2} + \| \dot{X}_s(t) - \dot{X}(t) \|^{2} \\
         = & 2 \int_{0}^{t} \left\langle \left( \begin{aligned}X_s(u) - X(u)\\  \dot{X}_s(u) - \dot{X}(u) \end{aligned}\right), \frac{\dd}{\dd u}\left( \begin{aligned}
& X_s(u) - X(u) \\ 
& \dot{X}_s(u) - \dot{X}(u)
\end{aligned} \right )\right\rangle \dd u+ \left\| X_s(0) - X(0) \right\|^{2} + \| \dot{X}_s(0) - \dot{X}(0) \|^{2} \\
       \leq & 2 \mathcal{L}_{1}\int_{0}^{t}  \left\| X_s(u) - X(u) \right\|^{2} + \| \dot{X}_s(u) - \dot{X}(u) \|^{2} \dd u + \left[\left( \mathcal{C}_{1} + \mathcal{C}_{3}\right)\mathcal{C}_{2} \sqrt{\mu} t + \frac{4\sqrt{s}}{(1 + \sqrt{\mu s})^{2}} \left\| \nabla f(x_{0})\right\|^{2} \right] \sqrt{s} \\
       \leq &  2 \mathcal{L}_{1}\int_{0}^{t}  \left\| X_s(u) - X(u) \right\|^{2} + \| \dot{X}_s(u) - \dot{X}(u) \|^{2} \dd u + \mathcal{C}_{4} \sqrt{s}.   
         \end{align*}
         According to Lemma~\ref{lem: Granwall-inequality}, we have
         \[
         \left\| X_s(t) - X(t) \right\|^{2} + \| \dot{X}_s(t) - \dot{X}(t) \|^{2} \leq \mathcal{C}_{4} \sqrt{s} \ee^{2\mathcal{L}_{1} t}.
         \]
 \item For NAG-\texttt{SC}, the phase-space representations~\eqref{eqn: high_ode_nag-sc_phase} and~\eqref{eqn: high_ode_nag-sc_phase_low} tell us that    
          \begin{align*}
         \frac{\dd}{\dd t}\left( \begin{aligned}
& X_s - X \\ 
& \dot{X}_s - \dot{X}
\end{aligned} \right ) = \left( \begin{aligned}
&\quad \dot{X}_s - \dot{X} \\ 
-\mu  &\left( \dot{X}_s - \dot{X} \right)  - \left(\nabla f(X_s) - \nabla f(X)\right)
\end{aligned} \right) - \sqrt{s} \left(\begin{aligned} &0 \\ & \nabla^{2} f(X_s) \dot{X}_s+ \sqrt{\mu}\nabla f(X_s) \end{aligned}\right)
         \end{align*}   
          Similarly, by the boundedness~\eqref{eqn: grad_norm_bound},~\eqref{eqn: velocity_bound} and~\eqref{eqn: velocity_bound_low} and the inequality~\eqref{eqn:Lip_low}, we have
         \begin{align*}
            &\left\| X_s(t) - X(t) \right\|^{2} + \| \dot{X}_s(t) - \dot{X}(t) \|^{2} \\
         = & 2 \int_{0}^{t} \left\langle \left( \begin{aligned}X_s(u) - X(u)\\  \dot{X}_s(u) - \dot{X}(u) \end{aligned}\right), \frac{\dd}{\dd u}\left( \begin{aligned}
& X_s(u) - X(u) \\ 
& \dot{X}_s(u) - \dot{X}(u)
\end{aligned} \right )\right\rangle \dd u+ \left\| X_s(0) - X(0) \right\|^{2} + \| \dot{X}_s(0) - \dot{X}(0) \|^{2} \\
       \leq & 2 \mathcal{L}_{1}\int_{0}^{t}  \left\| X_s(u) - X(u) \right\|^{2} + \| \dot{X}_s(u) - \dot{X}(u) \|^{2} \dd u\\
              & + \left[\left( \mathcal{C}_{1} + \mathcal{C}_{3}\right)\left( L \mathcal{C}_1 + \mathcal{C}_{2} \sqrt{\mu}\right)t + \frac{4\sqrt{s}}{(1 + \sqrt{\mu s})^{2}} \left\| \nabla f(x_{0})\right\|^{2} \right] \sqrt{s} \\
       \leq &  2 \mathcal{L}_{1}\int_{0}^{t}  \left\| X_s(u) - X(u) \right\|^{2} + \| \dot{X}_s(u) - \dot{X}(u) \|^{2} \dd u + \mathcal{C}_{5} \sqrt{s}      
         \end{align*}
           According to Lemma~\ref{lem: Granwall-inequality}, we have
         \[
         \left\| X_s(t) - X(t) \right\|^{2} + \| \dot{X}_s(t) - \dot{X}(t) \|^{2} \leq \mathcal{C}_{5} \sqrt{s} \ee^{2\mathcal{L}_{1} t}
         \]
\end{itemize}
The proof is complete.
\end{proof}

\begin{lem}
\label{lem: discrete-sc}
 The two methods, heavy-ball method and NAG-\texttt{SC}, converge to their low-resolution ODE~\eqref{eqn:ode_old_nagsc_hb} in the sense that
\[
\lim_{s \rightarrow 0} \max_{0 \le k \le T/\sqrt{s}} \left\| x_k - X(k\sqrt{s}) \right\| = 0
\]
for any fixed $T > 0$.
\end{lem}
This result has bee studied in~\cite{wilson2016lyapunov} and the method for proof refer to~\cite[Appendix 2]{su2016differential}. Combined with Corollary~\ref{coro: nag-sc-heavy-ball}, Lemma~\ref{lem: approximating_sc} and Lemma~\ref{lem: discrete-sc}, we complete the proof of Proposition~\ref{prop: exist_unique_SC}.

\subsubsection{Proof of Proposition~\ref{prop: exist_unique_C}}
\label{sec: proof_exist_unique_C}
\paragraph{Global Existence and Uniqueness}

Similar as Appendix~\ref{sec: proof_exist_unique_SC}, we first emphasize the fact that if $X_s = X_s(t)$ is the solution of high-resolution ODE~\eqref{eqn: nag-c_first}, there exists some constant $\mathcal{C}_{6}$ such that
\begin{align}
\label{eqn: velocity_bound_c}
\sup_{\frac{3\sqrt{s}}{2} \leq t < \infty}\left\| \dot{X_s}(t)\right\| \leq \mathcal{C}_6, 
\end{align}
which is only according to the following Lyapunov function
\begin{align}
\label{eqn:ef_simple_c}
\mathcal{E}(t) = \left(1 + \frac{3\sqrt{s}}{2t} \right) \left( f(X_s) - f(x^{\star}) \right)+ \frac{1}{2} \left\| \dot{X}_s\right\|^{2}.
\end{align}

Now, we proceed to prove  the global existence and uniqueness of solution to the high-resolution ODEs~\eqref{eqn: nag-c_first}.  Recall initial value problem (IVP) for first-order nonautonomous system in $\mathbb{R}^{m}$  as 
\begin{equation}
\label{eqn: IVP_nonauto}
\dot{x} = b(x, t), \quad x(0) = x_{0},
\end{equation}
of which the classical theory about global existence and uniqueness of solution is shown as below.
\begin{thm}
\label{thm: existence_uniqueness_nonauto}
Let $M \in \mathbb{R}^{m}$ be a compact manifold and $b \in C^{1}(M \times I)$, where $I = [t_{0}, \infty)$. If the vector field $b$ satisfies the global Lipschitz condition
\[
\left\| b(x, t) - b(y, t) \right\| \leq \mathfrak{L} \left\| x - y \right\|
\]
for all $(x, t), (y, t) \in M \times I$. Then for any $x_{0} \in M$, the IVP~\eqref{eqn: IVP_nonauto} has a unique solution $x(t)$ defined for all $t \in I$.
\end{thm}

The proof is consistent with Theorem $3$ and Theorem $4$ of Chapter $3.1$ in~\cite{perko2013differential} except the Lipschitz condition for the vector field
\[
\left\| b(x, t) - b(y, t) \right\| \leq \mathfrak{L} \left\| x - y \right\|
\]
instead of 
\[
\left\| b(x) - b(y) \right\| \leq \mathfrak{L} \left\| x - y \right\|
\]
for any $x, y \in M$.  The readers can also refer to~\cite{guckenheimer2013nonlinear}. Similarly, the set \[M_{\mathcal{C}_6} = \left\{ \left. (X_s, \dot{X}_s) \in \mathbb{R}^{2n} \right| \| \dot{X}_s \| \leq \mathcal{C}_6 \right\}\] is a compact manifold satisfying Theorem~\ref{thm: existence_uniqueness_nonauto} with $m = 2n$.

For NAG-\texttt{C}, the phase-space representation of  high-resolution ODE~\eqref{eqn: nag-sc_first} is
\begin{equation}
\label{eqn: high_ode_nag-c_phase}
\frac{\dd}{\dd t} \left( \begin{aligned}
& X_s \\ 
& \dot{X}_s 
\end{aligned} \right ) = \left( \begin{aligned}
& \dot{X}_s \\ 
-\frac{3}{t} \cdot &  \dot{X}_s - \sqrt{s} \nabla^{2}f(X_s) \dot{X}_s - \left(1 + \frac{3\sqrt{s}}{2t} \right) \nabla f(X_s)
\end{aligned} \right).
\end{equation}
For any $(X_s, \dot{X}_s, t), (Y_s, \dot{Y}_s, t) \in M_{\mathcal{C}_{6}} \times \left[ (3/2)\sqrt{s}, \infty\right)$, we have
\begin{align}
\label{eqn: Lipschitz_nag-c_ODE}
 & \left\| \left( \begin{aligned}
& \dot{X}_s \\ 
-\frac{3}{t}\cdot &\dot{X}_s - \sqrt{s} \nabla^{2}f(X_s)\dot{X}_s - \left(1 + \frac{3\sqrt{s}}{2t}\right) \nabla f(X_s)
\end{aligned} \right) - \left( \begin{aligned}
& \dot{Y}_s \\ 
-\frac{3}{t}\cdot  & \dot{Y}_s - \sqrt{s} \nabla^{2}f(Y_s) \dot{Y}_s- \left(1 + \frac{3\sqrt{s}}{2t}\right) \nabla f(Y_s) \end{aligned} \right)\right\| \nonumber \\
=& 
\left\| \left( \begin{aligned}
& \dot{X}_s - \dot{Y}_s \\ 
-\left(\frac{3}{t}\cdot  \bm{I}+  \sqrt{s} \nabla^2 f(X_s)\right)( &   \dot{X}_s - \dot{Y}_s )
\end{aligned} \right) \right\| 
+
\sqrt{s}  \left\| \left( \begin{aligned}
& 0 \\ 
&  \left( \nabla^2 f(X_s) -  \nabla^2 f(Y_s)\right) \dot{Y}_s
\end{aligned} \right)\right\| \nonumber \\
& \qquad +
\left(1 + \frac{3\sqrt{s}}{2t}\right) \left\| \left( \begin{aligned}
& 0 \\ 
&  \nabla f(X_s) -  \nabla f(Y_s)
\end{aligned} \right)\right\| \nonumber \\
 \leq & \sqrt{ 1 + \frac{18}{t_{0}^2}  + 2sL^{2}} \left\| \dot{X}_s - \dot{Y}_s \right\| + \left[ \sqrt{s} \mathcal{C}_{6} L'+ \left(1 + \frac{3\sqrt{s}}{2t_{0}}\right)L\right] \left\|  X_s - Y_s \right\| \nonumber \\
\leq &  2 \max\left\{ \sqrt{ 1 + \frac{8}{s}  + 2sL^{2}}, \sqrt{s} \mathcal{C}_{6} L'+ 2 L \right\} \left\| \left( \begin{aligned}
& X_s \\ 
& \dot{X}_s 
\end{aligned} \right) - \left( \begin{aligned}
& Y_s \\ 
& \dot{Y}_s 
\end{aligned} \right)\right\|. 
\end{align}

Based on the phase-space representation~\eqref{eqn: high_ode_nag-c_phase}, together with~\eqref{eqn: Lipschitz_nag-c_ODE}, Theorem~\ref{thm: existence_uniqueness_nonauto} leads the following Corollary.
\begin{coro}
\label{coro: nag-c}
For any $f \in \mathcal{F}^2(\mathbb{R}^n) := \cup_{L > 0} \mathcal{F}^2_L(\mathbb{R}^n)$, the ODE~\eqref{eqn: nag-c_first} with the specified initial conditions has a unique global solution $X \in C^2(I; \mathbb{R}^n)$.
\end{coro}

\paragraph{Approximation}
Using a linear transformation $t + (3/2)\sqrt{s}$ instead of $t$, we can rewrite high-resolution ODE~\eqref{eqn: nag-c_first}  as
\begin{align}
\label{eqn: high_resolution_new_c}
\ddot{X_s}(t) + \frac{3}{t + 3\sqrt{s}/2} \dot{X_s}(t) + \sqrt{s} \nabla^{2} f(X_s(t)) \dot{X_s}(t) + \left(1 + \frac{3\sqrt{s}}{2t + 3\sqrt{s}}\right) \nabla f(X_s(t)) = 0
\end{align}
for $t \ge 0$, with initial $X_s(0) = x_0$ and $\dot X_s(0) = -\sqrt{s} \nabla f(x_0)$,
of which the phase-space representation is 
\begin{equation}
\label{eqn: high_ode_nag-c_phase_new}
\frac{\dd}{\dd t} \left( \begin{aligned}
& X_s \\ 
& \dot{X}_s 
\end{aligned} \right ) = \left( \begin{aligned}
& \dot{X}_s \\ 
-\frac{3}{t + 3\sqrt{s}/2} \cdot &  \dot{X}_s - \sqrt{s} \nabla^{2}f(X_s) \dot{X}_s - \left(1 + \frac{3\sqrt{s}}{2t + 3\sqrt{s}} \right) \nabla f(X_s)
\end{aligned} \right).
\end{equation}
Here, we adopt the technique $\max\{\delta, t\}$ instead of $t$ for any $\delta > 0$ to overcome the singular point $t = 0$,  which is used firstly in~\cite{su2016differential}. Then~\eqref{eqn: high_ode_nag-c_phase_new} is replaced into 
\begin{equation}
\label{eqn: high_ode_nag-c_phase_new_modif}
\frac{\dd}{\dd t} \left( \begin{aligned}
& X_s^{\delta} \\ 
& \dot{X}_s^{\delta} 
\end{aligned} \right ) = \left( \begin{aligned}
& \dot{X}_s^{\delta} \\ 
-\frac{3}{\max\{\delta, t\} + 3\sqrt{s}/2} \cdot &  \dot{X}_s^{\delta} - \sqrt{s} \nabla^{2}f(X_s) \dot{X}_s^{\delta} - \left(1 + \frac{3\sqrt{s}}{2\max\{\delta, t\} + 3\sqrt{s}} \right) \nabla f(X_s^{\delta})
\end{aligned} \right),
\end{equation}
with the initial $X_s^\delta(0) = x_0$ and $\dot{X}_s^\delta(0) = -\sqrt{s} \nabla f(x_0)$.
Recall the low-resolution ODE~\eqref{eqn:ode_old_nagmc}, with the above technique, the phase-space representation is proposed as
\begin{equation}
\label{eqn: high_ode_nag-c_phase_low_modif}
\frac{\dd}{\dd t}\left( \begin{aligned}
& X^{\delta} \\ 
& \dot{X}^{\delta} 
\end{aligned} \right )
 = \left( \begin{aligned}
& \dot{X}^{\delta} \\ 
-\frac{3}{\max\{t, \delta\}} \cdot &  \dot{X}^{\delta} - \nabla f(X^{\delta})
\end{aligned} \right),
\end{equation}
with the initial $X_s^\delta(0) = x_0$ and $\dot{X}_s^\delta(0) =0$. Then according to~\eqref{eqn: high_ode_nag-c_phase_new_modif} and~\eqref{eqn: high_ode_nag-c_phase_low_modif}, if we can prove for any $\delta > 0$ and any $t \in [0, T]$, the following equality holds
\[
\lim_{s \rightarrow 0} \|X_{s}^{\delta}(t) - X^{\delta}(t) \| =0.
\]
Then, we can obtain the desired result as
\[
\lim_{s \rightarrow 0}\left\|X_{s}(t) - X(t) \right\| = \lim_{s \rightarrow 0} \lim_{\delta \rightarrow 0} \|X_{s}^{\delta}(t) - X^{\delta}(t) \| =   \lim_{\delta \rightarrow 0} \lim_{s \rightarrow 0} \|X_{s}^{\delta}(t) - X^{\delta}(t) \| =0.
\]

Similarly, using Lyapunov function argument, we can show that the solutions $X_{s}^{\delta}$ and $X^{\delta}$ satisfy
\begin{align}
\label{eqn: bound_new_c}
\sup_{0 \leq t < \infty}\left\| \dot{X}^{\delta}_s(t)\right\| \leq \mathcal{C}_7 \quad \text{and }\quad \sup_{0 \leq t < \infty}\left\| \nabla f(X_s^{\delta}(t))\right\| \leq \mathcal{C}_8,
\end{align}
and
\begin{align}
\label{eqn: velocity_bound_low_c}
\sup_{0 \leq t < \infty}\left\| \dot{X}^{\delta}(t)\right\| \leq \mathcal{C}_9 \quad \text{and} \quad \sup_{0 \leq t < \infty}\left\| \nabla f(X^{\delta}(t))\right\| \leq \mathcal{C}_{10}.
\end{align}

Simple calculation tells us that for any $(X, \dot{X}), (Y, \dot{Y}) \in \mathbb{R}^{2n}$, there exists some constant $\mathcal{L}_{2} > 0$ such that
\begin{align}
\label{eqn:Lip_low_c}
 &\left\| \left( \begin{aligned}
& \dot{X} \\ 
-\frac{3}{\max\{t, \delta\} + (3/2)\sqrt{s}} \cdot&  \dot{X}  -  \nabla f(X)
\end{aligned} \right) - \left( \begin{aligned}
& \dot{Y} \\ 
-\frac{3}{\max\{t, \delta\}  + (3/2)\sqrt{s}} \cdot &  \dot{Y}  -  \nabla f(Y) \end{aligned} \right)\right\| \nonumber \\
\leq &\mathcal{L}_{2}  \left\| \left( \begin{aligned}
& X \\ 
& \dot{X} 
\end{aligned} \right) - \left( \begin{aligned}
& Y \\ 
& \dot{Y} 
\end{aligned} \right)\right\|.
\end{align}
for all $t \geq 0$. Now, we proceed to show the approximation.
\begin{lem}
\label{lem: approximating_c}
Denote the solution to high-resolution ODE~\eqref{eqn: nag-c_first} as $X = X_{s}(t)$ and that to~\eqref{eqn:ode_old_nagmc} as $X = X(t)$.
We have
\begin{align}
\label{eqn approximating_c}
\lim_{s \rightarrow 0} \max_{0 \le t \le T} \left\| X_s(t) - X(t) \right\| = 0
\end{align}
for any fixed $T > 0$
\end{lem}
In order to prove~\eqref{eqn approximating_c}, we  prove a stronger result
\begin{align}
\label{eqn approximating_c_energy}
\lim_{s \rightarrow 0} \max_{0 \le t \le T}\left(  \left\| X_s(t) - X(t) \right\|^{2} + \| \dot{X}_s(t) - \dot{X}(t) \|^{2}\right) = 0.
\end{align}
\begin{proof}[Proof of Lemma~\ref{lem: approximating_c}]
The phase-space representation~\eqref{eqn: high_ode_nag-c_phase_new_modif} and~\eqref{eqn: high_ode_nag-c_phase_low_modif} tell us that
\begin{align*}
        \frac{\dd}{\dd t}\left( \begin{aligned}
& X^\delta_s - X^\delta \\ 
& \dot{X}^\delta_s - \dot{X}^\delta
\end{aligned} \right ) = &\left( \begin{aligned}
&\quad \dot{X}_s^\delta - \dot{X}^\delta \\ 
-  \frac{3}{\max \{t, \delta\} + (3/2)\sqrt{s}} \cdot&\left( \dot{X}_s^\delta - \dot{X}^\delta \right)  - \left(\nabla f(X_s^\delta) - \nabla f(X^\delta)\right)
\end{aligned} \right) \\
                                    &- \sqrt{s} \left(\begin{aligned} &0 \\ & \nabla^{2} f(X^\delta_s) \dot{X}^\delta_s+ \frac{3}{2\max \{t, \delta\} + 3\sqrt{s}} \cdot \nabla f(X_s) - \frac{9}{\max \{t, \delta\} \left(2\max \{t, \delta\} + 3\sqrt{s}\right)} \nabla f(X) \end{aligned}\right)
\end{align*}

By the boundedness~\eqref{eqn: bound_new_c} and~\eqref{eqn: velocity_bound_low_c} and the Lipschitz inequality~\eqref{eqn:Lip_low_c}, we have
         \begin{align*}
            &\left\| X^\delta_s(t) - X^\delta(t) \right\|^{2} + \left\| \dot{X}^\delta_s(t) - \dot{X}^\delta(t) \right\|^{2} \\
         = & 2 \int_{0}^{t} \left\langle \left( \begin{aligned}X^\delta_s(u) - X^\delta(u)\\  \dot{X}^\delta_s(u) - \dot{X}^\delta(u) \end{aligned}\right), \frac{\dd}{\dd u}\left( \begin{aligned}
& X^\delta_s(u) - X^\delta(u) \\ 
& \dot{X}^\delta_s(u) - \dot{X}^\delta(u)
\end{aligned} \right )\right\rangle \dd u+ \left\| X^\delta_s(0) - X^\delta(0) \right\|^{2} + \left \| \dot{X}^\delta_s(0) - \dot{X}^\delta(0) \right \|^{2} \\
       \leq & 2 \mathcal{L}_{2}\int_{0}^{t} \left\| X^\delta_s(u) - X^\delta(u) \right\|^{2} + \left \| \dot{X}^\delta_s(u) - \dot{X}^\delta(u) \right \|^{2} \dd u\\
              & + \left[ \left( \mathcal{C}_{7} + \mathcal{C}_{9}\right)\left( L \mathcal{C}_7  +  \frac{3\mathcal{C}_8}{2\delta} + \frac{9\mathcal{C}_{10}}{2\delta^{2}} \right)t + \sqrt{s} \left\| \nabla f(x_{0})\right\|^{2} \right] \sqrt{s} \\
       \leq &  2 \mathcal{L}_{2}\int_{0}^{t}  \left\| X_s(u) - X(u) \right\|^{2} + \| \dot{X}_s(u) - \dot{X}(u) \|^{2} \dd u + \mathcal{C}_{11} \sqrt{s}      
         \end{align*}
 According to Lemma~\ref{lem: Granwall-inequality}, we obtain the result as~\eqref{eqn approximating_c_energy}
         \[
         \left\| X^\delta_s(t) - X^\delta(t) \right\|^{2} + \| \dot{X}^\delta_s(t) - \dot{X}^\delta(t) \|^{2} \leq \mathcal{C}_{11} \sqrt{s} \ee^{2\mathcal{L}_{2} t}
         \]
The proof is complete.         
\end{proof}

\begin{lem}[Theorem $2$~\cite{su2016differential}]
\label{lem: discrete-c}
NAG-\texttt{C} converges to its low-resolution ODE in the sense that
\[
\lim_{s \rightarrow 0} \max_{0 \le k \le T/\sqrt{s}} \left\| x_k - X(k\sqrt{s}) \right\| = 0
\]
for any fixed $T > 0$.
\end{lem}
Combined with Corollary~\ref{coro: nag-c}, Lemma~\ref{lem: approximating_c} and Lemma~\ref{lem: discrete-c}, we complete the proof of Proposition~\ref{prop: exist_unique_C}.

\subsection{Closed-Form Solutions for Quadratic Functions}
\label{sec: qf_chf}
In this section, we propose the closed-form solutions to the three high-resolution ODEs for the quadratic objective function
\begin{align}\label{eqn: quadratic_objective}
	f(x) = \frac{1}{2}\theta x^{2}.
\end{align}
where $\theta$ is the parameter suitable for the function in $\mathcal{S}_{\mu, L}^{2}(\mathbb{R}^{n})$ and $\mathcal{F}_{L}^{2}(\mathbb{R}^{n})$.
 We compare them with the corresponding low-resolution ODEs and show the key difference.   Throughout this section, both $c_{1}$ and $c_{2}$ are arbitrary real constants.

\subsubsection{Oscillations and Non-Oscillations}
\label{subsec: quad_strongl_convex}

For any function $f(x) \in \mathcal{S}_{\mu, L}^{2}(\mathbb{R}^{n})$, the parameter $\theta$ is set in $\left[ \mu, L \right]$.
 First, plugging the quadratic objective~(\ref{eqn: quadratic_objective}) into the low-resolution ODE~(\ref{eqn:ode_old_nagsc_hb}) of both NAG-\texttt{SC} and heavy-ball method, we have
\begin{align}\label{eqn: low_resol_strongly_quadratic}
\ddot{X} + 2\sqrt{\mu}\dot{X} + \theta X = 0.
\end{align}
The closed-form solution of~(\ref{eqn: low_resol_strongly_quadratic}) can be shown from the theory of ODE, as below.
\begin{itemize}
\item When $\theta > \mu$, that is, $4 \mu - 4 \theta < 0$, the closed-form solution is the superimposition of two independent oscillation solutions
        \[
        X(t) = c_{1} \ee^{- \sqrt{\mu} t} \cos \left( \sqrt{\theta - \mu} \cdot t\right) + c_{2} \ee^{- \sqrt{\mu} t} \sin \left( \sqrt{\theta - \mu} \cdot t \right),
        \]
        of which the asymptotic estimate is 
        \[
        \left\| X(t) \right\| = \Theta\left( \ee^{- \sqrt{\mu} t} \right).
        \]
\item When $\theta = \mu$, that is, $4\mu - 4 \theta = 0$, the closed-form solution is the superimposition of two independent non-oscillation solutions
        \[
        X(t) = \left(c_{1} + c_{2}t \right) \ee^{- \sqrt{\mu} t},
        \]
        of which the asymptotic estimate is 
        \[
        \left\| X(t) \right\| = \Theta\left( t \ee^{- \sqrt{\mu} t} \right).
        \]
\end{itemize}

Second, plugging the quadratic objective~(\ref{eqn: quadratic_objective}) into the high-resolution ODE~(\ref{eqn: nag-sc_first}) of  NAG-\texttt{SC}, we have
\begin{align}\label{eqn: high_resol_strongly_quadratic_nag-sc}
\ddot{X} + (2\sqrt{\mu} + \sqrt{s} \theta )\dot{X} + (1 + \sqrt{\mu s}) \theta X = 0.
\end{align}
The closed-form solutions to~(\ref{eqn: high_resol_strongly_quadratic_nag-sc}) are shown as below.
\begin{itemize}
\item When $s < \frac{4(\theta - \mu)}{\theta^{2}}$, that is,  $4(\mu - \theta) + s\theta^{2} < 0$, the closed-form solution is the superimposition of two independent oscillation solutions
        \[
        X(t) = \ee^{- \left( \sqrt{\mu} + \frac{\sqrt{s} \theta}{2}\right) t} \left[ c_{1}  \cos \left( \sqrt{ (\theta - \mu) - \frac{1}{4}s\theta^{2}} \cdot t\right) + c_{2}  \sin \left(  \sqrt{ (\theta - \mu) - \frac{1}{4} s\theta^{2}} \cdot t\right) \right],
        \]
        the asymptotic estimate of which is 
        \[
        \left\| X(t) \right\| = \Theta\left( \ee^{- \left( \sqrt{\mu} + \frac{\sqrt{s} \theta}{2}\right) t} \right) \leq o\left( \ee^{- \sqrt{\mu} t} \right).
        \]
\item When $s = \frac{4(\theta - \mu)}{\theta^{2}}$, that is,  $4(\mu - \theta) + s \theta^{2} = 0$, the closed-form solution is the superimposition of two independent non-oscillation solutions
        \[
        X(t) = \left(c_{1} + c_{2}t \right) \ee^{-\left( \sqrt{\mu} + \frac{\sqrt{s} \theta}{2}\right) t},
        \]
        the asymptotic estimate of which is 
        \[
        \left\| X(t) \right\| \leq O\left( t\ee^{- \left( \sqrt{\mu} + \frac{\sqrt{s} \theta}{2}\right) t} \right) \leq o\left( \ee^{- \sqrt{\mu} t} \right).
        \]
\item When $s > \frac{4(\theta - \mu)}{\theta^{2}}$, that is,  $4(\mu - \theta) + s \theta^{2} > 0$, the closed-form solution is also the superimposition of two independent non-oscillation solutions
        \[
        X(t) = c_{1}  \ee^{-\left( \sqrt{\mu} + \frac{\sqrt{s} \theta}{2} + \sqrt{(\mu - \theta ) + \frac{s\theta^{2}}{4}}\right) t} + c_{2}  \ee^{-\left( \sqrt{\mu} + \frac{\sqrt{s} \theta}{2} - \sqrt{( \mu - \theta) + \frac{s\theta^{2}}{4}}\right) t} ,
        \]
          the asymptotic estimate of which is 
        \[
        \left\| X(t) \right\| \leq O\left( \ee^{-\left( \sqrt{\mu} + \frac{\sqrt{s} \theta}{2} - \sqrt{( \mu  - \theta) + \frac{s\theta^{2}}{4}}\right) t} \right) \leq o\left( \ee^{- \sqrt{\mu} t} \right).
        \]
\end{itemize}
Note that a simple calculation shows
\[
\frac{4 (\theta - \mu)}{\theta^{2}}= \frac{4}{\theta - \mu + \frac{\mu^2}{\theta - \mu} + 2} \leq \frac{2}{1 + \mu},\qquad \text{for}\;\; \theta \geq \mu.
\]
Hence, when the step size satisfies $s \geq 2$, there is always no oscillation in the closed-form solution of~(\ref{eqn: high_resol_strongly_quadratic_nag-sc}).

Finally, plugging the quadratic objective~(\ref{eqn: quadratic_objective}) into the high-resolution ODE~(\ref{eqn: heavy_ball_first}) of  the heavy-ball method, we have
\begin{align}\label{eqn: high_resol_strongly_quadratic_heavy-b}
\ddot{X} + 2\sqrt{\mu}\dot{X} +(1 + \sqrt{\mu s}) \theta X = 0.
\end{align}
Since $4\mu - 4  (1 + \sqrt{\mu s}) \theta < 0$ is well established, the closed-form solution of~\eqref{eqn: high_resol_strongly_quadratic_heavy-b} is the superimposition of two independent oscillation solutions
\[
  X(t) = c_{1} e^{- \sqrt{\mu} t} \cos \left( \sqrt{(1 + \sqrt{\mu s})\theta - \mu}  \cdot t\right) + c_{2} e^{- \sqrt{\mu} t} \sin \left( \sqrt{(1 + \sqrt{\mu s}) \theta - \mu} \cdot t \right),
\]
the asymptotic estimate is
\[
\left\| X(t) \right\| = \Theta\left( e^{- \sqrt{\mu} t} \right).
\]

In summary, both the closed-form solutions to~(\ref{eqn: low_resol_strongly_quadratic}) and~\eqref{eqn: high_resol_strongly_quadratic_heavy-b}
are oscillated except the fragile condition $\theta = \mu$ and the speed of linear convergence is  $\Theta\left( e^{- \sqrt{\mu} t} \right)$. However, the rate of convergence in the closed-form solution to the high-resolution ODE~(\ref{eqn: high_resol_strongly_quadratic_nag-sc}) is always faster than $\Theta\left( e^{- \sqrt{\mu} t} \right)$. Additionally, when the step size $s \geq 2$,  there is always no oscillation in the closed-form solution of the high-resolution ODE~(\ref{eqn: high_resol_strongly_quadratic_nag-sc}).

\subsubsection{Kummer's Equation and Confluent Hypergeometric Function}
\label{subsec: quad_convex}

For any function $f(x) \in \mathcal{F}^{2}_{L}(\mathbb{R}^{n})$, the parameter $\theta$ is required to located in $(0, L]$. Plugging the quadratic objective~(\ref{eqn: quadratic_objective}) into the low-resolution ODE~(\ref{eqn:ode_old_nagmc}) of NAG-\texttt{C}, we have
\[
\ddot{X}+ \frac{3}{t} \dot{X} + \theta X = 0,
\]
the closed-form solution of which has been proposed in~\cite{su2016differential}
\[
X(t) = \frac{1}{\sqrt{\theta} t } \cdot \left[ c_{1}J_{1}\left( \sqrt{\theta} t \right) + c_{2}Y_{1}\left( \sqrt{\theta} t \right) \right],
\]
where $J_{1}(\cdot)$ and $Y_{1}(\cdot)$ are the Bessel function of the first kind and the second kind, respectively. According to the asymptotic property of Bessel functions, 
\[
J_{1}(\sqrt{\theta} t) \sim \frac{1}{\sqrt{t}} \quad \text{and} \quad Y_{1}( \sqrt{\theta} t) \sim \frac{1}{\sqrt{ t}},
\]
we obtain the following estimate
\[
\left\| X(t) \right\| = \Theta\left(\frac{1}{t^{\frac{3}{2}}}\right).
\]

Now, we plug the quadratic objective~(\ref{eqn: quadratic_objective}) into the high-resolution ODE~(\ref{eqn: nag-c_first}) of NAG-\texttt{C} and obtain
\begin{align}\label{eqn: quadratic_high_resolution}
\ddot{X} + \left( \frac{3}{t} + \theta \sqrt{s} \right) \dot{X} + \left( 1 + \frac{3\sqrt{s}}{2t} \right)\theta X= 0.
\end{align}
For convenience, we define two new parameters as 
\[
\xi = \sqrt{s\theta^2 - 4\theta} \quad \text{and} \quad \rho = \frac{\theta \sqrt{s} + \sqrt{s\theta^{2} - 4\theta}}{2}. 
\]
Let $Y = X \ee^{ \rho t}$ and $t' = \xi t$, the high-resolution ODE~(\ref{eqn: quadratic_high_resolution}) can be rewritten as
\[
t'\ddot{Y}(t') + (3 - t') \dot{Y}(t') - (3/2)  Y(t') = 0,
\]
which actually corresponds to the Kummer's equation.  According to the closed-form solution to Kummer's equation, the high-resolution ODE~\eqref{eqn: quadratic_high_resolution} for quadratic function can be solved analytically as 
\begin{align}
\label{eqn: solution_ode_nagm-c}
X(t) = \ee^{-\rho t} \left[c_{1} M \left(\frac{3}{2}, 3, \xi t \right) + c_{2} U \left(\frac{3}{2}, 3, \xi t \right) \right]
\end{align}
where $M(\cdot, \cdot, \cdot)$ and $U(\cdot, \cdot, \cdot)$ are the confluent hypergeometric functions of the first kind and the second kind. The integral expressions of $M(\cdot, \cdot, \cdot)$ and $U(\cdot, \cdot, \cdot)$ are given as 
$$
\left\{\begin{aligned}
        &  M \left(\frac{3}{2}, 3, \xi t \right) = \frac{\Gamma(3)}{\Gamma\left(\frac{3}{2}\right)^{2}} \int_{0}^{1} e^{\xi t u} u^{\frac{1}{2}} (1 - u)^{\frac{1}{2}}du\\
        & U \left(\frac{3}{2}, 3, \xi t \right) = \frac{1}{\Gamma\left(\frac{3}{2}\right) } \int_{0}^{1} e^{\xi t u} u^{\frac{1}{2} } (1 - u)^{\frac{1}{2}}du.
        \end{aligned}\right.
$$
Since the possible value of $\arg (\xi t)$ either $0$ or $\pi/2$,  we have
\begin{equation}
\label{eqn: estimate_confluent}
\left\{
\begin{aligned}
        &  M \left(\frac{3}{2}, 3, \xi t \right) \sim  \Gamma(3) \left( \frac{e^{\xi t} (\xi t)^{ -\frac{3}{2}}}{\Gamma\left( \frac{3}{2} \right)} + \frac{(- \xi t)^{- \frac{3}{2}}}{\Gamma\left(\frac{3}{2} \right)}\right)\\
        & U \left(\frac{3}{2}, 3, \xi t \right)  \sim (- \xi t)^{- \frac{3}{2}}.
 \end{aligned}
 \right.
 \end{equation}
Apparently, from the asymptotic estimate of~\eqref{eqn: estimate_confluent}, we have 
        \begin{itemize}
        \item When $s < 4/ \theta$, that is, $s \theta^{2} - 4 \theta < 0$, the closed-form solution~\eqref{eqn: solution_ode_nagm-c} is estimated as 
                 \[
                 \left\| X(t) \right\| \leq \Theta\left( t ^{ - \frac{3}{2}} \ee^{-\frac{\sqrt{s} \theta  t}{2}} \right).
                 \] 
                Hence, when the step size satisfies $s < 4/L$, the above upper bound always holds.
        \item When $s \geq 4/\theta$, that is, $s \theta^{2} - 4 \theta \geq 0$, the closed-form solution~\eqref{eqn: solution_ode_nagm-c} is estimated as 
                 \[
                 \left\| X(t) \right\| \sim \ee^{-\frac{  \sqrt{s} \theta - \sqrt{s \theta^{2} - 4 \theta}}{2} \cdot t}  t^{- \frac{3}{2} } . 
                 \]    
                Apparently, we can bound 
                 \[
                   \left\| X(t) \right\| \leq O \left(e^{- \frac{t}{\sqrt{s}}} t^{- \frac{3}{2}} \right) = O\left( e^{- \frac{t}{\sqrt{s}} - \frac{3 \log t}{2}}\right)
                 \]
                 and 
                  \[
                   \left\| X(t) \right\| \geq \Omega \left(e^{- \frac{2t}{\sqrt{s}}} t^{- \frac{3}{2}} \right) = \Omega \left( e^{- \frac{2t}{\sqrt{s}} - \frac{3 \log t}{2}}\right).
                  \]
           \end{itemize}      

\section{Technical Details in Section~\ref{sec:strongly}}
\label{sec: heavyball_e}
\subsection{Proof of Lemma~\ref{lm:heavyball_e}}
\label{subsec: heavy_ball_continuous}

With Cauchy-Schwarz inequality 
\[
\|  \dot{X} + 2\sqrt{\mu} ( X - x^{\star} ) \|^{2} \leq 2 \left( \| \dot{X} \|^{2}  + 4 \mu \left\| X - x^{\star} \right\|_{2}^{2}   \right),
\]
the Lyapunov function~(\ref{eqn: energy_heavy-b_ode}) can be estimated as 
\begin{align}
\label{eqn: estimate_ef_heavy-b_ode}
\mathcal{E} \leq (1 + \sqrt{\mu s})\left( f(X) - f(x^{\star}) \right) + \frac{3}{4}  \| \dot{X} \|^{2} + 2 \mu \left\| X - x^{\star} \right\|^{2}.
\end{align}
Along the solution to the high-resolution ODE~(\ref{eqn: heavy_ball_first}), the time derivative of the Lyapunov function~(\ref{eqn: energy_heavy-b_ode}) is
\begin{align*}
\frac{\dd \mathcal{E}}{\dd t} & =  (1 + \sqrt{\mu s} ) \left\langle \nabla f(X), \dot{X} \right\rangle + \frac{1}{2}\left\langle \dot{X}, - 2 \sqrt{\mu} \dot{X} -  (1 + \sqrt{\mu s} ) \nabla f(X) \right\rangle \\
                                         & \quad + \frac{1}{2} \left\langle  \dot{X} + 2\sqrt{\mu} \left( X - x^{\star} \right), -  (1 + \sqrt{\mu s} ) \nabla f(X) \right\rangle \\
                                         & = - \sqrt{\mu} \left[ \| \dot{X} \|_{2}^{2} + ( 1 + \sqrt{\mu s}  ) \left\langle \nabla f(X), X - x^{\star} \right\rangle \right].
\end{align*}    
With~\eqref{eqn: estimate_ef_heavy-b_ode} and the inequality for any function $f(x) \in \mathcal{S}_{\mu, L}^{2}(\mathbb{R}^{n})$ 
\[
f(x^{\star}) \geq f(X) + \left\langle \nabla f(X), x^{\star} - X \right\rangle + \frac{\mu}{2} \left\| X - x^{\star} \right\|_{2}^{2},
\]
the time derivative of the Lyapunov function can be estimated as
\begin{align*}
\frac{\dd \mathcal{E}}{\dd t}  & \leq - \sqrt{\mu} \left[ ( 1 + \sqrt{\mu s}  ) ( f(X) - f(x^{\star}) ) +   \| \dot{X} \|_{2}^{2} + \frac{\mu}{2} \left\| X - x^{\star} \right\|_{2}^{2} \right] \\
                                         & \leq - \frac{\sqrt{\mu}}{4} \mathcal{E}
\end{align*}
Hence, the proof is complete.

\subsection{Completing the Proof of Lemma \ref{lm:nag_sc_ek1}}
\label{subsec:proof-lemma-nsc}


\subsubsection{Derivation of~(\ref{eq:proof_app_is})}
Here, we first point out that
\begin{align}
\label{eq:to_be_proved_appendix}
\mathcal{E}(k+1) - \mathcal{E}(k) \le  & - \frac{\sqrt{\mu s}}{1 - \sqrt{\mu s}} \left[   \frac{1 + \sqrt{\mu s} }{1 - \sqrt{\mu s} } \left(\left\langle \nabla f(x_{k + 1}), x_{k + 1} - x^{\star} \right\rangle  -  s\left\| \nabla f(x_{k + 1}) \right\|^{2}\right) + \left\| v_{k + 1} \right\|^{2} \right ] \nonumber \\
                                                           &- \frac{1}{2}\left( \frac{1 + \sqrt{\mu s}}{1 - \sqrt{\mu s}} + \frac{1 - \sqrt{\mu s}}{1 + \sqrt{\mu s}} \right) \left( \frac{1}{L} -  s\right) \left\| \nabla f(x_{k + 1}) - \nabla f(x_{k}) \right\|^{2}
\end{align}
implies \eqref{eq:proof_app_is} with $s \leq 1/L$.  With~\eqref{eq:to_be_proved_appendix}, 
 noting the basic inequality for $f(x) \in \mathcal{S}_{\mu, L}^{1}(\mathbb{R}^n)$ as 
$$
	\left\{
	\begin{aligned}
	&  f(x^{\star}) \geq f(x_{k + 1}) + \left\langle \nabla f(x_{k + 1}), x^{\star} - x_{k+1}\right\rangle + \frac{1}{2L} \left\| \nabla f(x_{k + 1}) \right\|_{2}^{2}        \\
	&  f(x^{\star}) \geq f(x_{k + 1}) + \left\langle \nabla f(x_{k + 1}), x^{\star} - x_{k+1}\right\rangle + \frac{\mu}{2} \left\| x_{k + 1} - x^{\star} \right\|_{2}^{2},
	\end{aligned}
	\right.
$$
when the step size satisfies $s \leq 1/(2L) \leq 1/L$, 
we have 
\begin{align*}
	\mathcal{E}(k + 1) - \mathcal{E}(k) & \leq - \frac{\sqrt{\mu s}}{1 - \sqrt{\mu s}} \left[ \left( \frac{1 + \sqrt{\mu s} }{1 - \sqrt{\mu s}}\right) \left( f(x_{k + 1}) - f(x^{\star}) \right)  +   \frac{1}{2L} \left( \frac{ \sqrt{\mu s} }{1 - \sqrt{\mu s}}\right) \left\|\nabla f(x_{k + 1})\right\|^{2} \right. \\
	& \qquad \left.+ \frac{\mu}{2} \left( \frac{1 }{1 - \sqrt{\mu s}}\right) \left\|x_{k + 1} - x^{\star} \right\|^{2}   -  \left( \frac{1 + \sqrt{\mu s} }{1 - \sqrt{\mu s}}\right) s \left\| \nabla f(x_{k + 1}) \right\|^{2} + \left\| v_{k + 1}\right\|^{2}  \right] \\
	& \leq - \sqrt{\mu s} \left[   \left(  \frac{1}{ 1 - \sqrt{\mu s}  }  \right)^{2} \left( f(x_{k + 1}) - f(x^{\star}) - s \left\| \nabla f(x_{k + 1}) \right\|^{2} \right) \right. \\
	&\qquad\qquad\quad \left.   + \frac{\sqrt{\mu s} }{(1 - \sqrt{\mu s})^{2}}  \left( f(x_{k + 1}) - f(x^{\star}) - \frac{s}{2} \left\| \nabla f(x_{k + 1}) \right\|^{2} \right)\right. \\
	&\qquad\qquad\quad \left. + \frac{\mu}{2(1 - \sqrt{\mu s} )^{2}} \left\| x_{k + 1} - x^{\star} \right\|^{2} + \frac{1}{ 1 - \sqrt{\mu s} } \left\| v_{k + 1} \right\|^{2} \right]\\
	& \leq  - \sqrt{\mu s} \left[ \frac{1 - 2Ls}{ \left( 1 - \sqrt{\mu s}  \right)^{2} }  \left( f(x_{k + 1}) - f(x^{\star})  \right) +  \frac{1}{ 1 - \sqrt{\mu s} } \left\| v_{k + 1} \right\|^{2}   \right. \\
	&\qquad\qquad\quad\left. + \frac{\mu}{2(1 - \sqrt{\mu s} )^{2}} \left\| x_{k + 1} - x^{\star} \right\|^{2} \right.\\
	&\qquad\qquad\quad\left.+  \frac{\sqrt{\mu s} }{(1 - \sqrt{\mu s})^{2}}  \left( f(x_{k + 1}) - f(x^{\star}) - \frac{s}{2} \left\| \nabla f(x_{k + 1}) \right\|^{2} \right)\right]. 
	\end{align*}

\subsubsection{Derivation of~(\ref{eq:to_be_proved_appendix})}
Now, we show the derivation of~\eqref{eq:to_be_proved_appendix}. Recall the discrete Lyapunov function~(\ref{eqn:lypunov_NAGM-SC_strongly}),  
\begin{align*}
\mathcal{E}(k) =  & \underbrace{\left( \frac{1 + \sqrt{\mu s} }{1 - \sqrt{\mu s} } \right)  \left( f(x_{k}) - f(x^{\star}) \right)}_{\mathbf{I}} +  \underbrace{\frac{1}{4} \left\| v_{k} \right\|^{2}}_{\mathbf{II}} + \underbrace{\frac{1}{4} \left\|  v_{k} + \frac{2\sqrt{\mu}}{1 - \sqrt{\mu s}} ( x_{k + 1} - x^{\star} ) +   \sqrt{s} \nabla f(x_{k}) \right\|^{2}}_{\mathbf{III}}\\
&  \underbrace{- \frac{s}{2} \left( \frac{1}{1 - \sqrt{\mu s} } \right) \left\| \nabla f(x_{k})\right\|^{2}}_{\textbf{additional\; term}}.
\end{align*}     
For convenience, we calculate the difference between 
$\mathcal{E}(k)$ and  $\mathcal{E}(k + 1)$ by the three parts, $\mathbf{I}$, $\mathbf{II}$ and $\mathbf{III}$ respectively.
	\begin{itemize}
		\item For the part $\mathbf{I}$, potential, with the convexity, we have
		\begin{align*}
		&  \left( \frac{1 + \sqrt{\mu s} }{1 - \sqrt{\mu s} } \right)  \left( f(x_{k + 1}) - f(x^{\star}) \right) - \left( \frac{1 + \sqrt{\mu s} }{1 - \sqrt{\mu s} } \right) \left( f(x_{k}) - f(x^{\star}) \right) \\
		\leq  & \left( \frac{1 + \sqrt{\mu s} }{1 - \sqrt{\mu s} } \right) \left[ \left\langle \nabla f(x_{k + 1}), x_{k + 1} - x_{k} \right\rangle - \frac{1}{2L} \left\| \nabla f(x_{k + 1}) - \nabla f(x_{k}) \right\|^{2} \right] \\
		\leq  &     \underbrace{\left( \frac{1 + \sqrt{\mu s} }{1 - \sqrt{\mu s} } \right) \sqrt{s} \left\langle \nabla f(x_{k + 1}), v_{k}  \right\rangle}_{\mathbf{I}_{1}}  \underbrace{- \frac{1}{2L} \left( \frac{1 + \sqrt{\mu s} }{1 - \sqrt{\mu s}} \right) \left\| \nabla f(x_{k + 1}) - \nabla f(x_{k}) \right\|^{2}}_{\mathbf{I}_{2}}.   
		\end{align*}
		\item For the part $\mathbf{II}$, kinetic energy, with the phase representation of NAG-\texttt{SC}~(\ref{eqn: Nesterov_sc_symplectic}), we have
		\begin{align*}
		\frac{1}{4} \left\| v_{k + 1} \right\|^{2} - \frac{1}{4} \left\| v_{k} \right\|^{2} & = \frac{1}{2} \left\langle v_{k + 1} - v_{k}, v_{k + 1} \right\rangle - \frac{1}{4} \left\| v_{k + 1} - v_{k} \right\|^{2} \\
		& = - \frac{\sqrt{\mu s} }{1 - \sqrt{\mu s}} \left\| v_{k + 1} \right\|^{2} - \frac{\sqrt{s}}{2} \left\langle \nabla f(x_{k + 1}) - \nabla f(x_{k}), v_{k + 1} \right\rangle \\
		& \quad - \frac{1 + \sqrt{\mu s} }{1 - \sqrt{\mu s} } \cdot \frac{\sqrt{s}}{2} \left\langle \nabla f(x_{k + 1}), v_{k + 1} \right\rangle - \frac{1}{4} \left\| v_{k + 1} - v_{k} \right\|^{2} \\
		& = \underbrace{- \frac{\sqrt{\mu s} }{1 - \sqrt{\mu s}} \left\| v_{k + 1} \right\|^{2}}_{\mathbf{II}_1} \underbrace{- \frac{\sqrt{s}}{2} \cdot \frac{1 - \sqrt{\mu s} }{1 + \sqrt{\mu s} }\left\langle \nabla f(x_{k + 1}) - \nabla f(x_{k}), v_{k} \right\rangle}_{\mathbf{II}_2} \\
		& \quad + \underbrace{ \frac{1 - \sqrt{\mu s} }{1 + \sqrt{\mu s} } \cdot \frac{s}{2} \left\| \nabla f(x_{k + 1}) - \nabla f(x_{k}) \right\|^{2}}_{\mathbf{II}_{3}} +\underbrace{ \frac{s}{2} \left\langle  \nabla f(x_{k + 1}) - \nabla f(x_{k}), \nabla f(x_{k + 1}) \right\rangle}_{\mathbf{II}_4} \\
		& \quad \underbrace{- \frac{1 + \sqrt{\mu s} }{1 - \sqrt{\mu s} } \cdot \frac{\sqrt{s}}{2} \left\langle \nabla f(x_{k + 1}), v_{k + 1} \right\rangle}_{\mathbf{II}_{5}} \underbrace{- \frac{1}{4} \left\| v_{k + 1} - v_{k} \right\|^{2}}_{\mathbf{II}_{6}}.
		\end{align*}
		\item For the part $\mathbf{III}$, mixed energy, with the phase representation of NAG-\texttt{SC}~(\ref{eqn: Nesterov_sc_symplectic}), we have
		\begin{align*}
		& \frac{1}{4} \left\| v_{k + 1} + \frac{2\sqrt{\mu}}{1 - \sqrt{\mu s}} ( x_{k + 2} - x^{\star} ) +   \sqrt{s} \nabla f(x_{k + 1})\right\|^{2} - \frac{1}{4} \left\| v_{k} + \frac{2\sqrt{\mu}}{1 - \sqrt{\mu s}} ( x_{k + 1} - x^{\star} ) +   \sqrt{s} \nabla f(x_{k})\right\|^{2} \\
		= & \frac{1}{2} \left\langle - \frac{1 + \sqrt{\mu s} }{1 - \sqrt{\mu s} } \sqrt{s} \nabla f(x_{k + 1}), \frac{1 + \sqrt{\mu s} }{1 - \sqrt{\mu s} } v_{k + 1} + \frac{2\sqrt{\mu}}{1 - \sqrt{\mu s} } (x_{k + 1} - x^{\star}) +  \sqrt{s} \nabla f(x_{k + 1}) \right\rangle  \\
		& - \frac{1}{4} \left( \frac{1 + \sqrt{\mu s} }{1 - \sqrt{\mu s} }\right)^{2} s \left\| \nabla f(x_{k + 1})\right\|^{2} \\
		= & \underbrace{- \frac{\sqrt{\mu s}}{1 - \sqrt{\mu s}}  \frac{1 + \sqrt{\mu s} }{1 - \sqrt{\mu s} }  \left\langle \nabla f(x_{k + 1}), x_{k + 1}- x^{\star} \right\rangle}_{\mathbf{III}_{1}} \underbrace{- \frac{1}{2} \left( \frac{1 + \sqrt{\mu s} }{1 - \sqrt{\mu s} } \right)^{2} \sqrt{s} \left\langle \nabla f(x_{k + 1}), v_{k + 1} \right\rangle}_{\mathbf{III}_{2}} \\
		&  \underbrace{ - \frac{1}{2} \left( \frac{1 + \sqrt{\mu s}  }{1 - \sqrt{\mu s} } \right) s \left\| \nabla f(x_{k + 1})\right\|^{2}}_{\mathbf{III}_{3}} \underbrace{ - \frac{1}{4} \left( \frac{1 + \sqrt{\mu s} }{1 - \sqrt{\mu s} }\right)^{2} s \left\| \nabla f(x_{k + 1})\right\|^{2} }_{\mathbf{III}_{4}}.
		\end{align*}       
	\end{itemize}
Both $\mathbf{II}_{2}$ and $\mathbf{III}_{3}$ above are the discrete correspondence of the terms $ -  \frac{\sqrt{s}}{2} \left\| \nabla f(X(t))\right\|^{2}$ and $ -  \frac{\sqrt{s}}{2} \dot{X}(t)^{\top}\nabla^{2}  f(X(t))\dot{X}(t)$ in~\eqref{eqn: nag-sc_ode_conver-rate}.
The impact can be found in the calculation.	Now, we calculate the difference of discrete Lyapunov function~(\ref{eqn:lypunov_NAGM-SC_strongly}) at $k$-th iteration by the simple operation 
	\begin{align*}
        &    \mathcal{E}(k + 1) - \mathcal{E}(k) \\
\leq  &    \underbrace{\left( \frac{1 + \sqrt{\mu s} }{1 - \sqrt{\mu s} } \right) \sqrt{s} \left\langle \nabla f(x_{k + 1}), v_{k}  \right\rangle}_{\mathbf{I}_{1}}  \underbrace{- \frac{1}{2L} \left( \frac{1 + \sqrt{\mu s} }{1 - \sqrt{\mu s} } \right) \left\| \nabla f(x_{k + 1}) - \nabla f(x_{k}) \right\|^{2}}_{\mathbf{I}_{2}} \\	
        &    \underbrace{- \frac{\sqrt{\mu s} }{1 - \sqrt{\mu s} } \left\| v_{k + 1} \right\|^{2}}_{\mathbf{II}_1} \underbrace{- \frac{\sqrt{s}}{2} \cdot \frac{1 - \sqrt{\mu s} }{1 + \sqrt{\mu s} }\left\langle \nabla f(x_{k + 1}) - \nabla f(x_{k}), v_{k} \right\rangle}_{\mathbf{II}_2}  + \underbrace{ \frac{1 - \sqrt{\mu s} }{1 + \sqrt{\mu s} } \cdot \frac{s}{2} \left\| \nabla f(x_{k + 1}) - \nabla f(x_{k}) \right\|^{2}}_{\mathbf{II}_{3}} \\
        &   +\underbrace{ \frac{s}{2} \left\langle  \nabla f(x_{k + 1}) - \nabla f(x_{k}), \nabla f(x_{k + 1}) \right\rangle}_{\mathbf{II}_4}  \underbrace{- \frac{1 + \sqrt{\mu s} }{1 - \sqrt{\mu s} } \cdot \frac{\sqrt{s}}{2} \left\langle \nabla f(x_{k + 1}), v_{k + 1} \right\rangle}_{\mathbf{II}_{5}} \underbrace{- \frac{1}{4} \left\| v_{k + 1} - v_{k} \right\|^{2}}_{\mathbf{II}_{6}}\\
         & \underbrace{- \frac{\sqrt{\mu s}}{1 - \sqrt{\mu s}}  \frac{1 + \sqrt{\mu s} }{1 - \sqrt{\mu s} }  \left\langle \nabla f(x_{k + 1}), x_{k + 1}- x^{\star} \right\rangle}_{\mathbf{III}_{1}} \underbrace{- \frac{1}{2} \left( \frac{1 + \sqrt{\mu s} }{1 - \sqrt{\mu s} } \right)^{2} \sqrt{s} \left\langle \nabla f(x_{k + 1}), v_{k + 1} \right\rangle}_{\mathbf{III}_{2}} \\
		&  \underbrace{ - \frac{1}{2} \left( \frac{1 + \sqrt{\mu s}  }{1 - \sqrt{\mu s} } \right) s \left\| \nabla f(x_{k + 1})\right\|^{2}}_{\mathbf{III}_{3}} \underbrace{ - \frac{1}{4} \left( \frac{1 + \sqrt{\mu s} }{1 - \sqrt{\mu s} }\right)^{2} s \left\| \nabla f(x_{k + 1})\right\|^{2} }_{\mathbf{III}_{4}}\\
		&  \underbrace{  - \frac{s}{2} \left( \frac{1}{1 - \sqrt{\mu s} } \right)\left( \left\| \nabla f(x_{k + 1})\right\|^{2} -    \left\| \nabla f(x_{k})\right\|^{2} \right) }_{\mathbf{additional \; term}}\\
\leq  &\underbrace{ - \frac{\sqrt{\mu s}}{1 - \sqrt{\mu s}} \left(   \frac{1 + \sqrt{\mu s} }{1 - \sqrt{\mu s} }  \left\langle \nabla f(x_{k + 1}), x_{k + 1} - x^{\star} \right\rangle + \left\| v_{k + 1} \right\|^{2}  \right)}_{\mathbf{II}_1 + \mathbf{III}_1} \\
	&\underbrace{ - \frac{1}{2} \left( \frac{1 +  \sqrt{\mu s} }{1 - \sqrt{\mu s} }\right) \left [\sqrt{s} \left\langle \nabla f(x_{k + 1}), \left( \frac{1 + \sqrt{\mu s} }{1 - \sqrt{\mu s} }\right) v_{k + 1} - v_{k}\right\rangle + s \left\| \nabla f(x_{k + 1}) \right\|^{2} \right]}_{\frac{1}{2}\mathbf{I}_{1} + \mathbf{III}_{2} + \mathbf{III}_{3} } \\
	&\underbrace{- \frac{\sqrt{s}}{2} \cdot \frac{1 - \sqrt{\mu s} }{1 + \sqrt{\mu s} }\left\langle \nabla f(x_{k + 1}) - \nabla f(x_{k}), v_{k} \right\rangle}_{\mathbf{II}_{2}} + \underbrace{\frac{s}{2} \left\langle  \nabla f(x_{k + 1}) - \nabla f(x_{k}), \nabla f(x_{k + 1}) \right\rangle}_{\mathbf{II}_{4}} \\
	&  \underbrace{- \frac{1}{4} \left[ \left\| v_{k + 1} - v_{k} \right\|^{2} + 2 \left( \frac{1 + \sqrt{\mu s} }{1 - \sqrt{\mu s} } \right) \sqrt{s} \left\langle \nabla f(x_{k + 1}), v_{k + 1} - v_{k}\right\rangle + \left( \frac{1 + \sqrt{\mu s} }{1 - \sqrt{\mu s} }\right)^{2} s \left\| \nabla f(x_{k + 1})\right\|^{2} \right]}_{\frac{1}{2}\mathbf{I}_{1} + \mathbf{II}_{5} + \mathbf{II}_{6} + \mathbf{III}_{4}} \\
        &\underbrace{- \frac{1}{2} \left[ \frac{1}{L} \left(\frac{1 + \sqrt{\mu s}}{1 - \sqrt{\mu s}}\right) - s \left(\frac{1 - \sqrt{\mu s}}{1 + \sqrt{\mu s}}\right) \right] \left\| \nabla f(x_{k + 1}) - \nabla f(x_{k}) \right\|^{2}}_{ \mathbf{I}_{2} + \mathbf{II}_{3} }\\
	& \underbrace{ - \frac{1}{2}\left( \frac{1}{1 - \sqrt{\mu s} }\right) s \left( \left\| \nabla f(x_{k + 1}) \right\|^{2} - \left\| \nabla f(x_{k}) \right\|^{2}\right) }_{\mathbf{additional \; term}}
        \end{align*}
Now, the term, $(1/2)\mathbf{I}_{1} + \mathbf{II}_{5} + \mathbf{II}_{6} + \mathbf{III}_{4}$, can be calculated as 
\begin{align*}
\frac{1}{2}\mathbf{I}_{1} + \mathbf{II}_{5} + \mathbf{II}_{6} + \mathbf{III}_{4} & = - \frac{1}{4} \left[ \left\| v_{k + 1} - v_{k} \right\|^{2} + 2 \left( \frac{1 + \sqrt{\mu s} }{1 - \sqrt{\mu s} } \right) \sqrt{s} \left\langle \nabla f(x_{k + 1}), v_{k + 1} - v_{k}\right\rangle \right.\\
                                                               &\qquad \qquad \left.+ \left( \frac{1 + \sqrt{\mu s} }{1 - \sqrt{\mu s} }\right)^{2} s \left\| \nabla f(x_{k + 1})\right\|^{2} \right] \\
                                                                                                                         & = - \frac{1}{4} \left\| v_{k + 1} - v_{k} + \left( \frac{1 + \sqrt{\mu s}}{1 - \sqrt{\mu s}} \right) \sqrt{s} \nabla f(x_{k}) \right\|^{2}\\
                                                                                                                         & \leq 0.
\end{align*}
With phase representation of NAG-\texttt{SC}~(\ref{eqn: Nesterov_sc_symplectic}), we have
\begin{align*}
\frac{1}{2}\mathbf{I}_{1} + \mathbf{III}_{2} + \mathbf{III}_{3} & = - \frac{1}{2} \left( \frac{1 +  \sqrt{\mu s} }{1 - \sqrt{\mu s} }\right) \left [\sqrt{s} \left\langle \nabla f(x_{k + 1}), \left( \frac{1 + \sqrt{\mu s} }{1 - \sqrt{\mu s} }\right) v_{k + 1} - v_{k}\right\rangle + s \left\| \nabla f(x_{k + 1}) \right\|^{2} \right] \\
                                                                                              & = \frac{1}{2}\left( \frac{1 +  \sqrt{\mu s} }{1 - \sqrt{\mu s} }\right) s \left( \left\langle  \nabla f(x_{k + 1}) - \nabla f(x_{k}), \nabla f(x_{k + 1}) \right\rangle + \frac{2\sqrt{\mu s}}{1 - \sqrt{\mu s} } \left\| \nabla f(x_{k + 1}) \right\|^{2}\right)\\
                                                                                               & =\underbrace{\frac{1}{2}\left( \frac{1 +  \sqrt{\mu s} }{1 - \sqrt{\mu s} }\right) \cdot s \cdot \left\langle  \nabla f(x_{k + 1}) - \nabla f(x_{k}), \nabla f(x_{k + 1}) \right\rangle}_{\mathbf{IV}_{1}} \\
                                                                                               & \quad + \underbrace{\left( \frac{1 +  \sqrt{\mu s} }{1 - \sqrt{\mu s} }\right)  \cdot \frac{\sqrt{\mu s}}{1 - \sqrt{\mu s} } \cdot s\left\| \nabla f(x_{k + 1}) \right\|^{2}}_{\mathbf{IV}_{2}}
\end{align*}
For convenience, we note the term $\mathbf{IV} = (1/2)\mathbf{I}_{1} + \mathbf{III}_{2} + \mathbf{III}_{3}$.
 Then, with phase representation of NAG-\texttt{SC}~(\ref{eqn: Nesterov_sc_symplectic}), the difference of Lyapunov function~(\ref{eqn:lypunov_NAGM-SC_strongly}) is
        \begin{align*}
                    \mathcal{E}(k + 1) - \mathcal{E}(k) \leq  & \underbrace{- \frac{\sqrt{\mu s}}{1 - \sqrt{\mu s}} \left(  \frac{1 + \sqrt{\mu s} }{1 - \sqrt{\mu s} } \left(\left\langle \nabla f(x_{k + 1}), x_{k + 1} - x^{\star} \right\rangle  -  s \left\| \nabla f(x_{k + 1}) \right\|^{2}\right) + \left\| v_{k + 1} \right\|^{2}  \right)}_{\mathbf{II}_{1} + \mathbf{III}_{1} + \mathbf{IV}_{2}} \\   
	&   \underbrace{ - \frac{1}{2} \cdot \frac{1 - \sqrt{\mu s} }{1 + \sqrt{\mu s} }\left\langle \nabla f(x_{k + 1}) - \nabla f(x_{k}), x_{k + 1} - x_{k} \right\rangle}_{\mathbf{II}_{2}} \\
	&   + \underbrace{\left( \frac{1}{1 - \sqrt{\mu s} }\right) s \left\langle  \nabla f(x_{k + 1}) - \nabla f(x_{k}), \nabla f(x_{k + 1}) \right\rangle}_{\mathbf{II}_{4} + \mathbf{IV}_{1}} \\
	&\underbrace{- \frac{1}{2} \left[ \frac{1}{L} \left(\frac{1 + \sqrt{\mu s}}{1 - \sqrt{\mu s}}\right) - s \left(\frac{1 - \sqrt{\mu s}}{1 + \sqrt{\mu s}}\right) \right] \left\| \nabla f(x_{k + 1}) - \nabla f(x_{k}) \right\|^{2}}_{ \mathbf{I}_{2} + \mathbf{II}_{3} }\\
	& \underbrace{ - \frac{1}{2}\left( \frac{1}{1 - \sqrt{\mu s} }\right) s \left( \left\| \nabla f(x_{k + 1}) \right\|^{2} - \left\| \nabla f(x_{k}) \right\|^{2}\right) }_{\mathbf{additional \; term}}
	\end{align*}
Now, we can find the impact of additional term in the Lyapunov function~(\ref{eqn:lypunov_NAGM-SC_strongly}). In other words, the $\mathbf{II}_{4} + \mathbf{IV}_{1}$ term added the additional term is a perfect square, as below
\begin{align*}
\mathbf{II}_{4} + \mathbf{IV}_{1} + \mathbf{additional\;term} & = \left( \frac{1}{1 - \sqrt{\mu s} }\right) s \left\langle  \nabla f(x_{k + 1}) - \nabla f(x_{k}), \nabla f(x_{k + 1}) \right\rangle \\
                                                                                                & \quad - \frac{1}{2}\left( \frac{1}{1 - \sqrt{\mu s} }\right) s \left( \left\| \nabla f(x_{k + 1}) \right\|^{2} - \left\| \nabla f(x_{k}) \right\|^{2}\right) \\
                                                                                                & =  \frac{1}{2}\left( \frac{1}{1 - \sqrt{\mu s} }\right) s \left\| \nabla f(x_{k + 1}) - \nabla f(x_{k})\right\|^{2}
\end{align*}
Merging all the similar items, $\mathbf{II}_{4} + \mathbf{IV}_{1} + \mathbf{additional\;term}$, $\mathbf{I}_{2} + \mathbf{II}_{3}$, we have
\begin{align*}
        & \quad (\mathbf{II}_{4} + \mathbf{IV}_{1} + \mathbf{additional\;term}) + ( \mathbf{I}_{2}  +  \mathbf{II}_{3}) \\
 =     &\quad  \frac{1}{2} \left( \frac{1}{1 - \sqrt{\mu s}} +  \frac{1 - \sqrt{\mu s}}{1 + \sqrt{\mu s}}  -  \frac{1 +\sqrt{\mu s}}{1 - \sqrt{\mu s}} \cdot \frac{1}{Ls}\right) s \left\| \nabla f(x_{k + 1}) - \nabla f(x_{k})\right\|^{2} \\
 \leq &\quad  \frac{1}{2} \left( \frac{1 + \sqrt{\mu s}}{1 - \sqrt{\mu s}} +  \frac{1 - \sqrt{\mu s}}{1 + \sqrt{\mu s}}  -  \frac{1 +\sqrt{\mu s}}{1 - \sqrt{\mu s}} \cdot \frac{1}{Ls}\right) s \left\| \nabla f(x_{k + 1}) - \nabla f(x_{k})\right\|^{2}
\end{align*}
Now, we obtain that the difference of Lyapunov function~(\ref{eqn:lypunov_NAGM-SC_strongly}) is
\begin{align*}
\mathcal{E}(k + 1) - \mathcal{E}(k)  & \leq  - \frac{\sqrt{\mu s}}{1 - \sqrt{\mu s}} \left(    \frac{1 + \sqrt{\mu s} }{1 - \sqrt{\mu s} } \left(\left\langle \nabla f(x_{k + 1}), x_{k + 1} - x^{\star} \right\rangle   -  s \left\| \nabla f(x_{k + 1}) \right\|^{2}\right)  + \left\| v_{k + 1} \right\|^{2}  \right) \\
                                                       & \quad - \frac{1}{2} \cdot \frac{1 - \sqrt{\mu s} }{1 + \sqrt{\mu s} }\left\langle \nabla f(x_{k + 1}) - \nabla f(x_{k}), x_{k + 1} - x_{k} \right\rangle \\
                                                       & \quad +  \frac{1}{2}\left( \frac{1 + \sqrt{\mu s}}{1 - \sqrt{\mu s}} +  \frac{1 - \sqrt{\mu s}}{1 + \sqrt{\mu s}}  -  \frac{1 +\sqrt{\mu s}}{1 - \sqrt{\mu s}} \cdot \frac{1}{Ls} \right)  s \left\| \nabla f(x_{k + 1}) - \nabla f(x_{k})\right\|^{2}
\end{align*}
With the inequality for any function $f(x) \in \mathcal{S}_{\mu, L}^{1}(\mathbb{R}^{n})$
	\[
	\left\| \nabla f(x_{k + 1}) - \nabla f(x_{k})\right\|^{2} \leq L \left\langle \nabla f(x_{k + 1}) - \nabla f(x_{k}), x_{k + 1} - x_{k}\right\rangle, 
	\]
we have		
	\begin{align*}   
\mathcal{E}(k + 1) - \mathcal{E}(k)	& \leq  - \frac{\sqrt{\mu s}}{1 - \sqrt{\mu s}} \left[     \frac{1 + \sqrt{\mu s} }{1 - \sqrt{\mu s} } \left(\left\langle \nabla f(x_{k + 1}), x_{k + 1} - x^{\star} \right\rangle   -  s \left\| \nabla f(x_{k + 1}) \right\|^{2}\right)  + \left\| v_{k + 1} \right\|^{2}  \right] \\
                                                          & \quad - \frac{1}{2} \cdot \frac{1 - \sqrt{\mu s} }{1 + \sqrt{\mu s} } \cdot \frac{1}{L} \cdot \left\| \nabla f(x_{k + 1}) - \nabla f(x_{k})\right\|^{2} \\
                                                          & \quad +  \frac{1}{2}\left( \frac{1 + \sqrt{\mu s}}{1 - \sqrt{\mu s}} +  \frac{1 - \sqrt{\mu s}}{1 + \sqrt{\mu s}}  -  \frac{1 +\sqrt{\mu s}}{1 - \sqrt{\mu s}} \cdot \frac{1}{Ls} \right)  s \left\| \nabla f(x_{k + 1}) - \nabla f(x_{k})\right\|^{2} \\
                                                         & \leq - \frac{\sqrt{\mu s}}{1 - \sqrt{\mu s}} \left(   \frac{1 + \sqrt{\mu s} }{1 - \sqrt{\mu s} } \left(\left\langle \nabla f(x_{k + 1}), x_{k + 1} - x^{\star} \right\rangle   -  s \left\| \nabla f(x_{k + 1}) \right\|^{2}\right)  + \left\| v_{k + 1} \right\|^{2}  \right) \\
                                                         & \quad - \frac{1}{2}\left( \frac{1 + \sqrt{\mu s}}{1 - \sqrt{\mu s}} + \frac{1 - \sqrt{\mu s}}{1 + \sqrt{\mu s}} \right) \left( \frac{1}{L} - s \right) \left\| \nabla f(x_{k + 1}) - \nabla f(x_{k})\right\|^{2}. 
         \end{align*}  

\subsection{Proof of Lemma~\ref{lm:heavy_dis_add}}
\label{subsec: heavy_ball_discrete}

With the phase representation of the heavy-ball method~(\ref{eqn: polyak_heavy_ball_symplectic}) and Cauchy-Schwarz inequality, we have
\begin{align*}
\left\| v_{k} + \frac{2\sqrt{\mu}}{1 - \sqrt{\mu s}}(x_{k + 1} - x^{\star})  \right\|_{2}^{2} & = \left\|  \frac{1 + \sqrt{\mu s} }{1 - \sqrt{\mu s} }v_{k} +\frac{2\sqrt{\mu}}{1 - \sqrt{\mu s}}(x_{k} - x^{\star}) \right\|_{2}^{2} \\
                                                                                                                                  & \leq 2 \left[ \left( \frac{1 + \sqrt{\mu s} }{1 - \sqrt{\mu s} } \right)^{2} \left\| v_{k} \right\|_{2}^{2} + \frac{4 \mu}{(1 - \sqrt{\mu s})^{2}} \left\| x_{k} - x^{\star} \right\|_{2}^{2}  \right].
\end{align*}
The discrete Lyapunov function~\eqref{eqn: lypunov_PHBM_strongly} can be estimated as
\begin{align}
\label{eqn: heavy-b_energy_estimate}
\mathcal{E}(k) \leq  \frac{1 + \sqrt{\mu s} }{1 - \sqrt{\mu s} } \left( f(x_{k}) - f(x^{\star}) \right) + \frac{1 + \mu s}{ (1 - \sqrt{\mu s} )^{2} } \left\| v_{k} \right\|_{2}^{2} + \frac{2\mu}{ (1 - \sqrt{\mu s} )^{2} } \left\| x_{k} - x^{\star} \right\|_{2}^{2}.
\end{align}
For convenience, we also split the discrete Lyapunov function~\eqref{eqn: lypunov_PHBM_strongly} into three parts and mark them as below
\[
\mathcal{E}(k) = \underbrace{\frac{1 + \sqrt{\mu s} }{1 - \sqrt{\mu s} } \left( f(x_{k}) - f(x^{\star}) \right)}_{\mathbf{I}} + \underbrace{\frac{1}{4} \left\| v_{k} \right\|^{2}}_{\mathbf{II}} + \underbrace{\frac{1}{4} \left\| v_{k} + \frac{2\sqrt{\mu}}{1 - \sqrt{\mu s}}(x_{k + 1} - x^{\star}) \right\|^{2} }_{\mathbf{III}},
\]
where the three parts $\mathbf{I}$, $\mathbf{II}$ and $\mathbf{III}$ are corresponding to potential, kinetic energy and mixed energy in classical mechanics, respectively. 
	\begin{itemize}
		\item For the part $\mathbf{I}$, potential, with the basic convex of $f(x) \in \mathcal{S}_{\mu, L}^{1}(\mathbb{R}^{n})$ 
		         \[
		         f(x_{k}) \geq f(x_{k + 1}) + \left\langle \nabla f(x_{k + 1}), x_{k} - x_{k + 1} \right\rangle + \frac{1}{2L} \left\| \nabla f(x_{k + 1}) - \nabla f(x_{k}) \right\|_{2}^{2},
		         \]
		         we have
		         \begin{align*}
		         		&  \left( \frac{1 + \sqrt{\mu s} }{1 - \sqrt{\mu s} } \right)  \left( f(x_{k + 1}) - f(x^{\star}) \right) - \left( \frac{1 + \sqrt{\mu s} }{1 - \sqrt{\mu s} } \right) \left( f(x_{k}) - f(x^{\star}) \right) \\
		\leq  &     \underbrace{\left( \frac{1 + \sqrt{\mu s} }{1 - \sqrt{\mu s} } \right) \sqrt{s} \left\langle \nabla f(x_{k + 1}), v_{k}  \right\rangle}_{\mathbf{I}_{1}}  \underbrace{- \frac{1}{2L} \left( \frac{1 + \sqrt{\mu s} }{1 - \sqrt{\mu s}} \right) \left\| \nabla f(x_{k + 1}) - \nabla f(x_{k}) \right\|^{2}}_{\mathbf{I}_{2}}. 
		         \end{align*}
                 \item For the part $\mathbf{II}$, kinetic energy, with the phase representation of the heavy-ball method~(\ref{eqn: polyak_heavy_ball_symplectic}), we have
                          \begin{align*}
                          \frac{1}{4} \left\| v_{k+1}\right\|^{2} - \frac{1}{4} \left\| v_{k}\right\|^{2} & = \frac{1}{2} \left\langle v_{k + 1} - v_{k}, v_{k + 1}\right\rangle - \frac{1}{4} \left\| v_{k + 1} - v_{k}\right\|^{2} \\
                                                                                                                                      & =\underbrace{ - \frac{\sqrt{\mu s}}{1 - \sqrt{\mu s}} \left\| v_{k + 1} \right\|^{2}}_{\mathbf{II}_{1}} \underbrace{- \frac{1}{2} \cdot \frac{1 + \sqrt{\mu s}}{1 - \sqrt{\mu s}} \cdot \sqrt{s} \left\langle \nabla f(x_{k + 1}), v_{k + 1} \right\rangle}_{\mathbf{II}_{2}} \\
                                                                                                                                      &\quad \underbrace{- \frac{1}{4} \left\| v_{k + 1} - v_{k}\right\|^{2}}_{\mathbf{II}_{3}}
                          \end{align*}
                 \item For the part $\mathbf{III}$, mixed energy, with the phase representation of the heavy-ball method~(\ref{eqn: polyak_heavy_ball_symplectic}), we have
                 \begin{align*}
                 & \frac{1}{4} \left\|  v_{k + 1} + \frac{2\sqrt{\mu}}{1 - \sqrt{\mu s}}(x_{k + 2} - x^{\star})  \right\|^{2} - \frac{1}{4} \left\| v_{k} +\frac{2\sqrt{\mu}}{1 - \sqrt{\mu s}}(x_{k + 1} - x^{\star})  \right\|^{2} \\
             =  & \frac{1}{4} \left\langle  v_{k + 1} - v_{k} + \frac{2\sqrt{\mu}}{1 - \sqrt{\mu s}} (x_{k+2} - x_{k +1}) ,  v_{k+1} + v_{k} + \frac{2\sqrt{\mu}}{1 - \sqrt{\mu s}} (x_{k + 2} + x_{k + 1}- 2x^{\star})  \right\rangle \\
             =  & - \frac{1}{2} \cdot \frac{1 + \sqrt{\mu s}}{1 - \sqrt{\mu s}} \cdot \sqrt{s} \left\langle \nabla f(x_{k+1}) ,  v_{k+1} + \frac{2\sqrt{\mu}}{1 - \sqrt{\mu s}} (x_{k + 2} - x^{\star}) \right\rangle - \frac{s}{4} \left( \frac{1 + \sqrt{\mu s}}{1 - \sqrt{\mu s}}\right)^{2} \left\| \nabla f(x_{k+1})\right\|^{2} \\
             = & \underbrace{- \frac{1 + \sqrt{\mu s}}{1 - \sqrt{\mu s}} \cdot \frac{\sqrt{\mu s}}{1 - \sqrt{\mu s}} \left\langle \nabla f(x_{k + 1}),  x_{k+1} - x^{\star}\right\rangle}_{\mathbf{III}_{1}} \underbrace{- \frac{1}{2}\left( \frac{1 + \sqrt{\mu s}}{1 - \sqrt{\mu s}}\right)^{2} \sqrt{s} \left\langle \nabla f(x_{k+1}), v_{k + 1} \right\rangle}_{\mathbf{III}_{2}}\\
                & \underbrace{- \frac{s}{4} \left( \frac{1 + \sqrt{\mu s}}{1 - \sqrt{\mu s}}\right)^{2} \left\| \nabla f(x_{k+1})\right\|^{2}}_{\mathbf{III}_{3}}
                 \end{align*}
         \end{itemize}    
Now, we calculate the difference of discrete Lyapunov function~(\ref{eqn:lypunov_NAGM-SC_strongly}) at the $k$-th iteration by the simple operation as         
\begin{align*}         
 \mathcal{E}(k + 1) - \mathcal{E}(k) & \leq    \underbrace{\left( \frac{1 + \sqrt{\mu s} }{1 - \sqrt{\mu s} } \right) \sqrt{s} \left\langle \nabla f(x_{k + 1}), v_{k}  \right\rangle}_{\mathbf{I}_{1}}  \underbrace{- \frac{1}{2L} \left( \frac{1 + \sqrt{\mu s} }{1 - \sqrt{\mu s}} \right) \left\| \nabla f(x_{k + 1}) - \nabla f(x_{k}) \right\|^{2}}_{\mathbf{I}_{2}} \\
                                                         & \quad \underbrace{ - \frac{\sqrt{\mu s}}{1 - \sqrt{\mu s}} \left\| v_{k + 1} \right\|^{2}}_{\mathbf{II}_{1}} \underbrace{-\frac{1}{2} \cdot \frac{1 + \sqrt{\mu s}}{1 - \sqrt{\mu s}} \cdot \sqrt{s} \left\langle \nabla f(x_{k + 1}), v_{k + 1} \right\rangle}_{\mathbf{II}_{2}} \underbrace{- \frac{1}{4} \left\| v_{k + 1} - v_{k}\right\|^{2}}_{\mathbf{II}_{3}}  \\
                                                         & \quad   \underbrace{- \frac{1 + \sqrt{\mu s}}{1 - \sqrt{\mu s}} \cdot \frac{\sqrt{\mu s}}{1 - \sqrt{\mu s}} \left\langle \nabla f(x_{k + 1}),  x_{k+1} - x^{\star}\right\rangle}_{\mathbf{III}_{1}} \underbrace{- \frac{1}{2}\left( \frac{1 + \sqrt{\mu s}}{1 - \sqrt{\mu s}}\right)^{2} \sqrt{s} \left\langle \nabla f(x_{k+1}), v_{k + 1} \right\rangle}_{\mathbf{III}_{2}}\\
                                                        & \quad \underbrace{- \frac{s}{4} \left( \frac{1 + \sqrt{\mu s}}{1 - \sqrt{\mu s}}\right)^{2} \left\| \nabla f(x_{k+1})\right\|^{2}}_{\mathbf{III}_{3}} \\
                                                        & = \underbrace{- \frac{\sqrt{\mu s}}{1 - \sqrt{\mu s}} \left(  \frac{1 + \sqrt{\mu s}}{1 - \sqrt{\mu s}} \left\langle \nabla f(x_{k + 1}), x_{k+1} - x^{\star} \right\rangle + \left\| v_{k + 1} \right\|^{2}  \right)}_{\mathbf{II}_{1} + \mathbf{III}_{1}}    \\
                                                        & \quad \underbrace{- \frac{1}{2L} \left( \frac{1 + \sqrt{\mu s} }{1 - \sqrt{\mu s}} \right) \left\| \nabla f(x_{k + 1}) - \nabla f(x_{k}) \right\|^{2}}_{\mathbf{I}_{2}} \\
                                                        & \quad  \underbrace{ - \frac{1}{2}\left( \frac{1 + \sqrt{\mu s} }{1 - \sqrt{\mu s} } \right) \sqrt{s} \left\langle \nabla f(x_{k + 1}), \left( \frac{1 + \sqrt{\mu s} }{1 - \sqrt{\mu s} } \right)v_{k + 1} - v_{k}  \right\rangle}_{\frac{1}{2}\mathbf{I}_{1} + \mathbf{III}_{2}} \\
                                                        & \quad \underbrace{- \frac{1}{4}\left( \left\| v_{k + 1} - v_{k} \right\|^{2} + 2\sqrt{s} \cdot \frac{1 + \sqrt{\mu s} }{1 - \sqrt{\mu s} }  \left\langle \nabla f(x_{k+1}), v_{k+1} - v_{k}\right\rangle +s \left( \frac{1 + \sqrt{\mu s}}{1 - \sqrt{\mu s}}\right)^{2} \left\| \nabla f(x_{k+1})\right\|^{2}\right)}_{\frac{1}{2}\mathbf{I}_{1} + \mathbf{II}_{2} + \mathbf{II}_{3} + \mathbf{III}_{3} }
\end{align*} 
 
With the phase representation of the heavy-ball method~(\ref{eqn: polyak_heavy_ball_symplectic}), we have
 \begin{align*}
 \frac{1}{2} \mathbf{I}_{1} + \mathbf{III}_{2} & = - \frac{1}{2}\left( \frac{1 + \sqrt{\mu s} }{1 - \sqrt{\mu s} } \right) \sqrt{s} \left\langle \nabla f(x_{k + 1}), \left( \frac{1 + \sqrt{\mu s} }{1 - \sqrt{\mu s} } \right)v_{k + 1} - v_{k}  \right\rangle \\
                                                                     & =  \frac{s}{2} \left( \frac{1 + \sqrt{\mu s}}{1 - \sqrt{\mu s}}\right)^{2} \left\| \nabla f(x_{k + 1}) \right\|^{2};
 \end{align*}
 and
\begin{align*}
     &\frac{1}{2}\mathbf{I}_{1} + \mathbf{II}_{2} + \mathbf{II}_{3} + \mathbf{III}_{3} \\
 =  & - \frac{1}{4}\left[ \left\| v_{k + 1} - v_{k} \right\|^{2} + 2\sqrt{s} \cdot \frac{1 + \sqrt{\mu s} }{1 - \sqrt{\mu s} }  \left\langle \nabla f(x_{k+1}), v_{k+1} - v_{k}\right\rangle +s \left( \frac{1 + \sqrt{\mu s}}{1 - \sqrt{\mu s}}\right)^{2} \left\| \nabla f(x_{k+1})\right\|^{2}\right] \\
 =  & - \frac{1}{4} \left\| v_{k + 1} - v_{k} + \frac{1 + \sqrt{\mu s}}{1 - \sqrt{\mu s}} \cdot \sqrt{s} \nabla f(x_{k + 1})\right\|^{2} \\
 \leq &0.
\end{align*}
Now, the difference of discrete Lyapunov function~(\ref{eqn: lypunov_PHBM_strongly}) can be rewritten as
\begin{align*}
\mathcal{E}(k + 1) - \mathcal{E}(k) & \leq     - \frac{\sqrt{\mu s}}{1 - \sqrt{\mu s}} \left(  \frac{1 + \sqrt{\mu s}}{1 - \sqrt{\mu s}} \left\langle \nabla f(x_{k + 1}), x_{k+1} - x^{\star} \right\rangle + \left\| v_{k + 1} \right\|^{2} \right) \\
                                                       & \quad  - \frac{1}{2L} \left( \frac{1 + \sqrt{\mu s} }{1 - \sqrt{\mu s}} \right) \left\| \nabla f(x_{k + 1}) - \nabla f(x_{k}) \right\|^{2}\\
                                                       & \quad   + \frac{s}{2} \left( \frac{1 + \sqrt{\mu s}}{1 - \sqrt{\mu s}}\right)^{2} \left\| \nabla f(x_{k + 1}) \right\|^{2}.
\end{align*}
With the inequality for any function $f(x) \in \mathcal{S}_{\mu, L}^{1}(\mathbb{R}^{n})$ 
\[
 f(x^{\star}) \geq f(x_{k + 1}) +  \left\langle \nabla f(x_{k + 1}), x^{\star} - x_{k + 1} \right\rangle + \frac{\mu}{2} \left\| x_{k + 1} - x^{\star}\right\|^{2}, 
 \]
 we have
\begin{align*}
\mathcal{E}(k + 1) - \mathcal{E}(k) & \leq - \sqrt{\mu s} \left[ \frac{1 + \sqrt{\mu s}}{(1 - \sqrt{\mu s})^{2}} \left( f(x_{k + 1}) - f(x^{\star}) \right) + \frac{\mu}{2} \cdot  \frac{1 + \sqrt{\mu s}}{(1 - \sqrt{\mu s})^{2}} \left\| x_{k + 1} - x^{\star}\right\|^{2} + \frac{1}{1 - \sqrt{\mu s}}\left\| v_{k+1} \right\|^{2} \right] \\
                                                       & \quad + \frac{s}{2} \left( \frac{1 + \sqrt{\mu s}}{1 - \sqrt{\mu s}}\right)^{2} \left\| \nabla f(x_{k + 1}) \right\|^{2}\\
                                                       & \leq - \sqrt{\mu s} \left[  \frac{1 + \sqrt{\mu s}}{1 - \sqrt{\mu s}} \left( f(x_{k + 1}) - f(x^{\star}) \right) + \frac{\mu}{2} \cdot  \frac{1 + \sqrt{\mu s}}{1 - \sqrt{\mu s}} \left\| x_{k + 1} - x^{\star}\right\|^{2} + \frac{1}{1 - \sqrt{\mu s}}\left\| v_{k+1} \right\|^{2} \right] \\
                                                       & \quad + \frac{s}{2} \left( \frac{1 + \sqrt{\mu s}}{1 - \sqrt{\mu s}}\right)^{2} \left\| \nabla f(x_{k + 1}) \right\|^{2}\\
                                                      & \leq - \sqrt{\mu s} \left[  \frac{1}{4}\cdot\frac{1 + \sqrt{\mu s}}{1 - \sqrt{\mu s}} \left( f(x_{k + 1}) - f(x^{\star}) \right) + \frac{1}{1 - \sqrt{\mu s}}\left\| v_{k+1} \right\|^{2} + \frac{\mu}{2} \cdot  \frac{1 + \sqrt{\mu s}}{1 - \sqrt{\mu s}} \left\| x_{k + 1} - x^{\star}\right\|^{2}   \right] \\
                                                       & \quad - \left[ \frac{3}{4} \sqrt{\mu s} \left( \frac{1 + \sqrt{\mu s}}{1 - \sqrt{\mu s}}\right) \left( f(x_{k + 1}) - f(x^{\star})\right) - \frac{s}{2} \left( \frac{1 + \sqrt{\mu s}}{1 - \sqrt{\mu s}}\right)^{2} \left\| \nabla f(x_{k + 1}) \right\|^{2}\right].
\end{align*}
Comparing the coefficient of the estimate of Lyapunov function~(\ref{eqn: heavy-b_energy_estimate}), we have
\begin{align*}
\mathcal{E}(k + 1) - \mathcal{E}(k) & \leq - \sqrt{\mu s} \min\left\{ \frac{1 - \sqrt{\mu s}}{1 + \sqrt{\mu s}}, \frac{1}{4}\right\} \mathcal{E}(k + 1) \\
                                                       & \quad - \left[ \frac{3}{4} \sqrt{\mu s} \left( \frac{1 + \sqrt{\mu s}}{1 - \sqrt{\mu s}}\right) \left( f(x_{k + 1}) - f(x^{\star})\right) - \frac{s}{2} \left( \frac{1 + \sqrt{\mu s}}{1 - \sqrt{\mu s}}\right)^{2} \left\| \nabla f(x_{k + 1}) \right\|^{2}\right].
\end{align*}
The proof is complete.

\section{Technical Details in Section~\ref{sec:nagm-c_analysis}} 
\label{sec: general-convex}
\subsection{Technical Details in Proof of Theorem~\ref{thm:NAGM-C_original}}
\label{subsec: technical_detail}

\subsubsection{Iterates $(x_{k}, y_{k})$ at $k = 1, 2, 3$}

The iterate $(x_{k}, y_{k})$ at $k = 1$ is
\begin{equation}
\label{eqn: iterate1}
 x_{1} = y_{1} = x_{0} - s \nabla f(x_{0}).
 \end{equation}
When $k = 2$,  the iterate $(x_{k}, y_{k})$  is
\begin{equation}
\label{eqn: iterate2} 
\left\{
 \begin{aligned}
 & y_{2} = x_{0} - s \nabla f(x_{0}) - s \nabla f(x_{0} - s \nabla f(x_{0})) \\
 & x_{2} = x_{0} - s \nabla f(x_{0}) - \frac{5}{4} s \nabla f(x_{0} - s \nabla f(x_{0})). 
 \end{aligned}\right.
 \end{equation}
When $k = 3$,  the iterate $(x_{k}, y_{k})$  is
\begin{equation}
\label{eqn: iterate3} 
\left\{ \begin{aligned}
 & y_{3} = x_{0} - s \nabla f(x_{0}) - \frac{5}{4}s \nabla f(x_{0} - s \nabla f(x_{0})) - s \nabla f\left(   x_{0} - s \nabla f(x_{0}) - \frac{5}{4}s \nabla f(x_{0} - s \nabla f(x_{0})) \right)\\
 & x_{3} = x_{0} - s \nabla f(x_{0}) - \frac{27}{20}s \nabla f(x_{0} - s \nabla f(x_{0})) - \frac{7}{5}s \nabla f\left(   x_{0} - s \nabla f(x_{0}) - \frac{5}{4}s \nabla f(x_{0} - s \nabla f(x_{0})) \right).
 \end{aligned}\right.
 \end{equation}
 
 \subsubsection{Estimate For $\left\|\nabla f(x_{k})\right\|^2$ at $k = 0, 1, 2, 3$}
According to~\eqref{eqn: iterate1}, we have 
 \begin{eqnarray}
 \label{eqn: gradient1}
 \left\| \nabla f(x_{1}) \right\|^{2}  = \left\| \nabla f(x_{0} - s \nabla f(x_{0})) \right\|^{2} & \leq& L^{2} \left\|x_{0} - x^{\star} - s\nabla f(x_{0})\right\|^{2} \nonumber \\
                                                                               & \leq & 2L^{2} \left( \left\| x_{0} - x^{\star}\right\|^{2} + s^{2} \left\| \nabla f(x_{0})\right\|^{2}\right) \nonumber \\
                                                                               & \leq & 2L^{2} (1 +L^{2}s^{2}) \left\| x_{0} - x^{\star} \right\|^{2}.
 \end{eqnarray}
According to~\eqref{eqn: iterate2}, we have
\begin{eqnarray}
 \label{eqn: gradient2}
  \left\| \nabla f(x_{2}) \right\|^{2}  & =  & \left\| \nabla f \left(x_{0} - s \nabla f(x_{0}) - \frac{5}{4} s \nabla f \left(x_{0} - s\nabla f(x_{0}) \right) \right) \right\|^{2}   \nonumber \\
                                                    & \leq & L^{2} \left\|x_{0} - x^{\star} - s\nabla f(x_{0}) - \frac{5}{4} s \nabla f \left(x_{0} - s\nabla f(x_{0}) \right) \right\|^{2}  \nonumber \\
                                                    & \leq & 3L^{2} \left( \left\| x_{0} - x^{\star}\right\|^{2} + s^{2} \left\| \nabla f(x_{0})\right\|^{2} + \frac{25}{16}s^{2} \left\| \nabla f(x_{0} - s\nabla f(x_{0}))\right\|^2\right)  \nonumber \\
                                                    & \leq & 3L^{2} \left[ (1 + L^{2}s^{2}) \left\| x_{0} - x^{\star}\right\|^{2} + \frac{25}{16}L^{2}s^{2} \left\| x_{0} - x^{\star} - s\nabla f(x_{0})\right\|^{2} \right]  \nonumber \\
                                                   & \leq & 3L^{2} \left[ (1 + L^{2}s^{2}) \left\| x_{0} - x^{\star}\right\|^{2} + \frac{25}{8}L^{2}s^{2}\left( \left\| x_{0} - x^{\star}\right\|^{2} + s^{2}\left\|\nabla f(x_{0})\right\|^{2} \right)\right]  \nonumber \\
                                                   & \leq & 3L^{2} \left( 1 + \frac{33}{8}L^{2}s^{2} + \frac{25}{8}L^{4}s^{4} \right) \left\| x_{0} - x^{\star}\right\|^{2}. 
 \end{eqnarray}
With~\eqref{eqn: iterate1}-\eqref{eqn: iterate3}, we have
\begin{eqnarray}
 \label{eqn: gradient3}
 \left\| \nabla f(x_{3})\right\|^{2} & \leq &L^{2} \left\| x_{3} - x^{\star} \right\|^{2}  \nonumber \\
                                                 & \leq &L^{2}  \left\| x_{0} - x^{\star} - s \nabla f(x_{0}) - \frac{27}{20}s \nabla f(x_{1})  - \frac{7}{5}s \nabla f\left(   x_{2} \right)  \right\|^{2}  \nonumber \\
                                                 &  =  & 4L^{2} \left( \left\| x_{0} - x^{\star} \right\|^{2} + s^{2} \left\| \nabla f(x_{0})\right\|^{2} + \frac{729}{400}s^{2} \left\| \nabla f(x_{1})\right\|^{2} + \frac{49}{25}s^{2} \left\| \nabla f(x_{2}) \right\|^{2}\right)  \nonumber \\
                                                 &  =  & 4L^{2} \left[ 1 + L^{2}s^{2} + \frac{729}{200} L^{2}s^{2} (1 + L^{2}s^{2}) + \frac{147}{25}L^{2}s^{2} \left( 1 + \frac{33}{8}L^{2}s^{2} + \frac{25}{8}L^{4}s^{4} \right)\right] \left\| x_{0} - x^{\star}\right\|^{2}  \nonumber \\
                                                 &  =  &\frac{L^{2}(40 + 381L^{2}s^{2} + 1156L^{4}s^{4} + 735L^{6}s^{6})}{10} \left\| x_{0} - x^{\star}\right\|^{2}.
 \end{eqnarray}
 Taking $s \leq 1/(3L)$ and using~\eqref{eqn: gradient1},~\eqref{eqn: gradient2} and~\eqref{eqn: gradient3}, we have
 \begin{align*}
  & \left\| \nabla f(x_{0}) \right\|^{2} \leq \frac{\left\| x_{0} - x^{\star} \right\|^{2}}{9s^{2}} , && \left\| \nabla f(x_{1}) \right\|^{2} \leq \frac{20 \left\| x_{0} - x^{\star} \right\|^{2} }{81s^{2}},\\
  &  \left\| \nabla f(x_{2}) \right\|^{2} \leq \frac{485\left\| x_{0} - x^{\star} \right\|^{2}}{972s^{2}}, && \left\| \nabla f(x_{3}) \right\|^{2} \leq \frac{2372 \left\| x_{0} - x^{\star} \right\|^{2} }{2187s^{2}}.
 \end{align*}
 
 \subsubsection{Estimate For $f(x_{k}) - f(x^\star)$ at $k = 0, 1$}
 According to~\eqref{eqn: iterate1}, we have
 \begin{eqnarray}
 \label{eqn: function_value_estimate1}
 f(x_{1}) - f(x^{\star}) & \leq &\frac{L}{2} \left\| x_{1} - x^{\star} \right\|^{2} \nonumber \\
                                 & \leq &\frac{L}{2} \left\| x_{0} - s\nabla f(x_{0}) - x^{\star} \right\|^{2} \nonumber \\
                                 & \leq &L \left(  \left\| x_{0} - x^{\star} \right\|^{2} + s^{2} \left\| \nabla f(x_{0})\right\|^{2} \right) \nonumber \\
                                 & \leq &L (1 + L^2s^2)\left\| x_{0} - x^{\star} \right\|^{2}.
  \end{eqnarray}
   Taking $s \leq 1/(3L)$,~\eqref{eqn: function_value_estimate1} tells us that
 \begin{align*}
  & f(x_{0}) - f(x^{\star}) \leq \frac{\left\| x_{0} - x^{\star} \right\|^{2}}{6s} , &&  f(x_{1}) - f(x^{\star})  \leq \frac{10 \left\| x_{0} - x^{\star} \right\|^{2} }{27s}.
 \end{align*}

 \subsubsection{Estimate for Lyapunov function $\mathcal{E}(2)$ and $\mathcal{E}(3)$}
 With the phase-space representation form~\eqref{eqn: Nesterov_convex_symplectic2}, we have
 \begin{equation}
 \label{eqn: velocity_2}
 v_{2} = \frac{x_{3} - x_{2}}{\sqrt{s}} = \frac{1}{10} \nabla f(x_{1}) + \frac{7}{5} \nabla f(x_{2}).
 \end{equation}
 According to~\eqref{eqn: NAGM-C_original_lypunov}, the Lyapunov function $\mathcal{E}(2)$ can be written as 
 \begin{align*}
 \mathcal{E}(2)  =      15s \left( f(x_{2}) - f(x^{\star})\right) + \frac{1}{2} \left\| 2(x_{2} - x^{\star}) + 5\sqrt{s} v_{2} + 3s \nabla f(x_{2})\right\|.^{2} \end{align*}
With~\eqref{eqn: velocity_2} and Cauchy-Schwarz inequality, we have
 \begin{align*}
 \mathcal{E}(2)  & \leq  \frac{15Ls}{2} \left\| x_{2} - x^{\star}\right\|^{2} + \frac{3}{2} \left( 4 \left\| x_{2} - x^{\star} \right\|^{2} + 25s \left\| v_{2}\right\|^{2} + 9s^{2} \left\| \nabla f(x_{2})\right\|^{2}\right) \\
                         & \leq \left( \frac{15Ls}{2} + 6 \right) \left\| x_{2} - x^{\star} \right\|^{2} + \frac{27}{2}s^{2} \left\| \nabla f(x_{2})\right\|^{2}  + \frac{75}{2}s^{2} \left\| \frac{1}{10} \nabla f(x_{1}) + \frac{7}{5} \nabla f(x_{2}) \right\|^{2} \\
                         & \leq  \left( \frac{15Ls}{2} + 6 \right) \left\| x_{2} - x^{\star} \right\|^{2} + \frac{27}{2}s^{2} \left\| \nabla f(x_{2})\right\|^{2}  + \frac{3}{4}s^{2} \left\|  \nabla f(x_{1}) \right\|^{2} + 147s^{2} \left\|\nabla f(x_{2})\right\|^{2}  \\
                          & =     \left( \frac{15Ls}{2} + 6 \right) \left\| x_{2} - x^{\star} \right\|^{2} + \frac{321}{2}s^{2} \left\| \nabla f(x_{2})\right\|^{2} + \frac{3}{4}s^{2} \left\|  \nabla f(x_{1}) \right\|^{2} .
 \end{align*}  
Furthermore, with~\eqref{eqn: iterate2}, we have 
 \begin{align*}                           
   \mathcal{E}(2)  \leq   \left( \frac{15Ls}{2} + 6 \right) \left\| x_{0}- x^{\star} -s\nabla f(x_{0}) - \frac{5}{4}s \nabla f(x_{0} - s\nabla f(x_{0}))  \right\|^{2} + \frac{321}{2}s^{2} \left\| \nabla f(x_{2})\right\|^{2} + \frac{3}{4}s^{2} \left\|  \nabla f(x_{1}) \right\|^{2}.  
    \end{align*}
Finally, with~\eqref{eqn: gradient1}-\eqref{eqn: gradient2}, Cauchy-Schwarz inequality tells 
 \begin{eqnarray}
\label{eqn: estimate_energy2}
\mathcal{E}(2) &\leq & \left\{ \left[\frac{3}{16} \left(12 + 15Ls\right) + \frac{963}{16}L^{2}s^{2}\right] \left(8 + 33L^{2}s^{2} + 25L^{4}s^{4} \right) + \frac{3}{2}L^{2}s^{2}(1 + L^{2}s^{2})\right\}  \cdot \left\| x_{0} - x^{\star} \right\|^{2} \nonumber \\
                        & =   &\frac{288 + 360Ls + 8916L^{2}s^{2} + 1485L^{3}s^{3} + 32703L^{4}s^{4} + 1125L^{5}s^{5} + 24075L^{6}s^{6}}{16} \nonumber \\
                        &&\cdot \left\| x_{0} - x^{\star} \right\|^{2}.
 \end{eqnarray}  
By Lemma~\ref{lm:e_decay}, when the step size $s \leq 1/(3L)$, ~\eqref{eqn: estimate_energy2} tells us
\[
\mathcal{E}(3) \leq \mathcal{E}(2) \leq 119 \left\| x_{0} - x^{\star} \right\|^{2}.
\]                         
                               

%

\subsection{Proof of Theorem~\ref{thm: modified-NAGM-C}}
\label{subsec: position_acceleration}

Let $w_{k} = (1/2) \left[ (k + 2)x_{k} - ky_{k} + (k - 1)s \nabla f(y_{k}) \right]$ for convenience. Using the dynamics of  $\{(x_{k}, y_{k})\}_{k = 0}^{\infty}$ generated by the modified NAG-\texttt{C}~\eqref{eqn: modified-NAGM-C}, we have
$$
\begin{aligned}
w_{k + 1} & = \frac{1}{2} \left[ (k + 3) x_{k + 1} - (k + 1)y_{k + 1} + s k  \nabla f(y_{k + 1}) \right] \\
                & = \frac{1}{2} \left[ (k + 3)  \left(y_{k + 1} + \frac{k}{k + 3}(y_{k + 1} - y_{k}) - \frac{sk}{k + 3} \nabla f(y_{k + 1}) \right.\right.\\
                & \qquad \qquad \qquad \quad \left. \left. + \frac{s (k - 1)}{k + 3} \nabla f(y_{k}) \right)- (k + 1)y_{k + 1} + s k  \nabla f(y_{k + 1}) \right] \\
                & = \frac{1}{2} \left[ (k + 2)y_{k + 1} - ky_{k} + s (k - 1) \nabla f(y_{k - 1}) \right] \\
                & = w_{k} - \frac{s(k + 2)}{2} \nabla f(x_{k}).
\end{aligned}
$$
Hence, the difference between $\left\| w_{k + 1} - x^{\star} \right\|^{2}$ and $\left\| w_{k} - x^{\star} \right\|^{2}$ is
$$
\begin{aligned}
     \frac{1}{2}\left\| w_{k + 1} - x^{\star} \right\|^{2} - \frac{1}{2} \left\| w_{k} - x^{\star} \right\|^{2} 
 & =  \left\langle w_{k + 1} - w_{k}, \frac{w_{k + 1} + w_{k}}{2} - x^{\star}\right\rangle \\
 & =  \frac{s^{2}(k + 2)^{2} }{8} \left\| \nabla f(x_{k})\right\|^{2} - \frac{s(k + 2)}{2} \left\langle \nabla f(x_{k}), w_{k} - x^{\star} \right\rangle \\
 & = \frac{s^{2} (k + 2)^{2} }{8} \left\| \nabla f(x_{k})\right\|^{2} - \frac{s^{2} (k - 1)(k + 2)}{4} \left\langle \nabla f(x_{k}), \nabla f(y_{k})\right\rangle \\
 & \quad - \frac{s(k + 2)}{4} \left\langle \nabla f(x_{k}), (k + 2)x_{k} - ky_{k} - 2x^{\star} \right\rangle.
\end{aligned}
$$
If the step size satisfies $s \leq 1/L$, there exists a tighter basic inequality than~\cite[Equation (22)]{su2016differential} and~\cite[Lemma 3.6]{bubeck2015convex} for any function $f(x) \in \mathcal{F}^{1}_{L}(\mathbb{R}^n)$
\begin{align}\label{eqn: more_tight_general_convex}
f(x - s \nabla f(x)) \leq f(y) + \left\langle \nabla f(x), x - y \right\rangle - \frac{s}{2} \left\| \nabla f(x) \right\|^{2} - \frac{s}{2} \left\| \nabla f(x) - \nabla f(y) \right\|^{2}.
\end{align}
With~\eqref{eqn: more_tight_general_convex}, we can obtain that 
\begin{multline*}
      (k + 2) \left( f(y_{k + 1}) - f(x^{\star}) \right) - k \left( f(y_{k}) - f(x^{\star}) \right) \leq    \left\langle \nabla f(x_{k}), (k + 2)x_{k} - ky_{k} -2x^{\star} \right\rangle \\
                                                                                                                                      - \frac{s(k + 2)}{2} \left\| \nabla f(x_{k}) \right\|^{2} - \frac{sk}{2} \left\| \nabla f(x_{k}) - \nabla f(y_{k}) \right\|^{2}.
\end{multline*}
Consider the discrete Lyapunov function
\begin{equation}
\label{eqn: discrete_energy_functional}
\mathcal{E}(k) = \frac{s(k + 1)^{2}}{4} \left( f(y_{k}) - f(x^{\star}) \right) + \frac{1}{2} \left\| w_{k} - x^{\star} \right\|^{2}.
\end{equation}
Hence, the difference between $\mathcal{E}(k + 1)$ and $\mathcal{E}(k)$ in~\eqref{eqn: discrete_energy_functional} is
\begin{align}\label{eqn: td_estimate}
          \mathcal{E}(k + 1) - \mathcal{E}(k) & =       - \frac{1}{4} \left( f(y_{k}) - f(x^{\star}) \right) - \frac{s^{2} (k - 1)(k + 2)}{2} \left\langle \nabla f(x_{k}), \nabla f(y_{k}) \right\rangle \nonumber \\
                                                                 &\quad - \frac{s^{2}k (k +2)}{8} \left\| \nabla f(x_{k}) - \nabla f(y_{k}) \right\|^{2} \nonumber \\
                                                                 & \leq   - \frac{1}{4} \left( f(y_{k}) - f(x^{\star}) \right) - \frac{s^{2} (k - 1)(k + 2)}{8} \left\| \nabla f(x_{k}) + \nabla f(y_{k}) \right\|^{2}.
\end{align}
When $k \geq 2$, we have
\begin{align*}
\mathcal{E}(k + 1) - \mathcal{E}(2) & = \sum_{i = 2}^{k} \left( \mathcal{E}(i + 1) - \mathcal{E}(i) \right) \\
                                                  & \leq - \sum_{i = 2}^{k} \frac{s^{2} (i - 1)(i + 2)}{8} \left\| \nabla f(x_{i}) + \nabla f(y_{i}) \right\|^{2} \\
                                                  & \leq - \frac{s^{2}}{8} \min_{2 \leq i \leq k} \left\| \nabla f(x_{i}) + \nabla f(y_{i}) \right\|^{2} \sum_{i = 2}^{k} (i - 1)(i + 2) \\
                                                  & \leq - \frac{s^{2}}{24} \min_{2 \leq i \leq k} \left\| \nabla f(x_{i}) + \nabla f(y_{i}) \right\|^{2} \cdot  k (k^{2} + 3k - 4)  \\
                                                  & \leq  - \frac{s^{2}}{24} \min_{2 \leq i \leq k} \left\| \nabla f(x_{i}) + \nabla f(y_{i}) \right\|^{2} \cdot \frac{(k+1)^3}{7} \\
                                                  & = - \frac{s^{2} (k+1)^{3}}{168}\min_{2 \leq i \leq k} \left\| \nabla f(x_{i}) + \nabla f(y_{i}) \right\|^{2}.
\end{align*}
Furthermore, we have
\[
\min_{2 \leq i \leq k} \left\| \nabla f(x_{i}) + \nabla f(y_{i}) \right\|^{2} \leq \frac{168 \left[\mathcal{E}(2) - \mathcal{E}(k + 1)\right]}{s^{2} (k + 1)^{3}} \leq \frac{168 \mathcal{E}(2)}{s^{2} (k + 1)^{3}}. 
\]
Combining with~\eqref{eqn: td_estimate}, we obtain that
\begin{align*}
\min_{2 \leq i \leq k} \left\| \nabla f(x_{i}) + \nabla f(y_{i}) \right\|^{2} & \leq \frac{168 \mathcal{E}(1)}{s^{2} (k + 1)^{3}} \\
                                                                                                          & \leq \frac{168}{s^{2} (k + 1)^{3}} \left[ s \left( f(y_{1}) - f(x^{\star})\right) + \frac{1}{2} \left\| w_{1} - x^{\star} \right\|^{2}\right] \\
                                                                                                          & \leq \frac{168}{s^{2} (k + 1)^{3}} \left( \frac{Ls}{2} \left\| y_{1} - x^\star \right\|^2 + \frac{1}{2} \left\| w_{0} - s \nabla f(x_{0}) - x^{\star} \right\|^{2} \right)\\
                                                                                                          & =    \frac{168}{s^{2} (k + 1)^{3}} \left( \frac{Ls}{2} \left\| x_{0} - s\nabla f(x_{0}) - x^\star \right\|^2 + \frac{1}{2} \left\| x_{0} - \frac{3s}{2} \nabla f(x_{0}) - x^{\star} \right\|^{2} \right) \\
                                                                                                          & \leq \frac{882  \left\| x_{0} - x^{\star} \right\|^{2}}{s^{2} (k + 1)^{3}}.
\end{align*}
Similarly, when $s \leq 1/L$, for $k = 0$, we have
\[
\left\| \nabla f(x_{0}) + \nabla f(y_{0}) \right\|^{2} = 4 \left\| \nabla f(x_{0})\right\|^{2} \leq \frac{4 \left\|x_{0} - x^{\star}\right\|^{2}}{s^{2}};
\]
for $k = 1$, following the modified NAG-\texttt{C}~\eqref{eqn: modified-NAGM-C}, we obtain $(x_{1}, y_{1})$  as
\[
y_{1} = x_{0} - s\nabla f(x_{0}), \quad x_{1} = x_{0} - \frac{4}{3}s \nabla f(x_{0}),
\]
furthermore we have
\begin{align*}
\left\| \nabla f(x_{1}) + \nabla f(y_{1}) \right\|^{2} & \leq 2 \left( \left\| \nabla f(x_{1}) \right\|^{2} +  \left\| \nabla f(y_{1}) \right\|^{2} \right) \\
                                                                            & \leq \frac{2}{s^{2}} \left( \left\| x_{1} - x^{\star} \right\|^{2} +  \left\| y_{1} - x^{\star} \right\|^{2} \right) \\
                                                                            & \leq \frac{4}{s^{2}} \left[ \left( 1 + L^{2}s^{2} \right) \left\| x_{0} - x^{\star} \right\|^{2} +  \left( 1 + (16/9)L^{2}s^{2} \right) \left\| x_{0} - x^{\star} \right\|^{2} \right] \\
                                                                            & \leq \frac{172s^{2 } \left\| x_{0} - x^{\star} \right\|^{2}}{9}.
\end{align*}
For function value,~\eqref{eqn: td_estimate} tells
\[
f(y_{k}) - f(x^{\star}) \leq \frac{4\mathcal{E}(1)}{s (k + 1)^2} \leq \frac{21 \left\| x_{0} - x^{\star} \right\|^{2} }{s (k + 1)^{2}}
\]
for all $k \geq 1$. Together with
\[
f(y_{0}) - f(x^{\star}) \leq \frac{\left\| x_{0} - x^{\star} \right\|^{2}}{s},
\]
we complete the proof.

\subsection{Nesterov's Lower Bound}
\label{subsec: lower_bound}
Recall~\cite[Theorem 2.1.7]{nesterov2013introductory}, for any $k$, $1 \leq k \leq (1/2) (n - 1)$, and any $x_{0} \in \mathbb{R}^{n}$, there exists a function $f \in \mathcal{F}_{L}^{1}( \mathbb{R}^{n})$ such that any first-order method obeys
\[
f(x_{k}) - f(x^{\star}) \geq \frac{3L \left\| x_{0} - x^{\star} \right\|^{2}}{32(k+1)^{2}}.
\]
Using the basic inequality for $f(x) \in \mathcal{F}_{L}^{1}( \mathbb{R}^{n})$,
\[
\left\| \nabla f(x_{k})\right\| \left\|x_{k} - x^{\star} \right\| \geq \left\langle  \nabla f(x_{k}), x_{k} - x^{\star} \right\rangle \geq f(x_{k}) - f(x^{\star}),
\]
we have 
\[
\left\| \nabla f(x_{k})\right\| \geq \frac{3L \left\| x_{0} - x^{\star} \right\|^{2}}{32(k+1)^{2}  \max\limits_{1 \leq k \leq  \frac{n - 1}{2}}\left\|x_{k} - x^{\star} \right\|}
\]
for  $1 \leq k \leq (1/2) (n - 1)$.

\section{Technical Details in Section~\ref{sec:extension}}
\label{sec: appendix2}
\subsection{Proof of Theorem~\ref{thm:c_vary_alpha3}: Case $\alpha = 3$ }
\label{subsec: proof_gamma_=3}
Before starting to prove Theorem~\ref{thm:c_vary_alpha3}, we first look back our high-resolution ODE framework in Section~\ref{sec:techniques}. 
\begin{itemize}
\item \textbf{Step~$1$}, the generalized high-resolution ODE has been given in~\eqref{eqn: generalize_NAGM-C_ode}.  
\item \textbf{Step~$2$}, the continuous Lyapunov function is constructed as 
         \begin{multline}
         \label{eqn: ef_alpha=3}
         \mathcal{E}(t) = t \left[ t + \left(\frac{3}{2} - \beta \right) \sqrt{s} \right] \left( f(X(t)) - f(x^{\star}) \right) \\
                                   + \frac{1}{2} \left\| 2(X(t) - x^{\star}) + t\left( \dot{X}(t) + \beta \sqrt{s} \nabla f(X(t)) \right) \right\|^{2}.
         \end{multline}
         Following this Lyapunov function~\eqref{eqn: ef_alpha=3}, we can definitely obtain similar results as Theorem~\ref{thm: first-order_NAGM-C_ode} and Corollary~\ref{coro:c_f_bound}. The detailed calculation, about the estimate of the optimal constant $\beta$ and how the constant $\beta$ influence the initial point, is left for readers.  
\item \textbf{Step $3$},  before constructing discrete Lyapunov functions,  we show the phase-space representation~\eqref{eqn: generalize_NAG-C_position} as   
        \begin{equation}
        \label{eqn: generalize_NAGM-C_SES}
        \begin{aligned}
       & x_{k} - x_{k - 1} = \sqrt{s}v_{k - 1} \\
       & v_{k} - v_{k - 1} = - \frac{\alpha}{k} v_{k} - \beta \sqrt{s} \left( \nabla f(x_{k}) - \nabla f(x_{k - 1}) \right) - \left(1 + \frac{\alpha}{k} \right)\sqrt{s} \nabla f(x_{k}).
       \end{aligned}
       \end{equation}   
\end{itemize}
Now, we show how to construct the discrete Lyapunov function and analyze the algorithms~\eqref{eqn: generalize_NAG-C_position} with $\alpha = 3$ in order to prove Theorem~\ref{thm:c_vary_alpha3}.


\subsubsection{Case: $\beta < 1$}
When $\beta < 1$, we know that the function
\[
g(k) = \frac{k + 3}{k + 3 - \beta}
\]
decreases monotonically. Hence we can construct the discrete Lyapunov function as
\begin{multline}
\label{eqn: NAG-C_generalize=3_lypunov_<1}
\mathcal{E}(k) = s (k + 4)(k + 1) \left( f(x_{k}) - f(x^{\star}) \right) \\
                          + \frac{k + 3}{2(k + 3 - \beta)} \left\|2 (x_{k+1} - x^{\star}) + \sqrt{s} (k + 1)\left( v_{k} + \beta \sqrt{s} \nabla f(x_{k}) \right) \right\|^{2},
\end{multline}
which is slightly different from the discrete Lyapunov function~\eqref{eqn: NAGM-C_original_lypunov} for NAG-\texttt{C}. When $\beta \rightarrow 1$, the discrete Lyapunov function~\eqref{eqn: NAG-C_generalize=3_lypunov_<1} approximate to~\eqref{eqn: NAGM-C_original_lypunov} as $k \rightarrow \infty$. 

With the phase-space representation~\eqref{eqn: generalize_NAGM-C_SES} for $\alpha = 3$, we can obtain 
\begin{align}
\label{eqn: generalize_NAGM-C_SES=3}
( k + 3 )\left( v_{k} + \beta \sqrt{s} \nabla f(x_{k}) \right) -  k \left( v_{k - 1} + \beta \sqrt{s} \nabla f(x_{k - 1}) \right) =  - \sqrt{s} \left( k + 3 - 3 \beta \right)  \nabla f(x_{k}).
\end{align}
The difference of the discrete Lyapunov function~(\ref{eqn: NAG-C_generalize=3_lypunov_<1}) of the $k$-th iteration is
\begin{align*}
\mathcal{E}(k + 1) - \mathcal{E}(k)  & =  s (k + 5)(k + 2) \left( f(x_{k + 1}) - f(x^{\star}) \right) -s (k + 4)(k + 1) \left( f(x_{k}) - f(x^{\star}) \right) \\
                                                         &\quad + \frac{k + 4}{2(k + 4 - \beta)} \left\|2 (x_{k+2} - x^{\star}) + \sqrt{s} (k + 2)  \left( v_{k + 1} + \beta \sqrt{s} \nabla f(x_{k + 1}) \right) \right\|^{2} \\
                                                         &\quad - \frac{k + 3}{2(k + 3 - \beta)} \left\|2 (x_{k+1} - x^{\star}) + \sqrt{s} (k + 1)  \left( v_{k} + \beta \sqrt{s} \nabla f(x_{k}) \right) \right\|^{2}\\
                                                         & \leq  s \left( k + 4 \right)(k + 1)\left( f(x_{k + 1}) - f(x_{k}) \right) +  s (2k + 6) \left( f(x_{k + 1}) - f(x^{\star}) \right) \\
                                                         & \quad + \frac{k + 4}{k + 4-\beta}\left [\left\langle 2(x_{k + 2} - x_{k + 1}) + \sqrt{s} (k + 2) \left( v_{k + 1} + \beta \sqrt{s} \nabla f(x_{k + 1}) \right)\right.\right. \\
                                                        & \qquad \qquad \qquad \qquad \qquad \qquad \qquad - \sqrt{s} (k + 1)\left(v_{k} + \beta \sqrt{s} \nabla f(x_{k}) \right), \\
                                                        & \qquad \qquad \qquad \qquad   \left. \left.2 (x_{k + 2} - x^{\star}) + \sqrt{s} (k + 2) \left(v_{k + 1} + \beta \sqrt{s} \nabla f(x_{k + 1}) \right)\right\rangle \right.\\
                                                         & \qquad \qquad \qquad \quad \left. - \frac{1}{2} \left\|2(x_{k + 2} - x_{k + 1}) + \sqrt{s} (k + 2)\left( v_{k + 1} + \beta \sqrt{s} \nabla f(x_{k + 1}) \right) \right.\right.\\
                                                        &\qquad \qquad \qquad \qquad \qquad\qquad \qquad \qquad  \left.\left. - \sqrt{s} (k + 1) \left( v_{k} + \beta \sqrt{s} \nabla f(x_{k}) \right) \right\|^{2} \right]\\
                                                        & = s \left( k + 4 \right) (k + 1) \left( f(x_{k + 1}) - f(x_{k}) \right) + s (2k + 6) \left( f(x_{k + 1}) - f(x^{\star}) \right) \\
                                                         &\quad - \left\langle  s(k + 4)  \nabla f(x_{k + 1}), 2 (x_{k + 2} - x^{\star}) + \sqrt{s}(k + 2) \left( v_{k + 1} + \beta\sqrt{s}\nabla f(x_{k + 1}) \right)\right\rangle \\
                                                         &\quad - \frac{1}{2} s^{2} (k + 4) \left( k + 4 - \beta \right)  \left\|  \nabla f(x_{k + 1}) \right\|^{2}. 
\end{align*}    
With the basic inequality of any function $f(x) \in \mathcal{F}_{L}^{1}(\mathbb{R}^{n})$
$$
\left\{ \begin{aligned}
         & f(x_{k})       \geq f(x_{k + 1}) + \left\langle \nabla f(x_{k + 1}), x_{k} - x_{k + 1} \right\rangle + \frac{1}{2L} \left\| \nabla f(x_{k + 1}) - \nabla f(x_{k})  \right\|^{2} \\
         & f(x^{\star})  \geq f(x_{k + 1}) + \left\langle \nabla f(x_{k + 1}), x^{\star} - x_{k + 1} \right\rangle,                
         \end{aligned}\right. 
$$
and the phase-space representation~\eqref{eqn: generalize_NAGM-C_SES}
\[
x_{k + 2} = x_{k + 1} +\sqrt{s}v_{k + 1},
\]
the difference of the discrete Lyapunov function~(\ref{eqn: NAG-C_generalize=3_lypunov_<1}) can be estimated as                                                 
\begin{align*}                                                        
\mathcal{E}(k + 1) - \mathcal{E}(k) & \leq s(k + 4)(k + 1) \left( \left\langle \nabla f(x_{k+1}), x_{k + 1} - x_{k} \right\rangle -  \frac{1}{2L} \left\| \nabla f(x_{k + 1}) - \nabla f(x_{k})  \right\|^{2} \right) \\
                                                        & \quad + s(2k + 6) \left(f(x_{k + 1}) -  f(x^{\star})  \right) -  s(2k + 8) \left\langle \nabla f(x_{k + 1}), x_{k + 1} - x^{\star}\right\rangle \\
                                                        & \quad - s^{\frac{3}{2}} (k + 4)^{2} \left\langle \nabla f(x_{k+1}), v_{k + 1} \right\rangle - \beta s^{2} (k + 2)(k + 4) \left\| \nabla f(x_{k + 1}) \right\|^{2}\\
                                                        & \quad  - \frac{1}{2} s^{2} (k + 4) \left( k + 4 - \beta \right)  \left\|  \nabla f(x_{k + 1}) \right\|^{2} \\
                                                        & \leq    - s^{\frac{3}{2}} (k + 4)  \left\langle \nabla f(x_{k + 1}), (k + 4)v_{k + 1} - (k + 1) v_{k}  \right\rangle \\
                                                        & \quad - \frac{s (k + 4)(k + 1)}{2L} \left\| \nabla f(x_{k + 1}) - \nabla f(x_{k}) \right\|^{2}\\
                                                        & \quad - 2s \left( f(x_{k + 1}) - f(x^{\star}) \right) \\
                                                        & \quad - s^{2}\left[ \beta(k +4) (k + 2) + \frac{1}{2} (k + 4) \left(k + 4 - \beta\right) \right]  \left\| \nabla f(x_{k + 1}) \right\|^{2}. 
\end{align*}                                     
Utilizing the phase-space representation~\eqref{eqn: generalize_NAGM-C_SES} again, we calculate the difference of the discrete Lyapunov function~(\ref{eqn: NAG-C_generalize=3_lypunov_<1}) as
\begin{align*}          
 \mathcal{E}(k + 1) - \mathcal{E}(k)  & \leq    s^{\frac{3}{2}} (k + 4)  \left\langle \nabla f(x_{k + 1}), \beta \sqrt{s} (k + 1) \left( \nabla f(x_{k + 1}) - \nabla f(x_{k}) \right) + \sqrt{s} (k + 4) \nabla f(x_{k + 1}) \right\rangle \\
                                                        & \quad - \frac{s (k + 4)(k + 1)}{2L} \left\| \nabla f(x_{k + 1}) - \nabla f(x_{k}) \right\|^{2}\\
                                                        & \quad - s^{2}\left[ \beta(k +4) (k + 2) + \frac{1}{2} (k + 4) \left(k + 4 - \beta\right) \right]  \left\| \nabla f(x_{k + 1}) \right\|^{2}\\
                                                       & \leq \beta s^{2} (k + 4)(k + 1) \left\langle \nabla f(x_{k + 1}), \nabla f(x_{k + 1}) - \nabla f(x_{k})  \right\rangle \\
                                                        & \quad - \frac{s (k + 4) (k + 1)}{2L} \left\| \nabla f(x_{k + 1}) - \nabla f(x_{k}) \right\|^{2}\\
                                                        & \quad - \left[ (k + 2) (k +4) \beta - \frac{1}{2} \left(k + 4 + \beta \right)(k + 4)  \right] s^{2} \left\| \nabla f(x_{k + 1}) \right\|^{2} \\
                                                        & \leq   \frac{L\beta^2  s^{3} }{2} (k + 4)(k + 1) \left\| \nabla f(x_{k+ 1})\right\|^{2}\\
                                                        & \quad - \left[ (k + 2) (k +4) \beta - \frac{1}{2} \left(k + 4 + \beta \right)(k + 4)  \right] s^{2} \left\| \nabla f(x_{k + 1}) \right\|^{2} \\
                                                        & =  - \left[ \beta(k + 2)  - \frac{1}{2}\left( k + 4 + \beta \right) - \frac{L\beta^2 s}{2} (k + 1) \right](k +4) s^{2} \left\| \nabla f(x_{k + 1}) \right\|^{2}.
 \end{align*}
To guarantee that the Lyapunov function $\mathcal{E}(k)$ is decreasing, a sufficient condition is
\begin{align}\label{eqn: sufficient_generalize_=3}
\beta(k + 2)  - \frac{1}{2}\left( k + 4 + \beta \right) - \frac{L\beta^2 s}{2} (k + 1) \geq 0.
\end{align}
Simple calculation tells us that~\eqref{eqn: sufficient_generalize_=3} can be rewritten as
\begin{align}\label{eqn: sufficient_generalize_=3_step}
s \leq \frac{(2 \beta - 1)k + 3 \beta - 4 }{(k + 1) L\beta^{2}} = \frac{1}{L\beta^{2}}\left( 2\beta- 1 + \frac{\beta - 3}{k + 1} \right).
\end{align}
Apparently, when $\beta \rightarrow 1$, the step size satisfies
\[
0 < s \leq \frac{k-1}{k + 1} \cdot \frac{1}{L}
\]
which is consistent with~\eqref{eqn: NAGM-C-condition}.
Now, we turn to discuss the parameter $0 \leq \beta < 1$ case by case.
\begin{itemize}
\item When the parameter $\beta \leq 1/2 $, the sufficient condition~(\ref{eqn: sufficient_generalize_=3}) for the Lyapunov function $\mathcal{E}(k)$ decreasing cannot be satisfied for sufficiently large $k$.
\item When the parameter $1/2 < \beta <1$,  since the function $h(k) = \frac{1}{L\beta^{2}}\left( 2\beta- 1 + \frac{\beta - 3}{k + 1} \right)$ increases monotonically for $k \geq 0$, there exists $k_{3, \beta} = \left\lfloor \frac{4 - 3\beta}{2\beta - 1}\right\rfloor + 1$ such that the step size
\[
s \leq  \frac{(2 \beta - 1)k_{3,\beta} + 3 \beta - 4 }{(k_{3,\beta} + 1) L\beta^{2}}
\]
works for any $k \geq k_{3,\beta}$ ($k_{3,\beta} \rightarrow 2$ with $\beta \rightarrow 1$). Then, the difference of the discrete Lyapunov function~(\ref{eqn: NAG-C_generalize=3_lypunov_<1}) can be estimated as
         \[
          \mathcal{E}(k + 1) - \mathcal{E}(k) \leq -s^{2} \left( \frac{2\beta- 1 - L\beta^2s}{2} \right) (k - k_{3,\beta})^{2}  \left\| \nabla f(x_{k + 1}) \right\|^{2}.
         \]
Here, the proof is actually complete. Without loss of generality, we briefly show the expression is consistent with Theorem~\ref{thm:c_vary_alpha3} and omit the proofs for the following facts.  When $k \geq k_{3, \beta} + 1$, there exists some constant $\mathfrak{C}^{0}_{3, \beta} > 0$ such that
\[
\mathcal{E}(k + 1) - \mathcal{E}(k) \leq -s^{2} \mathfrak{C}^{0}_{3, \beta} (k + 1)^{2}  \left\| \nabla f(x_{k + 1}) \right\|^{2}.
\]
For $k \leq k_{3, \beta}$, using mathematic induction,  there also exists some constant $\mathfrak{C}^{1}_{3, \beta} > 0$ such that for $s = O(1/L)$, we have
\[
\left\| \nabla f(x_{k + 1}) \right\|^{2} \leq \frac{\mathfrak{C}^{1}_{3, \beta} \left\| x_{0} - x^{\star} \right\|^{2}}{s^{2}} \quad \text{and}\quad f(x_{k})- f(x^\star)\leq \frac{\mathcal{E}(k)}{4s} \leq \frac{\mathfrak{C}^{1}_{3, \beta} \left\| x_{0} - x^{\star} \right\|^{2}}{s}.
\]
\end{itemize}

%
%
%
%


\subsubsection{Case: $\beta \geq 1$}
When $\beta \geq 1$, we know that the function
\[
g(k) = \frac{k + 2}{k + 3 - \beta}
\]
decreases monotonically. Hence we can construct the discrete Lyapunov function as
\begin{multline}
\label{eqn: NAG-C_generalize=3_lypunov_>1}
\mathcal{E}(k) = s (k + 3)(k + 1) \left( f(x_{k}) - f(x^{\star}) \right) \\
                          + \frac{k + 2}{2(k + 3 - \beta)} \left\|2 (x_{k+1} - x^{\star}) + \sqrt{s} (k + 1)\left( v_{k} + \beta \sqrt{s} \nabla f(x_{k}) \right) \right\|^{2}.
\end{multline}
which for $\beta = 1$ is consistent with the discrete Lyapunov function~\eqref{eqn: NAGM-C_original_lypunov} for NAG-\texttt{C}. 

With the expression~\eqref{eqn: generalize_NAGM-C_SES=3}
\[
\label{eqn: simple_transform_generalize_NAGM-C_SES=3}
( k + 3 )\left( v_{k} + \beta \sqrt{s} \nabla f(x_{k}) \right) -  k \left( v_{k - 1} + \beta \sqrt{s} \nabla f(x_{k - 1}) \right) =  - \sqrt{s} \left( k + 3 - 3 \beta \right)  \nabla f(x_{k}),
\]
the difference of the discrete Lyapunov function~(\ref{eqn: NAG-C_generalize=3_lypunov_>1}) of the $k$-th iteration is
\begin{align*}
\mathcal{E}(k + 1) - \mathcal{E}(k)  & =  s (k + 4)(k + 2) \left( f(x_{k + 1}) - f(x^{\star}) \right) -s (k + 3)(k + 1) \left( f(x_{k}) - f(x^{\star}) \right) \\
                                                         &\quad + \frac{k + 3}{2(k + 4 - \beta)} \left\|2 (x_{k+2} - x^{\star}) + \sqrt{s} (k + 2)  \left( v_{k + 1} + \beta \sqrt{s} \nabla f(x_{k + 1}) \right) \right\|^{2} \\
                                                         &\quad - \frac{k + 2}{2(k + 3 - \beta)} \left\|2 (x_{k+1} - x^{\star}) + \sqrt{s} (k + 1)  \left( v_{k} + \beta \sqrt{s} \nabla f(x_{k}) \right) \right\|^{2}\\
                                                         & \leq  s \left( k + 3 \right)(k + 1)\left( f(x_{k + 1}) - f(x_{k}) \right) +  s (2k + 5) \left( f(x_{k + 1}) - f(x^{\star}) \right) \\
                                                         & \quad + \frac{k + 3}{k + 4-\beta}\left [\left\langle 2(x_{k + 2} - x_{k + 1}) + \sqrt{s} (k + 2) \left( v_{k + 1} + \beta \sqrt{s} \nabla f(x_{k + 1}) \right)\right.\right. \\
                                                        & \qquad \qquad \qquad \qquad \qquad \qquad \qquad - \sqrt{s} (k + 1)\left(v_{k} + \beta \sqrt{s} \nabla f(x_{k}) \right), \\
                                                        & \qquad \qquad \qquad \qquad   \left. \left.2 (x_{k + 2} - x^{\star}) + \sqrt{s} (k + 2) \left(v_{k + 1} + \beta \sqrt{s} \nabla f(x_{k + 1}) \right)\right\rangle \right.\\
                                                         & \qquad \qquad \qquad \quad \left. - \frac{1}{2} \left\|2(x_{k + 2} - x_{k + 1}) + \sqrt{s} (k + 2)\left( v_{k + 1} + \beta \sqrt{s} \nabla f(x_{k + 1}) \right) \right.\right.\\
                                                        &\qquad \qquad \qquad \qquad \qquad\qquad \qquad \qquad  \left.\left. - \sqrt{s} (k + 1) \left( v_{k} + \beta \sqrt{s} \nabla f(x_{k}) \right) \right\|^{2} \right]\\
                                                        & = s \left( k + 3 \right) (k + 1) \left( f(x_{k + 1}) - f(x_{k}) \right) + s (2k + 5) \left( f(x_{k + 1}) - f(x^{\star}) \right) \\
                                                         &\quad - \left\langle  s(k + 3)  \nabla f(x_{k + 1}), 2 (x_{k + 2} - x^{\star}) + \sqrt{s}(k + 2) \left( v_{k + 1} + \beta\sqrt{s}\nabla f(x_{k + 1}) \right)\right\rangle \\
                                                         &\quad - \frac{1}{2} s^{2} (k + 3) \left( k + 4 - \beta \right)  \left\|  \nabla f(x_{k + 1}) \right\|^{2}. 
\end{align*}   
With the basic inequality of any function $f(x) \in \mathcal{F}_{L}^{1}(\mathbb{R}^{n})$
$$
\left\{ \begin{aligned}
         & f(x_{k})       \geq f(x_{k + 1}) + \left\langle \nabla f(x_{k + 1}), x_{k} - x_{k + 1} \right\rangle + \frac{1}{2L} \left\| \nabla f(x_{k + 1}) - \nabla f(x_{k})  \right\|^{2} \\
         & f(x^{\star})  \geq f(x_{k + 1}) + \left\langle \nabla f(x_{k + 1}), x^{\star} - x_{k + 1} \right\rangle,                
         \end{aligned}\right. 
$$
and the phase-space representation~\eqref{eqn: generalize_NAGM-C_SES}
\[
x_{k + 2} = x_{k + 1} +\sqrt{s}v_{k + 1},
\]
the difference of the discrete Lyapunov function~(\ref{eqn: NAG-C_generalize=3_lypunov_>1}) can be estimated as                                                 
\begin{align*}                                                        
\mathcal{E}(k + 1) - \mathcal{E}(k) & \leq s(k + 3)(k + 1) \left( \left\langle \nabla f(x_{k+1}), x_{k + 1} - x_{k} \right\rangle -  \frac{1}{2L} \left\| \nabla f(x_{k + 1}) - \nabla f(x_{k})  \right\|^{2} \right) \\
                                                        & \quad + s(2k + 5) \left(f(x_{k + 1}) -  f(x^{\star})  \right) -  s(2k + 6) \left\langle \nabla f(x_{k + 1}), x_{k + 1} - x^{\star}\right\rangle \\
                                                        & \quad - s^{\frac{3}{2}} (k + 3)(k + 4) \left\langle \nabla f(x_{k+1}), v_{k + 1} \right\rangle - \beta s^{2} (k + 2)(k + 3) \left\| \nabla f(x_{k + 1}) \right\|^{2}\\
                                                        & \quad  - \frac{1}{2} s^{2} (k + 3) \left( k + 4 - \beta \right)  \left\|  \nabla f(x_{k + 1}) \right\|^{2} \\
                                                        & \leq    - s^{\frac{3}{2}} (k + 3)  \left\langle \nabla f(x_{k + 1}), (k + 4)v_{k + 1} - (k + 1) v_{k}  \right\rangle \\
                                                        & \quad - \frac{s (k + 3)(k + 1)}{2L} \left\| \nabla f(x_{k + 1}) - \nabla f(x_{k}) \right\|^{2}\\
                                                        & \quad - 2s \left( f(x_{k + 1}) - f(x^{\star}) \right) \\
                                                        & \quad - s^{2}\left[ \beta(k +3) (k + 2) + \frac{1}{2} (k + 3) \left(k + 4 - \beta\right) \right]  \left\| \nabla f(x_{k + 1}) \right\|^{2}. 
\end{align*}           
Utilize the phase-space representation~\eqref{eqn: generalize_NAGM-C_SES} again, we calculate the difference of the discrete Lyapunov function~(\ref{eqn: NAG-C_generalize=3_lypunov_>1}) as         
\begin{align*}          
 \mathcal{E}(k + 1) - \mathcal{E}(k)  & \leq    s^{\frac{3}{2}} (k + 3)  \left\langle \nabla f(x_{k + 1}), \beta \sqrt{s} (k + 1) \left( \nabla f(x_{k + 1}) - \nabla f(x_{k}) \right) + \sqrt{s} (k + 4) \nabla f(x_{k + 1}) \right\rangle \\
                                                        & \quad - \frac{s (k + 3)(k + 1)}{2L} \left\| \nabla f(x_{k + 1}) - \nabla f(x_{k}) \right\|^{2}\\
                                                        & \quad - s^{2}\left[ \beta(k +3) (k + 2) + \frac{1}{2} (k + 3) \left(k + 4 - \beta\right) \right]  \left\| \nabla f(x_{k + 1}) \right\|^{2}\\
                                                       & \leq \beta s^{2} (k + 3)(k + 1) \left\langle \nabla f(x_{k + 1}), \nabla f(x_{k + 1}) - \nabla f(x_{k})  \right\rangle \\
                                                        & \quad - \frac{s (k + 3) (k + 1)}{2L} \left\| \nabla f(x_{k + 1}) - \nabla f(x_{k}) \right\|^{2}\\
                                                        & \quad - \left[ (k + 2) (k +3) \beta - \frac{1}{2} \left(k + 4 + \beta \right)(k + 3)  \right] s^{2} \left\| \nabla f(x_{k + 1}) \right\|^{2} \\
                                                        & \leq   \frac{L\beta^2  s^{3} }{2} (k + 3)(k + 1) \left\| \nabla f(x_{k+ 1})\right\|^{2}\\
                                                        & \quad - \left[ (k + 2) (k +3) \beta - \frac{1}{2} \left(k + 4 + \beta \right)(k + 3)  \right] s^{2} \left\| \nabla f(x_{k + 1}) \right\|^{2} \\
                                                        & =  - \left[ \beta(k + 2)  - \frac{1}{2}\left( k + 4 + \beta \right) - \frac{L\beta^2 s}{2} (k + 1) \right](k +3) s^{2} \left\| \nabla f(x_{k + 1}) \right\|^{2}.
 \end{align*}
Consistently, we can obtain the sufficient condition for the Lyapunov function $\mathcal{E}(k)$ decreasing~\eqref{eqn: sufficient_generalize_=3} and the sufficient condition for step size~\eqref{eqn: sufficient_generalize_=3_step}.

Now, we turn to discuss the parameter $\beta \geq 1$ case by case.
\begin{itemize}
\item When the parameter $\beta \geq 3$,  since the function $h(k) = \frac{1}{L\beta^{2}}\left( 2\beta- 1 + \frac{\beta - 3}{k + 1} \right)$  decreases monotonically for $k \geq 0$,  then the condition of the step size 
         \[
         s \leq \frac{2 \beta - 1}{(1 + \epsilon)L\beta^2} < \frac{2 \beta - 1}{L\beta^2}
         \]
         holds for~(\ref{eqn: sufficient_generalize_=3}), where $\epsilon > 0$ is a real number. Hence, when $k \geq k_{3,\beta} + 1$, where \[k_{3,\beta}= \max\left\{0, \left\lfloor \beta - 3\right\rfloor + 1, \left\lfloor \frac{4 - 3\beta + L\beta^{2} s}{2\beta - 1 - L\beta^{2}s} \right\rfloor + 1 \right\},\] the difference of the discrete Lyapunov function~(\ref{eqn: NAG-C_generalize=3_lypunov_>1}) can be estimated as
         \[
          \mathcal{E}(k + 1) - \mathcal{E}(k) \leq - s^{2}\left( \frac{2\beta - 1 - L\beta^2s}{2} \right)  (k - k_{3,\beta})^{2}  \left\| \nabla f(x_{k + 1}) \right\|^{2}.
         \]
\item When the parameter $1 \leq \beta < 3$,   since the function $h(k) = \frac{1}{L\beta^{2}}\left( 2\beta- 1 + \frac{\beta - 3}{k + 1} \right)$ increases monotonically for $k \geq 0$, there exists $k_{3,\beta} =\max\left\{0, \left\lfloor \beta - 3\right\rfloor + 1, \left\lfloor \frac{4 - 3\beta}{2\beta - 1}\right\rfloor + 1\right\}$ such that the step size
\[
s \leq  \frac{(2 \beta - 1)k_{3,\beta} + 3 \beta - 4 }{(k_{3,\beta} + 1) L\beta^{2}}
\]
works for any $k \geq k_{3, \beta}$. When $\beta = 1$, the step size satisfies
\[0 < s \leq \frac{k - 1}{k + 1} \cdot \frac{1}{L}\]
which is consistent with~\eqref{eqn: NAGM-C-condition} and $k_{3,\beta} = 2$.
Then, the difference of the discrete Lyapunov function~(\ref{eqn: NAG-C_generalize=3_lypunov_<1}) can be estimated as
         \[
          \mathcal{E}(k + 1) - \mathcal{E}(k) \leq -  s^{2} \left( \frac{2\beta- 1 - L\beta^2s}{2} \right) (k - k_{3,\beta})^{2} \left\| \nabla f(x_{k + 1}) \right\|^{2}.
         \]
for all  $k \geq k_{3,\beta} + 1$.

\end{itemize}

By simple calculation, we complete the proof.

%

\subsection{Proof of Theorem~\ref{thm:c_vary_alpha3}: Case $\alpha > 3$}
\label{subsec: proof_gamma_>3}
Before starting to prove Theorem~\ref{thm:c_vary_alpha3}: Case $\alpha > 3$, we first also look back our high-resolution ODE framework in Section~\ref{sec:techniques}. 
\begin{itemize}
\item \textbf{Step $1$}, the generalized high-resolution ODE has been given in~\eqref{eqn: generalize_NAGM-C_ode}.  
\item \textbf{Step $2$}, the continuous Lyapunov function is constructed as 
         \begin{multline}
         \label{eqn: ef_alpha>3}
         \mathcal{E}(t) = t \left[ t + \left(\frac{\alpha}{2} - \beta\right) \sqrt{s} \right] \left( f(X(t)) - f(x^{\star}) \right) \\
                                   + \frac{1}{2} \left\| (\alpha - 1)(X(t) - x^{\star}) + t\left( \dot{X}(t) + \beta \sqrt{s} \nabla f(X(t)) \right) \right\|^{2},
         \end{multline}
         which is consistent with~\eqref{eqn: ef_alpha=3} for $\alpha \rightarrow 3$. Following this Lyapunov function~\eqref{eqn: ef_alpha>3}, we can obtain 
         \begin{equation}
         \label{eqn:convergence_ode_>3}
         \begin{aligned}
         & f(X(t)) - f(x^{\star}) \leq O\left( \frac{\|X(t_{0}) - x^{\star}\|^{2}}{(t - t_{0})^{2}}\right) \\
         & \int_{t_{0}}^{t} u \left(f(X(u)) - f(x^{\star}) \right) + \sqrt{s} u^{2} \left\| \nabla f(X(u)) \right\|^{2}du  \leq O\left( \|X(t_{0}) - x^{\star}\|^{2}\right)
         \end{aligned}
         \end{equation}
         for any $t > t_{0} = \max\left\{  \sqrt{s} (\alpha/2 - \beta)(\alpha -2)/(\alpha-3), \sqrt{s}(\alpha/2)\right\}$.  The two inequalities of~\eqref{eqn:convergence_ode_>3} for the convergence rate of function value is stronger than Corollary~\ref{coro:c_f_bound}. The detailed calculation, about the estimate of the optimal constant $\beta$ and how the constant $\beta$ influences the initial point, is left for readers.  
\item \textbf{Step $3$},  before constructing discrete Lyapunov functions,  we look back the phase-space representation~\eqref{eqn: generalize_NAGM-C_SES}     
       \[
        \begin{aligned}
       & x_{k} - x_{k - 1} = \sqrt{s}v_{k - 1} \\
       & v_{k} - v_{k - 1} = - \frac{\alpha}{k} v_{k} - \beta \sqrt{s} \left( \nabla f(x_{k}) - \nabla f(x_{k - 1}) \right) - \left(1 + \frac{\alpha}{k} \right)     
       \sqrt{s} \nabla f(x_{k}).
       \end{aligned}
       \] 
      The discrete functional is constructed as 
       \begin{multline}
       \label{eqn: energy_functional_ses_>3}
       \mathcal{E}(k) = s (k + 1) (k + \alpha - \beta + 1) \left( f(x_{k}) - f(x^{\star}) \right) \\
                                  + \frac{1}{2} \left\| (\alpha - 1)(x_{k + 1} - x^{\star}) + \sqrt{s}(k + 1)\left(v_{k} + \beta \sqrt{s} \nabla f(x_{k})\right)   \right\|^{2}.
       \end{multline}
       When $\beta = 1$, with $\alpha \rightarrow 3$, the discrete Lyapunov function $\mathcal{E}(k)$ degenerates to~\eqref{eqn: NAGM-C_original_lypunov}.
\end{itemize}
 Now, we procced to \textbf{Step $4$} to analyze the algorithms~\eqref{eqn: generalize_NAG-C_position} with $\alpha > 3$ in order to prove Theorem~\ref{thm: OA_generalize_>3}. The simple transformation of~(\ref{eqn: generalize_NAGM-C_SES}) for $\alpha > 3$ is  
\begin{align}
\label{eqn: simple_transform_generalize_NAGM-C_SES>3}
( k + \alpha )\left( v_{k} + \beta\sqrt{s} \nabla f(x_{k}) \right) -  k \left( v_{k - 1} + \beta\sqrt{s} \nabla f(x_{k - 1}) \right) =  - \sqrt{s} \left( k + \gamma   - \gamma \beta \right) \nabla f(x_{k}).
\end{align}
Thus, the difference of the Lyapunov function~(\ref{eqn: energy_functional_ses_>3}) on the $k$-th iteration is
\begin{align*}
     \mathcal{E}(k + 1) - \mathcal{E}(k) & = s (k + 2) (k + \alpha - \beta + 2) \left( f(x_{k}) - f(x^{\star}) \right) \\
                                                             & \quad + \frac{1}{2} \left\| (\alpha - 1)(x_{k + 2} - x^{\star}) + \sqrt{s}(k + 2)\left(v_{k + 1} + \beta \sqrt{s} \nabla f(x_{k + 1})\right)   \right\|^{2} \\
                                                            & \quad -s (k + 1) (k + \alpha - \beta + 1) \left( f(x_{k}) - f(x^{\star}) \right) \\
                                                            & \quad - \frac{1}{2} \left\| (\alpha - 1)(x_{k + 1} - x^{\star}) + \sqrt{s}(k + 1)\left(v_{k} + \beta \sqrt{s} \nabla f(x_{k})\right)   \right\|^{2}\\
                                                            &=  s( k  + 1) \left(  k + \alpha - \beta + 1  \right) \left( f(x_{k + 1}) - f(x_{k}) \right) + s \left( 2k + \alpha - \beta + 3 \right) \left(f(x_{k + 1}) - f(x^{\star}) \right)\\
                                                             &  \quad + \left\langle (\alpha - 1)(x_{k + 2} - x_{k + 1}) + \sqrt{s} (k + 2) \left( v_{k + 1} + \beta \sqrt{s} \nabla f(x_{k + 1}) \right)\right. \\
                                                             & \qquad \qquad\qquad \qquad\qquad \qquad \qquad -\sqrt{s} (k + 1) \left( v_{k} + \beta\sqrt{s} \nabla f(x_{k}) \right),  \\
                                                              & \qquad \quad \left. ( \alpha - 1) (x_{k + 2} - x^{\star}) + \sqrt{s} (k + 2) \left( v_{k + 1} + \beta\sqrt{s} \nabla f(x_{k + 1}) \right)  \right\rangle \\
                                                              &   \quad - \frac{1}{2} \left\| (\alpha - 1)(x_{k + 2} - x_{k + 1}) + (k + 2)\sqrt{s} \left( v_{k + 1} + \beta\sqrt{s} \nabla f(x_{k + 1}) \right) \right.\\
                                                              &    \qquad \qquad\qquad\qquad\qquad\qquad \qquad\left.- (k + 1)\sqrt{s} \left( v_{k} + \beta\sqrt{s} \nabla f(x_{k}) \right)\right\|^{2} \\
                                                              &=  s( k  + 1) \left(  k + \alpha - \beta + 1 \right) \left( f(x_{k + 1}) - f(x_{k}) \right) \\
                                                             &  \quad + s \left( 2k + \alpha - \beta + 3 \right) \left(f(x_{k + 1}) - f(x^{\star}) \right)\\
                                                              &\quad- \left\langle  s \left(  k + \alpha - \beta + 1 \right) \nabla f(x_{k + 1}), (\alpha - 1) (x_{k + 1} - x^{\star}) + \sqrt{s} (k + \alpha +1)  v_{k + 1}\right.\\
                                                              &  \quad\qquad \qquad\qquad\qquad\qquad\qquad \qquad\qquad\qquad\qquad \qquad\left.+ \beta s (k + 2)  \nabla f(x_{k + 1}) \right\rangle \\
                                                               & \quad  - \frac{1}{2} s^{2}(k + \alpha - \beta + 1)^{2}\left\| \nabla f(x_{k + 1}) \right\|^{2}.
\end{align*}
With the basic inequality of convex function $f(x) \in \mathcal{F}_{L}^{1}(\mathbb{R}^{n})$, 
$$
\left\{ \begin{aligned}
         & f(x_{k})       \geq f(x_{k + 1}) + \left\langle \nabla f(x_{k + 1}), x_{k} - x_{k + 1} \right\rangle + \frac{1}{2L} \left\| \nabla f(x_{k + 1}) - \nabla f(x_{k})  \right\|^{2} \\
         & f(x^{\star})  \geq f(x_{k + 1}) + \left\langle \nabla f(x_{k + 1}), x^{\star} - x_{k + 1} \right\rangle                
         \end{aligned}\right. 
$$
and the phase-space representation~\eqref{eqn: generalize_NAGM-C_SES}
\[
x_{k + 2} = x_{k + 1} +\sqrt{s}v_{k + 1},
\]
the difference of the discrete Lyapunov function~(\ref{eqn: energy_functional_ses_>3}) can be estimated as
\begin{align*}     
        \mathcal{E}(k + 1) - \mathcal{E}(k)     &=  s( k  + 1) \left(  k + \alpha - \beta + 1 \right)  \left( \left\langle \nabla f(x_{k+1}), x_{k + 1} - x_{k} \right\rangle -  \frac{1}{2L} \left\| \nabla f(x_{k + 1}) - \nabla f(x_{k})  \right\|^{2} \right) \\
                                                                    &  \quad + s \left( 2k + \alpha - \beta + 3 \right) \left(f(x_{k + 1}) - f(x^{\star}) \right) -  s (\alpha - 1)  \left(  k + \alpha - \beta + 1 \right)\left\langle \nabla f(x_{k + 1}),  x_{k + 1} - x^{\star}\right\rangle\\
                                                                    &\quad- \left\langle  s \left(  k + \alpha - \beta + 1 \right) \nabla f(x_{k + 1}), \sqrt{s} (k + \alpha +1)  v_{k + 1} \right\rangle \\
                                                                  & \quad  - \frac{1}{2} s^{2}(k + \alpha - \beta + 1) \left[ (2\beta + 1) k   + \alpha + 3\beta + 1\right]\left\| \nabla f(x_{k + 1}) \right\|^{2} \\
                                                                  & \leq  - s^{\frac{3}{2}}\left( k + \alpha - \beta + 1 \right) \left\langle \nabla f(x_{k + 1}),  ( k + \alpha + 1) v_{k + 1} - (k + 1) v_{k} \right\rangle   \\
                                                                  &\quad - \frac{s ( k + 1) \left(  k + \alpha - \beta + 1\right)}{2L} \left\| \nabla f(x_{k + 1}) - \nabla f(x_{k}) \right\|_{2}^{2} \\
                                                                & \quad - s \left[ (\alpha - 3)k  + (\alpha - 2)\left( \alpha - \beta + 1 \right) -  2 \right] \left(f(x_{k + 1}) - f(x^{\star})\right) \\
                                                                &\quad - \frac{1}{2} s^{2}(k + \alpha - \beta + 1) \left[ (2\beta + 1) k   + \alpha + 3\beta + 1\right]\left\| \nabla f(x_{k + 1}) \right\|^{2}.
 \end{align*}
Utilizing the phase-space representation~\eqref{eqn: generalize_NAGM-C_SES} again, we calculate the difference of the discrete Lyapunov  function~(\ref{eqn: energy_functional_ses_>3}) as
 \begin{align*}
 \mathcal{E}(k + 1) - \mathcal{E}(k)     & = \beta s^{2}( k  + 1) \left(  k + \alpha - \beta + 1 \right) \left\langle \nabla f(x_{k + 1}), \nabla f(x_{k + 1}) - \nabla f(x_{k}) \right\rangle \\
                                                             & \quad +  s^{2} (k + \alpha + 1) \left(  k + \alpha - \beta + 1 \right) \left\|\nabla f(x_{k + 1}) \right\|^{2} \\
                                                             &\quad - \frac{s ( k + 1) \left(  k + \alpha - \beta + 1\right)}{2L} \left\| \nabla f(x_{k + 1}) - \nabla f(x_{k}) \right\|_{2}^{2} \\
                                                              & \quad - s \left[ (\alpha - 3)k  + (\alpha - 2)\left( \alpha - \beta + 1 \right) -  2 \right] \left(f(x_{k + 1}) - f(x^{\star})\right) \\
                                                              &\quad - \frac{1}{2} s^{2}(k + \alpha - \beta + 1) \left[ (2\beta + 1) k   + \alpha + 3\beta + 1\right]\left\| \nabla f(x_{k + 1}) \right\|^{2}\\
                                                              & = \beta s^{2}( k  + 1) \left(  k + \alpha - \beta + 1 \right) \left\langle \nabla f(x_{k + 1}), \nabla f(x_{k + 1}) - \nabla f(x_{k}) \right\rangle \\
                                                             &\quad - \frac{s ( k + 1) \left(  k + \alpha - \beta + 1\right)}{2L} \left\| \nabla f(x_{k + 1}) - \nabla f(x_{k}) \right\|_{2}^{2} \\
                                                              & \quad - s \left[ (\alpha - 3)k  + (\alpha - 2)\left( \alpha - \beta + 1 \right) -  2 \right] \left(f(x_{k + 1}) - f(x^{\star})\right) \\
                                                              &\quad - \frac{1}{2} s^{2}(k + \alpha - \beta + 1) \left[ (2\beta - 1) k   - \alpha + 3\beta - 1\right]\left\| \nabla f(x_{k + 1}) \right\|^{2}\\
                                                              & \leq \frac{L \beta^{2} s^{3}}{2}( k  + 1) \left(  k + \alpha - \beta + 1 \right) \left\| \nabla f(x_{k + 1})\right\|^{2} \\
                                                              & \quad - s \left[ (\alpha - 3)k  + (\alpha - 2)\left( \alpha - \beta + 1 \right) -  2 \right] \left(f(x_{k + 1}) - f(x^{\star})\right) \\
                                                              &\quad - \frac{1}{2} s^{2}(k + \alpha - \beta + 1) \left[ (2\beta - 1) k   - \alpha + 3\beta - 1\right]\left\| \nabla f(x_{k + 1}) \right\|^{2}\\
                                                              & =   - s \left[ (\alpha - 3)k  + (\alpha - 2)\left( \alpha - \beta + 1 \right) -  2 \right] \left(f(x_{k + 1}) - f(x^{\star})\right) \\
                                                              &\quad - \frac{1}{2} s^{2}(k + \alpha - \beta + 1) \left[ (2\beta - 1) k   - \alpha + 3\beta - 1 - L\beta^{2}s (k + 1)\right]\left\| \nabla f(x_{k + 1}) \right\|^{2}
\end{align*}
To guarantee the Lyapunov function $\mathcal{E}(k)$ decreasing, a sufficient condition is
\begin{equation} \label{eqn: sufficient_condition_>3}
          (2\beta - 1) k   - \alpha + 3\beta - 1 - L\beta^{2}s (k + 1) \geq 0.
\end{equation}
With the inequality~\eqref{eqn: sufficient_condition_>3}, the step size can be estimated as
\[
s \leq \frac{2\beta - 1}{L\beta^{2}} - \frac{\alpha - \beta}{(k + 1)L\beta^{2}}.
\]
\begin{itemize}
\item When the parameter $\beta > 1/2$ and $\alpha <\beta$, since the function 
        $h(k) = \frac{2\beta - 1}{L\beta^{2}} - \frac{\alpha - \beta}{(k + 1)L\beta^{2}}$ decreases monotonically for $k \geq 0$, thus the step size 
        \[
        s \leq \frac{2\beta - 1}{(1 + \epsilon)L\beta^{2}} < \frac{2\beta - 1}{L\beta^{2}}
        \]
        holds for~\eqref{eqn: sufficient_condition_>3}, where $\epsilon > 0$ is a real number. Hence, when $k \geq k_{\alpha,\beta} + 1$, where
        \[
       k_{\alpha,\beta} = \max\left\{0, \left\lfloor\frac{2 - (\alpha - 2)(\alpha - \beta + 1)}{\alpha - 3}\right\rfloor + 1, \left\lfloor \frac{4 - 3\beta + L\beta^{2}s}{-1 + 2\beta - L\beta^{2}s}\right\rfloor + 1, \left\lfloor \beta - \alpha - 1  \right\rfloor + 1 \right\},
       \]
       the difference of the discrete Lyapunov function~\eqref{eqn: energy_functional_ses_>3} can be estimated as
       \[
       \mathcal{E}(k + 1) - \mathcal{E}(k) \leq -s (\alpha - 3) \left( k - k_{\alpha,\beta} \right) \left( f(x_{k + 1}) - f(x^{\star}) \right) - s^{2} \left(\frac{2\beta - 1 -L\beta^{2}s}{2} \right) \left( k - k_{\alpha,\beta} \right) ^{2}\left\| \nabla f(x_{k + 1})\right\|^{2}.
       \]
\item When the parameter $\beta > 1/2$ and $\alpha \geq \beta$, since the function $h(k) = \frac{2\beta - 1}{L\beta^{2}} - \frac{\alpha - \beta}{(k + 1)L\beta^{2}}$ increases monotonically for $k \geq 0$, there exists \[k_{\alpha,\beta} = \max\left\{ 0, \left\lfloor\frac{2 - (\alpha - 2)(\alpha - \beta + 1)}{\alpha - 3}\right\rfloor + 1, \left\lfloor \beta - \alpha - 1  \right\rfloor + 1,  \left\lfloor \frac{1 + \alpha - 3\beta}{2\beta - 1} \right\rfloor + 1 \right\}\]
such that the step size satisfies
\[
s \leq \frac{(2\beta - 1)k_{\alpha, \beta} - \alpha + 3\beta - 1}{L\beta^{2}(k_{\alpha, \beta} + 1)}.
\]
When $\beta = 1$, the step size satisfies
\[
s \leq \frac{1}{L} \cdot \frac{k_{\alpha, \beta} - \alpha + 2}{(k_{\alpha, \beta} + 1)} \rightarrow \frac{1}{L} \cdot \frac{k_{\alpha, \beta} - 1}{k_{\alpha, \beta} + 1} \quad \text{with}\quad \alpha \rightarrow 3,
\]
which is consistent with ~\eqref{eqn: NAGM-C-condition}. Then, the difference of the discrete Lyapunov function~\eqref{eqn: energy_functional_ses_>3} can be estimated as
       \[
       \mathcal{E}(k + 1) - \mathcal{E}(k) \leq -s (\alpha - 3) \left( k - k_{\alpha,\beta} \right) \left( f(x_{k + 1}) - f(x^{\star}) \right) - s^{2} \left(\frac{2\beta - 1 -L\beta^{2}s}{2} \right) \left( k - k_{\alpha,\beta} \right)^2 \left\| \nabla f(x_{k + 1})\right\|^{2}.
       \]
\end{itemize}

\subsection{A Simple Counterexample}
\label{subsec: counter}
The simple counterexample is constructed as 
\[
f(x_{k}) - f(x^{\star}) = \left\{ \begin{aligned}
                                             & \frac{L \left\|x_{0} - x^{\star} \right\|^{2}}{(k + 1)^{2}}, && k = j ^{2} \\
                                             & 0,                                                                          && k \neq j ^{2} 
                                            \end{aligned}\right.
\]
where $j \in \mathbb{N}$. Plugging it into~\eqref{eqn: function_value_k^2}, we have
\[
\sum_{k = 0}^{\infty} (k + 1) \left( f(x_{k}) - f(x^{\star})\right) = L \left\|x_{0} - x^{\star} \right\|^{2} \cdot \sum_{j = 0}^{\infty} \left( \frac{1}{j^{2} + 1}\right) < \infty.
\]
Hence, Proposition~\ref{thm: OA_generalize_>3} cannot guarantee the faster convergence rate.


\subsection{Super-Critical Regime: Sharper Convergence Rate $o(1/t^{2})$ and $o(L/k^{2})$}
\label{subsec: high_friction_faster}
\subsubsection{The ODE Case}
\label{subsec: continuous_faster}

Here, we still turn back to our high-resolution ODE framework in Section~\ref{sec:techniques}. The generalized high-resolution ODE has been still shown in~\eqref{eqn: generalize_NAGM-C_ode}. A more general Lyapunov function is constructed as 
\begin{multline}
\label{eqn: continous_generalize_energy_faster}
\mathcal{E}_{\nu}(t) =  t\left[t + \left(\frac{\alpha}{2} - \beta\right)\sqrt{s} + (\alpha - \nu - 1) \beta \sqrt{s} \right] \left( f(X(t)) - f(x^{\star}) \right) \\
+ \frac{\nu (\alpha - \nu - 1)}{2} \left\| X(t) - x^{\star}\right\|^{2}+ \frac{1}{2} \left\| \nu (X(t) - x^{\star}) + t\left( \dot{X}(t) + \beta \sqrt{s} \nabla f(X(t)) \right) \right\|^{2}
\end{multline}
where $2 < \nu \leq \alpha - 1$. When $\nu = \alpha - 1$, the Lyapunov function~\eqref{eqn: continous_generalize_energy_faster} degenerates to~\eqref{eqn: ef_alpha>3}. Furthermore, when  $\nu = \alpha - 1 \rightarrow 2$, the Lyapunov function~\eqref{eqn: continous_generalize_energy_faster} degenerates to~\eqref{eqn: ef_alpha=3}. Finally, when  $2 =\nu = \alpha - 1$ and $\beta = 1$,  the Lyapunov function~\eqref{eqn: continous_generalize_energy_faster} is consistent with~\eqref{eqn: lypunov_NAGM-C_first-order_ode2}. 
 We assume that initial time is
 \[
 t_{\alpha, \beta, \nu} = \max\left\{ \sqrt{s} \left( \beta - \frac{\alpha}{2} \right) , \sqrt{s} \left( \frac{\beta(\alpha - 2)}{\nu - 2}  - \frac{\alpha (\nu - 1)}{2(\nu - 2)}\right), \frac{\sqrt{s} \alpha}{2}  \right\}.
 \]
 Based on the Lyapunov function~\eqref{eqn: continous_generalize_energy_faster}, we have the following results.
 \begin{thm}
\label{thm: faster_convergence}
Let $f(x) \in \mathcal{F}_{L}^{2}(\mathbb{R}^{n})$ and $X = X(t)$ be the solution of the ODE~\eqref{eqn: generalize_NAGM-C_ode} with $\alpha > 3$ and $\beta > 0$. Then, there exists $t_{\alpha, \beta, \nu}>0$ such that
\begin{equation}
\label{eqn: faster_convergence}
\left\{ \begin{aligned}
          & \lim_{t\rightarrow \infty} t^{2} \left (\left( f(X(t)) - f(x^{\star})\right) +  \left\| \dot{X}(t) + \beta \sqrt{s} \nabla f(X(t))\right\|^{2}\right) = \mathfrak{C}^{2}_{\alpha, \beta, \nu} \left\|x_{0} - x^{\star}\right\|^{2} \\
          & \int_{t_{0}}^{t}\left[ u\left( f(X(u)) - f(x^{\star})\right) + u \left\| \dot{X}(u) + \beta \sqrt{s} \nabla f(X(u))\right\|^{2}\right] \dd u < \infty
         \end{aligned} \right.
\end{equation}
for all $t\geq t_{\alpha, \beta, \nu}$,  where the positive constant $\mathfrak{C}^{2}_{\alpha, \beta,\nu}$ and the integer $t_{\alpha, \beta, \nu}$ depend only on $\alpha$, $\beta$ and $\nu$. In other words, the equivalent expression of~\eqref{eqn: faster_convergence} is
\[
 f(X(t)) - f(x^{\star}) +  \left\| \dot{X}(t) + \beta \sqrt{s} \nabla f(X(t))\right\|^{2} \leq o\left( \frac{\left\|x_{0} - x^{\star}\right\|^{2}}{t^{2}}\right).
\]
\end{thm}

 Now, we start to show the proof. Since $X = X(t)$ is the solution of the ODE~\eqref{eqn: generalize_NAGM-C_ode} with $\alpha > 3$ and $\beta > 0$, when $t > t_{\alpha, \beta, \nu}$,
 the time derivative of Lyapunov function~\eqref{eqn: continous_generalize_energy_faster}  is
\begin{align}
\label{eqn: ef_estimate_general_nu}
\frac{\dd \mathcal{E}_{\nu}(t)}{\dd t} & = \left[ 2t + \left(\frac{\alpha}{2} - \beta\right)\sqrt{s} + (\alpha - \nu - 1) \beta \sqrt{s}  \right] \left( f(X(t)) - f(x^{\star}) \right) \nonumber \\
                                                  &\quad + t\left[t + \left(\frac{\alpha}{2} - \beta\right)\sqrt{s} + (\alpha - \nu - 1) \beta \sqrt{s}  \right] \left\langle \nabla f(X(t)), \dot{X}(t) \right\rangle + \nu (\alpha - \nu - 1) \left\langle X(t) - x^{\star}, \dot{X}(t) \right\rangle \nonumber \\
                                                  &\quad- \left\langle (\alpha - 1 - \nu)\dot{X}(t) + \left[t + \left( \frac{\alpha}{2} - \beta \right) \sqrt{s} \right] \nabla f(X(t))  , \nu (X(t) - x^{\star}) + t\left( \dot{X}(t) +\beta \sqrt{s} \nabla f(X(t)) \right) \right\rangle \nonumber \\
                                                  & =  \left[ 2t + \left(\frac{\alpha}{2} - \beta\right)\sqrt{s} + (\alpha - \nu - 1) \beta \sqrt{s} \right] \left( f(X(t)) - f(x^{\star}) \right)  - (\alpha - 1 - \nu)t \left\| \dot{X}(t) \right\|^{2} \nonumber \\
                                                  &\quad  - \nu \left[ t + \left( \frac{\alpha}{2} - \beta \right) \sqrt{s} \right]  \left\langle \nabla f(X(t)), X(t) - x^{\star} \right\rangle  \\
                                                  & \quad - \beta t\sqrt{s}\left[ t + \left( \frac{\alpha}{2} - \beta \right) \sqrt{s} \right]  \left\| \nabla f(X(t)) \right\|^{2}.
\nonumber 
\end{align}
With the basic inequality for any $f(x) \in \mathcal{F}_{L}^{2}(\mathbb{R}^{n})$
\[
f(x^{\star}) \geq f(X(t)) + \left\langle \nabla f(X(t)), x^{\star} - X(t) \right\rangle, 
\]
the time derivative of Lyapunov function~\eqref{eqn: ef_estimate_general_nu} can be estimated as
\begin{multline*}
\frac{\dd \mathcal{E}_{\nu}(t)}{\dd t} \leq - \left\{ (\nu - 2)t + \sqrt{s} \left[\frac{\alpha(\nu - 1)}{2} - (\alpha - 2)\beta\right]\right\}\left( f(X(t)) - f(x^{\star}) \right) \\
- (\alpha - 1 - \nu)t \left\| \dot{X}(t) \right\|^{2}  - \beta t\sqrt{s}\left[t + \left( \frac{\alpha}{2} - \beta \right) \sqrt{s} \right]  \left\| \nabla f(X(t)) \right\|^{2}.
\end{multline*}
With the Lyapunov function $\mathcal{E}_{\nu}(t) \geq 0$ and the technique for integral, for any $t > t_{0}$ we have
\[
\int_{t_{0}}^{t} u (f(X(u)) - f(x^{\star})) du \leq \int_{t_{0}}^{t_{0}+ \delta} u (f(X(u)) - f(x^{\star})) du+ \left( 1 + \frac{t_{0}}{\delta}\right)  \int_{t_{0}+ \delta}^{t}(u - t_{0}) (f(X(u)) - f(x^{\star})) du,
\]
where $\delta < t -t_{0}$. Thus, we can obtain the following Lemma.
\begin{lem}
\label{lem: lem1_ode_faster}
Under the same assumption of Theorem~\ref{thm: faster_convergence}, the following limits exist
\[
\lim_{t \rightarrow \infty} \mathcal{E}_{\nu}(t), \; \lim_{t \rightarrow \infty}\int_{t_{0}}^{t} u(f(X(u)) - f(x^{\star})) \dd u, \; \lim_{t \rightarrow \infty}\int_{t_{0}}^{t} u \left\| \dot{X}(u) \right\|^{2} \dd u, \; \lim_{t\rightarrow \infty}\int_{t_{0}}^{t} u^{2}  \left\| \nabla f(X(u)) \right\|^{2} du.
\]
\end{lem}

With~\eqref{eqn: ef_estimate_general_nu} and Lemma~\ref{lem: lem1_ode_faster}, the following Lemma holds.
\begin{lem}
\label{lem: lem2_ode_faster}
Under the same assumption of Theorem~\ref{thm: faster_convergence}, the following limit exists
\[
\lim_{t \rightarrow \infty}\int_{t_{0}}^{t} u  \left\langle \nabla f(X(u)), X(u) - x^{\star} \right\rangle du.
\]
\end{lem}

\begin{lem}
\label{lem: lem3_ode_faster}
Under the same assumption of Theorem~\ref{thm: faster_convergence}, the following limits exist
\[
\lim_{t \rightarrow \infty} \left\| X(t) - x^{\star}\right\| \quad \text{and} \quad \lim_{t \rightarrow \infty} t \left\langle X(t) - x^{\star}, \dot{X}(t) + \beta\sqrt{s} \nabla f(X(t)) \right\rangle.
\]
\end{lem}
\begin{proof}[Proof of Lemma~\ref{lem: lem3_ode_faster}]
Taking $\nu \neq \nu' \in [2, \gamma - 1]$, we have
\begin{multline*}
\mathcal{E}_{\nu}(t) - \mathcal{E}_{\nu'}(t) = (\nu - \nu') \left[-\beta \sqrt{s} t\left( f(X(t)) - f(x^{\star}) \right)\right. \\ \left.+ t \left\langle X(t) - x^{\star}, \dot{X}(t) + \beta \sqrt{s}\nabla f(X(t))\right\rangle + \frac{\alpha - 1}{2}\left\|X(t) - x^{\star}\right\|^2 \right]
\end{multline*}
With Lemma~\ref{lem: lem1_ode_faster} and~\eqref{eqn:convergence_ode_>3}, the following limit exists
\begin{align}
\label{eqn: exist_limitation1}
\lim_{t \rightarrow \infty} \left[ t \left\langle X(t) - x^{\star}, \dot{X}(t) + \beta \sqrt{s}\nabla f(X(t))\right\rangle + \frac{\alpha - 1}{2}\left\|X(t) - x^{\star}\right\|^2 \right].
\end{align}
Define a new function about time variable $t$: 
\[
\pi(t) := \frac{1}{2} \left\| X(t) - x^{\star} \right\|^{2} + \beta \sqrt{s}\int_{t_{0}}^{t }  \left\langle \nabla f(X(u)), X(u) - x^{\star} \right\rangle du.
\]
If we can prove the existence of the limit $\pi(t)$ with $t \rightarrow \infty$, we can guarantee $\lim\limits_{t \rightarrow \infty} \left\| X(t) - x^{\star}\right\|$ exists with Lemma~\ref{lem: lem2_ode_faster}.  We observe the following equality
\begin{multline*}
    t\dot{\pi}(t) + (\alpha - 1) \pi(t)  \\
=  \beta (\alpha - 1) \sqrt{s}  \int_{t_{0}}^{t } \left\langle \nabla f(X(u)), X(u) - x^{\star} \right\rangle du + t \left\langle X(t) - x^{\star}, \dot{X}(t) +  \beta\sqrt{s}\nabla f(X(t))\right\rangle + \frac{\alpha - 1}{2}\left\|X(t) - x^{\star}\right\|^2 . 
\end{multline*}
With~\eqref{eqn: exist_limitation1}  and Lemma~\ref{lem: lem2_ode_faster}, we obtain that the following limit exists
\[
\lim_{t\rightarrow \infty} \left[ t\dot{\pi}(t) + (\alpha- 1) \pi(t)\right],
\]
that is, there exists some constant $\mathfrak{C}^{3}$ such that the following equality holds,
\[
\lim_{t\rightarrow \infty}\frac{\frac{ \dd (t^{\alpha - 1}\pi(t))}{\dd t}}{t^{\alpha - 2}} = \lim_{t\rightarrow \infty} \left[ t\dot{\pi}(t) + (\alpha - 1) \pi(t)\right] = \mathfrak{C}^{3}.
\]
For any $\epsilon > 0$, there exists $t_{0}>0$ such that when $t \geq t_{0}$, we have
\[
t^{\alpha - 1} \left(\pi(t) - \frac{\mathfrak{C}^{3}}{\alpha - 1} \right)- t_{0}^{\alpha - 1} \left(\pi(t_{0}) - \frac{\mathfrak{C}^{3}}{\alpha - 1} \right) \leq \frac{\epsilon }{\alpha - 1}  \cdot \left(t^{\alpha - 1} - t_{0}^{\alpha - 1}\right)
\]
that is,
\[
\left|\pi(t) - \frac{\mathfrak{C}^3}{\alpha - 1}\right| \leq \left| \pi(t_0) - \frac{\mathfrak{C}^3}{\alpha - 1}\right| \left( \frac{t_{0}}{t}\right)^{\alpha - 1} + \frac{\epsilon}{\alpha - 1}.
\]
The proof is complete.
\end{proof}

Finally, we finish the proof for Theorem~\ref{thm: faster_convergence}.

\begin{proof}[Proof of Theorem~\ref{thm: faster_convergence}]
When $t > t_{\alpha, \beta, \nu}$, we expand the Lyapunov function~\eqref{eqn: continous_generalize_energy_faster} as
\begin{multline*}
\mathcal{E}_{\nu}(t) =  t \left[t + \left(\frac{\alpha}{2} - \beta\right)\sqrt{s} + (\alpha - \nu - 1) \beta \sqrt{s} \right] \left( f(X(t)) - f(x^{\star}) \right) + \frac{\nu (\alpha  - 1)}{2} \left\| X(t) - x^{\star}\right\|^{2}\\
+ \frac{t^{2}}{2} \left\| \dot{X}(t) + \beta \sqrt{s} \nabla f(X(t))  \right\|^{2} +  t \left\langle X(t) - x^{\star}, \dot{X}(t) + \beta \sqrt{s} \nabla f(X(t)) \right\rangle.
\end{multline*}
With Lemma~\ref{lem: lem1_ode_faster} and Lemma~\ref{lem: lem3_ode_faster}, we obtain the first equation of~\eqref{eqn: faster_convergence}. Furthermore, Cauchy-Scharwz inequality tells that
\begin{align*}
       &  \left[t + \left(\frac{\alpha}{2} - \beta\right)\sqrt{s} + (\alpha - \nu - 1) \beta \sqrt{s} \right] \left( f(X(t)) - f(x^{\star}) \right) + \frac{t}{2} \left\| \dot{X}(t) + \beta\sqrt{s} \nabla f(X(t))  \right\|^{2} \\
\leq  & \left[t + \left(\frac{\alpha}{2} - \beta\right)\sqrt{s} + (\alpha - \nu - 1) \beta \sqrt{s} \right] \left( f(X(t)) - f(x^{\star}) \right) + t\left\| \dot{X}(t)   \right\|^{2} + \beta^2 s t \left\| \nabla f(X(t)) \right\|^{2}.
\end{align*}
With Lemma~\ref{lem: lem1_ode_faster}, we obtain the second equation of~\eqref{eqn: faster_convergence}. With basic calculation, we complete the proof.
\end{proof}

\subsubsection{Proof of Theorem~\ref{thm:faster_rate_0}}
\label{subsec: discrete_faster}
Similarly, under the assumption of Theorem~\ref{thm:faster_rate_0}, if we can show a discrete version of~\eqref{eqn: faster_convergence}, that is, there exists some constant $\mathfrak{C}^4_{\alpha, \beta, \nu} > 0$ and $\mathfrak{c}_{\alpha, \beta, \nu}> 0$ such that when the step size satisfies $0 < s \leq \mathfrak{c}_{\alpha, \beta, \nu}/L$, the following relationship holds
\begin{equation}
\label{eqn: faster_convergence_discrete}
\left\{ \begin{aligned}
          & \lim_{k\rightarrow \infty} (k+1)^{2} \left ( f(x_{k}) - f(x^{\star}) +  \left\| v_{k} + \beta \sqrt{s} \nabla f(x_{k})\right\|^{2}\right) = \frac{\mathfrak{C}^4_{\alpha, \beta, \nu} \left\|x_{0} - x^{\star}\right\|^{2}}{s} \\
          & \sum_{k=0}^{\infty}(k + 1) \left( \left( f(x_{k}) - f(x^{\star})\right) +  \left\| v_{k} + \beta \sqrt{s} \nabla f(x_{k})\right\|^{2}\right) < \infty.
         \end{aligned} \right.
\end{equation}
Thus, we obtain the sharper convergence rate as 
\[
f(x_{k}) - f(x^{\star}) +  \left\| v_{k} + \beta \sqrt{s} \nabla f(x_{k})\right\|^{2} \leq o\left( \frac{\left\|x_{0} - x^{\star}\right\|^{2}}{s k^{2}} \right).
\]

Now we show the derivation of the inequality~\eqref{eqn: faster_convergence_discrete}. The discrete Lyapunov function is constructed as
\begin{multline}
\label{eqn: discrete_generalize_energy_faster}
\mathcal{E}(k) =  \underbrace{s( k  + 1) \left[  k + \alpha + 1 - \beta + \frac{( k + 2 )(\alpha - 1 -\nu) \beta}{k+ \alpha + 1}\right] \left( f(x_{k}) - f(x^{\star}) \right)}_{\mathbf{I}}\\
                           + \underbrace{\frac{\nu(\alpha - \nu - 1)}{2} \left\| x_{k+1} - x^{\star}\right\|^{2}}_{\mathbf{II}} + \underbrace{\frac{1}{2} \left\| \nu (x_{k + 1} - x^{\star}) + (k + 1)\sqrt{s} \left( v_{k} + \beta \sqrt{s} \nabla f(x_{k}) \right)  \right\|^{2}}_{\mathbf{III}},
 \end{multline}  
 where $2 \leq \nu < \alpha - 1$ and parts $\mathbf{I}$, $\mathbf{II}$ and $\mathbf{III}$ are potential, Euclidean distance and mixed energy respectively.  Apparently, when $\nu = \alpha - 1$, the discrete Lyapunov function~\eqref{eqn: discrete_generalize_energy_faster} is consistent with~\eqref{eqn: energy_functional_ses_>3}. When $\beta = 1$ and $\nu = \alpha - 1 \rightarrow 2$,  the discrete Lyapunov function~\eqref{eqn: discrete_generalize_energy_faster} degenerates to~\eqref{eqn: NAGM-C_original_lypunov}, 
Now, we turn to estimate the difference of Lyapunov function~\eqref{eqn: discrete_generalize_energy_faster}.  
 \begin{itemize}
 \item For the part $\mathbf{I}$, potential, we have
          \begin{align*}
            &   s ( k  + 2) \left[  k + \alpha + 2 - \beta + \frac{( k + 3 )(\alpha - 1 -\nu)\beta}{k+ \alpha + 2}\right] \left( f(x_{k + 1}) - f(x^{\star}) \right) \\
            &   - s ( k  + 1) \left[  k + \alpha + 1 - \beta + \frac{( k + 2 )(\alpha - 1 -\nu) \beta }{k+ \alpha + 1}\right] \left( f(x_{k}) - f(x^{\star}) \right)\\
          = & s ( k  + 1) \left[  k + \alpha + 1 - \beta + \frac{( k + 2 )(\alpha - 1 -\nu) \beta }{k+ \alpha + 1}\right] \left( f(x_{k + 1}) - f(x_{k}) \right) \\
             &  +s \left(2k + \alpha + 3- \beta\right) \left( f(x_{k + 1}) - f(x^{\star}) \right) \\
             &  + s(k + 2)(\alpha - 1 - \nu)\beta  \left[ \frac{k + 3}{k + \alpha + 2} - \frac{k + 1}{k + \alpha + 1} \right] \left( f(x_{k + 1}) - f(x^{\star}) \right) \\
       \leq &  \underbrace{s ( k  + 1) \left[  k + \alpha + 1 - \beta + \frac{( k + 2 )(\alpha - 1 -\nu) \beta }{k+ \alpha + 1}\right] \left( f(x_{k + 1}) - f(x_{k}) \right) }_{\mathbf{I_{1}}} \\
             & + \underbrace{s \left[2k + \alpha + 3 + (2\alpha - 3 - 2\nu)\beta\right] \left( f(x_{k + 1}) - f(x^{\star}) \right)}_{\mathbf{I}_{2}},
          \end{align*}
where the last inequality follows $k+ \alpha +2 > k+ \alpha +1 > k + 2$.        

\item  For the part $\mathbf{II}$, Euclidean distance, we have
          \begin{align*}
           &\frac{\nu (\alpha - \nu - 1)}{2} \left\| x_{k+2} - x^{\star}\right\|^{2} - \frac{\nu (\alpha - \nu - 1)}{2} \left\| x_{k+1} - x^{\star}\right\|^{2} \\
     =    &  \underbrace{\nu (\alpha - \nu - 1) \left\langle  x_{k + 2} - x_{k + 1}, x_{k + 2} - x^{\star}\right\rangle}_{\mathbf{II}_{1}} \underbrace{- \frac{\nu (\alpha - \nu - 1)}{2} \left\|x_{k + 2} - x_{k + 1} \right\|^{2}}_{\mathbf{II}_{2}}.
          \end{align*}
\item For the part $\mathbf{III}$, mixed energy, with the simple transformation~\eqref{eqn: simple_transform_generalize_NAGM-C_SES>3} for $\alpha > 3$
\begin{align*}
( k + \alpha )\left( v_{k} + \beta\sqrt{s} \nabla f(x_{k}) \right) -  k \left( v_{k - 1} + \beta\sqrt{s} \nabla f(x_{k - 1}) \right) =  - \sqrt{s} \left( k + \gamma   - \gamma \beta \right) \nabla f(x_{k}),
\end{align*}
we have
\begin{align*}
     &   \frac{1}{2} \left\| \nu (x_{k + 2} - x^{\star}) + (k + 2)\sqrt{s} \left( v_{k + 1} + \beta\sqrt{s} \nabla f(x_{k + 1}) \right)  \right\|^{2} \\
     & \qquad \qquad - \frac{1}{2} \left\| \nu (x_{k + 1} - x^{\star}) + (k + 1)\sqrt{s} \left( v_{k} + \beta \sqrt{s} \nabla f(x_{k}) \right)  \right\|^{2} \\
=   &   \left\langle \nu (x_{k +2} - x_{k + 1}) + (k + 2)\sqrt{s} \left( v_{k + 1} + \beta \sqrt{s} \nabla f(x_{k+1})\right) - (k + 1)\sqrt{s}\left( v_{k} + \beta \sqrt{s} \nabla f(x_{k})\right),  \right.\\
                                                         & \quad \qquad \qquad \qquad \qquad \qquad \qquad \qquad \quad    \left. \nu (x_{k+2} - x^{\star}) + (k + 2)\sqrt{s} \left( v_{k + 1} + \beta \sqrt{s} \nabla f(x_{k+1})\right)\right\rangle\\
                                                         & \quad - \frac{1}{2}\left\| \nu (x_{k +2} - x_{k + 1}) + (k + 2)\sqrt{s} \left( v_{k + 1} + \beta\sqrt{s} \nabla f(x_{k+1})\right)- (k + 1)\sqrt{s}\left( v_{k} + \beta\sqrt{s} \nabla f(x_{k})\right) \right\|^{2} \\
=   & - \left\langle   s\left( k + \alpha + 1- \beta\right) \nabla f(x_{k+1}) + (\alpha - 1 - \nu) (x_{k+2} - x_{k+1}),\right.\\
                                                         & \qquad\qquad  \left. \nu (x_{k+2} - x^{\star}) + (k + 2 )(x_{k+2} - x_{k+1}) + \beta s(k + 2) \nabla f(x_{k+1})\right\rangle\\
                                                         &  - \frac{1}{2}\left\| s\left( k + \alpha + 1- \beta\right) \nabla f(x_{k+1}) + (\alpha - 1 - \nu) (x_{k+2} - x_{k+1})\right\|^2    \\
=                                                       &  - \nu (\alpha - \nu - 1) \left\langle  x_{k + 2} - x_{k + 1}, x_{k + 2} - x^{\star}\right\rangle - (k + 2)(\alpha - \nu - 1)\left\| x_{k+2} - x_{k+1} \right\|^{2} \\
                                                         & - \beta s(k+2)(\alpha -1 - \nu) \left\langle \nabla f(x_{k+1}), x_{k+2} - x_{k+1} \right\rangle \\
                                                         &  -  \left\langle    s\left( k + \alpha + 1- \beta\right) \nabla f(x_{k+1}),\right.\\
                                                         & \qquad\quad  \left. \nu (x_{k+1} - x^{\star}) + (k + 2 + \nu )(x_{k+2} - x_{k+1}) + \beta s (k + 2) \nabla f(x_{k+1})\right\rangle\\
                                                         &  - \frac{1}{2}\left\| s\left( k + \alpha + 1- \beta\right) \nabla f(x_{k+1}) \right\|^2  \\
                                                         &  - \left\langle s\left( k + \alpha + 1- \beta\right) \nabla f(x_{k+1}), (\alpha - 1 - \nu) (x_{k + 2} - x_{k + 1})\right\rangle  \\
                                                         &  - \frac{\left( \alpha - 1 - \nu \right)^{2} }{2} \left\| x_{k + 2} - x_{k + 1}\right\|^{2}  \\
 = & \underbrace{- \nu (\alpha - \nu - 1) \left\langle  x_{k + 2} - x_{k + 1}, x_{k + 2} - x^{\star}\right\rangle}_{\mathbf{III}_{1}} \underbrace{- \frac{(2k + \alpha +3 - \nu)(\alpha - \nu - 1)}{2} \left\| x_{k + 2} - x_{k + 1} \right\|^{2}}_{\mathbf{III}_{2}}\\
    & \underbrace{- s(k + \alpha + 1) \left[ k + \alpha + 1-\beta + \frac{(k + 2)(\alpha - 1 - \nu)\beta}{k + \alpha + 1}\right]   \left\langle \nabla f(x_{k + 1}), x_{k + 2} - x_{k + 1} \right\rangle}_{\mathbf{III}_{3}}   \\
    & \underbrace{-   s\nu\left( k + \alpha + 1- \beta\right) \left\langle \nabla f(x_{k + 1}), x_{k + 1} - x^{\star} \right\rangle}_{\mathbf{III}_{4}} \\
    & \underbrace{- \frac{1}{2}s^2\left[  k + \alpha + 1- \beta+ 2(k + 2)\beta \right] \left( k + \alpha + 1- \beta\right)\left\| \nabla f(x_{k + 1}) \right\|^{2} }_{\mathbf{III}_{5}}.                                                                                                 
\end{align*}
\end{itemize}
Apparently, we can observe that
\[
\mathbf{II}_{1} + \mathbf{III}_{1} = 0,
\]
and 
\[
\mathbf{II}_{2} + \mathbf{III}_{2} = - \frac{s(2k + \alpha + 3)(\alpha - \nu - 1)}{2}  \left\| v_{k + 1}\right\|^{2}.
\]
Using the basic inequality for $f(x) \in \mathcal{F}_{L}^{1}(\mathbb{R}^n)$                    
\[
 f(x_{k}) \geq f(x_{k + 1}) + \left\langle \nabla f(x_{k + 1}), x_{k} - x_{k + 1} \right\rangle + \frac{1}{2L} \left\| \nabla f(x_{k + 1}) - \nabla f(x_{k}) \right\|^{2}, 
\]
 we have  
\begin{align*}
\mathbf{I}_{1} + \mathbf{III}_{3} + \mathbf{III}_{5} & = s ( k  + 1) \left[  k + \alpha + 1 - \beta + \frac{( k + 2 )(\alpha - 1 -\nu) \beta }{k+ \alpha + 1}\right] \left( f(x_{k + 1}) - f(x_{k}) \right)\\
                                                                               & \quad  - s(k + \alpha + 1) \left[ k + \alpha + 1-\beta + \frac{(k + 2)(\alpha - 1 - \nu)\beta}{k + \alpha + 1}\right]   \left\langle \nabla f(x_{k + 1}), x_{k + 2} - x_{k + 1} \right\rangle\\
                                                                               & \quad - \frac{1}{2}s^2\left[  k + \alpha + 1- \beta+ 2(k + 2)\beta \right] \left( k + \alpha + 1- \beta\right)\left\| \nabla f(x_{k + 1}) \right\|^{2}\\
                                                                               & \leq - s^{\frac{3}{2}} \left[  k + \alpha + 1 - \beta + \frac{( k + 2 )(\alpha - 1 -\nu) \beta }{k+ \alpha + 1}\right] \left\langle \nabla f(x_{k+1}), (k + \alpha + 1)v_{k+1} - (k+1)v_{k} \right\rangle   \\
                                                                              &   \quad - \frac{s(k + 1)}{2L}  \left[  k + \alpha + 1 - \beta + \frac{( k + 2 )(\alpha - 1 -\nu) \beta }{k+ \alpha + 1}\right] \left\|\nabla f(x_{k + 1}) - \nabla f(x_{k})\right\|^2\\
                                                                              &   \quad- \frac{1}{2}s^2\left[  k + \alpha + 1- \beta+ 2(k + 2)\beta \right] \left( k + \alpha + 1- \beta\right)\left\| \nabla f(x_{k + 1}) \right\|^{2}.                         
\end{align*}
Utilizing~\eqref{eqn: simple_transform_generalize_NAGM-C_SES>3} again, 
we have
\begin{align*}
\mathbf{I}_{1} + \mathbf{III}_{3} + \mathbf{III}_{5}  & \leq   \beta s^{2} ( k  + 1) \left[  k + \alpha + 1 - \beta + \frac{( k + 2 )(\alpha - 1 -\nu) \beta }{k+ \alpha + 1}\right] \left\langle \nabla f(x_{k+1}), \nabla f(x_{k + 1}) - \nabla f(x_{k}) \right\rangle   \\
                                                    &  \quad + s^{2} (k + \alpha + 1) \left[  k + \alpha + 1 - \beta + \frac{( k + 2 )(\alpha - 1 -\nu) \beta }{k+ \alpha + 1}\right] \left\| \nabla f(x_{k + 1}) \right\|^{2} \\
                                                      &   \quad - \frac{s(k + 1)}{2L}  \left[  k + \alpha + 1 - \beta + \frac{( k + 2 )(\alpha - 1 -\nu) \beta }{k+ \alpha + 1}\right] \left\|\nabla f(x_{k + 1}) - \nabla f(x_{k})\right\|^2\\
                                                                              &   \quad- \frac{1}{2}s^2\left[  k + \alpha + 1- \beta+ 2(k + 2)\beta \right] \left( k + \alpha + 1- \beta\right)\left\| \nabla f(x_{k + 1}) \right\|^{2}\\            
                                                    & \leq \frac{L\beta^2s^2}{2}  ( k  + 1) \left[  k + \alpha + 1 - \beta + \frac{( k + 2 )(\alpha - 1 -\nu) \beta }{k+ \alpha + 1}\right] \left\| \nabla f(x_{k + 1}) \right\|^{2} \\   
                                                     &  \quad + s^{2} (k + \alpha + 1) \left[  k + \alpha + 1 - \beta + \frac{( k + 2 )(\alpha - 1 -\nu) \beta }{k+ \alpha + 1}\right] \left\| \nabla f(x_{k + 1}) \right\|^{2} \\
                                                     &   \quad- \frac{1}{2}s^2\left[  k + \alpha + 1- \beta+ 2(k + 2)\beta \right] \left( k + \alpha + 1- \beta\right)\left\| \nabla f(x_{k + 1}) \right\|^{2}\\
                                                     & = s^2\left[  \frac{L\beta^2 s}{2}  (k + 1) + (k + \alpha + 1)\right]  \left[  k + \alpha + 1 - \beta + \frac{( k + 2 )(\alpha - 1 -\nu) \beta }{k+ \alpha + 1}\right] \left\| \nabla f(x_{k + 1}) \right\|^{2} \\
                                                      &   \quad- \frac{1}{2}s^2\left[  (2\beta + 1)k + \alpha + 1+3 \beta \right] \left( k + \alpha + 1- \beta\right)\left\| \nabla f(x_{k + 1}) \right\|^{2}\\
                                                       & \leq s^2\left[  \frac{L\beta^2 s}{2}  (k + 1) + (k + \alpha + 1)\right]  \left[  k + \alpha + 1 - \beta + (\alpha - 1 -\nu) \beta \right] \left\| \nabla f(x_{k + 1}) \right\|^{2} \\
                                                      &   \quad- \frac{1}{2}s^2\left[  (2\beta + 1)k + \alpha + 1+3 \beta \right] \left( k + \alpha + 1- \beta\right)\left\| \nabla f(x_{k + 1}) \right\|^{2}
\end{align*}
Since $\beta > 1/2$, let $n \in \mathbb{N}^{+}$ satisfy
\[
n = \left\lfloor \frac{2}{2\beta - 1}\right\rfloor + 1.
\]
When $k \geq n (\alpha - 1 - \nu) \beta - (\alpha + 1 - \beta)$, we have
\begin{align*}
\mathbf{I}_{1} + \mathbf{III}_{3} + \mathbf{III}_{5}   & \leq s^2\left[  \frac{L\beta^2 s}{2}  (k + 1) + (k + \alpha + 1)\right]  \left[  k + \alpha + 1 - \beta + (\alpha - 1 -\nu) \beta \right] \left\| \nabla f(x_{k + 1}) \right\|^{2} \\
                                                      &   \quad- \frac{s^2 n}{2(n + 1)} \cdot \left[  (2\beta + 1)k + \alpha + 1+3 \beta \right]  \left[  k + \alpha + 1 - \beta + (\alpha - 1 -\nu) \beta \right]  \left\| \nabla f(x_{k + 1}) \right\|^{2} 
\end{align*}

With the monotonicity of the following function about $k$
\begin{align*}
h(k) = &\frac{\left(\frac{n(2 \beta + 1)}{2(n + 1)} -1\right)k + \frac{n}{2 (n + 1)} \cdot (\alpha + 1 + 3\beta) - \alpha - 1}{\frac{L \beta^2 (k + 1)}{2}} \\= &\frac{(2\beta n- n -2)(k + 1)+ (\beta - \alpha )n - 2\alpha  }{L\beta^{2} (n + 1)(k + 1)},
\end{align*}
we know there exists some constant $\mathfrak{c}_{\alpha, \beta, \nu}$ and $k_{1,\alpha, \beta, \nu}$ such that the step size satisfies $0 < s \leq \mathfrak{c}_{\alpha, \beta, \nu}/L$.
When $k \geq k_{1,\alpha, \beta, \nu}$, the following inequality holds
\[
\mathbf{I}_{1} + \mathbf{III}_{3} + \mathbf{III}_{5}  \leq -  \frac{s^{2}}{2} \left( \frac{2 \beta n}{n + 1}  - \frac{n + 2}{n + 1}- L\beta^2s\right) (k - k_{1,\alpha, \beta, \nu})^{2}\left\| \nabla f(x_{k + 1}) \right\|^{2}.
\]

With the basic inequality for $f(x) \in \mathcal{F}_{L}^{1}(\mathbb{R}^{n})$,
\[
f(x^{\star}) \geq f(x_{k + 1}) + \left\langle \nabla f(x_{k + 1}), x^{\star} - x_{k + 1} \right\rangle,
\]
we know that there exists $ k_{2,\alpha,\beta,\nu}$ such that when $k \geq k_{2,\alpha,\beta,\nu}$,
\[
\mathbf{I}_{2} + \mathbf{III}_{4} \leq - s (\nu - 2) (k - k_{2,\alpha,\beta,\nu})  \left\langle \nabla f(x_{k + 1}), x_{k + 1} - x^{\star} \right\rangle.
\] 

Let $k_{\alpha,\beta,\nu} = \max \{k_{1,\alpha,\beta,\nu}, k_{2,\alpha,\beta,\nu}\} + 1$. Summing up all the estimates above, when $\beta > 1/2$, the difference of discrete Lyapunov function, for any $k \geq k_{\alpha,\beta,\nu}$,
\begin{align*}
\mathcal{E}(k + 1) - \mathcal{E}(k)    & \leq  - \frac{s^{2} }{2} \left( \frac{2 \beta n}{n + 1}  - \frac{n + 2}{n + 1}- L\beta^2s\right)(k - k_{\alpha, \beta, \nu})^{2}\left\| \nabla f(x_{k + 1}) \right\|^{2} \\
                                                           &\quad - s (\nu - 2) (k - k_{\alpha,\beta,\nu})  \left\langle \nabla f(x_{k + 1}), x_{k + 1} - x^{\star} \right\rangle \\
                                                           & \quad - \frac{s(2k + \alpha + 3)(\alpha - \nu - 1)}{2}  \left\| v_{k + 1}\right\|^{2}.
\end{align*}
With the basic inequality for any function $f(x) \in \mathcal{F}_{L}^{1}(\mathbb{R}^{n})$
\[
 \left\langle \nabla f(x_{k + 1}), x_{k + 1} - x^{\star} \right\rangle \geq f(x_{k + 1}) - f(x^{\star}),
\]
we can obtain the following lemma. 
\begin{lem}
\label{lem: lem1_discrete_faster}
Under the same assumption of Theorem~\ref{thm:faster_rate_0}, the following limit exists
\[
\lim_{k \rightarrow \infty} \mathcal{E}(k)
\]
and the summation of the following series exist 
\begin{align*}
& \sum_{k = 0}^{\infty} (k + 1)^2\left\| \nabla f(x_{k + 1}) \right\|^{2}, && \sum_{k = 0}^{\infty} (k + 1)\left\langle \nabla f(x_{k + 1}), x_{k + 1} - x^{\star} \right\rangle, \\
& \sum_{k =0}^{\infty} (k + 1)(f(x_{k +1}) - f(x^{\star})), && \sum_{k =0}^{\infty}(k + 1)\left\| v_{k + 1}\right\|^{2}.
\end{align*}
\end{lem}

\begin{lem}
\label{lem: lem2_discrete_faster}
Under the same assumption of Theorem~\ref{thm:faster_rate_0}, the following limits exist
\[
\lim_{k \rightarrow \infty} \left\| x_{k} - x^{\star} \right\|\quad \mathbf{and} \quad \lim_{k \rightarrow \infty} (k + 1)\left\langle x_{k + 1} -x^{\star}, v_{k} + \beta \sqrt{s} \nabla f(x_{k})\right\rangle.
\]
\end{lem}

\begin{proof}[Proof of Lemma~\ref{lem: lem2_discrete_faster}]
Taking $\nu \neq \nu' \in (2, \gamma - 1]$, we have
\begin{align*}
\mathcal{E}_{\nu}(k) - \mathcal{E}_{\nu'}(k) & = (\nu - \nu') \left[ - s \beta \cdot \frac{(k + 1)(k + 2) }{k + \alpha + 1} \left( f(x_{k}) - f(x^{\star}) \right)\right. \\
                                                                     & \qquad \qquad  \quad \left.+ (k + 1)\sqrt{s}\left\langle x_{k + 1} - x^{\star}, v_{k} + \beta \sqrt{s} \nabla f(x_{k})\right\rangle+  \frac{ (\alpha - 1)}{2} \left\| x_{k + 1} - x^{\star} \right\|^{2} \right]
\end{align*}
With Lemma~\ref{lem: lem1_discrete_faster}, the following limit exists 
\begin{equation}
\label{eqn: discrete_general_ok^2}
\lim_{k \rightarrow \infty}\left[ (k + 1)\sqrt{s} \left\langle x_{k + 1} -x^{\star}, v_{k} + \beta \sqrt{s} \nabla f(x_{k})\right\rangle  + \frac{\alpha - 1}{2} \left\| x_{k + 1} - x^{\star} \right\|^{2}\right].
\end{equation}
Define a new function about $k$:
\begin{align*}
\pi(k) := \frac{1}{2} \left\| x_{k } - x^{\star} \right\|^{2} + \beta s \sum_{i = k_{0}}^{k - 1} \left\langle \nabla f(x_{i}), x_{i + 1} - x^{\star} \right\rangle.
\end{align*}
If we can show the existence of the limit $\pi(k)$ with $k \rightarrow \infty$, we can guarantee $\lim\limits_{k \rightarrow \infty} \left\| x_{k + 1} - x^{\star} \right\|$ exists with Lemma~\ref{lem: lem1_discrete_faster}. We observe the following equality
\begin{align*}
    & (k + 1) (\pi (k +1) - \pi(k)) + (\alpha - 1)\pi(k + 1) - s\left( \alpha - 1 \right)\beta \sum_{i = 0}^{k} \left\langle \nabla f(x_{i}), x_{i + 1} - x^{\star} \right\rangle \\
= & (k + 1) \left\langle x_{k + 1} - x_{k}, x_{k + 1} - x^{\star}\right\rangle - \frac{(k+1)s}{2} \left\| v_{k}\right\|^{2} + \frac{\alpha - 1}{2} \left\| x_{k + 1} - x^{\star} \right\|^{2} + s (k + 1)\beta \left\langle \nabla f(x_{k }), x_{k + 1} - x^{\star} \right\rangle \\
= & (k + 1)\sqrt{s} \left\langle x_{k + 1} -x^{\star}, v_{k} + \beta \sqrt{s} \nabla f(x_{k})\right\rangle - \frac{(k+1)s}{2} \left\| v_{k}\right\|^{2} + \frac{\alpha - 1}{2} \left\| x_{k + 1} - x^{\star} \right\|^{2}.
\end{align*}
Lemma~\ref{lem: lem1_discrete_faster} and~\eqref{eqn: discrete_general_ok^2} tell us there exists some constant $\mathfrak{C}^5$ such that
\[
\lim_{k \rightarrow \infty} \left[ (k + \alpha) \pi(k + 1) - (k + 1) \pi(k) \right] = \mathfrak{C}^5,
\]
that is, taking a simple translation $ \pi'(k) = \pi(k) - \mathfrak{C}^5/(\gamma - 1)$, we have
\[
\lim_{k \rightarrow \infty} \left[ (k + \alpha) \pi'(k + 1) - (k + 1) \pi'(k) \right] =0.
\]
Since $\mathcal{E}(k)$ decreases for $k \geq k_{\alpha, \beta, \nu}$, thus, $ \left\| x_{k } - x^{\star} \right\|^{2}$ is bounded. With  Lemma~\ref{lem: lem1_discrete_faster}, we obtain that $\pi(k)$ is bounded, that is, $\pi'(k)$ is bounded. Then we have
\[
\lim_{k \rightarrow \infty}\frac{(k+2)^{\alpha - 1} \pi'(k+1) - (k + 1)^{\alpha - 1} \pi'(k)}{ (k + 1)^{\alpha - 2}} = 0,
\]
that is, for any $\epsilon>0$, there exists $k'_{0} > 0$ such that 
\[
 \left| \pi'(k) \right| \leq \left( \frac{k'_{0} + 1}{k + 1}\right)^{\alpha - 1} \left| \pi'(k'_{0}) \right| + \frac{\epsilon \sum\limits_{i = k'_{0}}^{k-1} (i + 1)^{\alpha - 2} }{(k + 1)^{\alpha - 1}} .
\]
With arbitrary $\epsilon > 0$, we complete the proof of Lemma~\ref{lem: lem2_discrete_faster}.
\end{proof}
\begin{proof}[Proof of~\eqref{eqn: faster_convergence_discrete}]
When $k \geq k_{\alpha, \beta, \nu}$, we expand the discrete Lyapunov function~\eqref{eqn: discrete_generalize_energy_faster} as
\begin{multline*}
\mathcal{E}(k) =  s( k  + 1)  \left[  k + \alpha + 1 - \beta + \frac{( k + 2 )(\alpha - 1 -\nu) \beta}{k+ \alpha + 1}\right] \left( f(x_{k}) - f(x^{\star}) \right)\\
                           +\sqrt{s} (k + 1)\nu\left\langle x_{k + 1} - x^{\star}, v_{k}   + \beta \sqrt{s} \nabla f(x_{k}) \right\rangle\\ + \frac{\nu(\alpha  - 1)}{2} \left\| x_{k+1} - x^{\star}\right\|^{2}+ \frac{s(k  +1)^{2}}{2} \left\| v_{k} + \beta \sqrt{s} \nabla f(x_{k})  \right\|^{2}.
 \end{multline*}  
 With Lemma~\ref{lem: lem1_discrete_faster} and Lemma~\ref{lem: lem2_discrete_faster}, we obtain the first equation of~\eqref{eqn: faster_convergence_discrete}. Additionally, we have
 \begin{align*}
         & s  \left[  k + \alpha + 1 - \beta + \frac{( k + 2 )(\alpha - 1 -\nu) \beta}{k+ \alpha + 1}\right] \left( f(x_{k}) - f(x^{\star}) \right) + \frac{(k  +1)s}{2} \left\| v_{k} + \beta \sqrt{s} \nabla f(x_{k})  \right\|^{2} \\
 \leq & s   \left[  k + \alpha + 1 - \beta + \frac{( k + 2 )(\alpha - 1 -\nu) \beta}{k+ \alpha + 1}\right] \left( f(x_{k}) - f(x^{\star}) \right)  + (k + 1)s\left\| v_{k}   \right\|^{2}  + (k+1)\beta^{2}s^{2} \left\|\nabla f(x_{k})  \right\|^{2}.
 \end{align*}
 With Lemma~\ref{lem: lem1_discrete_faster}, we obtain the second equation of~\eqref{eqn: faster_convergence_discrete}.  
\end{proof}

\end{document}